\documentclass[a4paper]{amsart}

\usepackage{geometry}
\usepackage{amsmath}
\usepackage{amsfonts}
\usepackage{amssymb}
\usepackage{mathtools}
\usepackage{colonequals}
\usepackage{stmaryrd}
\usepackage{tabularx}
\usepackage{ltxtable}
\usepackage{xcolor}
\usepackage{colortbl}
\usepackage{float} %For option [H] for floating objects (also available in "algorithm")
\usepackage{url}
\usepackage{tikz}
\usepackage{overpic}
\usepackage{calc}
\usetikzlibrary{calc}

\newtheorem{theorem}{Theorem}[section]
\newtheorem{lemma}[theorem]{Lemma}

\newtheorem{corollary}[theorem]{Corollary}
\newtheorem{conjecture}[theorem]{Conjecture}
\theoremstyle{definition}

\theoremstyle{remark}

\newtheorem{example}[theorem]{Example}

\numberwithin{equation}{section}

\newlength{\tabcolseptemp}

\newcounter{theoremitemcounter}

\newenvironment{theoremtable}[1][rX]{
\setcounter{theoremitemcounter}{0}%
\noindent\ignorespaces%
\setlength{\tabcolseptemp}{\the\tabcolsep}%
\setlength{\tabcolsep}{2pt}%
\\[-0.5\baselineskip]%
\tabularx{\textwidth}{@{}#1@{}}%
}{
\endtabularx%
\\[-0.75\baselineskip]\vphantom{a}%
\setlength{\tabcolsep}{\tabcolseptemp}%
\par\noindent%
\ignorespacesafterend%
}

\newcommand{\theoremitem}[1]{
{\textup{#1}}%
\renewcommand{\thetheoremitemcounter}{{#1}}%Evil hack
\refstepcounter{theoremitemcounter}%
}
\newcommand{\proofitem}[2][\quad\qquad ]{{\rlap{#2:}#1}}

\makeatletter %changes the catcode of @ to 11
\newcommand{\push}[1]{\ifmeasuring@#1\else\omit$\displaystyle#1$\ignorespaces\fi}

\newcommand{\pushleft}[1]{\push{#1\hfill}}
\makeatother %changes the catcode of @ back to 12

\newcommand{\ignore}[1]{}

\newcommand{\set}[1]{\left\{#1\right\}}
\newcommand{\setl}[1]{\left\{#1\right.\mathclose{}\kern-\nulldelimiterspace}
\newcommand{\setr}[1]{\kern-\nulldelimiterspace\mathopen{}\left.#1\right\}}
\newcommand{\intint}[2]{{\llbracket#1,#2\rrbracket}}
\newcommand{\intintuptoex}[1]{{\underline{\smash{#1}}\vphantom{#1}}}
\newcommand{\intintuptoin}[1]{{\overline{#1}}}

\newcommand{\ex}{\exists\:}

\newcommand{\fa}{\forall\:}
\newcommand{\ce}{\:\colonequals\:}
\newcommand{\rce}{\:\equalscolon\:}

\newcommand{\abs}[1]{\left\lvert#1\right\rvert}

\newcommand{\floor}[1]{\left\lfloor#1\right\rfloor}
\newcommand{\ceiling}[1]{\left\lceil#1\right\rceil}

\newcommand{\powerset}{\mathcal{P}}
\newcommand{\id}{\operatorname{id}}

\newcommand{\modulo}{\%}
\newcommand{\modulus}[1]{{\mod{#1}}}

\newcommand{\Gcd}[1]{\operatorname{gcd}\left(#1\right)}

\newcommand{\Lcm}[1]{\operatorname{lcm}\left(#1\right)}
\newcommand{\Sgn}[1]{\operatorname{sgn}\left(#1\right)}

\newcommand{\I}{\mathrm{i}}
\renewcommand{\P}{\mathbb{P}}
\newcommand{\N}{\mathbb{N}}
\newcommand{\Z}{\mathbb{Z}}
\newcommand{\Q}{\mathbb{Q}}
\newcommand{\R}{\mathbb{R}}

\newcommand{\Nz}{{\mathbb{N}_0}}

\newcommand{\CoSequences}{{\mathbf S}}
\newcommand{\Sf}[1]{{#1}}
\newcommand{\SFinitial}[1]{{\mathcal{I}\left(#1\right)}}
\newcommand{\SFperiodic}[1]{{\mathcal{P}\left(#1\right)}}

\newcommand{\SFposition}[2]{{\operatorname{pos}\left(#1,#2\right)}}
\newcommand{\SFcount}[2]{{\operatorname{cnt}\left(#1,#2\right)}}
\newcommand{\SPlength}[1]{{\operatorname{len}_{#1}}}
\newcommand{\SPfinite}{{\operatorname{fin}}}
\newcommand{\SPempty}{{\operatorname{emp}}}
\newcommand{\SPboundedby}[1]{{\operatorname{bnd}_{#1}}}
\newcommand{\SPprefix}[1]{{\operatorname{pre}_{#1}}}
\newcommand{\SPsuffix}[1]{{\operatorname{suf}_{#1}}}
\newcommand{\SPperiodic}{{\operatorname{per}}}
\newcommand{\SPultimatelyperiodic}{{\operatorname{uper}}}
\newcommand{\SPaperiodic}{{\operatorname{aper}}}
\newcommand{\SPweakblock}[2]{{\operatorname{w-block}_{#1,#2}}}
\newcommand{\SPblock}[2]{{\operatorname{block}_{#1,#2}}}

\newcommand{\CoSequenceTables}{{\mathcal{S}}}
\newcommand{\STf}[1]{{\mathrm{#1}}}
\newcommand{\STFdomain}{{\operatorname{dom}}}
\newcommand{\STPdomain}[1]{{\operatorname{dom}_{#1}}}
\newcommand{\STPlength}[1]{{\operatorname{len}_{#1}}}
\newcommand{\STPfinite}{{\operatorname{fin}}}
\newcommand{\STPempty}{{\operatorname{emp}}}
\newcommand{\STPboundedby}[1]{{\operatorname{bnd}_{#1}}}

\newcommand{\SoDigitTables}[1]{{\mathcal{D}_{#1}}}
\newcommand{\DTPweakblockat}[1]{{\operatorname{w-block}_{#1}}}
\newcommand{\DTPweakblock}{{\operatorname{w-block}}}
\newcommand{\DTPblockat}[1]{{\operatorname{block}_{#1}}}
\newcommand{\DTPblock}{{\operatorname{block}}}
\newcommand{\DTPweakblockFat}[1]{{\operatorname{w-block-F}_{#1}}}
\newcommand{\DTPweakblockF}{{\operatorname{w-block-F}}}
\newcommand{\DTPweakblockSat}[1]{{\operatorname{w-block-S}_{#1}}}
\newcommand{\DTPweakblockS}{{\operatorname{w-block-S}}}

\newcommand{\SoFibredFunctions}[1]{{\mathcal{F}_{#1}}}
\newcommand{\FFf}[1]{{\mathrm{#1}}}
\newcommand{\FFFdomain}{{\operatorname{dom}}}
\newcommand{\FFFST}{{\STf{S}}}
\newcommand{\FFFDT}{{\STf{D}}}
\newcommand{\FFFSoPP}{{\operatorname{PerP}}}
\newcommand{\FFFSoUPP}{{\operatorname{UPerP}}}
\newcommand{\FFFSoAPP}{{\operatorname{APerP}}}
\newcommand{\FFFSoGenPP}{{\operatorname{PerP-Gen}}}
\newcommand{\FFFSoGenUPP}{{\operatorname{UPerP-Gen}}}
\newcommand{\FFFSoGenAPP}{{\operatorname{APerP-Gen}}}
\newcommand{\FFPweakcanonicalform}{{\operatorname{w-canf}}}
\newcommand{\FFPcanonicalform}{{\operatorname{canf}}}
\newcommand{\FFPdomain}[1]{{\operatorname{dom}_{#1}}}
\newcommand{\FFPboundedby}[1]{{\operatorname{bnd}_{#1}}}
\newcommand{\FFPclosed}{{\operatorname{closed}}}
\newcommand{\FFPweakblockat}[1]{{\operatorname{w-block}_{#1}}}
\newcommand{\FFPweakblock}{{\operatorname{w-block}}}
\newcommand{\FFPblockat}[1]{{\operatorname{block}_{#1}}}
\newcommand{\FFPblock}{{\operatorname{block}}}
\newcommand{\FFPweakblockFat}[1]{{\operatorname{w-block-F}_{#1}}}
\newcommand{\FFPweakblockF}{{\operatorname{w-block-F}}}
\newcommand{\FFPweakblockSat}[1]{{\operatorname{w-block-S}_{#1}}}
\newcommand{\FFPweakblockS}{{\operatorname{w-block-S}}}
\newcommand{\FFPavoiding}{{\operatorname{avoid}}}
\newcommand{\FFPpolynomialcoefficientsdegree}[2]{{\operatorname{poly}_{#1,#2}}}
\newcommand{\FFPpolynomialcoefficients}[1]{{\operatorname{poly}_{#1}}}
\newcommand{\FFPpolynomial}{{\operatorname{poly}}}
\newcommand{\FFPlinearpolynomialcoefficients}[1]{{\operatorname{lin-poly}_{#1}}}
\newcommand{\FFPlinearpolynomial}{{\operatorname{lin-poly}}}
\newcommand{\FFPcontractive}{{\operatorname{contr}}}
\newcommand{\FFPexpansive}{{\operatorname{exp}}}
\newcommand{\FFPmixed}{{\operatorname{mix}}}
\newcommand{\FFPcontractsdenominators}{{\operatorname{d-contr}}}
\newcommand{\FFPexpandsdenominators}{{\operatorname{d-exp}}}
\newcommand{\FFPmixesdenominators}{{\operatorname{d-mix}}}
\newcommand{\FFPsatisfies}[1]{{\left[#1\right]}}
\newcommand{\FFPperiodicon}[1]{{\operatorname{per-on}_{#1}}}
\newcommand{\FFPultimatelyperiodicon}[1]{{\operatorname{uper-on}_{#1}}}
\newcommand{\FFPaperiodicon}[1]{{\operatorname{aper-on}_{#1}}}

\newcommand{\SoFibredRationalFunctions}[1]{{\mathcal{R}_{#1}}}
\newcommand{\FRFf}[1]{{\mathrm{#1}}}
\newcommand{\FRFFdomain}{{\operatorname{dom}}}
\newcommand{\FRFF}[2]{{#1_{#2}}}
\newcommand{\FRFFST}[1]{{\STf{S}_{#1}}}
\newcommand{\FRFlinA}[2]{{A_{#1}(#2)}}
\newcommand{\FRFlinB}[2]{{B_{#1}(#2)}}
\newcommand{\FRFFint}{{\operatorname{int}}}
\newcommand{\FRFPdomain}[1]{{\operatorname{dom}_{#1}}}
\newcommand{\FRFPboundedby}[1]{{\operatorname{bnd}_{#1}}}
\newcommand{\FRFPclosed}{{\operatorname{closed}}}
\newcommand{\FRFPintegral}{{\operatorname{integral}}}
\newcommand{\FRFPavoiding}{{\operatorname{avoid}}}
\newcommand{\FRFPpolynomialcoefficientsdegree}[2]{{\operatorname{poly}_{#1,#2}}}
\newcommand{\FRFPpolynomialcoefficients}[1]{{\operatorname{poly}_{#1}}}
\newcommand{\FRFPpolynomial}{{\operatorname{poly}}}
\newcommand{\FRFPlinearpolynomialcoefficients}[1]{{\operatorname{lin-poly}_{#1}}}
\newcommand{\FRFPlinearpolynomial}{{\operatorname{lin-poly}}}

\newcommand{\FPweaklysuitableat}[3]{{\operatorname{w-suit}_{#1,#2,#3}}}
\newcommand{\FPweaklysuitable}[2]{{\operatorname{w-suit}_{#1,#2}}}
\newcommand{\FPsuitableat}[3]{{\operatorname{suit}_{#1,#2,#3}}}
\newcommand{\FPsuitable}[2]{{\operatorname{suit}_{#1,#2}}}
\newcommand{\FPavoiding}[2]{{\operatorname{avoid}_{#1,#2}}}
\newcommand{\FPintegral}[2]{{\operatorname{integral}_{#1,#2}}}

\newcommand{\SoSystems}[1]{{\overline{\mathcal{F}}_{#1}}}
\newcommand{\SoFunctions}[1]{{\overline{\mathcal{Z}}_{#1}}}
\newcommand{\SoTables}[1]{{\overline{\mathcal{D}}_{#1}}}
\newcommand{\SoPermutations}[1]{{\overline{\mathcal{P}}_{#1}}}

\newcommand{\vl}{{\mathrlap{\smash{\,\rule[-1pt]{0.4pt}{7pt}}}}}

\makeatletter
\def\myparagraph{\@startsection{paragraph}{4}%
  \z@\z@{-\fontdimen2\font}%
  {\normalfont\bfseries\\[-0.5\baselineskip]}}
\makeatother

\newcommand{\myparagraphtoc}[1]{
\myparagraph{#1}
\addcontentsline{toc}{subsection}{\makebox[\widthof{\thesection.}]{}\quad\;\;--\;#1}
}

\title[An introduction to $p$-adic systems]{An introduction to $p$-adic systems: A new kind of number systems inspired by the Collatz $3n+1$ conjecture}

\author[M.~Weitzer]{Mario~Weitzer}
\address{Institute of Analysis and Number Theory,
Graz University of Technology, Steyrergasse 30/II, 8010 Graz, AUSTRIA}
\email{weitzer@math.tugraz.at}

\subjclass[2010]{Primary: 11S82, Secondary: 11A63, 11C08}
\keywords{$p$-adic systems, number systems, permutation polynomials, Collatz conjecture}

\begin{document}

\begin{abstract}
This article introduces a new kind of number systems on $p$-adic integers which is inspired by the well-known $3n+1$ conjecture of Lothar Collatz. A $p$-adic system is a piecewise function on $\mathbb{Z}_p$ which has branches for all residue classes modulo $p$ and whose dynamics can be used to define digit expansions of $p$-adic integers which respect congruency modulo powers of $p$ and admit a distinctive ``block structure''. $p$-adic systems generalize several notions related to $p$-adic integers such as permutation polynomials and put them under a common framework, allowing for results and techniques formulated in one setting to be transferred to another. The general framework established by $p$-adic systems also provides more natural versions of the original Collatz conjecture and first results could be achieved in the context. A detailed formal introduction to $p$-adic systems and their different interpretations is given. Several classes of $p$-adic systems defined by different types of functions such as polynomial functions or rational functions are characterized and a group structure on the set of all $p$-adic systems is established, which altogether provides a variety of concrete examples of $p$-adic systems. Furthermore, $p$-adic systems are used to generalize Hensel's Lemma on polynomials to general functions on $\mathbb{Z}_p$, analyze the original Collatz conjecture in the context of other ``linear-polynomial $p$-adic systems'', and to study the relation between ``polynomial $p$-adic systems'' and permutation polynomials with the aid of ``trees of cycles'' which encode the cycle structure of certain permutations of $\mathbb{Z}_p$. To outline a potential roadmap for future investigations of $p$-adic systems in many different directions, several open questions and problems in relation to $p$-adic systems are listed.
\end{abstract}

\maketitle

\tableofcontents

\newpage

\section{Introduction and motivation}
\label{SIntroduction}
The aim of this paper is to introduce a new kind of number systems on the $p$-adic integers denoted by \emph{$p$-adic systems} and to derive first non-trivial results. It is the author's hope that $p$-adic systems will provide a useful framework to describe several deeply mysterious phenomena (a central one being the famous Collatz conjecture or $3n+1$-problem) which will allow to express related results thus far contained in conceptionally and notationally isolated papers in terms of a common language, but also to gain entirely new results and insights. While problems such as the Collatz conjecture may remain a distant goal at the horizon, embedding them into a set of related yet more accessible questions will, hopefully, indicate a path in the right direction.

In order to motivate the definition of $p$-adic systems, we will start by repeating the statement of the $3n+1$-problem formulated by the German mathematician Lothar Collatz in 1937. For that we define the transformation
\begin{align}
\label{ECollatzZ}
\FFf{F}_C:\N&\to\N.\\
\nonumber
n&\mapsto
\begin{cases}
\frac{n}{2}&\text{if }n\equiv0\modulus{2}\\
\frac{3n+1}{2}&\text{if }n\equiv1\modulus{2}
\end{cases}
\end{align}
Applying $\FFf{F}_C$ repeatedly, one finds that the orbit of any natural number up to $2^{60}$ \cite{Roosendaal:2019,Silva:2019} eventually reaches $1$ where it enters the $2$-cycle $(1,2)$. Remarkably, despite the extremely simple formulation and high popularity of the problem (for an extensive overview of related work see \cite{Lagarias:1985,Lagarias:2003,Lagarias:2006,Wirsching:1998} which can all be found in the book \cite{Lagarias:2010}, and \cite{Wirsching:1998}), it has remained unproven for more than $80$ years and it appears that we are no closer to a solution than Lothar Collatz was when he first found it. To continue we consider the slightly modified transformation
\begin{align}
\FFf{F}_2:\Nz&\to\Nz.\\
\nonumber
n&\mapsto
\begin{cases}
\frac{n}{2}&\text{if }n\equiv0\modulus{2}\\
\frac{n-1}{2}&\text{if }n\equiv1\modulus{2}
\end{cases}
\end{align}
It is of course the transformation which can be used to compute the standard binary expansion of a natural number by taking its orbit modulo $2$. Obviously, the orbit of any natural number under $\FFf{F}_2$ eventually enters the $1$-cycle $(0)$. By simply replacing the linear polynomial $3x+1$ by another linear polynomial ($x-1$) the question for the ultimate behavior of the corresponding transformation changes from extremely hard to trivial. Yet, there is one decisive property which the orbits produced by both transformations have in common, i.e. that they can be used to define a ``number system'' on the $p$-adic integers that satisfies a rather natural condition. It is this property which will be central to the definition of $p$-adic systems. Clearly, the definitions of both $\FFf{F}_C$ and $\FFf{F}_2$ naturally extend to the $2$-adic integers (for an introduction to $p$-adic integers see e.g. \cite{Mahler:1981,Gouvea:1997,Katok:2007}) with the only parts to change being the domains and codomains:
\begin{align}
\label{ECollatzZtwo}
\FFf{F}_C:\Z_2&\to\Z_2,&\FFf{F}_2:\Z_2&\to\Z_2.\\
\nonumber
n&\mapsto
\begin{cases}
\frac{n}{2}&\text{if }n\equiv0\modulus{2}\\
\frac{3n+1}{2}&\text{if }n\equiv1\modulus{2}
\end{cases}
&
n&\mapsto
\begin{cases}
\frac{n}{2}&\text{if }n\equiv0\modulus{2}\\
\frac{n-1}{2}&\text{if }n\equiv1\modulus{2}
\end{cases}
\end{align}
The term ``number system'' above is put in quotation marks, as there is no strict definition of what a number system ``of'' or ``on'' some set $X$ actually is. To the author's mind it is something that can be used to give unique ``names'' (in our case infinite strings over a finite alphabet) to all elements of said set $X$ which, ideally, are not chosen at random but follow certain rules and encode some information on the represented elements. Examples would be the usual binary or decimal representations of natural numbers which come with easy algorithms that allow fast addition and multiplication of the represented numbers but, somewhat mysteriously, are of no good use when trying to obtain the factors of, say, the product of two large primes. The information on the factors is encoded in the digits of the product, but cannot easily be extracted. Alternatively, a natural number can be represented by giving a list of its prime factors which could also be considered a number system on $\N$. In this setting multiplication and factorization are straight forward, but in return addition is just as hard as factorization is in the other setting. That there appears to be no number system allowing for fast addition, multiplication and factorization all at once, is a phenomenon at the heart of many of the biggest open problems in mathematics today. The idea of defining number systems on arbitrary sets $X$ by iterative application of some transformation on $X$ led to the very general definition of \emph{fibred systems} in \cite{Schweiger:1995} (cf. also \cite{BaratBertheLiardetThuswaldner:2006}). Examples of fibred systems which have been the focus of extensive research in recent years and decades are \emph{positional notation systems} (standard and non-standard), \emph{double base number systems} \cite{DimitrovJullienMiller:1999,DimitrovImbertPradeep:2008}, \emph{continued fractions}, \emph{$\beta$-expansions} \cite{Renyi:1957,Parry:1960,Bertrand-Mathis:1989,Blanchard:1989,FrougnySolomyak:1992,Frougny:2000,Lothaire:2002,Sidorov:2003}, \emph{canonical number systems (CNS)} \cite{Knuth:1960,Penney:1965,KataiSzabo:1975,KataiKovacs:1980,KataiKovacs:1981,Kovacs:1981,Gilbert:1981,Pethoe:1991,Knuth:1998,KovacsPethoe:2006}, and \emph{shift radix systems} \cite{AkiyamaBorbelyBrunottePethoeThuswaldner:2005,AkiyamaBrunottePethoeThuswaldner:2006,AkiyamaBrunottePethoeThuswaldner:2008a,AkiyamaBrunottePethoeThuswaldner:2008b,BrunotteKirschenhoferThuswaldner:2011,KirschenhoferThuswaldner:2014,Weitzer:2015a,Weitzer:2015b,PethoeVargaWeitzer:2015,Weitzer:2015c}.

In our setting the ``names'' of $p$-adic integers are of course obtained by taking the orbits modulo $2$ which, in the case of $\FFf{F}_2$, yields the usual binary representation. The tables below show the initial parts of the orbits (which we will refer to as \emph{sequences} in the following) of several natural numbers as well as the resulting ``names'' (\emph{(digit)-expansions}).
\begin{table}[H]
\bgroup
\fontsize{8}{10}\selectfont
\begin{align*}
&
\begin{array}{r|rrrrr}
1&1&2&1&2&\cdots\\[-3pt]
2&2&1&2&1&\cdots\\[-3pt]
3&3&5&8&4&\cdots\\[-3pt]
4&4&2&1&2&\cdots\\[-3pt]
5&5&8&4&2&\cdots\\[-3pt]
6&6&3&5&8&\cdots\\[-3pt]
7&7&11&17&26&\cdots\\[-3pt]
8&8&4&2&1&\cdots\\[-3pt]
9&9&14&7&11&\cdots\\[-3pt]
10&10&5&8&4&\cdots\\[-3pt]
11&11&17&26&13&\cdots\\[-3pt]
12&12&6&3&5&\cdots\\[-3pt]
13&13&20&10&5&\cdots\\[-3pt]
14&14&7&11&17&\cdots\\[-3pt]
15&15&23&35&53&\cdots\\[-3pt]
16&16&8&4&2&\cdots\\[-3pt]
\vdots&\vdots&\vdots&\vdots&\vdots&\ddots\\
\hline
\FFFST(\FFf{F}_C)&0&1&2&3&\cdots
\end{array}
&
\begin{array}{r|rrrrr}
1&\smash{\mathrlap{\rule[6pt]{48.82pt}{0.4pt}}}1\vl&0\vl&1\vl&0\vl&\cdots\\[-3pt]
2&0\vl&1\vl&0\vl&1\vl&\cdots\\[-3pt]
3&\smash{\mathrlap{\rule[6pt]{6pt}{0.4pt}}}1&1\vl&0\vl&0\vl&\cdots\\[-3pt]
4&0&0\vl&1\vl&0\vl&\cdots\\[-3pt]
5&\smash{\mathrlap{\rule[6pt]{20pt}{0.4pt}}}1&0&0\vl&0\vl&\cdots\\[-3pt]
6&0&1&1\vl&0\vl&\cdots\\[-3pt]
7&1&1&1\vl&0\vl&\cdots\\[-3pt]
8&0&0&0\vl&1\vl&\cdots\\[-3pt]
9&\smash{\mathrlap{\rule[6pt]{34.4pt}{0.4pt}}}1&0&1&1\vl&\cdots\\[-3pt]
10&0&1&0&0\vl&\cdots\\[-3pt]
11&1&1&0&1\vl&\cdots\\[-3pt]
12&0&0&1&1\vl&\cdots\\[-3pt]
13&1&0&0&1\vl&\cdots\\[-3pt]
14&0&1&1&1\vl&\cdots\\[-3pt]
15&1&1&1&1\vl&\cdots\\[-3pt]
16&\smash{\mathrlap{\rule[-1pt]{48.82pt}{0.4pt}}}0&0&0&0\vl&\cdots\\[-3pt]
\vdots&\vdots&\vdots&\vdots&\vdots&\ddots\\
\hline
\FFFDT(\FFf{F}_C)&0&1&2&3&\cdots
\end{array}\\
&
\begin{array}{r|rrrrr}
1&1&\phantom{0}0&\phantom{0}0&\phantom{0}0&\cdots\\[-3pt]
2&2&1&0&0&\cdots\\[-3pt]
3&3&1&0&0&\cdots\\[-3pt]
4&4&2&1&0&\cdots\\[-3pt]
5&5&2&1&0&\cdots\\[-3pt]
6&6&3&1&0&\cdots\\[-3pt]
7&7&3&1&0&\cdots\\[-3pt]
8&8&4&2&1&\cdots\\[-3pt]
9&9&4&2&1&\cdots\\[-3pt]
10&10&5&2&1&\cdots\\[-3pt]
11&11&5&2&1&\cdots\\[-3pt]
12&12&6&3&1&\cdots\\[-3pt]
13&13&6&3&1&\cdots\\[-3pt]
14&14&7&3&1&\cdots\\[-3pt]
15&15&7&3&1&\cdots\\[-3pt]
16&16&8&4&2&\cdots\\[-3pt]
\vdots&\vdots&\vdots&\vdots&\vdots&\ddots\\
\hline
\mathrlap{\FFFST(\FFf{F}_2)}\phantom{\FFFST(\FFf{F}_C)}&0&1&2&3&\cdots
\end{array}
&
\begin{array}{r|rrrrr}
1&\smash{\mathrlap{\rule[6pt]{48.82pt}{0.4pt}}}1\vl&0\vl&0\vl&0\vl&\cdots\\[-3pt]
2&0\vl&1\vl&0\vl&0\vl&\cdots\\[-3pt]
3&\smash{\mathrlap{\rule[6pt]{6pt}{0.4pt}}}1&1\vl&0\vl&0\vl&\cdots\\[-3pt]
4&0&0\vl&1\vl&0\vl&\cdots\\[-3pt]
5&\smash{\mathrlap{\rule[6pt]{20pt}{0.4pt}}}1&0&1\vl&0\vl&\cdots\\[-3pt]
6&0&1&1\vl&0\vl&\cdots\\[-3pt]
7&1&1&1\vl&0\vl&\cdots\\[-3pt]
8&0&0&0\vl&1\vl&\cdots\\[-3pt]
9&\smash{\mathrlap{\rule[6pt]{34.4pt}{0.4pt}}}1&0&0&1\vl&\cdots\\[-3pt]
10&0&1&0&1\vl&\cdots\\[-3pt]
11&1&1&0&1\vl&\cdots\\[-3pt]
12&0&0&1&1\vl&\cdots\\[-3pt]
13&1&0&1&1\vl&\cdots\\[-3pt]
14&0&1&1&1\vl&\cdots\\[-3pt]
15&1&1&1&1\vl&\cdots\\[-3pt]
16&\smash{\mathrlap{\rule[-1pt]{48.82pt}{0.4pt}}}0&0&0&0\vl&\cdots\\[-3pt]
\vdots&\vdots&\vdots&\vdots&\vdots&\ddots\\
\hline
\mathrlap{\FFFDT(\FFf{F}_2)}\phantom{\FFFDT(\FFf{F}_C)}&0&1&2&3&\cdots
\end{array}
\end{align*}
\egroup
\caption{Sequences and expansions of natural numbers as given by $\FFf{F}_C$ and $\FFf{F}_2$.}
\label{TSTDTEx}
\end{table}

\noindent
It can be seen that the tables of expansions of both $\FFf{F}_C$ and $\FFf{F}_2$ admit a specific block structure which translates to a very natural condition on any number system on the $p$-adic integers:
\begin{align}
\tag{block}
\text{the first $k$ elements of the expansions of $m$ and $n$ coincide}\;\;\Leftrightarrow\;\; m\equiv n\modulus{p^k}.
\end{align}
It is this block structure which is the essential condition in the definition of $p$-adic systems. We will provide a formal definition in the upcoming section, but conclude this introduction by a verbal description and by a summary of the above observations: A $p$-adic system is a number system on $p$-adic integers which assigns an infinite string (expansion) over the alphabet $\set{0,\ldots,p-1}$ (the \emph{digits}) to any element of $\Z_p$ such that the complete ``table of expansions'' satisfies the block property. The main goal of studying $p$-adic system is to understand which parameters control what kind of expansions one can get on specific subsets of $\Z_p$, such as the integers or rational numbers in $\Z_p$. Simple changes in the definition of a $p$-adic system (such as going from $\FFf{F}_2$ to $\FFf{F}_C$) can shift questions on the resulting expansions from trivial to very hard. The hope is that by studying the entirety of $p$-adic systems one can find examples of intermediate difficulty which may shed some light on the true nature of the hard problems and help identify the ``right questions to ask''. A possible list of such examples and questions will be provided in the upcoming sections and it is the author's hope that they will arouse the curiosity of many and convince them to join in a common effort to approach them.

\newcommand{\si}[1]{{\textbf{#1}}}

\myparagraphtoc{Structure.}
This manuscript is structured as follows:\\[0.25\baselineskip]
\si{Section~\ref{SIntroduction}} summarizes the motivation and philosophy behind $p$-adic systems.\\[0.25\baselineskip]
\si{Section~\ref{SNotDef}} provides rigorous definitions of basic concepts and notions which are used throughout the paper, up to and including the central objects ``$p$-adic systems''.\\[0.25\baselineskip]
\si{Section~\ref{SInterpretations}} lists different interpretations (``ordinary functions on $\Z_p$ with block property'', ``$p$-digits tables with block property'', and ``$p$-adic permutations'') of $p$-adic systems and outlines how to translate between these different viewpoints. Furthermore, a group structure on the set of $p$-adic systems is established and the notion of ``trees of cycles'' is introduced.\\[0.25\baselineskip]
\si{Section~\ref{SCharacterization}} provides a characterization of $p$-adic systems which considers the defining functions independently from one another. This allows for the complete characterization of all $p$-adic systems which are defined by polynomials in $\Z_p[x]$, $\Q_p[x]$ and certain rational functions on $\Z_p$. These classes provide a multitude of concrete examples of $p$-adic systems.\\[0.25\baselineskip]
\si{Section~\ref{SHensel}} shows how $p$-adic systems can be used to generalize Hensel's Lemma on polynomials to general functions on $\mathbb{Z}_p$ in two different ways. The notion of ``$p$-fibred rational functions'' is introduced and investigated here.\\[0.25\baselineskip]
\si{Section~\ref{SPeriodicExpansions}} provides general results on periodic and ultimately periodic digit expansions of ``contractive'' and ``expansive'' $p$-adic systems.\\[0.25\baselineskip]
\si{Section~\ref{SLinPoly}} studies the class of ``linear-polynomial $p$-adic systems'' which are closest to the original Collatz transformation. Several conjectures which generalize the Collatz conjecture within the framework of $p$-adic systems along with first related results are listed here.\\[0.25\baselineskip]
\si{Section~\ref{SPermPoly}} describes the relation between $p$-adic systems and permutation polynomials and analyzes properties of the trees of cycles of $p$-adic systems from different classes.\\[0.25\baselineskip]
\si{Section~\ref{SQuestions}} provides a list of open questions and problems related to $p$-adic systems.\\[0.25\baselineskip]
\si{In the appendix} a list and short summary of all theorems (lemmas, corollaries, examples, etc.) can be found, as well as a list of all used symbols in order of first appearance.

\section{Notation and definitions}
\label{SNotDef}
The purpose of this somewhat technical section is to provide a solid conceptual and notational foundation for the clear and efficient, but at the same time comprehensible discussion of $p$-adic systems.

\myparagraphtoc{Basic notation.}
For any set $A$ and any set of predicates $\mathcal{P}$ we let
\begin{align}
\label{DSetPred}
A(\mathcal{P})\ce\set{a\in A\mid\fa P\in\mathcal{P}:P(a)}
\end{align}
denote the set of all elements of $A$ which satisfy all predicates in $\mathcal{P}$. If $\mathcal{P}=\set{P_1,\ldots,P_n}$ for some $n\in\Nz$, we also define the shorter version $A(P_1,\ldots,P_n)\ce A(\mathcal{P})$. For $a,b\in\R\cup\set{\pm\infty}$ let
\begin{align}
\label{DIntInt}
\intint{a}{b}&\ce\set{n\in\Z\mid a\leq n\leq b},
&
\intintuptoex{b}&\ce\intint{0}{b-1},
&
\intintuptoin{b}&\ce\intint{0}{b}.
\end{align}
For sets $A$ and $B$ let $B^A$ denote the \emph{set of all mappings} from $A$ to $B$, respectively the \emph{set of all indexed families} with index set $A$ and entries in $B$. For $n\in\Nz$ let also $B^n\ce B^\intintuptoex{n}$. For any set $A$ we may identify elements of $A$, $\set{\set{a}\mid a\in A}$, and $A^1$ by
\begin{align}
\label{EIdentify}
a\mapsto\set{a}\mapsto(a).
\end{align}

\noindent
Throughout the paper we will make heavy use of the \emph{modulo function} which shall be denoted by $\modulo$ (C++ notation). Specifically, for $0\neq m\in\Z$ and $a\in\Z$, $a\modulo m$ denotes the unique element of $\intintuptoex{\abs{m}}$ satisfying $a-a\modulo m\equiv0\modulus{m}$. Additionally, for any $2\leq p\in\N$, $a=\sum_{i=0}^\infty a_ip^i\in\Z_p$ with $a_i\in\intintuptoex{p}$ for all $i\in\Nz$, and $k\in\Nz$ let $a\modulo p^k\ce\sum_{i=0}^{k-1} a_ip^i\in\intintuptoex{p^k}$.

\myparagraphtoc{Sequences.}
In order to deal with orbits and digit expansions, we introduce the notion of sequences: the elements of any set $A^k$, where $A$ is a set and $k\in\Nz\cup\set{\infty}$, are called \emph{sequences}. The \emph{class of all sequences} shall be denoted by \label{DCoSequences}$\CoSequences$. For any sequence $\Sf{S}\in A^k$ we set \label{DSlength}$\abs{\Sf{S}}\ce k$, the \emph{length} or \emph{size} of $\Sf{S}$. For any set $A$ we define \label{DSsubseq}$\Sf{S}[A]$ to be the subsequence of $\Sf{S}$ consisting of the entries with indices in $A\cap\intintuptoex{\abs{\Sf{S}}}$ in increasing order. Furthermore, for $i,j\in\R\cup\set{\pm\infty}$ we define the shorter version \label{DSsubseqshort}$\Sf{S}[i,j]\ce\Sf{S}[\intint{i}{j}]$. Note that by Eqn.~(\ref{EIdentify}), $\Sf{S}[k]$ is defined for all $k\in\intintuptoex{\abs{\Sf{S}}}$ ($k\mapsto\set{k}$) and may be interpreted as the subsequence of length $1$ of $\Sf{S}$ which consists of the entry of $\Sf{S}$ with index $k$, or as this entry itself ($k\mapsto(k)$).

\noindent
We define the following predicates on $\CoSequences$ ($\Sf{S},\Sf{T}\in\CoSequences$, $A$ set):
\begin{flalign}
\pushleft{\SPlength{A}(\Sf{S})}&\Leftrightarrow\abs{\Sf{S}}\in A&\hspace{-2em}\text{$\Sf{S}$ has \emph{length in $A$} or, if $A=\set{a}$, $\Sf{S}$ is \emph{of length $a$}}\label{DSPlength}\\
\pushleft{\SPfinite(\Sf{S})}&\Leftrightarrow\SPlength{\Nz}(\Sf{S})&\text{$\Sf{S}$ is \emph{finite}, otherwise \emph{infinite}}\label{DSPfinite}\\
\pushleft{\SPempty(\Sf{S})}&\Leftrightarrow\SPlength{0}(\Sf{S})&\text{$\Sf{S}$ is \emph{empty}, otherwise \emph{non-empty}}\label{DSPempty}\\
\pushleft{\SPboundedby{A}(\Sf{S})}&\Leftrightarrow\fa n\in\intintuptoex{\abs{\Sf{S}}}:\Sf{S}[n]\in A&\text{$\Sf{S}$ is \emph{$A$-bounded}}\label{DSPboundedby}\\
\pushleft{\SPprefix{\Sf{T}}(\Sf{S})}&\Leftrightarrow\Sf{S}[\intintuptoex{\abs{\Sf{T}}}]=\Sf{T}&\text{$\Sf{S}$ has \emph{prefix $\Sf{T}$}}\label{DSPprefix}\\
\pushleft{\SPsuffix{\Sf{T}}(\Sf{S})}&\Leftrightarrow\Sf{S}[\abs{\Sf{S}}-\abs{\Sf{T}},\abs{\Sf{S}}-1]=\Sf{T}&\text{$\Sf{S}$ has \emph{suffix $\Sf{T}$}}\label{DSPsuffix}
\end{flalign}
\emph{Multiplication} \label{DSmult}$\Sf{S}\cdot\Sf{T}$ of two sequences $\Sf{S}$ and $\Sf{T}$ with $\Sf{S}$ being finite is defined by concatenation. If $\Sf{S}$ is infinite and $\Sf{T}$ is empty we set $\Sf{S}\cdot\Sf{T}\ce\Sf{S}$ and $\Sf{T}\cdot\Sf{S}\ce\Sf{S}$. Being the neutral element of sequence multiplication, the empty sequence shall be the result of empty products. The \emph{$n$-th power} \label{DSpow}$\Sf{S}^n$ of a finite sequence $\Sf{S}$ is the $n$-fold multiplication of $\Sf{S}$ by itself. If $\Sf{S}\neq()$, then \label{DSpowinf}$\Sf{S}^\infty$ is the infinite periodic sequence with period $\Sf{S}$, otherwise we set $()^\infty\ce()$. For an infinite sequence $\Sf{S}$ we define \label{DSFinitial}$\SFinitial{\Sf{S}}\in\CoSequences$ to be the \emph{initial part} of $\Sf{S}$ and \label{DSFperiodic}$\SFperiodic{\Sf{S}}\in\CoSequences(\SPfinite)$ to be the \emph{periodic part} of $\Sf{S}$, i.e. $\SFinitial{\Sf{S}}$ and $\SFperiodic{\Sf{S}}$ are chosen shortest possible (with $\SFinitial{\Sf{S}}$ having the precedence) such that $\Sf{S}=\SFinitial{\Sf{S}}\cdot\SFperiodic{\Sf{S}}^\infty$.

\noindent
We define the following additional predicates on $\CoSequences$ ($\Sf{S}\in\CoSequences$):
\begin{flalign}
\pushleft{\SPperiodic(\Sf{S})}&\Leftrightarrow\neg\SPfinite(\Sf{S})\land\SPempty(\SFinitial{\Sf{S}})&\text{$\Sf{S}$ is \emph{(purely) periodic}}\label{DSPperiodic}\\
\pushleft{\SPultimatelyperiodic(\Sf{S})}&\Leftrightarrow\neg\SPfinite(\Sf{S})\land\SPfinite(\SFinitial{\Sf{S}})&\text{$\Sf{S}$ is \emph{ultimately periodic}}\label{DSPultimatelyperiodic}\\
\pushleft{\SPaperiodic(\Sf{S})}&\Leftrightarrow\neg\SPfinite(\SFinitial{\Sf{S}})&\text{$\Sf{S}$ is \emph{aperiodic}}\label{DSPaperiodic}
\end{flalign}
Any function $f:A\to B$ between arbitrary sets $A$ and $B$ extends naturally to $\CoSequences(\SPboundedby{A})$ by \label{DSfunc}$\Sf{S}\mapsto(f(\Sf{S}[k]))_{k\in\intintuptoex{\abs{\Sf{S}}}}\in\CoSequences(\SPboundedby{B})$ (\emph{entry-wise application of $f$ to $\Sf{S}$}).

\myparagraphtoc{Sequence tables.}
The entirety of orbits of a transformation will be collected in ``tables'' (cf. Table~\ref{TSTDTEx}) which motivates the definition of sequence tables: the elements of any set $\CoSequences(\SPboundedby{B},\SPlength{k})^A$, where $A$ and $B$ are sets and $k\in\Nz\cup\set{\infty}$, are called \emph{sequence tables}. The \emph{class of all sequence tables} shall be denoted by \label{DCoSequenceTables}$\CoSequenceTables$. For any sequence table $\STf{S}\in\CoSequences(\SPboundedby{B},\SPlength{k})^A$ we set \label{DSTFdomain}$\STFdomain(\STf{S})\ce A$, the \emph{domain} of $\STf{S}$, and \label{DSTlength}$\abs{\STf{S}}\ce k$, the \emph{length} or \emph{size} of $\STf{S}$. For any $\STf{S}\in\CoSequenceTables$ and any $n\in\STFdomain(\STf{S})$ let \label{DSTentry}$\STf{S}[n]$ denote the $n$-th entry of $\STf{S}$ (the \emph{$n$-th row of $\STf{S}$} or \emph{$\STf{S}$-sequence of $n$}). For any subset $A$ of $\STFdomain(\STf{S})$ we define \label{DSTrestriction}$\STf{S}\vert_A\ce(\STf{S}[n])_{n\in A}$, the \emph{restriction of $\STf{S}$ to $A$}, and for any arbitrary set $A$ we define \label{DSTsubtable}$\STf{S}\llbracket A\rrbracket\ce(\STf{S}[n][A])_{n\in\STFdomain(\STf{S})}$. Furthermore, for $i,j\in\R\cup\set{\pm\infty}$ we define the shorter version \label{DSTsubtableshort}$\STf{S}\llbracket i,j\rrbracket\ce\STf{S}\llbracket\intint{i}{j}\rrbracket$. \label{DSTmultpowfunc}\emph{Multiplication} and \emph{exponentiation} of sequences as well as \emph{entry-wise applications of functions} to sequences carry over to sequence tables (row-wise).

\noindent
We define the following predicates on $\CoSequenceTables$ ($\STf{S}\in\CoSequenceTables$, $A$ set):
\begin{flalign}
\pushleft{\STPdomain{A}(\STf{S})}&\Leftrightarrow\STFdomain(\STf{S})=A&\text{$\STf{S}$ has \emph{domain $A$}}\label{DSTPdomain}\\
\pushleft{\STPlength{A}(\STf{S})}&\Leftrightarrow\abs{\STf{S}}\in A&\hspace{-4em}\text{$\STf{S}$ has \emph{length in $A$} or, if $A=\set{a}$, $\STf{S}$ is \emph{of length $a$}}\label{DSTPlength}\\
\pushleft{\STPfinite(\STf{S})}&\Leftrightarrow\STPlength{\Nz}(\STf{S})&\text{$\STf{S}$ is \emph{finite}, otherwise \emph{infinite}}\label{DSTPfinite}\\
\pushleft{\STPempty(\STf{S})}&\Leftrightarrow\STPlength{0}(\STf{S})&\text{$\STf{S}$ is \emph{empty}, otherwise \emph{non-empty}}\label{DSTPempty}\\
\pushleft{\STPboundedby{A}(\STf{S})}&\Leftrightarrow\fa n\in\STFdomain(\STf{S}):\SPboundedby{A}(\STf{S}[n])&\text{$\STf{S}$ is \emph{$A$-bounded}}\label{DSTPboundedby}
\end{flalign}

\myparagraphtoc{$p$-digit tables.}
In addition to sequence tables, we define the specialised $p$-digit tables which will be used to represent collections of expansions. Let $2\leq p\in\N$. An element $\STf{D}$ of any set $\CoSequenceTables(\STPdomain{A},\STPboundedby{\intintuptoex{p}})$, where $A\subseteq\Z_p$\footnote{Note that many authors define $p$-adic integers only for prime numbers $p$. The main reason is that if $p$ has at least two different prime factors, then the ring $\Z_p$ is no domain anymore making it less useful in many situations. However, in the setting of this paper the existence of zero divisors does not cause any problems and we thus do not limit our definition to prime numbers. Note that $\Z_p\simeq\Z_{p_1}\times\cdots\times\Z_{p_\ell}$ where $p_1,\ldots,p_\ell$ are the distinct prime factors of $p$. In the appendix a short discussion of the issue can be found.}, is called \emph{$p$-digit table} if it satisfies the condition
\begin{align}
\label{EDT}
&\fa n\in\STFdomain(\STf{D}):\STf{D}[n][0]=n\modulo p.
\end{align}
The \emph{set of all $p$-digit tables} shall be denoted by \label{DSoDigitTables}$\SoDigitTables{p}$. For any $\STf{D}\in\SoDigitTables{p}$ and $n\in\STFdomain(\STf{D})$ the \emph{$\STf{D}$-digit expansion of $n$} is given by \label{DDTentry}$\STf{D}[n]$ and for $k\in\Nz$ the \emph{$k$-th digit of $n$ with respect to $\STf{D}$} is given by \label{DDTentryentry}$\STf{D}[n][k]$.

\noindent
We define the following predicates on $\SoDigitTables{p}$ ($\STf{D}\in\SoDigitTables{p}$, $K\subseteq\Nz$) (cf. Table~\ref{TSTDTEx}):
\begin{flalign}
\pushleft{\DTPweakblockat{K}(\STf{D})}&\Leftrightarrow\fa k\in K:\fa m,n\in\STFdomain(\STf{D}):&\hspace{-4em}\text{$\STf{D}$ has the \emph{weak block property at $K$}}\label{DDTPweakblockat}\\
\nonumber
&\phantom{\vphantom{a}\Leftrightarrow\vphantom{a}}\quad m\equiv n\modulus{p^k}\Rightarrow \STf{D}[m][\intintuptoex{k}]=\STf{D}[n][\intintuptoex{k}]\\
\pushleft{\DTPweakblock(\STf{D})}&\Leftrightarrow\DTPweakblockat{\Nz}(\STf{D})&\text{$\STf{D}$ has the \emph{weak block property}}\label{DDTPweakblock}\\
\pushleft{\DTPblockat{K}(\STf{D})}&\Leftrightarrow\fa k\in K:\fa m,n\in\STFdomain(\STf{D}):&\text{$\STf{D}$ has the \emph{block property at $K$}}\label{DDTPblockat}\\
\nonumber
&\phantom{\vphantom{a}\Leftrightarrow\vphantom{a}}\quad m\equiv n\modulus{p^k}\Leftrightarrow\STf{D}[m][\intintuptoex{k}]=\STf{D}[n][\intintuptoex{k}]\\
\pushleft{\DTPblock(\STf{D})}&\Leftrightarrow\DTPblockat{\Nz}(\STf{D})&\text{$\STf{D}$ has the \emph{block property}}\label{DDTPblock}
\end{flalign}

\myparagraphtoc{$p$-fibred functions.}
The transformations $\FFf{F}_C$ and $\FFf{F}_2$ defined in the previous section are both piecewise functions on $\Z_2$ with branches for both residue classes modulo $2$. Throughout the paper we will use a very useful notation for functions given in this way. Let $2\leq p\in\N$. The elements of any set $\CoSequences(\SPboundedby{(\Z_p)^A},\SPlength{p})$, where $A\subseteq\Z_p$, are called \emph{$p$-fibred functions}, i.e. a $p$-fibred function $\FFf{F}=(\FFf{F}[0],\ldots,\FFf{F}[p-1])$ is a $p$-tuple of functions $\FFf{F}[r]:A\to\Z_p$, $r\in\intintuptoex{p}$, on some fixed subset $A$ of the $p$-adic integers. The \emph{set of all $p$-fibred functions} shall be denoted by \label{DSoFibredFunctions}$\SoFibredFunctions{p}$. For any $p$-fibred function $\FFf{F}\in\CoSequences(\SPboundedby{(\Z_p)^A},\SPlength{p})$ we set \label{DFFFdomain}$\FFFdomain(\FFf{F})\ce A$, the \emph{domain of $\FFf{F}$}. We interpret $\FFf{F}$ itself as a function on its domain in the following way:
\begin{align}
\label{EFibredFunction}
\FFf{F}:\FFFdomain(\FFf{F})&\to\Z_p\\
\nonumber
n&\mapsto\frac{\FFf{F}[n\modulo p](n)-\FFf{F}[n\modulo p](n)\modulo p}{p}
\end{align}
where \label{Dmodulo}$\modulo$ is the modulo function (C++ notation). Note that the subtrahend in the numerator of the fraction above has the mere function to guarantee that the result is divisible by $p$. For any subset $A$ of $\FFFdomain(\FFf{F})$ we define \label{DFFrestriction}$\FFf{F}\vert_A\ce(\FFf{F}[0]\vert_A,\ldots,\FFf{F}[p-1]\vert_A)$, the \emph{restriction of $\FFf{F}$ to $A$}. For any ordinary function on a subset $A$ of $\Z_p$ there is a $p$-fibred function showing the same behavior (i.e. the images of all elements of $A$ under both the ordinary function and the $p$-fibred function coincide). If $f:A\to\Z_p$ is such an ordinary function, then one such $p$-fibred function showing the same behavior is given by
\begin{align}
\label{ROrdinaryFunction}
(pf,\ldots,pf).
\end{align}
In general (indeed, in any case) there are several different $p$-fibred functions representing a given ordinary function. ``Representing the same ordinary function'' defines an equivalence relation on $\SoFibredFunctions{p}$:
\begin{align}
\label{EEquiv}
\FFf{F}\sim_p\FFf{G}\quad\Leftrightarrow\quad\FFFdomain(\FFf{F})=\FFFdomain(\FFf{G})\land\fa n\in\FFFdomain(\FFf{F}):\FFf{F}(n)=\FFf{G}(n).
\end{align}
For every equivalence class of $\sim_p$ there is a canonical representative fixed by the following predicate on $\SoFibredFunctions{p}$ ($\FFf{F}\in\SoFibredFunctions{p}$):
\begin{flalign}
\pushleft{\FFPcanonicalform(\FFf{F})}&\Leftrightarrow\fa r\in\intintuptoex{p}:\fa n\in\FFFdomain(\FFf{F}):&\text{$\FFf{F}$ is \emph{in canonical form}}\label{EPCanF}\\
\nonumber
&\phantom{\vphantom{a}\Leftrightarrow\vphantom{a}}\quad\FFf{F}[r](n)\in
\begin{cases}
p\Z_p&\text{if }n\equiv r\modulus{p}\\
\set{0}&\text{if }n\not\equiv r\modulus{p}
\end{cases}
\end{flalign}
The canonical representative of the equivalence class of a $p$-fibred function $\FFf{F}$ is called the \emph{canonical form} of $\FFf{F}$. In some cases it is useful to consider the following weaker predicate on $\SoFibredFunctions{p}$ ($\FFf{F}\in\SoFibredFunctions{p}$):
\begin{flalign}
\pushleft{\FFPweakcanonicalform(\FFf{F})}&\Leftrightarrow\fa r\in\intintuptoex{p}:\FFf{F}[r]((r+p\Z_p)\cap\FFFdomain(\FFf{F}))\subseteq p\Z_p&\text{$\FFf{F}$ is \emph{in weak canonical form}}\label{EPWCanF}
\end{flalign}
If $\FFf{G}$ is a $p$-fibred function in weak canonical form and $\FFf{F}\sim_p\FFf{G}$, then $\FFf{G}$ is called a \emph{weak canonical form} of $\FFf{F}$. If $\FFf{F}$ is a $p$-fibred function in weak canonical form, the definition of its corresponding ordinary function simplifies to
\begin{align}
\FFf{F}:\FFFdomain(\FFf{F})&\to\Z_p.\\
\nonumber
n&\mapsto\frac{\FFf{F}[n\modulo p](n)}{p}
\end{align}

\noindent
We define the following predicates on $\SoFibredFunctions{p}$ ($\FFf{F}\in\SoFibredFunctions{p}$, $A$ set):
\begin{flalign}
\pushleft{\FFPdomain{A}(\FFf{F})}&\Leftrightarrow\FFFdomain(\FFf{F})=A&\text{$\FFf{F}$ has \emph{domain $A$}}\label{DFFPdomain}\\
\pushleft{\FFPboundedby{A}(\FFf{F})}&\Leftrightarrow\FFf{F}(\FFFdomain(\FFf{F}))\subseteq A&\text{$\FFf{F}$ is \emph{$A$-bounded}}\label{DFFPboundedby}\\
\pushleft{\FFPclosed(\FFf{F})}&\Leftrightarrow\FFPboundedby{\FFFdomain(\FFf{F})}(\FFf{F})&\text{$\FFf{F}$ is \emph{closed}}\label{DFFPclosed}
\end{flalign}

\noindent
For any closed $p$-fibred function $\FFf{F}$ we call \label{DFFFST}$\FFFST(\FFf{F})\ce\big(\big(\FFf{F}^k(n)\big)_{k\in\Nz}\big)_{n\in\FFFdomain(\FFf{F})}\in\CoSequenceTables(\STPdomain{\FFFdomain(\FFf{F})},\neg\STPfinite)$ the \emph{$\FFf{F}$-sequence table}, and \label{DFFFDT}$\FFFDT(\FFf{F})\ce\FFFST(\FFf{F})\modulo p\in\SoDigitTables{p}(\STPdomain{\FFFdomain(\FFf{F})},\neg\STPfinite)$ the \emph{$\FFf{F}$-digit table}. For $n\in\FFFdomain(\FFf{F})$ the \label{DFseq}\emph{$\FFf{F}$-sequence of $n$} is given by the $\FFFST(\FFf{F})$-sequence of $n$, the \label{DFdigexp}\emph{$\FFf{F}$-digit expansion of $n$} by the $\FFFDT(\FFf{F})$-digit expansion of $n$, and for $k\in\Nz$ the \label{DFdigit}\emph{$k$-th digit of $n$ with respect to $\FFf{F}$} is given by the $k$-th digit of $n$ with respect to $\FFFDT(\FFf{F})$.

\noindent
The predicates \label{DFPblock}$\FFPweakblockat{K}$, $\FFPweakblock$, $\FFPblockat{K}$, and $\FFPblock$ carry over to $\SoFibredFunctions{p}(\FFPclosed)$ by $\FFFDT(\FFf{F})$.

\myparagraphtoc{$p$-adic systems.}
We are now in the position to define our central object of interest. A \emph{$p$-adic system} is a $p$-fibred function with domain $\Z_p$ (which implies that it is closed) which has the block property. The \emph{set of all $p$-adic systems} is thus given by \label{DSoSystems}$\SoSystems{p}\ce\SoFibredFunctions{p}(\FFPdomain{\Z_p},\FFPblock)$.

Following the definition above we are able to write
\begin{align}
\label{ECollatz}
\FFf{F}_C&\ce(x,3x+1)\\
\label{EBinary}
\FFf{F}_2&\ce(x,x-1)\sim_p(x,x).
\end{align}
$(x,x-1)$ is a weak canonical form of $(x,x)$ and the canonical form of $\FFf{F}_C$ is given by
\begin{align}
(x(x\equiv0\modulus{2}\;?\;1:0),(3x+1)(x\equiv1\modulus{2}\;?\;1:0))
\end{align}
where \label{Dcondfunc}$(P(x)\;?\;f(x):g(x))$ is a function on $\Z_p$ which reads ``if $P(x)$ then $f(x)$, else $g(x)$'' (again, C++ notation). Both $\FFf{F}_C$ and $\FFf{F}_2$ have the block property (as we will show later, Corollary~\ref{CPropPolyF}~(2)) and are thus examples of $p$-adic systems.

\section{Three and a half interpretations of $p$-adic systems}
\label{SInterpretations}
In this sections we will establish different ways to think about $p$-adic systems, all of which are valid interpretations in there own right along the one which we ultimately chose to be the formal definition. As with any introduction to a mathematical object, one has to decide which description to call ``definition'' and which ``characterization'' instead. Recall that the questions on $p$-adic systems we are most interested in are those on their ultimate behavior (such as whether the orbits of all natural numbers under $\FFf{F}_C$ actually end up in $1$). It is in this regard that the mentioned interpretations will be equivalent.

\myparagraphtoc{The formal definition: $p$-fibred functions with block property.}
We begin by repeating the formal definition which would be the first interpretation in our list of three and a half: A $p$-adic system is a $p$-fibred function $\FFf{F}$ with domain $\Z_p$, i.e. a piecewise function on $\Z_p$ with branches for all residue classes modulo $p$, such that the $p$-digit table $\FFFDT(\FFf{F})$ given by the $\FFf{F}$-digit expansions satisfies the block property: $\FFFDT(\FFf{F})[m][\intintuptoex{k}]=\FFFDT(\FFf{F})[n][\intintuptoex{k}]$ if and only if $m\equiv n\modulus{p^k}$ for all $k\in\Nz$ and for all $m,n\in\Z_p$. According to this definition a $p$-adic system is a ``dynamical object'' as the property we are most interested in (the corresponding digit expansions) is defined by the dynamical process of repeatedly applying the $p$-adic system. We recall that $\sim_p$ (cf. Eqn.~(\ref{EEquiv})) defines an equivalence relation on $\SoFibredFunctions{p}$ (we provide a formal proof in the lemma below), where equivalent $p$-fibred functions show exactly the same dynamical behavior and thus define equal $p$-digit tables. Since the $p$-digit tables are what we really care about, one might consider a $p$-adic system to be a whole equivalence class of $\sim_p$ instead of a single $p$-fibred function.

\begin{lemma}
\label{LFFCanForm}
Let $2\leq p\in\N$ and $\FFf{F}\in\SoFibredFunctions{p}$. Then, there is a unique $\FFf{F}_c\in\SoFibredFunctions{p}$ such that $\FFf{F}_c$ is a canonical form of $\FFf{F}$. In this case $\FFFST(\FFf{F}_c)=\FFFST(\FFf{F})$ and $\FFFDT(\FFf{F}_c)=\FFFDT(\FFf{F})$. If $\FFf{G}$ is another $p$-fibred function and $\FFf{G}_c$ its canonical form, then $\FFf{F}\sim_p\FFf{G}$ if and only if $\FFf{F}_c=\FFf{G}_c$. In particular, $\sim_p$ is an equivalence relation on $\SoFibredFunctions{p}$ and every equivalence class contains a unique canonical form.
\end{lemma}

\begin{proof}
Let $\FFf{F}_c\in\SoFibredFunctions{p}$ with
\begin{align}
\FFf{F}_c[r]:\FFFdomain(\FFf{F})&\to\FFFdomain(\FFf{F})\\
\nonumber
n&\mapsto
\begin{cases}
\FFf{F}[r](n)-\FFf{F}[r](n)\modulo p&\text{if }n\equiv r\modulus{p}\\
0&\text{if }n\not\equiv r\modulus{p}
\end{cases}
\end{align}
for all $r\in\intintuptoex{p}$. Then, $\FFf{F}_c$ is a canonical form of $\FFf{F}$ and its uniqueness follows by construction. Clearly, $\FFFST(\FFf{F}_c)=\FFFST(\FFf{F})$ and hence $\FFFDT(\FFf{F}_c)=\FFFDT(\FFf{F})$.

Let $\FFf{G}\in\SoFibredFunctions{p}$. If $\FFf{F}\sim_p\FFf{G}$, then $\FFFdomain(\FFf{F})=\FFFdomain(\FFf{G})$, and for all $r\in\intintuptoex{p}$ and all $n\in(r+p\Z_p)\cap\FFFdomain(\FFf{F})$ we get
\begin{align}
\FFf{F}_c[r](n)&=\FFf{F}[r](n)-\FFf{F}[r](n)\modulo p=p\FFf{F}(n)=p\FFf{G}(n)=\FFf{G}[r](n)-\FFf{G}[r](n)\modulo p=\FFf{G}_c[r](n)
\end{align}
and hence $\FFf{F}_c=\FFf{G}_c$.

If, however, $\FFf{F}_c=\FFf{G}_c$, then again $\FFFdomain(\FFf{F})=\FFFdomain(\FFf{G})$ and for all $r\in\intintuptoex{p}$ and all $n\in(r+p\Z_p)\cap\FFFdomain(\FFf{F})$ we get
\begin{align}
\FFf{F}(n)&=\frac{\FFf{F}[r](n)-\FFf{F}[r](n)\modulo p}{p}=\frac{\FFf{F}_c[r](n)}{p}=\frac{\FFf{G}_c[r](n)}{p}=\frac{\FFf{G}[r](n)-\FFf{G}[r](n)\modulo p}{p}=\FFf{G}(n)
\end{align}
and hence $\FFf{F}\sim_p\FFf{G}$.
\end{proof}

\noindent
From now on we identify elements of $\SoFibredFunctions{p}$ and $\SoFibredFunctions{p}\slash_{\sim_p}$ by $\FFf{F}\mapsto[\FFf{F}]_{\sim_p}$, e.g. $[\FFf{F}]_{\sim_p}(n)=\FFf{F}(n)$, $\FFFDT([\FFf{F}]_{\sim_p})=\FFFDT(\FFf{F})$, etc.

\myparagraphtoc{Ordinary functions on $\Z_p$ with block property.}
This alternative interpretation accounts for the ``half'' in the title of this section and has already been mentioned in the formal definition of $p$-adic systems in the previous section. Any $p$-fibred function defines an ordinary function on its domain (Eqn.~(\ref{EFibredFunction})) and, vice versa, any ordinary function on some subset of $\Z_p$ is equal to the ordinary function defined by some $p$-fibred function (Eqn.~(\ref{ROrdinaryFunction})). A class of the equivalence relation $\sim_p$ consists of exactly those $p$-fibred functions which represent the same ordinary function. For any $A\subseteq\Z_p$ there is thus a one-to-one correspondence between $\SoFibredFunctions{p}(\FFPdomain{A})\slash_{\sim_p}$ and $(\Z_p)^A$. The distinction between $p$-fibred functions and ordinary functions is therefore only a matter of notation and the question that remains is: how does the block property of $p$-fibred functions translate to ordinary functions? The answer is given by the following Lemma.

\begin{lemma}
\label{LOrdFuncBlock}
Let $2\leq p\in\N$, $\FFf{F}\in\SoFibredFunctions{p}(\FFPclosed)$, and $k\in\N\cup\set{\infty}$. Then,

\begin{theoremtable}
\theoremitem{(1)}&$(\fa\ell\in\intint{1}{k}:\fa m,n\in\FFFdomain(\FFf{F})$ with $m\equiv n\!\!\modulus{p}:m\equiv n\!\modulus{p^\ell}\Rightarrow\FFf{F}(m)\equiv\FFf{F}(n)\!\modulus{p^{\ell-1}})$\tabularnewline
&\quad$\Rightarrow\FFPweakblockat{\intintuptoin{k}}(\FFf{F})$\tabularnewline
\theoremitem{(2)}&$(\fa\ell\in\intint{1}{k}:\fa m,n\in\FFFdomain(\FFf{F})$ with $m\equiv n\!\modulus{p}:m\equiv n\!\modulus{p^\ell}\Leftrightarrow\FFf{F}(m)\equiv\FFf{F}(n)\!\modulus{p^{\ell-1}})$\tabularnewline
&\quad$\Leftrightarrow\FFPblockat{\intintuptoin{k}}(\FFf{F})$.
\end{theoremtable}
\end{lemma}

\noindent
Before we give a proof of the lemma, we analyze its meaning. (2) completely characterizes when an ordinary function $f:A\to A$, $A\subseteq\Z_p$ (which is the function given by the $p$-fibred function $\FFf{F}$ in the lemma with $A=\FFFdomain(\FFf{F})$) has the block property by only considering a single application of $f$. However, (1) only provides a sufficient condition for the weak block property as the following example shows.

\begin{example}
\label{EOrdFuncBlock}
Let $\FFf{F}\in\SoFibredFunctions{2}(\FFPdomain{\Z_2})$ with $\FFf{F}[0](n)=(n=8\;?\;4:0)$ and $\FFf{F}[1](n)=2$ for all $n\in\Z_2$. Then, $\FFFDT(\FFf{F})=((n\modulo2)^\infty)_{n\in\Z_2}$ and, hence, $\FFPweakblock(\FFf{F})$. At the same time we get $0\equiv8\modulus{2^3}$ but $\FFf{F}(0)=0\not\equiv2=\FFf{F}(8)\modulus{2^2}$.
\end{example}

\noindent
We will show later (Example~\ref{ENoCharacWBlk}) that no characterization of the weak block property of the above kind can exist that is both necessary and sufficient.

\begin{proof}[Proof of Lemma~\ref{LOrdFuncBlock}]
$ $\\
\proofitem{(1)}%
We prove $\FFPweakblockat{\ell}(\FFf{F})$ for all $\ell\in\intintuptoin{k}$ by induction on $\ell$: clearly, $\FFPweakblockat{0}(\FFf{F})$. Assume $\ell\in\intint{1}{k}$ and that the statement holds for $\ell-1$. Let $m,n\in\FFFdomain(\FFf{F})$ with $m\equiv n\modulus{p^\ell}$. Then $m\equiv n\modulus{p}$ and thus
\begin{align}
\FFFDT(\FFf{F})[m][0]=\FFFDT(\FFf{F})[n][0]
\end{align}
by definition of $\FFFDT(\FFf{F})$. Furthermore,
\begin{align}
m\equiv n\modulus{p^\ell}
&\Rightarrow\FFf{F}(m)\equiv\FFf{F}(n)\modulus{p^{\ell-1}}\\
\text{(ind. hyp.)\quad}&\Rightarrow\FFFDT(\FFf{F})[\FFf{F}(m)][0,\ell-2]=\FFFDT(\FFf{F})[\FFf{F}(n)][0,\ell-2]\\
&\Leftrightarrow\FFFDT(\FFf{F})[m][0]\cdot\FFFDT(\FFf{F})[\FFf{F}(m)][0,\ell-2]=\FFFDT(\FFf{F})[n][0]\cdot\FFFDT(\FFf{F})[\FFf{F}(n)][0,\ell-2]\\
&\Leftrightarrow\FFFDT(\FFf{F})[m][\intintuptoex{\ell}]=\FFFDT(\FFf{F})[n][\intintuptoex{\ell}]
\end{align}
which implies $\FFPweakblockat{\ell}(\FFf{F})$.

\noindent
\proofitem{(2)}%
We prove $``\Rightarrow$'' by showing $\FFPblockat{\ell}(\FFf{F})$ for all $\ell\in\intintuptoin{k}$ by induction on $\ell$: clearly, $\FFPblockat{0}(\FFf{F})$. Assume $\ell\in\intint{1}{k}$ and that the statement holds for $\ell-1$. Let $m,n\in\FFFdomain(\FFf{F})$ with $m\equiv n\modulus{p^\ell}$. Then $m\equiv n\modulus{p}$ and thus
\begin{align}
\FFFDT(\FFf{F})[m][0]=\FFFDT(\FFf{F})[n][0]
\end{align}
as before. Furthermore,
\begin{align}
m\equiv n\modulus{p^\ell}
&\Leftrightarrow\FFf{F}(m)\equiv\FFf{F}(n)\modulus{p^{\ell-1}}\\
\text{(ind. hyp.)\quad}&\Leftrightarrow\FFFDT(\FFf{F})[\FFf{F}(m)][0,\ell-2]=\FFFDT(\FFf{F})[\FFf{F}(n)][0,\ell-2]\\
&\Leftrightarrow\FFFDT(\FFf{F})[m][0]\cdot\FFFDT(\FFf{F})[\FFf{F}(m)][0,\ell-2]=\FFFDT(\FFf{F})[n][0]\cdot\FFFDT(\FFf{F})[\FFf{F}(n)][0,\ell-2]\\
&\Leftrightarrow\FFFDT(\FFf{F})[m][\intintuptoex{\ell}]=\FFFDT(\FFf{F})[n][\intintuptoex{\ell}]
\end{align}
which implies $\FFPblockat{\ell}(\FFf{F})$.

For the proof of ``$\Leftarrow$'' let $m,n\in\FFFdomain(\FFf{F})$ with $m\equiv n\modulus{p}$. Then,
\begin{align}
\FFFDT(\FFf{F})[m][0]=\FFFDT(\FFf{F})[n][0].
\end{align}
Furthermore,
\begin{align}
m\equiv n\modulus{p^\ell}
&\Leftrightarrow\FFFDT(\FFf{F})[m][\intintuptoex{\ell}]=\FFFDT(\FFf{F})[n][\intintuptoex{\ell}]\\
&\Leftrightarrow\FFFDT(\FFf{F})[m][0]\cdot\FFFDT(\FFf{F})[\FFf{F}(m)][0,\ell-2]=\FFFDT(\FFf{F})[n][0]\cdot\FFFDT(\FFf{F})[\FFf{F}(n)][0,\ell-2]\\
&\Leftrightarrow\FFFDT(\FFf{F})[\FFf{F}(m)][0,\ell-2]=\FFFDT(\FFf{F})[\FFf{F}(n)][0,\ell-2]\\
&\Leftrightarrow\FFf{F}(m)\equiv\FFf{F}(n)\modulus{p^{\ell-1}}.
\end{align}
\end{proof}

If we denote by \label{DSoFunctions}$\SoFunctions{p}$ the set of all functions $f:\Z_p\to\Z_p$ which satisfy\begin{align}
\fa k\in\N:\fa m,n\in\Z_p:m\equiv n\modulus{p^k}\Leftrightarrow f(m)\equiv f(n)\modulus{p^{k-1}},
\end{align}
as a consequence, there is a bijection between $\SoSystems{p}\slash_{\sim_p}$ and $\SoFunctions{p}$ given by
\begin{align}
\SoSystems{p}\slash_{\sim_p}&\leftrightarrow\SoFunctions{p}.\\
\nonumber
\FFf{F}&\mapsto(\FFf{F}:\Z_p\to\Z_p)\\
\nonumber
(pf,\ldots,pf)&\mapsfrom f
\end{align}

\myparagraphtoc{$p$-digit tables with block property.}
As argued above, both of the previous interpretations consider $p$-adic systems to be ``dynamical objects''. The dynamical behavior of a (closed) $p$-fibred function is used to define a ``static object'', the corresponding $p$-digit table, whose structure determines whether the $p$-fibred function is considered a $p$-adic system (which is the case precisely if the $p$-digit table has the block property). Surprisingly, it is also possible to go in the other direction, i.e. start out with the static object and use it to define a dynamical one. Indeed, there is a one-to-one correspondence between $\SoSystems{p}\slash_{\sim_p}$ and \label{DSoTables}$\SoTables{p}\ce\SoDigitTables{p}(\STPdomain{\Z_p},\neg\STPfinite,\DTPblock)$. Clearly, every $p$-adic system $\FFf{F}$ defines an infinite $p$-digit table with block property by its $\FFf{F}$-digit table $\FFFDT(\FFf{F})$ and any $p$-adic system that is equivalent to $\FFf{F}$ defines the same $p$-digit table. The other two facts implied by the existence of the mentioned one-to-one correspondence are probably less obvious: the $p$-digit tables of two $p$-adic systems $\FFf{F}$ and $\FFf{G}$ are identical if and only if $\FFf{F}$ and $\FFf{G}$ are equivalent and every infinite $p$-digit table with domain $\Z_p$ and block property is the $\FFf{F}$-digit table of some $p$-adic system $\FFf{F}$. In the following we will prove just that and we begin by interpreting any $p$-digit table as a multivalued function on its domain. We will show that if $\STf{D}$ is infinite and has domain $\Z_p$ and the block property, then this function actually defines the $p$-adic system we are looking for. For any $\STf{D}\in\SoDigitTables{p}$ let
\begin{align}
\label{DDTfunc}
\STf{D}:\STFdomain(\STf{D})&\to\powerset(\STFdomain(\STf{D}))\\
\nonumber
n&\mapsto\set{m\in\STFdomain(\STf{D})\mid\STf{D}[m][0,\abs{\STf{D}}-2]=\STf{D}[n][1,\abs{\STf{D}}-1]}
\end{align}
where \label{Dpowerset}$\powerset(A)$ denotes the power set of a set $A$. What $\STf{D}$ (as a function) does is, it takes an element $n$ from its domain, drops the first entry (entry with index $0$) from the $\STf{D}$-digit expansion of $n$, and returns all elements of its domain which have the resulting sequence as their initial $\STf{D}$-digit expansion. Using this  notation we continue with a characterization of all $p$-fibred functions which define a given $p$-digit table.

\begin{theorem}
\label{TDTFF}
Let $2\leq p\in\N$ and $\STf{D}\in\SoDigitTables{p}$. Then, for every $\FFf{F}\in\SoFibredFunctions{p}(\FFPdomain{\STFdomain(\STf{D})},\FFPclosed)$ we get $\FFFDT(\FFf{F})\llbracket\intintuptoex{\abs{\STf{D}}}\rrbracket=\STf{D}$ if and only if $\FFf{F}(n)\in\STf{D}(n)$ for every $n\in\STFdomain(\STf{D})$. In particular, there is a $\FFf{F}\in\SoFibredFunctions{p}(\FFPdomain{\STFdomain(\STf{D})},\FFPclosed)$ such that $\FFFDT(\FFf{F})\llbracket\intintuptoex{\abs{\STf{D}}}\rrbracket=\STf{D}$ if and only if $\STf{D}(n)\neq\emptyset$ for every $n\in\STFdomain(\STf{D})$.
\end{theorem}

\begin{proof}
For ``$\Rightarrow$'' let $n\in\STFdomain(\STf{D})$ and observe that $\FFFDT(\FFf{F})[n][1,\infty]=\FFFDT(\FFf{F})[\FFf{F}(n)]$ which implies that $\STf{D}[n][1,\abs{\STf{D}}-1]=\FFFDT(\FFf{F})[n][1,\abs{\STf{D}}-1]=\FFFDT(\FFf{F})[\FFf{F}(n)][0,\abs{\STf{D}}-2]=\STf{D}[\FFf{F}(n)][0,\abs{\STf{D}}-2]$. Therefore, $\FFf{F}(n)\in\STf{D}(n)$.

To show ``$\Leftarrow$'' we prove $\FFFDT(\FFf{F})\llbracket\intintuptoex{k}\rrbracket=\STf{D}\llbracket\intintuptoex{k}\rrbracket$ for all $k\in\intint{1}{\abs{\STf{D}}}$ by induction on $k$. If $k=1$ then this is clearly true by the definition of $p$-digit tables. Now assume $k\geq2$ and let $n\in\STFdomain(\STf{D})$. We get $\FFFDT(\FFf{F})\llbracket\intintuptoex{k-1}\rrbracket=\STf{D}\llbracket\intintuptoex{k-1}\rrbracket$ by the induction hypothesis. Furthermore, $\FFFDT(\FFf{F})[n][k-1]=\FFFDT(\FFf{F})[\FFf{F}(n)][k-2]=\STf{D}[\FFf{F}(n)][k-2]=\STf{D}[n][k-1]$. Thus, $\FFFDT(\FFf{F})\llbracket\intintuptoex{k}\rrbracket=\STf{D}\llbracket\intintuptoex{k}\rrbracket$ and consequently $\FFFDT(\FFf{F})\llbracket\intintuptoex{\abs{\STf{D}}}\rrbracket=\STf{D}$.
\end{proof}

It is clear from the previous theorem that in order to get a better understanding of the relation between $p$-adic systems and $p$-digit tables with block property, we need to investigate the structure of the sets $\STf{D}(n)$. Before doing so, we need to prove two basic but useful lemmas on finite $p$-digit tables with block property at their lengths. From now on we interpret any complete residue system (CRS) $R\subseteq\Z_p$ modulo $p^k$ (i.e. $\abs{R}=p^k\in\N$ and no two distinct elements of $R$ are congruent modulo $p^k$, e.g. $R=\intintuptoex{p^k}$) as a function $R:\Z_p\to R$ which maps any $n\in\Z_p$ to the unique \label{DCRSfunc}$R(n)\in R$ with $n\equiv R(n)\modulus{p^k}$.

\begin{lemma}
\label{LFinBlDTCompl}
Let $2\leq p\in\N$, $k\in\N$, $\STf{D}\in\SoDigitTables{p}(\STPlength{k},\DTPblockat{k})$, $R\subseteq\STFdomain(\STf{D})$ a CRS modulo $p^k$, and $\Sf{D}\in\CoSequences(\SPboundedby{\intintuptoex{p}},\SPlength{k})$. Then there is a unique $n\in R$ such that $\STf{D}[n]=\Sf{D}$.
\end{lemma}

\begin{proof}
Since $\abs{\CoSequences(\SPboundedby{\intintuptoex{p}},\SPlength{k})}=\abs{R}=p^k$ and $\DTPblockat{k}(\STf{D})$, every sequence in $\CoSequences(\SPboundedby{\intintuptoex{p}},\SPlength{k})$ occurs exactly once among the sequences $\STf{D}[r]$, $r\in R$.
\end{proof}

\begin{lemma}
\label{LSBlDTBl}
Let $2\leq p\in\N$, $k,\ell\in\Nz$ with $k<\ell$, $\STf{D}\in\SoDigitTables{p}(\STPlength{\ell},\DTPblockat{\ell},\DTPweakblockat{k})$ such that $\STFdomain(\STf{D})$ contains a CRS modulo $p^\ell$. Then $\DTPblockat{k}(\STf{D})$.
\end{lemma}

\begin{proof}
Let $m,n\in\STFdomain(\STf{D})$ such that $\STf{D}[m][\intintuptoex{k}]=\STf{D}[n][\intintuptoex{k}]\rce\Sf{D}$ and assume $m\not\equiv n\modulus{p^k}$. Furthermore, let $R\subseteq\STFdomain(\STf{D})$ be a CRS modulo $p^\ell$ and
\begin{align}
M\ce\set{r\in R\mid r\equiv m\modulus{p^k}\lor r\equiv n\modulus{p^k}}.
\end{align}
Then $\abs{M}=2p^{\ell-k}$ and because of $\DTPweakblockat{k}(\STf{D})$, we get $\STf{D}[r][\intintuptoex{k}]=\Sf{D}$ for every $r\in M$. Thus, $\set{\STf{D}[r]\mid r\in R}$ can have at most $p^\ell-p^{\ell-k}$ elements, which contradicts Lemma~\ref{LFinBlDTCompl} and hence implies $\DTPblockat{k}(\STf{D})$.
\end{proof}

The following theorem describes the structure of the sets $\STf{D}(n)$ in dependence of the structure of $\STf{D}$.

\begin{theorem}
\label{TDnStructure}
Let $2\leq p\in\N$, $k\in\N$, $\STf{D}\in\SoDigitTables{p}(\STPlength{k})$, $R,S\subseteq\STFdomain(\STf{D})$ CRSs modulo $p^k$ and $p^{k-1}$ respectively, and $n\in\STFdomain(\STf{D})$. Then,

\begin{theoremtable}
\theoremitem{(1)}&$\mathrlap{\DTPweakblockat{k}(\STf{D})}\phantom{\DTPblockat{k}(\STf{D})\land\DTPweakblockat{k-1}(\STf{D})}\Rightarrow\STf{D}(n)=(\STf{D}(R(n))\cap R+p^k\Z_p)\cap\STFdomain(\STf{D})$\tabularnewline
&$\phantom{\DTPblockat{k}(\STf{D})\land\DTPweakblockat{k-1}(\STf{D})\Rightarrow\vphantom{a}}$\rlap{In particular: }\phantom{In particular: }\textup{(1)}\;\,$\fa m\in(n+p^k\Z_p)\cap\STFdomain(\STf{D}):\STf{D}(n)=\STf{D}(m)$\tabularnewline
&$\phantom{\DTPblockat{k}(\STf{D})\land\DTPweakblockat{k-1}(\STf{D})\Rightarrow\vphantom{a}}$\phantom{In particular: }\textup{(2)}\;\,$\STf{D}(n)=(\STf{D}(n)+p^k\Z_p)\cap\STFdomain(\STf{D})$\tabularnewline
\theoremitem{(2)}&$\mathrlap{\DTPweakblockat{\set{k,k-1}}(\STf{D})}\phantom{\DTPblockat{k}(\STf{D})\land\DTPweakblockat{k-1}(\STf{D})}\Rightarrow\STf{D}(n)=(\STf{D}(R(n))\cap S+p^{k-1}\Z_p)\cap\STFdomain(\STf{D})$\tabularnewline
&$\phantom{\DTPblockat{k}(\STf{D})\land\DTPweakblockat{k-1}(\STf{D})\Rightarrow\vphantom{a}}$\rlap{In particular: }\phantom{In particular: }$\STf{D}(n)=(\STf{D}(n)+p^{k-1}\Z_p)\cap\STFdomain(\STf{D})$\tabularnewline
\theoremitem{(3)}&$\mathrlap{\DTPblockat{k}(\STf{D})}\phantom{\DTPblockat{k}(\STf{D})\land\DTPweakblockat{k-1}(\STf{D})}\Rightarrow\abs{\STf{D}(n)\cap R}=p$\tabularnewline
\theoremitem{(4)}&$\DTPblockat{k}(\STf{D})\land\DTPweakblockat{k-1}(\STf{D})\Rightarrow\fa l,m\in\STf{D}(n):l\equiv m\modulus{p^{k-1}}$\tabularnewline
&\leavevmode\phantom{$\DTPblockat{k}(\STf{D})\land\DTPweakblockat{k-1}(\STf{D})\Rightarrow\vphantom{a}$}In particular: $\abs{\STf{D}(n)\cap S}=1$ and\tabularnewline
&\leavevmode\phantom{$\DTPblockat{k}(\STf{D})\land\DTPweakblockat{k-1}(\STf{D})\Rightarrow\vphantom{a}$In particular: }$\STf{D}(n)=(s_{R(n)}+p^{k-1}\Z_p)\cap\STFdomain(\STf{D})$ where\tabularnewline
&\leavevmode\phantom{$\DTPblockat{k}(\STf{D})\land\DTPweakblockat{k-1}(\STf{D})\Rightarrow\vphantom{a}$In particular: }$s_{R(n)}$ is the unique element of $\STf{D}(n)\cap S$.
\end{theoremtable}
\end{theorem}

\noindent
Before we give a proof of the theorem, we will discuss its claims. (1) states that if $\STf{D}$ has the weak block property at $k$, then in order to know the structure of any $\STf{D}(n)$, it suffices to know the structures of $\STf{D}(r)$ within the finite set $R$ alone, where $r\in R$ is again one element of only finitely many. In addition, the ``In particular'' part states that $\STf{D}(n)$ is closed under addition of multiples of $p^k$ (if the result lands in the domain of $\STf{D}$). (2) and (3) give even more precise information on the structure of $\STf{D}(n)$ if more is known on the structure of $\STf{D}$. The results are stated in terms of arbitrary CRS $R$ and $S$. To get a better understanding of the theorem, it is helpful to consider the most important special case $\STFdomain(\STf{D})\supseteq\Nz$, $R=\intintuptoex{p^k}$, $S=\intintuptoex{p^{k-1}}$, and hence $R(n)=n\modulo p^k$ for all $n\in\Z_p$. As examples consider the three $2$-digit tables below which correspond to the three cases treated in the theorem. We get
\begin{align}
\STf{D}_1(29)&=(\STf{D}_1(5)\cap\intintuptoex{8}+8\Z_2)\cap\Nz=\set{1,3,7}+8\Nz\\
\STf{D}_2(29)&=(\STf{D}_2(5)\cap\intintuptoex{8}+8\Z_2)\cap\Nz=\set{0,2}+8\Nz\text{ (note: $\abs{\set{0,2}}=2$)}\\
\STf{D}_3(29)&=(\STf{D}_3(5)\cap\intintuptoex{4}+4\Z_2)\cap\Nz=2+4\Nz
\end{align}
just as claimed.

\begin{table}[H]
\bgroup
\fontsize{8}{10}\selectfont
\begin{align*}
&
\begin{array}{r|rrr}
0&0&0&0\\[-3pt]
1&1&0&0\\[-3pt]
2&0&0&0\\[-3pt]
3&1&0&0\\[-3pt]
4&0&1&0\\[-3pt]
5&1&1&0\\[-3pt]
6&0&1&0\\[-3pt]
7&1&0&0\\[-3pt]
8&\smash{\mathrlap{\rule[6pt]{34.4pt}{0.4pt}}}0&0&0\\[-3pt]
9&1&0&0\\[-3pt]
\vdots&\vdots&\vdots&\vdots\\
\hline
\STf{D}_1&0&1&2
\end{array}
&
\begin{array}{r|rrr}
0&0&0&0\\[-3pt]
1&1&0&1\\[-3pt]
2&0&0&1\\[-3pt]
3&1&1&0\\[-3pt]
4&0&1&0\\[-3pt]
5&1&0&0\\[-3pt]
6&0&1&1\\[-3pt]
7&1&1&1\\[-3pt]
8&\smash{\mathrlap{\rule[6pt]{34.4pt}{0.4pt}}}0&0&0\\[-3pt]
9&1&0&1\\[-3pt]
\vdots&\vdots&\vdots&\vdots\\
\hline
\STf{D}_2&0&1&2
\end{array}
&&
\begin{array}{r|rrr}
0&0&0&0\\[-3pt]
1&1&0&0\\[-3pt]
2&0&1&0\\[-3pt]
3&1&1&0\\[-3pt]
4&0&0&1\\[-3pt]
5&1&0&1\\[-3pt]
6&0&1&1\\[-3pt]
7&1&1&1\\[-3pt]
8&\smash{\mathrlap{\rule[6pt]{34.4pt}{0.4pt}}}0&0&0\\[-3pt]
9&1&0&0\\[-3pt]
\vdots&\vdots&\vdots&\vdots\\
\hline
\STf{D}_3&0&1&2\\
\end{array}
\end{align*}
\egroup
\caption{Three $2$-digit tables of length $3$ with domain $\Nz$. The initial block of $8$ rows is supposed to repeat periodically in all cases. They satisfy $\DTPweakblockat{3}(\STf{D}_1)$, $\DTPblockat{3}(\STf{D}_2)$, and $\DTPblockat{3}(\STf{D}_3)\land\DTPweakblockat{2}(\STf{D}_3)$.}
\label{TDTEx}
\end{table}

\begin{proof}[Proof of Theorem~\ref{TDnStructure}]
$ $\\
\proofitem{(1)}%
\begin{align}
\label{EDnStrucCompA}
\STf{D}(n)&=\set{m\in\STFdomain(\STf{D})\mid\STf{D}[m][0,k-2]=\STf{D}[n][1,k-1]}\\
\label{EDnStrucCompB}
(\DTPweakblockat{k}(\STf{D}))\;&=\set{m\in\STFdomain(\STf{D})\mid\STf{D}[R(m)][0,k-2]=\STf{D}[R(n)][1,k-1]}\\
\label{EDnStrucCompC}
&=\left(\set{m\in R\mid\STf{D}[m][0,k-2]=\STf{D}[R(n)][1,k-1]}+p^k\Z_p\right)\cap\STFdomain(\STf{D})\\
\label{EDnStrucCompD}
&=\left(\STf{D}(R(n))\cap R+p^k\Z_p\right)\cap\STFdomain(\STf{D}).
\end{align}

If $m\in(n+p^k\Z_p)\cap\STFdomain(\STf{D})$ then, $R(m)=R(n)$ and
\begin{align}
\STf{D}(n)&=(\STf{D}(R(n))\cap R+p^k\Z_p)\cap\STFdomain(\STf{D})
=(\STf{D}(R(m))\cap R+p^k\Z_p)\cap\STFdomain(\STf{D})
=\STf{D}(m)
\end{align}
which proves (1) of the ``In particular'' part. Furthermore, since $R(n)\in(n+p^k\Z_p)\cap\STFdomain(\STf{D})$, we get $\STf{D}(n)=\STf{D}(R(n))$ and thus
\begin{align}
\STf{D}(n)&=(\STf{D}(R(n))\cap R+p^k\Z_p)\cap\STFdomain(\STf{D})
=(\STf{D}(n)\cap R+p^k\Z_p)\cap\STFdomain(\STf{D})\\
&=(\STf{D}(n)+p^k\Z_p)\cap\STFdomain(\STf{D})
\end{align}
proving (2) of the ``In particular'' part.

\noindent
\proofitem{(2)}%
Analogously to Eqn.~(\ref{EDnStrucCompA})~--~(\ref{EDnStrucCompD}) we get
\begin{align}
\STf{D}(n)&=\set{m\in\STFdomain(\STf{D})\mid\STf{D}[m][0,k-2]=\STf{D}[n][1,k-1]}\\
(\DTPweakblockat{\set{k,k-1}}(\STf{D}))\;&=\set{m\in\STFdomain(\STf{D})\mid\STf{D}[S(m)][0,k-2]=\STf{D}[R(n)][1,k-1]}\\
&=\left(\set{m\in S\mid\STf{D}[m][0,k\!-\!2]=\STf{D}[R(n)][1,k\!-\!1]}+p^{k-1}\Z_p\right)\cap\STFdomain(\STf{D})\\
&=(\STf{D}(R(n))\cap S+p^{k-1}\Z_p)\cap\STFdomain(\STf{D}).
\end{align}

For the ``In particular'' part we again observe that $\STf{D}(n)=\STf{D}(R(n))$, hence
\begin{align}
\STf{D}(n)&=(\STf{D}(R(n))\cap S+p^{k-1}\Z_p)\cap\STFdomain(\STf{D})
=(\STf{D}(n)\cap S+p^{k-1}\Z_p)\cap\STFdomain(\STf{D})\\
&=(\STf{D}(n)+p^{k-1}\Z_p)\cap\STFdomain(\STf{D}).
\end{align}

\noindent
\proofitem{(3)}%
\begin{align}
\STf{D}(n)\cap R&=\set{m\in R\mid\STf{D}[m][0,k-2]=\STf{D}[n][1,k-1]}\\
&=\set{m\in R\mid\ex d\in\intintuptoex{p}:\STf{D}[m]=\STf{D}[n][1,k-1]\cdot(d)}.
\end{align}
Thus, $\DTPblockat{k}(\STf{D})$ and Lemma~\ref{LFinBlDTCompl} imply that $\abs{\STf{D}(n)\cap R}=\abs{\intintuptoex{p}}=p$.

\noindent
\proofitem{(4)}%
Let $l,m\in\STf{D}(n)$. Then, $\STf{D}[l][0,k-2]=\STf{D}[m][0,k-2]$ and hence $l\equiv m\modulus{p^{k-1}}$ by $\DTPblockat{k}(\STf{D})$, $\DTPweakblockat{k-1}(\STf{D})$, and Lemma~\ref{LSBlDTBl} (which implies $\DTPblockat{k-1}(\STf{D})$). The ``In particular'' part then follows directly from (2).
\end{proof}

\noindent
Using Theorem~\ref{TDTFF} and Theorem~\ref{TDnStructure} we can completely characterize all $2$-fibred functions which generate the finite $2$-digit tables given in Table~\ref{TDTEx}. If $\FFf{F}\in\SoFibredFunctions{2}(\FFPdomain{\Nz},\FFPclosed)$ then
\begingroup
\allowdisplaybreaks
\begin{align}
\FFFDT(\FFf{F})\llbracket\intintuptoex{3}\rrbracket=\STf{D}_1&\Leftrightarrow\fa n\in\Nz:\FFf{F}(n)\in M_{n\modulo8}+8\Nz\text{ where}\\
\nonumber
&\phantom{\vphantom{a}\Leftrightarrow\vphantom{a}}M_0\ce\set{0,2},\;M_1\ce\set{0,2},\;M_2\ce\set{0,2},\;M_3\ce\set{0,2},\\
\nonumber
&\phantom{\vphantom{a}\Leftrightarrow\vphantom{a}}M_4\ce\set{1,3,7},\;M_5\ce\set{1,3,7},\;M_6\ce\set{1,3,7},\;M_7\ce\set{0,2}\\
\FFFDT(\FFf{F})\llbracket\intintuptoex{3}\rrbracket=\STf{D}_2&\Leftrightarrow\fa n\in\Nz:\FFf{F}(n)\in M_{n\modulo8}+8\Nz\text{ where}\\
\nonumber
&\phantom{\vphantom{a}\Leftrightarrow\vphantom{a}}M_0\ce\set{0,2},\;M_1\ce\set{4,6},\;M_2\ce\set{4,6},\;M_3\ce\set{1,5},\\
\nonumber
&\phantom{\vphantom{a}\Leftrightarrow\vphantom{a}}M_4\ce\set{1,5},\;M_5\ce\set{0,2},\;M_6\ce\set{3,7},\;M_7\ce\set{3,7}\\
\FFFDT(\FFf{F})\llbracket\intintuptoex{3}\rrbracket=\STf{D}_3&\Leftrightarrow\fa n\in\Nz:\FFf{F}(n)\in m_{n\modulo8}+4\Nz\text{ where}\\
\nonumber
&\phantom{\vphantom{a}\Leftrightarrow\vphantom{a}}m_0\ce0,\;m_1\ce0,\;m_2\ce1,\;m_3\ce1\\
\nonumber
&\phantom{\vphantom{a}\Leftrightarrow\vphantom{a}}m_4\ce2,\;m_5\ce2,\;m_6\ce3,\;m_7\ce3.
\end{align}
\endgroup

The rows of the $\FFf{F}$-digit table of a closed $p$-fibred function $\FFf{F}$ are computed by iterative application of $\FFf{F}$. Actually performing these iterations can be a very difficult task if $\FFf{F}$ is complicated or highly expansive. The following corollary provides a useful method to actually compute $p$-digit tables of $p$-fibred functions that at least have the weak block property.

\begin{corollary}
\label{CComputeDT}
Let $2\leq p\in\N$, $k\in\N$, $\FFf{F},\FFf{G}_1,\FFf{G}_2,\FFf{H}_1,\FFf{H}_2\in\SoFibredFunctions{p}(\FFPclosed)$ of equal domain, and $R,S,T\subseteq\Z_p$ CRSs modulo $p^{k+1}$, $p^k$, and $p^k$ respectively, such that
\begin{align}
\FFf{G}_1[r](n)&=R(\FFf{F}[r](n))\\
\FFf{G}_2[r](n)&=R(\FFf{F}[r](T(n)))\\
\FFf{H}_1[r](n)&=S(\FFf{F}[r](n))\\
\FFf{H}_2[r](n)&=S(\FFf{F}[r](T(n)))
\end{align}
for all $r\in\intintuptoex{p}$ and all $n\in(r+p\Z_p)\cap\FFFdomain(\FFf{F})$, and
\begin{align}
R'&\ce\set{\frac{r-r\modulo p}{p}\mid r\in R},&
S'&\ce\set{\frac{s-s\modulo p}{p}\mid s\in S}.
\end{align}
Then,

\begin{theoremtable}
\theoremitem{(1)}&$\FFPboundedby{R'}(\FFf{G}_1)$, $\FFPboundedby{R'}(\FFf{G}_2)$, $\FFPboundedby{S'}(\FFf{H}_1)$, $\FFPboundedby{S'}(\FFf{H}_2)$\tabularnewline
\theoremitem{(2)}&$\FFPweakblockat{k}(\FFf{F})\Rightarrow$\tabularnewline
&\quad$\FFFDT(\FFf{F})\llbracket\intintuptoex{k}\rrbracket=\FFFDT(\FFf{G}_1)\llbracket\intintuptoex{k}\rrbracket$\tabularnewline
&\quad$\FFFDT(\FFf{F})\llbracket\intintuptoex{k}\rrbracket=\FFFDT(\FFf{G}_2)\llbracket\intintuptoex{k}\rrbracket,\;\fa n\in\FFFdomain(\FFf{F}):\FFf{G}_2(n)=\FFf{G}_2(T(n))$\tabularnewline
&$\FFPweakblockat{\set{k,k-1}}(\FFf{F})\Rightarrow$\tabularnewline
&\quad$\FFFDT(\FFf{F})\llbracket\intintuptoex{k}\rrbracket=\FFFDT(\FFf{H}_1)\llbracket\intintuptoex{k}\rrbracket$\tabularnewline
&\quad$\FFFDT(\FFf{F})\llbracket\intintuptoex{k}\rrbracket=\FFFDT(\FFf{H}_2)\llbracket\intintuptoex{k}\rrbracket,\;\fa n\in\FFFdomain(\FFf{F}):\FFf{H}_2(n)=\FFf{H}_2(T(n))$\tabularnewline
&$\FFPblockat{k}(\FFf{F})\land\FFPweakblockat{k-1}(\FFf{F})$ and $S'$ CRS modulo $p^{k-1}$ $\Rightarrow$\tabularnewline
&\quad$\fa n\in\FFFdomain(\FFf{F}):\FFf{H}_1(n)=\FFf{H}_1(T(n))=\FFf{H}_2(n)$.
\end{theoremtable}
\end{corollary}

\noindent
We will prove a slightly stronger version of the last statement of (2) in a later part of the paper (Corollary~\ref{CComputeDTStrong}~(3)).

As an example we consider $p=3$, $k=5$, $R=\intintuptoex{3^6}$, $S=T=\intintuptoex{3^5}$, and
\begin{align}
\FFf{F}&=(7x^3-4x^2+x-6,3x^7-x+1,x^2+6x+2)
\end{align}
with $\FFFdomain(\FFf{F})=\Z_3$. Then, $\FFPblock(\FFf{F})$ (as we will prove later, Corollary~\ref{CPropPolyF}~(2)) and
\begin{align}
\FFf{H}_1(n)&=(\FFf{F}[0](n)\phantom{\modulo3^5}\modulo3^5,\FFf{F}[1](n)\phantom{\modulo3^5}\modulo3^5,\FFf{F}[2](n)\phantom{\modulo3^5}\modulo3^5)\\
\FFf{H}_2(n)&=(\FFf{F}[0](n\modulo3^5)\modulo3^5,\FFf{F}[1](n\modulo3^5)\modulo3^5,\FFf{F}[2](n\modulo3^5)\modulo3^5)
\end{align}
for all $n\in\Z_3$. As claimed by the corollary, we get $\FFf{H}_1(n)=\FFf{H}_1(n\modulo3^5)=\FFf{H}_2(n)\in\intintuptoex{3^4}$ for all $n\in\Z_3$ and
\begin{align}
\FFFDT(\FFf{F})[17][\intintuptoex{5}]=\FFFDT(\FFf{H}_1)[17][\intintuptoex{5}]=(2,2,1,2,0)
\end{align}
but $\FFFST(\FFf{H}_1)[17][4]=63$ while $\FFFST(\FFf{F})[17][4]=2.51041\ldots\cdot10^{52}$.

\begin{proof}[Proof of Corollary~\ref{CComputeDT}]
$ $\\
\proofitem{(1)}%
Let $n\in\FFFdomain(\FFf{F})$. Then,
\begin{align}
\FFf{G}_1(n)
&=\frac{\FFf{G}_1[n\modulo p](n)-\FFf{G}_1[n\modulo p](n)\modulo p}{p}
=\frac{R(\FFf{F}[n\modulo p](n))-R(\FFf{F}[n\modulo p](n))\modulo p}{p}\in R'.
\end{align}
The remaining statements can be proven analogously.

\noindent
\proofitem{(2)}%
We will prove
\begin{align}
\FFf{G}_1(n),\FFf{G}_2(n),\FFf{H}_1(n),\FFf{H}_2(n)\in\FFFDT(\FFf{F})\llbracket\intintuptoex{k}\rrbracket(n)
\end{align}
for all $n\in\FFFdomain(\FFf{F})$ from which it follows by Theorem~\ref{TDTFF} that
\begin{align}
\label{EEqDT}
\FFFDT(\FFf{F})\llbracket\intintuptoex{k}\rrbracket=\FFFDT(\FFf{G}_1)\llbracket\intintuptoex{k}\rrbracket=\FFFDT(\FFf{G}_2)\llbracket\intintuptoex{k}\rrbracket=\FFFDT(\FFf{H}_1)\llbracket\intintuptoex{k}\rrbracket=\FFFDT(\FFf{H}_2)\llbracket\intintuptoex{k}\rrbracket.
\end{align}
Let $n\in\FFFdomain(\FFf{F})$. Then,
\begin{align}
\FFf{G}_1(n)
&=\frac{\FFf{G}_1[n\modulo p](n)-\FFf{G}_1[n\modulo p](n)\modulo p}{p}
=\frac{R(\FFf{F}[n\modulo p](n))-R(\FFf{F}[n\modulo p](n))\modulo p}{p}\\
&\equiv\frac{\FFf{F}[n\modulo p](n)-\FFf{F}[n\modulo p](n)\modulo p}{p}=\FFf{F}(n)\modulus{p^k}.
\end{align}
In addition, $\FFf{F}(n)\in\FFFDT(\FFf{F})\llbracket\intintuptoex{k}\rrbracket(n)$ and thus $\FFf{G}_1(n)\in\FFFDT(\FFf{F})\llbracket\intintuptoex{k}\rrbracket(n)$ by $\FFPweakblockat{k}(\FFf{F})$ and by (1) of the ``In particular'' part of Theorem~\ref{TDnStructure}~(1).

Analogously,
\begin{align}
\FFf{G}_2(n)
&=\frac{\FFf{G}_2[n\modulo p](n)-\FFf{G}_2[n\modulo p](n)\modulo p}{p}
=\frac{R(\FFf{F}[n\modulo p](T(n)))-R(\FFf{F}[n\modulo p](T(n)))\modulo p}{p}\\
&\equiv\frac{\FFf{F}[n\modulo p](T(n))-\FFf{F}[n\modulo p](T(n))\modulo p}{p}=\FFf{F}(T(n))\modulus{p^k}.
\end{align}
In addition, $\FFf{F}(T(n))\in\FFFDT(\FFf{F})\llbracket\intintuptoex{k}\rrbracket(T(n))=\FFFDT(\FFf{F})\llbracket\intintuptoex{k}\rrbracket(n)$ by (2) of the ``In particular'' part of Theorem~\ref{TDnStructure}~(1) which again implies $\FFf{G}_2(n)\in\FFFDT(\FFf{F})\llbracket\intintuptoex{k}\rrbracket(n)$.

It can be shown in a completely analogous fashion that
\begin{align}
\FFf{H}_1(n)&\equiv\FFf{F}(n)\modulus{p^{k-1}}\\
\FFf{H}_2(n)&\equiv\FFf{F}(T(n))\modulus{p^{k-1}}
\end{align}
and hence $\FFf{H}_1(n),\FFf{H}_2(n)\in\FFFDT(\FFf{F})\llbracket\intintuptoex{k}\rrbracket(n)$  by $\FFPweakblockat{\set{k,k-1}}(\FFf{F})$ and by the ``In particular'' part of Theorem~\ref{TDnStructure}~(2) which completes the proof of Eqn.~(\ref{EEqDT}).

Clearly, $\FFf{G}_2(n)=\FFf{G}_2(T(n))$ and $\FFf{H}_1(T(n))=\FFf{H}_2(n)=\FFf{H}_2(T(n))$. We are thus left to show that $\FFf{H}_1(n)=\FFf{H}_1(T(n))$ if $\FFPblockat{k}(\FFf{F})$, $\FFPweakblockat{k-1}(\FFf{F})$, and $S'$ is a CRS modulo $p^{k-1}$. We observe that $\FFf{H}_1(T(n))\equiv\FFf{F}(T(n))\modulus{p^{k-1}}$ and thus $\FFf{H}_1(T(n))\in\FFFDT(\FFf{F})\llbracket\intintuptoex{k}\rrbracket(n)$. But then $\FFf{H}_1(T(n))\equiv\FFf{H}_1(n)\modulus{p^{k-1}}$ by Theorem~\ref{TDnStructure}~(4), and since both $\FFf{H}_1(T(n))$ and $\FFf{H}_1(n)$ are in $S'$ by (1), they must be equal.
\end{proof}

Now that we know exactly how to generate a finite $p$-digit table using a $p$-fibred function, we are left with dealing with the infinite case which will finally establish the relation between $p$-adic systems and infinite $p$-digit tables with block property and prove the claimed existence of a one-to-one correspondence. We continue with two basic lemmas on infinite $p$-digit tables with block property and one corollary on $p$-adic systems.

\begin{lemma}
\label{LInfBlDTEq}
Let $2\leq p\in\N$, $\STf{D}\in\SoDigitTables{p}(\neg\STPfinite,\DTPblock)$, and $m,n\in\STFdomain(\STf{D})$. Then, $m=n$ if and only if $\STf{D}[m]=\STf{D}[n]$.
\end{lemma}

\begin{proof}
\begin{align}
m=n\Leftrightarrow\fa k\in\N:m\equiv n\modulus{p^k}\Leftrightarrow\fa k\in\N:\STf{D}[m][\intintuptoex{k}]=\STf{D}[n][\intintuptoex{k}]\Leftrightarrow\STf{D}[m]=\STf{D}[n].
\end{align}
\end{proof}

\noindent
From the previous lemma can be derived the following useful corollary on $p$-adic systems.

\begin{corollary}
\label{CInitPer}
Let $2\leq p\in\N$, $\FFf{F}\in\SoSystems{p}$, and $n\in\Z_p$. Then, $\abs{\SFinitial{\FFFST(\FFf{F})[n]}}=\abs{\SFinitial{\FFFDT(\FFf{F})[n]}}$ and $\abs{\SFperiodic{\FFFST(\FFf{F})[n]}}=\abs{\SFperiodic{\FFFDT(\FFf{F})[n]}}$. In particular, $\SPperiodic(\FFFST(\FFf{F})[n])\Leftrightarrow\SPperiodic(\FFFDT(\FFf{F})[n])$, $\SPultimatelyperiodic(\FFFST(\FFf{F})[n])\Leftrightarrow\SPultimatelyperiodic(\FFFDT(\FFf{F})[n])$, and $\SPaperiodic(\FFFST(\FFf{F})[n])\Leftrightarrow\SPaperiodic(\FFFDT(\FFf{F})[n])$.
\end{corollary}

\begin{proof}
Let $k,\ell\in\Nz$. Then,
\begin{align}
\FFFST(\FFf{F})[n][k+\ell,\infty]=\FFFST(\FFf{F})[n][k,\infty]
&\Leftrightarrow\FFFST(\FFf{F})[\FFf{F}^{k+\ell}(n)]=\FFFST(\FFf{F})[\FFf{F}^k(n)]\\
&\Leftrightarrow\FFf{F}^{k+\ell}(n)=\FFf{F}^k(n)\\
\text{(by Lemma~\ref{LInfBlDTEq})}&\Leftrightarrow\FFFDT(\FFf{F})[\FFf{F}^{k+\ell}(n)]=\FFFDT(\FFf{F})[\FFf{F}^k(n)]\\
&\Leftrightarrow\FFFDT(\FFf{F})[n][k+\ell,\infty]=\FFFDT(\FFf{F})[n][k,\infty].
\end{align}
If $k\ce\abs{\SFinitial{\FFFST(\FFf{F})[n]}}$ and $\ell\ce\abs{\SFperiodic{\FFFST(\FFf{F})[n]}}$, then $\FFFST(\FFf{F})[n][k+\ell,\infty]=\FFFST(\FFf{F})[n][k,\infty]$ and hence $\FFFDT(\FFf{F})[n][k+\ell,\infty]=\FFFDT(\FFf{F})[n][k,\infty]$. Thus,
\begin{align}
\abs{\SFinitial{\FFFDT(\FFf{F})[n]}}&\leq\abs{\SFinitial{\FFFST(\FFf{F})[n]}}\\
\abs{\SFperiodic{\FFFDT(\FFf{F})[n]}}&\leq\abs{\SFperiodic{\FFFST(\FFf{F})[n]}}.
\end{align}
Analogously, if $k\ce\abs{\SFinitial{\FFFDT(\FFf{F})[n]}}$ and $\ell\ce\abs{\SFperiodic{\FFFDT(\FFf{F})[n]}}$, then $\FFFDT(\FFf{F})[n][k+\ell,\infty]=\FFFDT(\FFf{F})[n][k,\infty]$ and hence $\FFFST(\FFf{F})[n][k+\ell,\infty]=\FFFST(\FFf{F})[n][k,\infty]$. Thus,
\begin{align}
\abs{\SFinitial{\FFFST(\FFf{F})[n]}}&\leq\abs{\SFinitial{\FFFDT(\FFf{F})[n]}}\\
\abs{\SFperiodic{\FFFST(\FFf{F})[n]}}&\leq\abs{\SFperiodic{\FFFDT(\FFf{F})[n]}}.
\end{align}
\end{proof}

\noindent
The following lemma is an analogue to the finite case treated in Lemma~\ref{LFinBlDTCompl}.
 
\begin{lemma}
\label{LInfBlDTCompl}
Let $2\leq p\in\N$, $\STf{D}\in\SoTables{p}$, and $\Sf{D}\in\CoSequences(\SPboundedby{\intintuptoex{p}},\neg\SPfinite)$. Then, there is a unique $n\in\Z_p$ such that $\STf{D}[n]=\Sf{D}$. In particular, $\abs{\STf{D}(n)}=1$ for every $n\in\Z_p$.
\end{lemma}

\begin{proof}
For every $i\in\N$ let $r_i$ be the unique (by Lemma~\ref{LFinBlDTCompl}) element of $\intintuptoex{p^i}$ such that $\STf{D}[r_i][\intintuptoex{i}]=\Sf{D}[\intintuptoex{i}]$. Furthermore, let $n_0\ce r_1\in\intintuptoex{p}$, $n_i\ce(r_{i+1}-r_i)/p^i\in\intintuptoex{p}$ for all $i\in\N$, and $n\ce\sum_{i=0}^\infty n_ip^i\in\Z_p$. Then, $n\equiv\sum_{i=0}^{k-1} n_ip^i=s_k\modulus{p^k}$ and hence $\STf{D}[n][\intintuptoex{k}]=\STf{D}[s_k][\intintuptoex{k}]=\Sf{D}[\intintuptoex{k}]$ for all $k\in\N$. The uniqueness of $n$ follows directly from Lemma~\ref{LInfBlDTEq}.
\end{proof}

\noindent
Using the above lemmas we are finally able to establish the claimed relation between $p$-adic systems and infinite $p$-digit tables with block property.

\begin{theorem}
\label{TASeqDTbl}
Let $2\leq p\in\N$ and $\STf{D}\in\SoTables{p}$. Then, there is a unique $\FFf{F}_\STf{D}\in\SoFibredFunctions{p}(\FFPcanonicalform)$ such that $\FFFDT(\FFf{F}_\STf{D})=\STf{D}$. In particular, $\FFf{F}_\STf{D}$ is a $p$-adic system.
\end{theorem}

\begin{proof}
From Lemma~\ref{LInfBlDTCompl} we get $\abs{\STf{D}(n)}=1$ for every $n\in\Z_p$. Let $\FFf{F}_\STf{D}\in\SoFibredFunctions{p}$ with
\begin{align}
\label{DFD}
\FFf{F}_\STf{D}[r]:\Z_p&\to\Z_p\\
\nonumber
n&\mapsto
\begin{cases}
p\STf{D}(n)&\text{if }n\equiv r\modulus{p}\text{\quad(cf. Eqn.~(\ref{EIdentify}))}\\
0&\text{if }n\not\equiv r\modulus{p}
\end{cases}
\end{align}
for all $r\in\intintuptoex{p}$. Then, $\FFPcanonicalform(\FFf{F}_\STf{D})$ and $\FFFDT(\FFf{F}_\STf{D})=\STf{D}$ by Theorem~\ref{TDTFF}. Furthermore, $\FFf{F}_\STf{D}$ is uniquely defined by this property by Lemma~\ref{LFFCanForm} and Theorem~\ref{TDTFF}.
\end{proof}

\noindent
Theorem~\ref{TASeqDTbl} finally proves that there is a one-to-one correspondence between $p$-adic systems (modulo $\sim_p$) and infinite $p$-digit tables with block property given by
\begin{align}
\SoSystems{p}\slash_{\sim_p}&\leftrightarrow\SoTables{p}.\\
\nonumber
\FFf{F}&\mapsto\FFFDT(\FFf{F})\\
\nonumber
\FFf{F}_\STf{D}&\mapsfrom\STf{D}
\end{align}
In this sense we might as well define a $p$-adic system to be an infinite $p$-digit table with domain $\Z_p$ and block property and consider the corresponding $p$-fibred function to be its feature instead of interpreting it the other way around. The decision to choose the dynamical interpretation over the statical one in the definition (and thus fix a certain mindset) is somewhat arbitrary, but will be explained to some extent in the upcoming section. Nevertheless, we may choose to go back and forth between both interpretations if certain things are easier to see or prove in one setting or the other.

Considering the above theorem, a natural question to ask is whether all $p$-digit tables that only have the weak block property but not the the block property, can also be realized as the $p$-digit table of a closed $p$-fibred function. That this is not the case, is proven by the following counter-example.

\begin{example}
\label{EWBlockF}
Let $\STf{D}\ce((n\equiv0\modulus{2}\;?\;(0)^\infty:(1,0)\cdot(1)^\infty))_{n\in\Z_2}$. Then, $\DTPweakblock(\STf{D})$ but, since $\STf{D}(1)=\emptyset$, Theorem~\ref{TDTFF} implies that there is no closed $p$-fibred function $\FFf{F}$ such that $\FFFDT(\FFf{F})=\STf{D}$.
\end{example}

As a first example of how the relation between $p$-adic systems and $p$-digit tables can be exploited, we show that any infinite $p$-digit table with block property whose domain is dense in $\Z_p$, can be extended to $\Z_p$ in a unique way such that the block property still holds. This implies an equal statement for $p$-adic systems: any $p$-fibred function with block property whose domain is dense in $\Z_p$, has a unique extension to a $p$-adic system (i.e. its domain can be extended to $\Z_p$ such that the block property is preserved).

\begin{lemma}
\label{LEquDT}
Let $2\leq p\in\N$ and $\STf{D},\STf{E}\in\SoDigitTables{p}(\DTPweakblock)$ of equal domain such that $\STf{D}\vert_A=\STf{E}\vert_A$ for some set $A\subseteq\STFdomain(\STf{D})$ that is dense in $\Z_p$. Then, $\STf{D}=\STf{E}$.
\end{lemma}

\begin{proof}
Assume to the contrary that $\STf{D}\neq\STf{E}$. Then there is an $n\in\STFdomain(\STf{D})$ and a $k\in\Nz$ such that $\STf{D}[n][k]\neq\STf{E}[n][k]$. Since $A$ is dense in $\Z_p$, there is an $N\in A$ with $N\equiv n\modulus{p^{k+1}}$. Then, $\DTPweakblock(\STf{D})$ and $\DTPweakblock(\STf{E})$ implies
\begin{align}
\STf{D}[n][\intintuptoin{k}]&=\STf{D}[N][\intintuptoin{k}]=\STf{E}[N][\intintuptoin{k}]=\STf{E}[n][\intintuptoin{k}].
\end{align}
In particular, $\STf{D}[n][k]=\STf{E}[n][k]$, which is a contradiction.
\end{proof}

\begin{lemma}
\label{LDTExt}
Let $2\leq p\in\N$, $\STf{D}\in\SoDigitTables{p}(\DTPweakblock)$ such that $\STFdomain(\STf{D})$ is dense in $\Z_p$, and let $\STf{E}\in\SoDigitTables{p}(\STPdomain{\Z_p})$ be defined in the following way: for $n\in\Z_p$ and $k\in\Nz$ let $\STf{E}[n][k]\ce\STf{D}[N][k]$ where $N\in\STFdomain(\STf{D})$ such that $N\equiv n\modulus{p^{k+1}}$ (well-defined due to $\DTPweakblock(\STf{D})$). Then, $\DTPweakblock(\STf{E})$ and $\STf{E}\vert_{\STFdomain(\STf{D})}=\STf{D}$ and $\STf{E}$ is uniquely defined by this property. Furthermore, if $\DTPblock(\STf{D})$, then $\DTPblock(\STf{E})$.
\end{lemma}

\begin{proof}
Uniqueness follows directly from Lemma~\ref{LEquDT}.

Let $n\in\Z_p$, $k\in\N$, and $N\in\STFdomain(\STf{D})$ such that $N\equiv n\modulus{p^k}$. Then $N\equiv n\modulus{p^{i+1}}$ and hence $\STf{E}[n][i]=\STf{D}[N][i]$ for all $i\in\intintuptoex{k}$ which implies $\STf{E}[n][\intintuptoex{k}]=\STf{D}[N][\intintuptoex{k}]$. If $n\in\STFdomain(\STf{D})$, this further implies that $\STf{E}[n][\intintuptoex{k}]=\STf{D}[N][\intintuptoex{k}]=\STf{D}[n][\intintuptoex{k}]$ and we conclude $\STf{E}\vert_{\STFdomain(\STf{D})}=\STf{D}$.

Now let $m,n\in\Z_p$, $k\in\N$ with $m\equiv n\modulus{p^k}$, and $M\in\STFdomain(\STf{D})$ such that $M\equiv m\modulus{p^k}$. Then, $M\equiv n\modulus{p^k}$ and $\STf{E}[m][\intintuptoex{k}]=\STf{D}[M][\intintuptoex{k}]=\STf{E}[n][\intintuptoex{k}]$. Therefore, $\DTPweakblock(\STf{E})$.

Finally, assume that $\DTPblock(\STf{D})$ and let $m,n\in\Z_p$, $k\in\N$ with $\STf{E}[m][\intintuptoex{k}]=\STf{E}[n][\intintuptoex{k}]$, and $M,N\in\STFdomain(\STf{D})$ such that $M\equiv m\modulus{p^k}$ and $N\equiv n\modulus{p^k}$. Then,
\begin{align}
\STf{D}[M][\intintuptoex{k}]&=\STf{E}[m][\intintuptoex{k}]=\STf{E}[n][\intintuptoex{k}]=\STf{D}[N][\intintuptoex{k}],
\end{align}
and since $\DTPblock(\STf{D})$, this implies that $m\equiv M\equiv N\equiv n\modulus{p^k}$. Therefore, $\DTPblock(\STf{E})$.
\end{proof}

\begin{corollary}
\label{CFFwbDTExt}
Let $\FFf{F}\in\SoFibredFunctions{p}(\FFPweakblock)$ with $\FFFdomain(\FFf{F})$ dense in $\Z_p$, $\FFf{G}\in\SoFibredFunctions{p}(\FFPdomain{\Z_p},\FFPweakblock)$ such that $\FFf{G}\vert_{\FFFdomain(\FFf{F})}=\FFf{F}$, and $\STf{E}\in\SoDigitTables{p}(\STPdomain{\Z_p},\DTPweakblock)$ such that $\STf{E}\vert_{\FFFdomain(\FFf{F})}=\FFFDT(\FFf{F})$ (cf. Lemma~\ref{LDTExt}). Then, $\STf{E}=\FFFDT(\FFf{G})$.
\end{corollary}

\begin{proof}
Follows directly from Lemma~\ref{LDTExt}.
\end{proof}

By Lemma~\ref{LEquDT} and Lemma~\ref{LDTExt} a $p$-digit table with weak block property is uniquely defined by its restriction to any dense subset of its domain. Because of the nature of the weak block property, we can even go one step further and drop more redundant information to gain a minimal representation of a given $p$-digit table with weak block property that still allows to recover the full table. Looking at the examples given in Table~\ref{TSTDTEx}, one can see that due to the repetition of blocks the full table can be reconstructed from the sequence gained by concatenating the rightmost rows in each block. Formally, if $2\leq p\in\N$ and $k\in\N\cup\set{\infty}$, then there is a bijection between the set of all $p$-digit tables of length $k$ that have domain $\Z_p$ and the weak block property and the set of all $\intintuptoex{p}$-bounded sequences of length $(p^{k+1}-1)/(p-1)-1$ with prefix $(0,\ldots,p-1)$ given by
\begin{align}
\label{EBijDTSeq}
\SoDigitTables{p}(\STPdomain{\Z_p},\STPlength{k},\DTPweakblock)&\leftrightarrow\CoSequences(\SPboundedby{\intintuptoex{p}},\SPlength{(p^{k+1}-1)/(p-1)-1},\SPprefix{(0,\ldots,p-1)})\\
\nonumber
\STf{D}&\mapsto\prod_{\ell=0}^{k-1}\prod_{n=0}^{p^{\ell+1}-1}(\STf{D}[n][\ell])\\
\nonumber
\left(\left(\Sf{D}\left[\frac{p^{\ell+1}-1}{p-1}-1+n\modulo p^{\ell+1}\right]\right)_{\ell\in\intintuptoex{k}}\right)_{n\in\Z_p}&\mapsfrom\Sf{D}
\end{align}
where the product denotes a product of one-element sequences (i.e. their concatenation). In this interpretation the $p$-digit tables with weak block property correspond exactly to the $\intintuptoex{p}$-bounded sequences of a certain length (and a specific prefix necessary for technical reasons). The question arises how the stronger block property of $p$-digit tables (which makes them $p$-adic systems after all) translates to the corresponding sequences. In order to answer this question, we define the following two predicates on $\CoSequences$ ($2\leq p\in\N$, $k\in\N\cup\set{\infty}$, $\Sf{S}\in\CoSequences$):
\begin{flalign}
\pushleft{\SPweakblock{p}{k}(\Sf{S})}&\Leftrightarrow\SPboundedby{\intintuptoex{p}}(\Sf{S})&\hspace{-15em}\text{$\Sf{S}$ has the \emph{weak $(p,k)$-block property}}\label{DSPweakblock}\\
\nonumber
&\phantom{\vphantom{a}\Leftrightarrow\vphantom{a}}\SPlength{(p^{k+1}-1)/(p-1)-1}(\Sf{S})\\
\nonumber
&\phantom{\vphantom{a}\Leftrightarrow\vphantom{a}}\SPprefix{(0,\ldots,p-1)}(\Sf{S})\\
\pushleft{\SPblock{p}{k}(\Sf{S})}&\Leftrightarrow\SPweakblock{p}{k}(\Sf{S})&\hspace{-13em}\text{$\Sf{S}$ has the \emph{$(p,k)$-block property}}\label{DSPblock}\\
\nonumber
&\phantom{\vphantom{a}\Leftrightarrow\vphantom{a}}\fa\ell\in\intintuptoex{k}:\fa n\in(p^{\ell+1}-1)/(p-1)-1+\intintuptoex{p^\ell}:\set{\Sf{S}[n+ip^\ell]\mid i\in\intintuptoex{p}}=\intintuptoex{p}
\end{flalign}

\noindent
A quick check of the definitions shows that a sequence with the $(p,k)$-block property corresponds to (by the above mapping) a $p$-digit table that has the block property and vice versa. Therefore, the restriction of the above mapping also defines a bijection between $\SoDigitTables{p}(\STPdomain{\Z_p},\STPlength{k},\DTPblock)$ and $\CoSequences(\SPblock{p}{k})$. This is true in particular for $k=\infty$, in which case we get a bijection between $\SoTables{p}$ and $\CoSequences(\SPblock{p}{\infty})$.

As an example we consider the well-known Thue-Morse sequence\label{DThueMorse}
\begin{align}
\Sf{T}\ce(0,1,1,0,1,0,0,1,1,0,0,1,0,1,1,0,1,0,0,1,0,1,1,0,0,1,1,0,1,0,0,1,\ldots)
\end{align}
which, according to its Wikipedia article, is ``obtained by starting with $0$ and successively appending the Boolean complement of the sequence obtained thus far''. By slight modification of its beginning it is actually possible to make it satisfy the $(2,\infty)$-block property and thus define a $2$-adic system. For that we set
\begin{align}
\Sf{S}&\ce(0,1)\cdot\Sf{T}[4,\infty]\\
&\phantom{\ce}\mathllap{=}\;\;(0,1,1,0,0,1,1,0,0,1,0,1,1,0,1,0,0,1,0,1,1,0,0,1,1,0,1,0,0,1,\ldots).
\end{align}
The corresponding $2$-adic system which generates this sequence can be found easily using the results of this subsection and is given by
\begin{align}
\FFf{F}&\ce(x+6-2(x\modulo8),x+3-2(x\modulo8)+2(x\modulo4)).
\end{align}
It can be readily verified that $\Sf{S}$ is indeed equal to the sequence obtained from $\Sf{D}=\FFFDT(\FFf{F})$ when using the bijection defined in Eqn.~(\ref{EBijDTSeq}). Despite best efforts, the author was unable to find any reference to this method for defining the Thue-Morse sequence in the literature. Furthermore, it appears worth noting that the corresponding sequence of the even simpler $2$-adic system
\begin{align}
\FFf{G}\ce(x,x+3-2(x\modulo4))
\end{align}
coincides with $1-T[2,\infty]$, the truncated Boolean complement of the Thue-Morse sequence.

\myparagraphtoc{Permutations of $p$-adic integers that respect congruence modulo powers of $p$.}
The last interpretation of $p$-adic systems we give in this section is that of permutations of $\Z_p$. Clearly, every $p$-adic system $\FFf{F}$ defines a bijection between $\Z_p$ and $\CoSequences(\SPboundedby{\intintuptoex{p}},\neg\SPfinite)$ by its $\FFf{F}$-digit table. We define
\begin{align}
\label{Dpsi}
\psi_\FFf{F}:\Z_p&\to\CoSequences(\SPboundedby{\intintuptoex{p}},\neg\SPfinite).\\
\nonumber
n&\mapsto\FFFDT(\FFf{F})[n]
\end{align}
If one interprets an infinite sequence with entries in $\intintuptoex{p}$ as the usual base $p$ expansion of a $p$-adic integer, then $\psi_\FFf{F}$ also defines a permutation of $\Z_p$. This is a special case of the following idea: for two $p$-adic systems $\FFf{F}$ and $\FFf{G}$ let
\begin{align}
\label{Dpi}
\pi_{\FFf{F},\FFf{G}}\ce{\psi_\FFf{G}}^{-1}\circ\psi_\FFf{F}:\Z_p&\to\Z_p.
\end{align}
Then $\pi_{\FFf{F},\FFf{G}}$ clearly defines a permutation of $\Z_p$. Interpreting the infinite sequences that $\psi_\FFf{F}$ yields as the usual base $p$ expansions, corresponds to the choice \label{DFp}$\FFf{G}=\FFf{F}_p=(x)^p=(x,\ldots,x)$ (cf. Eqn.~(\ref{EBinary})). As an example consider $\FFf{F}\ce\FFf{F}_C$ and $\FFf{G}\ce\FFf{F}_2$. Then we get $\FFFDT(\FFf{F})[1]=(1,0)^\infty=\FFFDT(\FFf{G})[-1/3]$ and hence $\pi_{\FFf{F},\FFf{G}}(1)=-1/3$.

Several properties of permutations of the form $\pi_{\FFf{F},\FFf{G}}$ are summarized in the following lemma.

\begin{lemma}
\label{LPermProp}
Let $2\leq p\in\N$, $\FFf{F}$ and $\FFf{G}$ two $p$-adic systems, $\pi\ce\pi_{\FFf{F},\FFf{G}}$, and $k\in\Nz$. Then,

\begin{theoremtable}
\theoremitem{(1)}&$\fa m,n\in\Z_p:m\equiv n\modulus{p^k}\Leftrightarrow\pi(m)\equiv\pi(n)\modulus{p^k}$\tabularnewline
&In particular: $\pi$ is measure preserving (and thus continuous) and it induces a permutation $\pi_k$ of $\Z_p\slash p^k\Z_p$ by $[n]\mapsto [\pi(n)]$\tabularnewline
\theoremitem{(2)}&$\fa m,n\in\Z_p$ with $n\equiv\pi(m)\modulus{p^k}:\fa M,N\in\Z_p$ with $N\equiv\pi(M)\modulus{p^{k+1}}:$\tabularnewline
&\quad$m\equiv M\modulus{p^k}\Rightarrow n\equiv N\modulus{p^k}$,\tabularnewline
&Note: $n\equiv\pi(m)\modulus{p^k}\Leftrightarrow n\in\pi_k([m])$ and $N\equiv\pi(M)\modulus{p^{k+1}}\Leftrightarrow N\in\pi_{k+1}([M])$\tabularnewline
\theoremitem{(3)}&$\pi^{-1}=\pi_{\FFf{G},\FFf{F}}$\tabularnewline
\theoremitem{(4)}&$\fa n\in\Z_p:\psi_\FFf{F}(\FFf{F}(n))=\psi_\FFf{F}(n)[1,\infty]$\tabularnewline
&In particular: $\fa\STf{D}\in\CoSequences(\STPboundedby{\intintuptoex{p}},\neg\STPfinite):\psi_\FFf{F}(\FFf{F}({\psi_\FFf{F}}^{-1}(\STf{D})))=\STf{D}[1,\infty]$\tabularnewline
\theoremitem{(5)}&$\fa n\in\Z_p:\pi(\FFf{F}(n))=\FFf{G}(\pi(n))$ and $\FFf{F}(\pi^{-1}(n))=\pi^{-1}(\FFf{G}(n))$.
\end{theoremtable}
\end{lemma}

\begin{proof}
$ $\\
\proofitem{(1)}%
\begin{align}
m\equiv n\modulus{p^k}&\Leftrightarrow\psi_\FFf{F}(m)[\intintuptoex{k}]=\psi_\FFf{F}(n)[\intintuptoex{k}]\;\;(\FFPblock(\FFf{F}))\\
&\Leftrightarrow\psi_\FFf{G}({\psi_\FFf{G}}^{-1}(\psi_\FFf{F}(m)))[\intintuptoex{k}]=\psi_\FFf{G}({\psi_\FFf{G}}^{-1}(\psi_\FFf{F}(n)))[\intintuptoex{k}]\\
(\FFPblock(\FFf{G}))\;&\Leftrightarrow{\psi_\FFf{G}}^{-1}(\psi_\FFf{F}(m))\equiv{\psi_\FFf{G}}^{-1}(\psi_\FFf{F}(n))\modulus{p^k}\\
&\Leftrightarrow\pi(m)\equiv\pi(n)\modulus{p^k}.
\end{align}

\noindent
\proofitem{(2)}%
From $m\equiv M\modulus{p^k}$ and (1) it follows that $n\equiv\pi(m)\equiv\pi(M)\equiv N\modulus{p^k}$.

\noindent
\proofitem{(3)}%
Follows directly from the definitions.

\noindent
\proofitem{(4)}%
For every $n\in\Z_p$ we have
\begin{align}
\psi_\FFf{F}(\FFf{F}(n))
&=\FFFDT(\FFf{F})[\FFf{F}(n)]
=\FFFDT(\FFf{F})[n][1,\infty]
=\psi_\FFf{F}(n)[1,\infty].
\end{align}

\noindent
For the ``In particular'' part set $n\ce{\psi_\FFf{F}}^{-1}(\STf{D})$. Then,
\begin{align}
\psi_\FFf{F}(\FFf{F}({\psi_\FFf{F}}^{-1}(\STf{D})))
&=\psi_\FFf{F}({\psi_\FFf{F}}^{-1}(\STf{D}))[1,\infty]
=\STf{D}[1,\infty].
\end{align}

\noindent
\proofitem{(5)}%
From (4) it follows that
\begin{align}
\psi_\FFf{G}(\pi(\FFf{F}(n)))
&=\psi_\FFf{G}({\psi_\FFf{G}}^{-1}(\psi_\FFf{F}(\FFf{F}(n))))
=\psi_\FFf{F}(n)[1,\infty]
=\psi_\FFf{G}(\FFf{G}({\psi_\FFf{G}}^{-1}(\psi_\FFf{F}(n))))\\
&=\psi_\FFf{G}(\FFf{G}(\pi(n)))
\end{align}
and hence $\pi(\FFf{F}(n))=\FFf{G}(\pi(n))$. To see $\FFf{F}(\pi^{-1}(n))=\pi^{-1}(\FFf{G}(n))$ just substitute $n$ with $\pi^{-1}(n)$.
\end{proof}

In the previous paragraph we observed that $p$-digit tables coming from $p$-adic systems have a specific property (the block property). We then tried to answer the question if every $p$-digit table with block property can be recognized as the $\FFf{F}$-digit table of a $p$-adic system $\FFf{F}$ and we found that this is indeed the case. We will now try to pursue the same strategy for permutations of $p$-adic integers, where the special property which we demand to hold is given by (1) of the previous lemma. Indeed, we call a permutation $\pi:\Z_p\to\Z_p$ of the $p$-adic integers a \emph{$p$-adic permutation} if the following two properties hold:
\begin{align}
\label{EPermA}
&\fa n\in\Z_p:\pi(n)\equiv n\modulus{p}\\
\label{EPermB}
&\fa k\in\N:\fa m,n\in\Z_p:m\equiv n\modulus{p^k}\Leftrightarrow\pi(m)\equiv\pi(n)\modulus{p^k}.
\end{align}
The \emph{set of $p$-adic permutations} shall be denoted by \label{DSoPermutations}$\SoPermutations{p}$. For every $p$-adic permutation $\pi$ and every $k\in\Nz$ we define (cf. (1) of the previous lemma)
\begin{align}
\label{Dpik}
\pi_k:\Z_p\slash p^k\Z_p&\to\Z_p\slash p^k\Z_p.\\
\nonumber
[n]&\mapsto[\pi(n)]
\end{align}
Clearly, every permutation of the form $\pi=\pi_{\FFf{F},\FFf{G}}$ where $\FFf{F}$ and $\FFf{G}$ are $p$-adic systems is a $p$-adic permutation by the previous lemma (cf. also Eqn.~(\ref{EDT})). The question is if every $p$-adic permutation can be expressed in this way and the answer is given in the following theorem.

\begin{theorem}
\label{TPermEq}
Let $2\leq p\in\N$, $\pi$ a $p$-adic permutation, and $\FFf{G}$ a $p$-adic system. Then there is a $p$-adic system $\FFf{F}$ such that $\pi=\pi_{\FFf{F},\FFf{G}}$. $\FFf{F}$ is uniquely defined up to $\sim_p$.
\end{theorem}

\begin{proof}
Let
\begin{align}
\STf{D}&\ce(\psi_{\FFf{G}}(\pi(n)))_{n\in\Z_p}.
\end{align}
We claim that $\STf{D}$ is a $p$-digit table. Then, Theorem~\ref{TASeqDTbl} implies that there is a unique (up to $\sim_p$) $p$-adic system $\FFf{F}$ such that $\STf{D}=\FFFDT(\FFf{F})$ and we thus get
\begin{align}
\psi_{\FFf{F}}(n)&=\STf{D}[n]=\psi_{\FFf{G}}(\pi(n))
\end{align}
for all $n\in\Z_p$, i.e. $\pi=\pi_{\FFf{F},\FFf{G}}$.

By Eqn.~(\ref{EPermA}) and $\FFPblock(\FFf{G})$ we get
\begin{align}
\STf{D}[n][0]&=\FFFDT(\FFf{G})[\pi(n)][0]=\FFFDT(\FFf{G})[n][0]=n\modulo p
\end{align}
for all $n\in\Z_p$. We are thus left to show $\DTPblock(\STf{D})$.

Let $k\in\N$ and $m,n\in\Z_p$. Then,
\begin{align}
m\equiv n\modulus{p^k}&\Leftrightarrow\pi(m)\equiv\pi(n)\modulus{p^k}\\
(\FFPblock(\FFf{G}))\;&\Leftrightarrow\psi_{\FFf{G}}(\pi(m))[\intintuptoex{k}]=\FFFDT(\FFf{G})[\pi(m)][\intintuptoex{k}]=\FFFDT(\FFf{G})[\pi(n)][\intintuptoex{k}]=\psi_{\FFf{G}}(\pi(n))[\intintuptoex{k}]\\
&\Leftrightarrow\STf{D}[m][\intintuptoex{k}]=\STf{D}[n][\intintuptoex{k}]
\end{align}
which completes the proof.
\end{proof}

Just as with $p$-digit tables with block property we established a one-to-one correspondence between $p$-adic permutations and $p$-adic systems. If $\FFf{G}$ is a fixed $p$-adic system then
\begin{align}
\label{Dgroupisom}
\Pi_{\FFf{G}}:\SoSystems{p}\slash_{\sim_p}&\to\SoPermutations{p}\\
\nonumber
\FFf{F}&\mapsto\pi_{\FFf{F},\FFf{G}}
\end{align}
defines a bijection.

As with $p$-digit tables this relation can also be used to see and prove certain properties of $p$-adic systems more easily as the following theorem shows.

\begin{theorem}
\label{TPermSubgroup}
Let $2\leq p\in\N$. Then the set of all $p$-adic permutations forms a subgroup of the set of all permutations of $\Z_p$ with respect to composition.
\end{theorem}

\begin{proof}
Let $\pi$, $\pi_1$ and $\pi_2$ be $p$-adic permutations. Then, it follows directly from the definition that both $\pi_2\circ\pi_1$ and $\pi^{-1}$ are also $p$-adic permutations.
\end{proof}

\noindent
We can now use $\Pi_\FFf{G}$ to transport this group structure to the set $\SoSystems{p}$ of $p$-adic systems. We define
\begin{align}
\label{Dgroupop}
\FFf{F}_1\circ_\FFf{G}\FFf{F}_2&\ce{\Pi_\FFf{G}}^{-1}(\Pi_\FFf{G}(\FFf{F}_1)\circ\Pi_\FFf{G}(\FFf{F}_2))
\end{align}
for every pair $\FFf{F}_1,\FFf{F}_2$ of $p$-adic systems which makes $\left(\SoSystems{p},\circ_\FFf{G}\right)$ a group with neutral element ${\Pi_\FFf{G}}^{-1}(\pi_{\FFf{G},\FFf{G}})=\FFf{G}$ and inverse element ${\Pi_\FFf{G}}^{-1}(\pi_{\FFf{G},\FFf{F}})$ of $\FFf{F}\in\SoSystems{p}$. Basic properties of these groups are summarized in the following lemma.

\begin{lemma}
\label{LGroupProp}
Let $2\leq p\in\N$, $\FFf{F}$, $\FFf{F}_1$, $\FFf{F}_2$, $\FFf{G}$, $\FFf{G}_1$, $\FFf{G}_2$ $p$-adic systems, $\pi$ a $p$-adic permutation, and $n\in\Z_p$. Then,

\begin{theoremtable}
\theoremitem{(1)}&$\Pi_\FFf{G}$ is an isomorphism and $\left(\SoSystems{p}\slash_{\sim_p},\circ_\FFf{G}\right)$ and $\left(\SoPermutations{p},\circ\right)$ are isomorphic\tabularnewline
&In particular: ${\Pi_{\FFf{G}_2}}^{-1}\circ\Pi_{\FFf{G}_1}$ is an isomorphism and\tabularnewline
&\leavevmode\phantom{In particular: }$\left(\SoSystems{p}\slash_{\sim_p},\circ_{\FFf{G}_1}\right)$ and $\left(\SoSystems{p}\slash_{\sim_p},\circ_{\FFf{G}_2}\right)$ are isomorphic\tabularnewline
\theoremitem{(2)}&$\Pi_\FFf{G}(\FFf{F})(n)=\pi_{\FFf{F},\FFf{G}}(n)={\psi_\FFf{G}}^{-1}(\psi_\FFf{F}(n))$\tabularnewline
\theoremitem{(3)}&${\Pi_\FFf{G}}^{-1}(\pi)(n)=\pi^{-1}(\FFf{G}(\pi(n)))$
\tabularnewline
\theoremitem{(4)}&$\FFf{F}_2\circ_\FFf{G}\FFf{F}_1(n)=\pi_{\FFf{G},\FFf{F}_1}(\FFf{F}_2(\pi_{\FFf{F}_1,\FFf{G}}(n)))$.
\end{theoremtable}
\end{lemma}

\begin{proof}
$ $\\
\proofitem{(1) and (2)}%
\qquad\quad Follow directly from the definitions.

\noindent
\proofitem{(3)}%
Let $\FFf{F}\ce{\Pi_\FFf{G}}^{-1}(\pi)$. Then, $\pi=\pi_{\FFf{F},\FFf{G}}$ and
\begin{align}
\pi^{-1}(\FFf{G}(\pi(n)))
&=\pi_{\FFf{G},\FFf{F}}(\FFf{G}(\pi_{\FFf{F},\FFf{G}}(n)))\;\text{(Lemma~\ref{LPermProp}~(3))}\\
&={\psi_\FFf{F}}^{-1}(\psi_\FFf{G}(\FFf{G}({\psi_\FFf{G}}^{-1}(\psi_\FFf{F}(n)))))\\
\text{(Lemma~\ref{LPermProp}~(4))}\;&={\psi_\FFf{F}}^{-1}(\psi_\FFf{F}(n)[1,\infty])\\
&={\psi_\FFf{F}}^{-1}(\FFFDT(\FFf{F})[n][1,\infty])
={\psi_\FFf{F}}^{-1}(\FFFDT(\FFf{F})[\FFf{F}(n)])
=\FFf{F}(n)
={\Pi_\FFf{G}}^{-1}(\pi)(n).
\end{align}

\noindent
\proofitem{(4)}%
\begin{align}
\FFf{F}_2\circ_\FFf{G}\FFf{F}_1(n)
&={\Pi_\FFf{G}}^{-1}(\Pi_\FFf{G}(\FFf{F}_2)\circ\Pi_\FFf{G}(\FFf{F}_1))(n)\\
(3)\;&=(\Pi_\FFf{G}(\FFf{F}_2)\circ\Pi_\FFf{G}(\FFf{F}_1))^{-1}(\FFf{G}(\Pi_\FFf{G}(\FFf{F}_2)\circ\Pi_\FFf{G}(\FFf{F}_1)(n)))\\
&={\Pi_\FFf{G}(\FFf{F}_1)}^{-1}({\Pi_\FFf{G}(\FFf{F}_2)}^{-1}(\FFf{G}(\Pi_\FFf{G}(\FFf{F}_2)(\Pi_\FFf{G}(\FFf{F}_1)(n)))))\\
(\text{Lemma~\ref{LPermProp}~(3)})\;&=\pi_{\FFf{G},\FFf{F}_1}(\pi_{\FFf{G},\FFf{F}_2}(\FFf{G}(\pi_{\FFf{F}_2,\FFf{G}}(\pi_{\FFf{F}_1,\FFf{G}}(n)))))\\
(\text{Lemma~\ref{LPermProp}~(5)})\;&=\pi_{\FFf{G},\FFf{F}_1}(\pi_{\FFf{G},\FFf{F}_2}(\pi_{\FFf{F}_2,\FFf{G}}(\FFf{F}_2(\pi_{\FFf{F}_1,\FFf{G}}(n)))))\\
(\text{Lemma~\ref{LPermProp}~(3)})\;&=\pi_{\FFf{G},\FFf{F}_1}(\FFf{F}_2(\pi_{\FFf{F}_1,\FFf{G}}(n))).
\end{align}
\end{proof}

\noindent
In order to see how $\circ_\FFf{G}$ operates on $\SoSystems{p}$, consider the following example.

\begin{example}
\label{EGroupProp}
Let $\FFf{G}\ce(x,x-1)$, $\FFf{F}_1\ce(x,3x+1)$, $\FFf{F}_2\ce\FFf{F}_C=(5x,x+1)$, and $n\ce5$. Then,
\begin{align}
&\psi_{\FFf{F}_1}(5)=(1,0,0)\cdot(0,1)^\infty=\psi_\FFf{G}(-13/3)\text{, hence }\pi_{\FFf{F}_1,\FFf{G}}(5)=-13/3,\\
&\FFf{F}_2(-13/3)=-5/3,\\
&\psi_\FFf{G}(-5/3)=(1,0)\cdot(0,1)^\infty=\psi_{\FFf{F}_1}(7/3)\text{, hence }\pi_{\FFf{G},\FFf{F}_1}(-5/3)=7/3.
\end{align}
Thus, Lemma~\ref{LGroupProp}~(4) implies $\FFf{F}_2\circ_\FFf{G}\FFf{F}_1(n)=7/3$.
\end{example}

We close this paragraph with an analysis of the cycle structure of the induced permutations $\pi_k$ of a $p$-adic permutation $\pi$ which leads to the definition of the \emph{tree of cycles} of $\pi$. In Section~\ref{SPermPoly} we will use this tree of cycles to prove that two classes of $p$-adic systems (polynomial $p$-adic systems and $p$-permutation polynomials) are indeed distinct. The method used there is quite general and should work for other classes that may be found in the future as well.

We define the \emph{cyclic shift function} on sequences by
\begin{align}
\label{ECycShift}
\sigma:\CoSequences\times\R&\to\CoSequences.\\
\nonumber
(\Sf{S},s)&\mapsto
\begin{cases}
\Sf{S}[\floor{s}\modulo\abs{\Sf{S}},\abs{\Sf{S}}-1]\cdot\Sf{S}[0,\floor{s}\modulo\abs{\Sf{S}}-1]&\text{if }\SPfinite(\Sf{S})\\
\Sf{S}[\floor{s},\infty]&\text{if }\neg\SPfinite(\Sf{S})
\end{cases}
\end{align}
For any permutation $\pi$ of a finite set we denote by \label{DSocycles}$\Sigma(\pi)\subseteq\CoSequences(\SPfinite)\slash_{\sim_\sigma}$ the \emph{set of cycles} of $\pi$, where
\begin{align}
\label{Dcyceq}
\Sf{S}\sim_\sigma\Sf{T}\Leftrightarrow\abs{\Sf{S}}=\abs{\Sf{T}}\land\ex s\in\Z:\sigma(\Sf{S},s)=\Sf{T}
\end{align}
for all $\Sf{S},\Sf{T}\in\CoSequences(\SPfinite)$, e.g. $(0,1,2,3)\sim_\sigma(3,0,1,2)$, since $\sigma((0,1,2,3),3)=(3,0,1,2)$. We set \label{Dcyclelength}$\abs{[\Sf{S}]_{\sim_\sigma}}\ce\abs{\Sf{S}}$ for all $[\Sf{S}]_{\sim_\sigma}\in\CoSequences(\SPfinite)\slash_{\sim_\sigma}$.

The following theorem is the basis of the definition of the tree of cycles of a $p$-adic permutation (and thus of $p$-adic systems).

\begin{theorem}
\label{TCycles}
Let $2\leq p\in\N$, $\pi\in\SoPermutations{p}$, $k\in\Nz$, and $\Sf{S}=[([a_0],\ldots,[a_{r-1}])]_{\sim_\sigma}\in\Sigma(\pi_k)$. Then there are $m\in\intint{1}{p}$, $s_0,\ldots,s_{m-1}\in\N$ with $s_n/r\in\intint{1}{p}$ for all $n\in\intintuptoex{m}$ and $\sum_{n=0}^{m-1}s_n/r=p$, and pairwise distinct $[([b_{0,0}],\ldots,[b_{0,s_0-1}])]_{\sim_\sigma},\ldots,[([b_{m-1,0}],\ldots,[b_{m-1,s_{m-1}-1}])]_{\sim_\sigma}\in\Sigma(\pi_{k+1})$ (the \emph{children of $\Sf{S}$}) such that
\begin{align}
&\fa n\in\intintuptoex{m}:\fa i\in\intintuptoex{r}:\fa j\in\intintuptoex{s_n/r}:a_i\equiv b_{n,jr+i}\modulus{p^k}.
\end{align}
In particular, $\Sigma(\pi_{k+1})$ is the disjoint union of the sets of children of all cycles in $\Sigma(\pi_k)$.
\end{theorem}

\begin{proof}
Let $b_{0,0}\in\set{a_0+ip^k\mid i\in\intintuptoex{p}}$ and $[([b_{0,0}],\ldots,[b_{0,s_0-1}])]_{\sim_\sigma}$ be the cycle of $\pi_{k+1}$ that contains $[b_{0,0}]$. Then Lemma~\ref{LPermProp}~(2) implies that $a_{i\modulo r}\equiv b_{i\modulo s_0}\modulus{p^k}$ for all $i\in\Z$ and hence $r\mid s_0$. If $s_0=rp$ we are done. Otherwise we choose $b_{1,0}\in\set{a_0+ip^k\mid i\in\intintuptoex{p}}$ such that $[b_{1,0}]$ does not occur among the entries of $[([b_{0,0}],\ldots,[b_{0,s_0-1}])]_{\sim_\sigma}$ and we consider the cycle $[([b_{1,0}],\ldots,[b_{1,s_1-1}])]_{\sim_\sigma}$ of $\pi_{k+1}$ that contains $[b_{1,0}]$ which will again satisfy $a_{i\modulo r}\equiv b_{i\modulo s_1}\modulus{p^k}$ for all $i\in\Z$ and $r\mid s_1$. After $m\leq p$ steps we have found the cycles $[([b_{0,0}],\ldots,[b_{0,s_0-1}])]_{\sim_\sigma},\ldots,[([b_{m-1,0}],\ldots,[b_{m-1,s_{m-1}-1}])]_{\sim_\sigma})$ of $\pi_{k+1}$ that have all the claimed properties.
\end{proof}

\begin{corollary}
\label{CCycleLengths}
Let $2\leq p\in\N$, $\pi\in\SoPermutations{p}$, $k\in\Nz$, and $\sigma\in\Sigma(\pi_k)$. Then the prime factors of $\abs{\sigma}$ are contained in $\intintuptoin{p}$.
\end{corollary}

\begin{proof}
Follows directly from Theorem~\ref{TCycles} and from the fact that $\pi_0:\set{\Z_p}\to\set{\Z_p}$, $[0]\mapsto[0]$ and thus $\Sigma(\pi_0)=\set{[([0])]_{\sim_\sigma}}$.
\end{proof}

As indicated above, Theorem~\ref{TCycles} implies the existence of a \emph{tree of cycles}: for any $2\leq p\in\N$ and $\pi\in\SoPermutations{p}$ let
\begin{align}
\label{Dtreeofcyclesvertices}
\mathcal{V}(\pi)&\ce\set{(k,\sigma)\mid k\in\Nz\land\sigma\in\Sigma(\pi_k)}\\
\label{Dtreeofcyclesedges}
\mathcal{E}(\pi)&\ce\setl{((k,[([a_0],\ldots,[a_r])]_{\sim_\sigma}),(\ell,[([b_0],\ldots,[b_s])]_{\sim_\sigma}))\in\mathcal{V}(\pi)^2\mid}\\
\nonumber
&\phantom{\vphantom{a}\ce\vphantom{a}}\qquad k+1=\ell\\[-1pt]
\nonumber
&\smash{\phantom{\vphantom{a}\ce\vphantom{a}}\qquad\setr{\ex i\in\Z:\fa j\in\Z:a_{j\modulo r}\equiv b_{(i+j)\modulo s}\modulus{p^k}}}\\
\label{Dtreeofcycles}
\mathcal{G}(\pi)&\ce(\mathcal{V}(\pi),\mathcal{E}(\pi))
\end{align}
and
\begin{align}
\label{Dtreeofcyclescolor}
c(\pi):\mathcal{E}(\pi)&\to\intint{1}{p}.\\
\nonumber
((k,\sigma),(k+1,\tau))&\mapsto\frac{\abs{\tau}}{\abs{\sigma}}
\end{align}

\begin{corollary}
\label{CCycles}
Let $2\leq p\in\N$, $\pi\in\SoPermutations{p}$, and $k\in\Nz$. Then $\mathcal{G}(\pi)$ is a directed, infinite, rooted tree with root $(0,[([0])]_{\sim_\sigma})$. The out-degrees of all vertices are contained in $\intint{1}{p}$ and the out-degree of the root is $p$. The cycle decomposition of $\pi_k$ is given by the $k$-th layer of $\mathcal{G}(\pi)$ (vertices of distance $k$ from the root). Furthermore, $c(\pi)$ defines an edge labeling of $\mathcal{G}(\pi)$. The labels of all outgoing edges of a given vertex sum up to $p$ and all edges going out of the root are labeled $1$. For every vertex of $\mathcal{G}(\pi)$ the length of the represented cycle coincides with the product of all edge labels along the unique path connecting the vertex with the root.
\end{corollary}

\begin{proof}
Follows directly from Theorem~\ref{TCycles} and from the fact that
\begin{align}
\Sigma(\pi_1)=\set{[([0])]_{\sim_\sigma},\ldots,[([p-1])]_{\sim_\sigma}}
\end{align}
by Eqn.~(\ref{EPermA}).
\end{proof}

\noindent
Figure~\ref{FTreeCyc} below gives two examples of trees of cycles of $p$-adic permutations. For comparison we give the lists of all cycles of $\pi_0,\ldots,\pi_5$ for $\pi\ce\pi_{(x,3x+1),(5x+18,x-7)}$, i.e. the first of the two examples given in Figure~\ref{FTreeCyc} (to improve readability we omit the square brackets indicating equivalence classes, i.e. we write $(0,2,4,6)$ for $[([0],[2],[4],[6])]_{\sim_\sigma}$):

\begin{align}
\Sigma(\pi_0):&\;\;(\mathbf{0})\\
\Sigma(\pi_1):&\;\;(\mathbf{0}),\;(\mathbf{1})\\
\Sigma(\pi_2):&\;\;(\mathbf{0},2),\;(\mathbf{1},3)\\
\Sigma(\pi_3):&\;\;(\mathbf{0},2,4,6),\;(\mathbf{1},7),\;(\mathbf{5},3)\\
\Sigma(\pi_4):&\;\;(\mathbf{0},10,4,14),\;(\mathbf{8},2,12,6),\;(\mathbf{1},15,9,7),\;(\mathbf{5},11),\;(\mathbf{13},3)\\
\Sigma(\pi_5):&\;\;(\mathbf{0},26,4,14,16,10,20,30),\;(\mathbf{8},2,28,22,24,18,12,6),\;(\mathbf{1},31,25,23),\;(\mathbf{17},15,9,7),\\
\nonumber
&\;\;(\mathbf{5},11),\;(\mathbf{21},27),\;(\mathbf{13},3,29,19).
\end{align}

\begin{figure}[H]
\centering
\includegraphics[height=4.66cm]{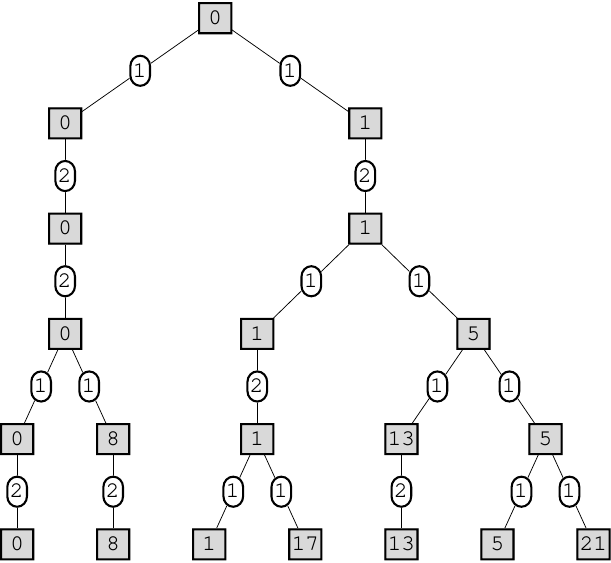}
\qquad
\includegraphics[height=4.66cm]{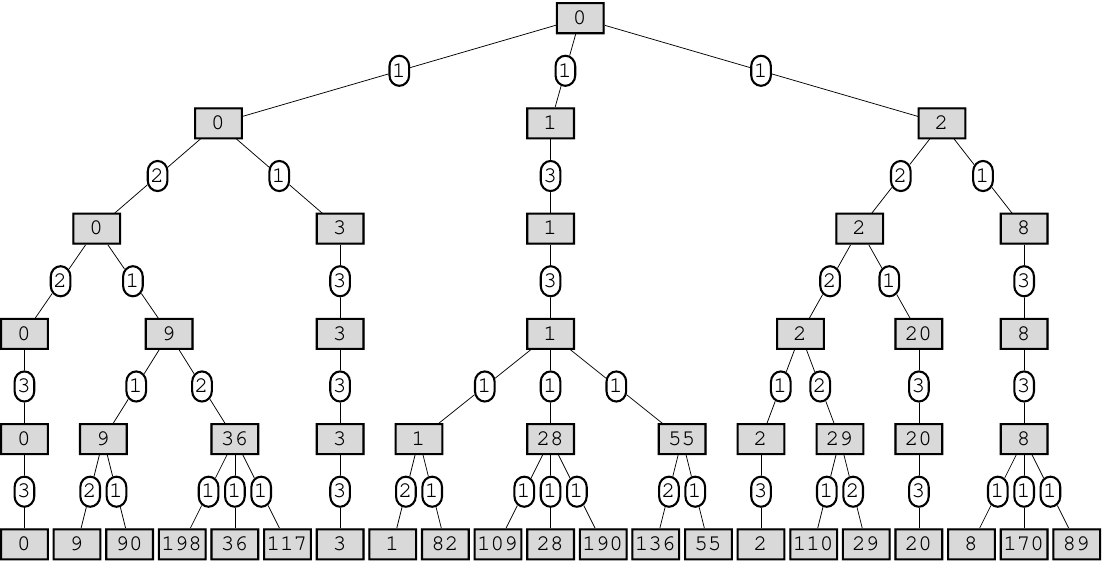}
\caption{Layers $0$ to $5$ of the trees of cycles of the $p$-adic permutations $\pi_{(x,3x+1),(5x+18,x-7)}$ (left) and $\pi_{(-5x-3,5x+1,-x+5),(-4x+3,-x+1,-2x+4)}$ (right). The vertex labels are class representatives of possible starting points of the respective cycles.}
\label{FTreeCyc}
\end{figure}

\myparagraphtoc{Summary.}
In this section we have discussed several interpretations of $p$-adic systems. The relations we have established are summarized below.
\begin{align}
\label{EEqSummary}
\SoSystems{p}\slash_{\sim_p}&\leftrightarrow\SoFunctions{p},&
\FFFDT:\SoSystems{p}\slash_{\sim_p}&\leftrightarrow\SoTables{p},&
\Pi_\FFf{G}(\FFf{F}):\SoSystems{p}\slash_{\sim_p}&\leftrightarrow\SoPermutations{p}.\\
\nonumber
\FFf{F}&\mapsto(\FFf{F}:\Z_p\to\Z_p)&
\FFf{F}&\mapsto\FFFDT(\FFf{F})&
\FFf{F}&\mapsto{\psi_\FFf{G}}^{-1}\circ\psi_\FFf{F}&\\
\nonumber
(pf,\ldots,pf)&\mapsfrom f&
\FFf{F}_\STf{D}&\mapsfrom\STf{D}&
\pi^{-1}\circ\FFf{G}\circ\pi&\mapsfrom\pi
\end{align}

\section{A useful characterization and many examples of $p$-adic systems}
\label{SCharacterization}
In the previous section we have discussed several interpretations of $p$-adic systems from which we ultimately chose ``$p$-fibred functions with block property'' to serve as their definition. As with any new mathematical object the obvious first question to ask is: do they exist? In the case of $p$-adic systems the question seems obsolete, as we already have given examples ($\FFf{F}_C$ and $\FFf{F}_2$ from the introduction) and have also proven that there is a one-one-one correspondence between $p$-fibred systems and, say, $p$-digit tables with block property, the latter of which clearly exist in abundance. The real question to ask in this case is, thus, a slightly different one: do many of them occur ``naturally'' in their $p$-fibred function form (which we favored over the other forms in the definition after all) or are most of them clumsy and ``artificial'' when interpreted as $p$-fibred functions? The answer to this question (which is: yes, many occur naturally) will be given by the following very useful characterization. But before we need to define the following predicates on the set $(\Z_p)^A$ of all mappings from $A$ to $\Z_p$, where $2\leq p\in\N$ and $A\subseteq\Z_p$ ($f\in(\Z_p)^A$, $r\in\intintuptoex{p}$, $K\subseteq\Nz$):
\begin{flalign}
\pushleft{\FPweaklysuitableat{p}{r}{K}(f)}&\Leftrightarrow\fa k\in K:\fa m,n\in(r+p\Z_p)\cap A:&\text{$f$ is \emph{weakly $(p,r)$-suitable at $K$}}\label{DFPweaklysuitableat}\\
\nonumber
&\phantom{\vphantom{a}\Leftrightarrow\vphantom{a}}\quad m\equiv n\modulus p^k\Rightarrow(f-f\modulo p)(m)\equiv(f-f\modulo p)(n)\modulus p^k\hspace{-11em}\\
\pushleft{\FPweaklysuitable{p}{r}(f)}&\Leftrightarrow\FPweaklysuitableat{p}{r}{\Nz}(f)&\text{$f$ is \emph{weakly $(p,r)$-suitable}}\label{DFPweaklysuitable}\\
\pushleft{\FPsuitableat{p}{r}{K}(f)}&\Leftrightarrow\fa k\in K:\fa m,n\in(r+p\Z_p)\cap A:&\text{$f$ is \emph{$(p,r)$-suitable at $K$}}\label{DFPsuitableat}\\
\nonumber
&\phantom{\vphantom{a}\Leftrightarrow\vphantom{a}}\quad m\equiv n\modulus p^k\Leftrightarrow(f-f\modulo p)(m)\equiv(f-f\modulo p)(n)\modulus p^k\hspace{-11em}\\
\pushleft{\FPsuitable{p}{r}(f)}&\Leftrightarrow\FPsuitableat{p}{r}{\Nz}(f)&\text{$f$ is \emph{$(p,r)$-suitable}}\label{DFPsuitable}
\end{flalign}

\noindent
Note that any function $f:A\to\Z_p$ satisfying $f((r+p\Z_p)\cap A)\subseteq p\Z_p$ (cf. the definition of $p$-fibred functions being in weak canonical form, Eqn.~(\ref{EPWCanF})) is weakly $(p,r)$-suitable if and only if $f\vert_{(r+p\Z_p)\cap A}$ satisfies the $1$-Lipschitz condition, and $(p,r)$-suitable if and only if $f\vert_{(r+p\Z_p)\cap A}$ is measure preserving (can be extended to a measure preserving function on $\Z_p$, to be exact) \cite{Yurova:2013}. In this case one clearly gets $f(n)\modulo p=0$ for all $n\in(r+p\Z_p)\cap A$ and thus the condition $(f-f\modulo p)(m)\equiv(f-f\modulo p)(n)\modulus p^k$ simplifies to $f(m)\equiv f(n)\modulus p^k$.

The following theorem summarizes how weakly $(p,r)$-suitable and $(p,r)$-suitable functions can be used to characterize closed $p$-fibred functions with weak block property and block property respectively.

\begin{theorem}
\label{TFuncSuit}
Let $2\leq p\in\N$, $\FFf{F}\in\SoFibredFunctions{p}(\FFPclosed)$, and $k\in\N$. Then,

\begin{theoremtable}
\theoremitem{(1)}&$\fa r\in\intintuptoex{p}:\FPweaklysuitableat{p}{r}{\intintuptoin{k}}(\FFf{F}[r])\Rightarrow\FFPweakblockat{\intintuptoin{k}}(\FFf{F})$\tabularnewline
&In particular: $\fa r\in\intintuptoex{p}:\FPweaklysuitable{p}{r}(\FFf{F}[r])\Rightarrow\FFPweakblock(\FFf{F})$\tabularnewline
\theoremitem{(2)}&$\fa r\in\intintuptoex{p}:\FPsuitableat{p}{r}{\intintuptoin{k}}(\FFf{F}[r])\Leftrightarrow\FFPblockat{\intintuptoin{k}}(\FFf{F})$\tabularnewline
&In particular: $\fa r\in\intintuptoex{p}:\FPsuitable{p}{r}(\FFf{F}[r])\Leftrightarrow\FFPblock(\FFf{F})$.
\end{theoremtable}
\end{theorem}

\begin{proof}
Let $\ell\in\intintuptoin{k}$, $r\in\intintuptoex{p}$, and $m,n\in(r+p\Z_p)\cap\FFFdomain(\FFf{F})$. Then,
\begin{align}
\FFf{F}(m)\equiv\FFf{F}(n)\modulus{p^{\ell-1}}&\Leftrightarrow\frac{\FFf{F}[r](m)-\FFf{F}[r](m)\modulo p}{p}\equiv\frac{\FFf{F}[r](n)-\FFf{F}[r](n)\modulo p}{p}\modulus{p^{\ell-1}}\\
&\Leftrightarrow\FFf{F}[r](m)-\FFf{F}[r](m)\modulo p\equiv\FFf{F}[r](n)-\FFf{F}[r](n)\modulo p\modulus{p^\ell}
\end{align}
and the statements follow from Lemma~\ref{LOrdFuncBlock}.
\end{proof}

\noindent
Despite the fact that the proof of this theorem is rather simple, it is quite remarkable. It states that the block property of a $p$-fibred function $\FFf{F}$ does not depend on the relation between the functions $\FFf{F}[0],\ldots,\FFf{F}[p-1]$, but is only a question of whether each $\FFf{F}[r]$ is ``suitable'' to be the $r$-th of the $p$ entries of $\FFf{F}$. The $p$ functions that define $\FFf{F}$ can be chosen completely independently from one another.

\myparagraphtoc{The weak block property revisited.}
A similar characterization is given for the weak block property, but in contrast to the characterization of the block property in (2), the condition in (1) is only sufficient but not necessary. A natural question to ask is whether the weak block property also permits a necessary and sufficient characterization that only considers the functions $\FFf{F}[0],\ldots,\FFf{F}[p-1]$ independently from one another. The following example shows that this is not the case.

\begin{example}
\label{ENoCharacWBlk}
Let $f,g,h:\Z_2\to\Z_2$ with $f(n)=(n=8\;?\;6:2)$, $g(n)=(n=3\;?-2:(n\equiv3\modulus{4}\;?\;6:0))$, and $h(n)=0$ for all $n\in\Z_2$. Furthermore, let $\FFf{F}\ce(f,h)$, $\FFf{G}\ce(h,g)$, and $\FFf{H}\ce(f,g)$. Then $\FFFDT(\FFf{H})[0][\intintuptoex{3}]=(0,1,0)\neq(0,1,1)=\FFFDT(\FFf{H})[2^3][\intintuptoex{3}]$ which implies that $\FFf{H}$ does not have the weak block property (at $3$). At the same time we get $\FFFDT(\FFf{F})=((n\equiv0\modulus{2}\;?\;(0,1)^\infty:(1,0)^\infty))_{n\in\Z_2}$ and $\FFFDT(\FFf{G})=((n\equiv0\modulus{2}\;?\;(0)^\infty:(n\equiv1\modulus{4}\;?\;(1)\cdot(0)^\infty:(1)^\infty))_{n\in\Z_2}$ which implies that both $\FFf{F}$ and $\FFf{G}$ have the weak block property. Note that $f(0)=2\not\equiv2+2^2=f(0+2^3)\modulus{2^3}$ and $g(3)=-2\not\equiv-2+2^3=g(3+2^4)\modulus{2^4}$ which means that neither $f$ nor $g$ are weakly $(2,0)$-, respectively weakly $(2,1)$-suitable.
\end{example}

\noindent
$f$ and $g$ of the previous example can both be part of $2$-fibred functions ($\FFf{F}$ and $\FFf{G}$ respectively) that have the weak block property, but the $2$-fibred function $\FFf{H}$, which contains both $f$ and $g$, does not have the weak block property. Thus, there cannot be a necessary and sufficient characterization of the weak block property which considers the entries of a $p$-fibred function independently.

A natural follow-up question is whether every $p$-digit table with weak block property that is the $p$-digit table of a $p$-fibred function can at least be realized as the $p$-digit table of a $p$-fibred function whose entries are weakly $(p,r)$-suitable functions. This is true for the $2$-digit tables $\FFFDT(\FFf{F})$ and $\FFFDT(\FFf{G})$ of the previous example: if $f_0,f_1,g_0,g_1:\Z_2\to\Z_2$ with $f_0(n)=2$, $f_1(n)=0$, $g_0(n)=0$, and $g_1(n)=(n\equiv1\modulus{4}\;?\;0:6)$ for all $n\in\Z_2$, then $\FFFDT(\FFf{F})=\FFFDT((f_0,f_1))$ and $\FFFDT(\FFf{G})=\FFFDT((g_0,g_1))$, and also $\FPweaklysuitable{2}{0}(f_0)$, $\FPweaklysuitable{2}{1}(f_1)$, $\FPweaklysuitable{2}{0}(g_0)$, and $\FPweaklysuitable{2}{1}(g_1)$. However, in general even this does not hold as the following example shows.

\begin{example}
\label{ENoCharacWBlkEx}
Let $f_0,f_1:\Z_2\to\Z_2$ with $f_0(n)=(n\modulo16=8\;?\;6:(n\modulo4=0\;?\;2:4))$ and $f_1(n)=(n\modulo4=1\;?\;4:0)$ for all $n\in\Z_2$. Furthermore, let $\FFf{F}\ce(f_0,f_1)$, and $\STf{D}\ce\FFFDT(\FFf{F})$. Then $\STf{D}=((n\equiv8\modulus{16}\;?\;(0,1,0,1):(n\equiv0\modulus{4}\;?\;(0,1):(n\equiv1\modulus{4}\;?\;(1):(n\equiv2\modulus{4}\;?\;():(1,0,1)))))\cdot(0)^\infty)_{n\in\Z_2}$ and hence $\STf{D}$ has the weak block property, but since $\STf{D}(0)=1+4\Z_2$ and $\STf{D}(8)=3+4\Z_2$, Theorem~\ref{TDTFF} implies that there cannot exist any $2$-fibred function $\FFf{G}=(g_0,g_1)$ with $g_0:\Z_2\to\Z_2$ weakly $(2,0)$-suitable and $g_1:\Z_2\to\Z_2$ weakly $(2,1)$-suitable such that $\STf{D}=\FFFDT(\FFf{G})$, because in this case $g_0(0)\in\set{2,3}+8\Z_2$ and $g_0(8)\in\set{6,7}+8\Z_2$ and hence $(g_0-g_0\modulo2)(0)\not\equiv(g_0-g_0\modulo2)(0+2^3)\modulus2^3$.
\end{example}

Example~\ref{EWBlockF}, Example~\ref{ENoCharacWBlk}, and Example~\ref{ENoCharacWBlkEx} indicate that there are actually three levels of the weak block property of $p$-digit tables of different generality. The weakest form is given by the weak block property itself, the stronger version (as proven by Example~\ref{EWBlockF}) requires the $p$-digit table to come from some $p$-fibred function, while the even stronger version (as proven by Example~\ref{ENoCharacWBlkEx}) requires the entries of this $p$-fibred function to be weakly $(p,r)$-suitable.

\noindent
We thus define the following predicates on $\SoDigitTables{p}$ ($\STf{D}\in\SoDigitTables{p}$, $K\subseteq\Nz$):
\begin{flalign}
\pushleft{\DTPweakblockFat{K}(\STf{D})}&\Leftrightarrow\ex\FFf{F}\in\SoFibredFunctions{p}(\FFPweakblockat{K}):&\hspace{-5em}\text{$\STf{D}$ has the \emph{weak block property F at $K$}}\label{DDTPweakblockFat}\\
\nonumber
&\phantom{\vphantom{a}\Leftrightarrow\vphantom{a}}\quad\STf{D}=\FFFDT(\FFf{F})\\
\pushleft{\DTPweakblockF(\STf{D})}&\Leftrightarrow\DTPweakblockFat{\Nz}(\STf{D})&\hspace{-2em}\text{$\STf{D}$ has the \emph{weak block property F}}\label{DDTPweakblockF}\\
\pushleft{\DTPweakblockSat{K}(\STf{D})}&\Leftrightarrow\ex\FFf{F}\in\SoFibredFunctions{p}(\FFPweakblockat{K}):&\hspace{-5em}\text{$\STf{D}$ has the \emph{weak block property S at $K$}}\label{DDTPweakblockSat}\\
\nonumber
&\phantom{\vphantom{a}\Leftrightarrow\vphantom{a}}\quad\STf{D}=\FFFDT(\FFf{F})\land\fa r\in\intintuptoex{p}:\FPweaklysuitableat{p}{r}{K}(\FFf{F}[r])\\
\pushleft{\mathrlap{\DTPweakblockS(\STf{D})}\hphantom{\DTPweakblockFat{K}(\STf{D})}}&\Leftrightarrow\DTPweakblockSat{\Nz}(\STf{D})&\hspace{-2em}\text{$\STf{D}$ has the \emph{weak block property S}}\label{DDTPweakblockS}
\end{flalign}

\noindent
Clearly,
\begin{align}
&\DTPblockat{K}(\STf{D})\Rightarrow\DTPweakblockSat{K}(\STf{D})\Rightarrow\DTPweakblockFat{K}(\STf{D})\Rightarrow\DTPweakblockat{K}(\STf{D})\\
&\mathrlap{\DTPblock(\STf{D})}\phantom{\DTPblockat{K}(\STf{D})}\Rightarrow\mathrlap{\DTPweakblockS(\STf{D})}\phantom{\DTPweakblockSat{K}(\STf{D})}\Rightarrow\mathrlap{\DTPweakblockF(\STf{D})}\phantom{\DTPweakblockFat{K}(\STf{D})}\Rightarrow\DTPweakblock(\STf{D})
\end{align}
for all $\STf{D}\in\SoDigitTables{p}$ and all $K\subseteq\Nz$ by definition but, as discussed above, neither of the converses of the four implications is true in general.

The predicates \label{DFFPweakblockFat}$\DTPweakblockFat{K}$ and \label{DFFPweakblockF}$\DTPweakblockF$ carry over to $\SoFibredFunctions{p}(\FFPclosed)$ by $\FFFDT(\FFf{F})$ but we define the following alternative predicates on $\SoFibredFunctions{p}(\FFPclosed)$ ($\FFf{F}\in\SoFibredFunctions{p}(\FFPclosed)$, $k\subseteq\Nz$):
\begin{flalign}
\pushleft{\FFPweakblockSat{K}(\FFf{F})}&\Leftrightarrow\fa r\in\intintuptoex{p}:\FPweaklysuitableat{p}{r}{K}(\FFf{F}[r])&\text{$\FFf{F}$ has the \emph{weak block property S at $K$}}\label{DFFPweakblockSat}\\
\pushleft{\FFPweakblockS(\FFf{F})}&\Leftrightarrow\FFPweakblockSat{\Nz}(\FFf{F})&\text{$\FFf{F}$ has the \emph{weak block property S}}\label{DFFPweakblockS}
\end{flalign}

\noindent
It follows from the definitions and from Theorem~\ref{TFuncSuit}~(2) that
\begin{align}
&\FFPblockat{K}(\FFf{F})\Rightarrow\FFPweakblockSat{K}(\FFf{F})\Rightarrow\DTPweakblockSat{K}(\FFFDT(\FFf{F}))\Rightarrow\FFPweakblockFat{K}(\FFf{F})\Leftrightarrow\FFPweakblockat{K}(\FFf{F})\\
&\mathrlap{\FFPblock(\FFf{F})}\phantom{\FFPblockat{K}(\FFf{F})}\Rightarrow\mathrlap{\FFPweakblockS(\FFf{F})}\phantom{\FFPweakblockSat{K}(\FFf{F})}\Rightarrow\mathrlap{\DTPweakblockS(\FFFDT(\FFf{F}))}\phantom{\DTPweakblockSat{K}(\FFFDT(\FFf{F}))}\Rightarrow\mathrlap{\FFPweakblockF(\FFf{F})}\phantom{\FFPweakblockFat{K}(\FFf{F})}\Leftrightarrow\FFPweakblock(\FFf{F})
\end{align}
for all $\FFf{F}\in\SoFibredFunctions{p}(\FFPclosed)$ and all $K\subseteq\Nz$ while, again, neither of the converses of the six implications is true in general.

In addition to being of general interest, this ``refinement'' of the weak block property has a first application as well in proving a stronger version of the last statement of Corollary~\ref{CComputeDT}~(2).

\begin{corollary}
\label{CComputeDTStrong}
Under the assumptions of Corollary~\ref{CComputeDT} we have

\begin{theoremtable}
\theoremitem{(3)}&$\FFPweakblockSat{k}(\FFf{F})\land\FFPweakblockat{k-1}(\FFf{F})$ and $S'$ CRS modulo $p^{k-1}$ $\Rightarrow$\tabularnewline
&\quad$\fa n\in\FFFdomain(\FFf{F}):\FFf{H}_1(n)=\FFf{H}_1(T(n))=\FFf{H}_2(n)$.
\end{theoremtable}
\end{corollary}

\begin{proof}
It follows from the assumptions that $\FFf{H}_1[r](n)=S(\FFf{F}[r](n))\equiv\FFf{F}[r](n)\modulus{p^k}$ for all $r\in\intintuptoex{p}$ and all $n\in(r+p\Z_p)\cap\FFFdomain(\FFf{F})$. Thus, $\FFPweakblockSat{k}(\FFf{F})$ implies that also $\FFPweakblockSat{k}(\FFf{H}_1)$, i.e. $\FPweaklysuitableat{p}{r}{k}(\FFf{H}_1[r])$ for all $r\in\intintuptoex{p}$.

Now let $n\in\FFFdomain(\FFf{F})$. Clearly, $\FFf{H}_1(T(n))=\FFf{H}_2(n)$. In addition, $T(n)\equiv n\modulus{p^k}$ and hence
\begin{align}
p\FFf{H}_1(T(n))
&=\FFf{H}_1[T(n)\modulo p](T(n))-\FFf{H}_1[T(n)\modulo p](T(n))\modulo p\\
&=\FFf{H}_1[n\modulo p](T(n))-\FFf{H}_1[n\modulo p](T(n))\modulo p\\
(\FPweaklysuitableat{p}{n\modulo p}{k}(\FFf{H}_1[n\modulo p]))\;&\equiv\FFf{H}_1[n\modulo p](n)-\FFf{H}_1[n\modulo p](n)\modulo p\\
&=p\FFf{H}_1(n)\modulus{p^k}.
\end{align}
We thus get $\FFf{H}_1(T(n))\equiv\FFf{H}_1(n)\modulus{p^{k-1}}$ and since both $\FFf{H}_1(T(n))$ and $\FFf{H}_1(n)$ are in $S'$ by Corollary~\ref{CComputeDT}~(1), they must be equal.
\end{proof}

\myparagraphtoc{Properties of suitable functions.}
Since $(p,r)$-suitable functions are the building blocks of $p$-adic systems, we summarize some basic facts on them below.

\begin{lemma}
\label{LSuitFProp}
Let $2\leq p\in\N$, $A\subseteq\Z_p$, $r\in\intintuptoex{p}$, and $f:A\to\Z_p$ $(p,r)$-suitable. Then,

\begin{theoremtable}
\theoremitem{(1)}&$f\vert_{(r+p\Z_p)\cap A}$ is injective\tabularnewline
\theoremitem{(2)}&$r+p\Z_p\subseteq A\land f(r+p\Z_p)\subseteq p\Z_p\Rightarrow f\vert_{r+p\Z_p}(r+p\Z_p)=p\Z_p$\tabularnewline
&$\phantom{r+p\Z_p\subseteq A\land f(r+p\Z_p)\subseteq p\Z_p\Rightarrow\vphantom{a}}$In particular: $g:r+p\Z_p\to\Z_p$, $n\mapsto f(n)/p$ is bijective.
\end{theoremtable}
\end{lemma}

\begin{proof}
$ $\\
\proofitem{(1)}%
Assume to the contrary that $f\vert_{(r+p\Z_p)\cap A}$ is not injective and let $m,n\in (r+p\Z_p)\cap A$ with $m\neq n$ such that $f(m)=f(n)$. Then there is a $k\in\N$ such that $m\not\equiv n\modulus{p^k}$, but $(f-f\modulo p)(m)\equiv(f-f\modulo p)(n)\modulus p^k$ which is a contradiction.

\noindent
\proofitem{(2)}%
Let $g:\Z_p\to\Z_p$ be an arbitrary extension of $f$, $\id_{\Z_p}$ the identity function on $\Z_p$, and $\FFf{F}\ce(\id_{\Z_p})^r\cdot(g)\cdot(\id_{\Z_p})^{p-r-1}$. Then $\FFf{F}\in\SoSystems{p}$ by Theorem~\ref{TFuncSuit}~(2) and hence Lemma~\ref{LInfBlDTCompl} implies that for every $n\in\Z_p$ there is a unique $m\in\Z_p$ such that
\begin{align}
\FFFDT(\FFf{F})[m]=(r)\cdot\FFFDT(\FFf{F})[n].
\end{align}
It follows that
\begin{align}
\FFFDT(\FFf{F})[\FFf{F}(m)]=\FFFDT(\FFf{F})[n]
\end{align}
and hence $\FFf{F}(m)=n$ by Lemma~\ref{LInfBlDTEq}. Furthermore, $m\in r+pR$ and thus
\begin{align}
f(m)=\FFf{F}[r](m)=p\FFf{F}(m)=pn.
\end{align}
We conclude that $f\vert_{r+p\Z_p}(r+p\Z_p)=p\Z_p$. Thus $g$ is surjective and by (1) it is also injective.
\end{proof}

\begin{corollary}
\label{CSuitFProp}
Let $2\leq p\in\N$ and $\FFf{F}\in\SoSystems{p}$. Then,

\begin{theoremtable}
\theoremitem{(1)}&$\fa r\in\intintuptoex{p}:\FFf{F}\vert_{r+p\Z_p}:r+p\Z_p\to\Z_p$ is surjective and one-to-one (i.e. bijective)\tabularnewline
&In particular: $\fa n\in\Z_p:\fa r\in\intintuptoex{p}:\abs{\FFf{F}^{-1}(n)\cap(r+p\Z_p)}=1$\tabularnewline
\theoremitem{(2)}&$\FFf{F}:\Z_p\to\Z_p$ is surjective and $p$-to-one.\tabularnewline
&In particular: $\fa n\in\Z_p:\abs{\FFf{F}^{-1}(n)}=p$.
\end{theoremtable}
\end{corollary}

\begin{proof}
By Lemma~\ref{LFFCanForm} we may assume without loss of generality that $\FFf{F}$ is in canonical form. Then all statements follow directly from Theorem~\ref{TFuncSuit}~(2), and Lemma~\ref{LSuitFProp}.
\end{proof}

We may construct new $(p,r)$-suitable functions from existing ones by multiplying them with weakly $(p,r)$-suitable functions whose values are coprime to $p$. This fact will later be used to identify many rational functions that are $(p,r)$-suitable.

\begin{theorem}
\label{TSuitFProd}
Let $2\leq p\in\N$, $A\subseteq\Z_p$, $r\in\intintuptoex{p}$, $k\subseteq\Nz$, $f:A\to\Z_p$ with $f((r+p\Z_p)\cap A)\subseteq p\Z_p$, and $g:A\to\Z_p$ weakly $(p,r)$-suitable at $\intintuptoin{k}$ with $\gcd(p,g(n)\modulo p)=1$ for all $n\in(r+p\Z_p)\cap A$. Then,

\begin{theoremtable}
\theoremitem{(1)}&$\FPweaklysuitableat{p}{r}{\intintuptoin{k}}(f)\Leftrightarrow\FPweaklysuitableat{p}{r}{\intintuptoin{k}}(fg)$\tabularnewline
\theoremitem{(2)}&$\FPsuitableat{p}{r}{\intintuptoin{k}}(f)\Leftrightarrow\FPsuitableat{p}{r}{\intintuptoin{k}}(fg)$.\tabularnewline[0.5\baselineskip]
\end{theoremtable}
In particular, $1/g:A\to\Z_p$ exists and satisfies the same properties as $g$ and thus $f/g$ is (weakly) $(p,r)$-suitable at $\intintuptoin{k}$ as well if and only if $f$ is (weakly) $(p,r)$-suitable at $\intintuptoin{k}$.
\end{theorem}

\begin{proof}
To prove ``$\Rightarrow$'' of (2) we need to show that
\begin{align}
m\equiv n\modulus{p^\ell}\quad\Leftrightarrow\quad((fg)-(fg)\modulo p)(m)\equiv((fg)-(fg)\modulo p)(n)\modulus p^\ell
\end{align}
for all $\ell\in\intintuptoin{k}$ and $m,n\in(r+p\Z_p)\cap A$ which is equivalent to
\begin{align}
m\equiv n\modulus{p^\ell}\quad\Leftrightarrow\quad(fg)(m)\equiv(fg)(n)\modulus p^\ell
\end{align}
since $f((r+p\Z_p)\cap A)\subseteq p\Z_p$.

We begin by proving ``$\Rightarrow$'', which is the easier direction. It follows from the assumptions that
\begin{align}
m\equiv n\modulus{p^\ell}&\Rightarrow\mathrlap{f(m)\equiv f(n)\modulus p^\ell}\phantom{f(m)g(m)\equiv f(n)g(m)\modulus{p^\ell}}\quad\land\quad g(m)\equiv g(n)\modulus p^\ell\\
&\Rightarrow f(m)g(m)\equiv f(n)g(m)\modulus{p^\ell}\quad\land\quad \frac{g(n)-g(m)}{p^\ell}\in\Z_p\\
&\Rightarrow f(m)g(m)\equiv f(n)g(m)+f(n)\frac{g(n)-g(m)}{p^\ell}p^\ell\modulus{p^\ell}\\
&\Rightarrow(fg)(m)\equiv(fg)(n)\modulus{p^\ell}.
\end{align}
Note that in the above deduction we only needed that $f$ is weakly $(p,r)$-suitable at $\intintuptoin{k}$ (in the first implication), which implies that we have already proven ``$\Rightarrow$'' of (1).

For ``$\Leftarrow$'' we proceed by induction on $\ell$. If $\ell=0$, this is clearly true. Now assume that it is also true for some $\ell\in\intintuptoex{k}$ and that
\begin{align}
(fg)(m)\equiv(fg)(n)\modulus{p^{\ell+1}}.
\end{align}
Then,
\begin{align}
(fg)(m)\equiv(fg)(n)\modulus{p^\ell}
\end{align}
and thus
\begin{align}
m\equiv n\modulus{p^\ell}
\end{align}
by the induction hypothesis. Consequently, since $g$ is weakly $(p,r)$-suitable at $\ell$,
\begin{align}
&g(m)\equiv g(n)\modulus{p^\ell}
\end{align}
and hence
\begin{align}
&f(n)g(m)\equiv f(n)g(n)\modulus{p^{\ell+1}}
\end{align}
since $f(n)\in p\Z_p$. But then we get
\begin{align}
&(f(m)-f(n))g(m)=f(m)g(m)-f(n)g(m)\equiv f(n)g(n)-f(n)g(n)=0\modulus{p^{\ell+1}}.
\end{align}
Since $\gcd(p,g(m)\modulo p)=1$, we have $1/g(m)\in\Z_p$ and it follows that
\begin{align}
&f(m)-f(n)=(f(m)-f(n))g(m)\frac{1}{g(m)}\equiv0\modulus{p^{\ell+1}}.
\end{align}
Consequently,
\begin{align}
&f(m)\equiv f(n)\modulus{p^{\ell+1}}
\end{align}
and since $f$ is $(p,r)$-suitable at $\ell+1$, we finally get
\begin{align}
&m\equiv n\modulus{p^{\ell+1}}
\end{align}
which proves that $fg$ is $(p,r)$-suitable at $\ell+1$.

For the ``In particular'' part we need to show that $1/g$ has the same properties as $g$, i.e. that $1/g$ is weakly $(p,r)$-suitable at $\intintuptoin{k}$ and $\gcd(p,(1/g)(n)\modulo p)=1$ for all $n\in(r+p\Z_p)\cap A$. The latter part follows trivially from the fact that $\gcd(p,g(n)\modulo p)=1$ and $g(n)(1/g)(n)=1$ for all $n\in(r+p\Z_p)\cap A$. In order to prove that $1/g$ is weakly $(p,r)$-suitable at $\intintuptoin{k}$, let $\ell\in\intintuptoin{k}$ and $m,n\in(r+p\Z_p)\cap A$. Since $(g(m)g(n))\modulo p=((g(m)\modulo p)(g(n)\modulo p))\modulo p=1$, we have $1/(g(m)g(n))\in\Z_p$. Thus,
\begin{align}
m\equiv n\modulus{p^\ell}
&\Rightarrow g(m)\equiv g(n)\modulus{p^\ell}\\
&\Rightarrow g(m)\frac{1}{g(m)g(n)}\equiv g(n)\frac{1}{g(m)g(n)}\modulus{p^\ell}\\
&\Rightarrow(1/g)(m)\equiv(1/g)(n)\modulus{p^\ell}.
\end{align}

Using ``$\Rightarrow$'' of (1) and (2) and the ``In particular'' part, it is now straightforward to prove ``$\Leftarrow$'' of (1) and (2). As a result we have,
\begin{align}
&\FPweaklysuitableat{p}{r}{\intintuptoin{k}}(fg)\Rightarrow\FPweaklysuitableat{p}{r}{\intintuptoin{k}}(fg(1/g))\Rightarrow\FPweaklysuitableat{p}{r}{\intintuptoin{k}}(f)\\
&\FPsuitableat{p}{r}{\intintuptoin{k}}(fg)\Rightarrow\FPsuitableat{p}{r}{\intintuptoin{k}}(fg(1/g))\Rightarrow\FPsuitableat{p}{r}{\intintuptoin{k}}(f).
\end{align}
which completes the proof.
\end{proof}

\myparagraphtoc{Polynomial $p$-adic systems.}
In this paragraph we investigate the consequences of Theorem~\ref{TFuncSuit} for $p$-fibred systems defined by polynomials over $\Z_p$ and $\Q_p$. We will prove the surprising result that ``most'' polynomial $p$-fibred functions have the block property and thus provide a first big class of $p$-adic systems that have a rather natural representation as $p$-fibred functions.

\begin{theorem}
\label{TPropPolyF}
Let $2\leq p\in\N$, $r\in\intintuptoex{p}$, $f=\sum_{i=0}^da_ix^i\in\Z_p[x]$ (note that for the whole theorem we define $0^0\ce1$), and $k\in\Nz$. Then,

\begin{theoremtable}
\theoremitem{(1)}&$\FPweaklysuitableat{p}{r}{k}(f)$\tabularnewline
\theoremitem{(2)}&$k\leq1\Rightarrow\FPsuitableat{p}{r}{k}(f)$\tabularnewline
\theoremitem{(3)}&$k\geq2\Rightarrow$\tabularnewline
&\quad$\FPsuitableat{p}{r}{k}(f)\Leftrightarrow\Gcd{p,f'(r)\modulo p}=1$\quad(note: $f'(r)\modulo p=\left(\sum_{i=1}^d(a_i\modulo p)ir^{i-1}\right)\modulo p$)\tabularnewline
&\quad In particular: if $a_0,a_2,\ldots,a_d$ are given then the set of all $a_1$ that make $f$ $(p,r)$-suitable at $k$\tabularnewline
&\leavevmode\phantom{\quad In particular: }is given by $\set{\left(a-\textstyle\sum_{i=2}^d(a_i\modulo p)ir^{i-1}\right)\modulo p\mid a\in\intintuptoex{p}\land\Gcd{p,a}=1}+p\Z_p$.
\end{theoremtable}
\end{theorem}

\begin{proof}
First we note that $p$ is not a zero divisor of $\Z_p$ and therefore $a=b\Leftrightarrow pa=pb$ for all $a,b\in\Z_p$ (otherwise $a\neq b$ and $pa=pb$ for some $a,b\in\Z_p$ and hence $a-b\neq0$ and $p(a-b)=0$ which implies that $p$ is a zero divisor). In particular,
\begin{align}
p^k\mid pa\Leftrightarrow\ex b\in\Z_p:pp^{k-1}b=pa\Leftrightarrow\ex b\in\Z_p:p^{k-1}b=a\Leftrightarrow p^{k-1}\mid a
\end{align}
for all $a\in\Z_p$ and $k\in\N$. Let $k\in\Nz$, $m,n\in r+p\Z_p$, and $b_{m,n}\in\Z_p$ such that $m-n=pb_{m,n}$. If $k\leq1$, then clearly $\FPsuitableat{p}{r}{k}(f)$, which proves (2). Otherwise, we get
\begin{align}
m\equiv n\modulus{p^k}&\Leftrightarrow p^k\mid m-n\Leftrightarrow p^k\mid pb_{m,n}\Leftrightarrow p^{k-1}\mid b_{m,n}
\end{align}
and
\begin{align}
f(m)\equiv f(n)\modulus{p^k}
&\Leftrightarrow
p^k\mid\sum_{i=0}^da_im^i-\sum_{i=0}^da_in^i\\
&\Leftrightarrow
p^k\mid\sum_{i=1}^da_i\left(m^i-n^i\right)\\
&\Leftrightarrow
p^k\mid (m-n)\sum_{i=1}^da_i\sum_{j=0}^{i-1}m^jn^{i-1-j}\\
&\Leftrightarrow
p^k\mid pb_{m,n}\sum_{i=1}^da_i\sum_{j=0}^{i-1}m^jn^{i-1-j}\\
&\Leftrightarrow
p^{k-1}\mid b_{m,n}\sum_{i=1}^da_i\sum_{j=0}^{i-1}m^jn^{i-1-j}\\
&\Leftrightarrow
p^{k-1}\mid b_{m,n}\left(p\sum_{i=1}^da_i\sum_{j=0}^{i-1}\frac{m^jn^{i-1-j}-r^{i-1}}{p}+\sum_{i=1}^da_iir^{i-1}\right)\\
&\Leftrightarrow
p^{k-1}\mid b_{m,n}\left(pc_{m,n}+s\right)
\end{align}
where
\begin{align}
c_{m,n}&\ce\sum_{i=1}^da_i\sum_{j=0}^{i-1}(m^jn^{i-1-j}-r^{i-1})/p\\
s&\ce\sum_{i=1}^da_iir^{i-1}
\end{align}
and $c_{m,n}\in\Z_p$, since $(m^jn^{i-1-j}-r^{i-1})\modulo p=(r^jr^{i-1-j}-r^{i-1})\modulo p=0$ for all $i\in\intint{1}{d}$ and $j\in\intintuptoex{i}$. Therefore, we get
\begin{align}
\fa m,n\in r+p\Z:\left(m\equiv n\modulus{p^k}\Rightarrow f(m)\equiv f(n)\modulus{p^k}\right)
\end{align}
which proves (1).

In order to prove (3) we claim
\begin{align}
\left(\fa m,n\in r+p\Z_p:p^{k-1}\mid b_{m,n}(pc_{m,n}+s)\Rightarrow p^{k-1}\mid b_{m,n}\right)\Leftrightarrow\Gcd{p,s\modulo p}=1
\end{align}
for all $k\geq2$.

For ``$\Rightarrow$'' assume to the contrary that $g\ce\Gcd{p,s\modulo p}\neq1$. Let $m\ce r$ and $n\ce r-p^k/g$. Then, $b_{m,n}=p^{k-1}/g$ and thus $p^{k-1}\nmid b_{m,n}$ since $g\neq1$ (note: $p,g,b_{m,n}\in\N$). Next we claim that $g\mid s$. Then $b_{m,n}(pc_{m,n}+s)=p^{k-1}(p/gc_{m,n}+s/g)$ with $p/gc_{m,n}+s/g\in\Z_p$, and hence $p^{k-1}\mid b_{m,n}(pc_{m,n}+s)$ which is a contradiction. To prove the claim let $s=\sum_{i=0}^\infty s_i p^i$. Then
\begin{align}
s/g
&=(s\modulo p)/g+\sum_{i=1}^\infty s_i p/gp^{i-1}
=(s\modulo p)/g+\sum_{i=0}^\infty s_{i+1} p/gp^i
\in\Z_p
\end{align}
and hence $g\mid s$.

For ``$\Leftarrow$'' we first claim that if $x=\sum_{i=0}^\infty x_i p^i$, $y=\sum_{i=0}^\infty y_i p^i\in\Z_p$ with $\Gcd{p,y\modulo p}=1$, and $\ell\in\Nz$ then
\begin{align}
\label{EDiv}
p^\ell\mid xy&\Leftrightarrow p^\ell\mid x.
\end{align}
We have
\begin{align}
xy
&=\left(x\modulo p^\ell+\frac{x-x\modulo p^\ell}{p^\ell}p^\ell\right)\left(y\modulo p^\ell+\frac{y-y\modulo p^\ell}{p^\ell}p^\ell\right)
=\left(x\modulo p^\ell\right)\left(y\modulo p^\ell\right)+p^\ell d_{x,y,\ell}
\end{align}
where
\begin{align}
d_{x,y,\ell}\ce\left(x\modulo p^\ell\right)\frac{y-y\modulo p^\ell}{p^\ell}
+\frac{x-x\modulo p^\ell}{p^\ell}\left(y\modulo p^\ell\right)
+\frac{x-x\modulo p^\ell}{p^\ell}\frac{y-y\modulo p^\ell}{p^\ell}p^\ell\in\Z_p.
\end{align}
Also, $\Gcd{p^\ell,y\modulo p^\ell}=1$, since $\Gcd{p,y\modulo p}=1$ and thus
\begin{align}
p^\ell\mid xy&\Leftrightarrow p^\ell\mid\left(x\modulo p^\ell\right)\left(y\modulo p^\ell\right)\Leftrightarrow p^\ell\mid x\modulo p^\ell\Leftrightarrow x\modulo p^\ell=0\Leftrightarrow p^\ell\mid x.
\end{align}

Now let $m,n\in r+p\Z_p$ with $p^{k-1}\mid b_{m,n}(pc_{m,n}+s)$. Then
\begin{align}
\Gcd{p,(pc_{m,n}+s)\modulo p}=\Gcd{p,s\modulo p}=1
\end{align}
and thus $p^{k-1}\mid b_{m,n}$ by Eqn.~(\ref{EDiv}).

For the "In particular" part we set $t\ce\sum_{i=2}^d(a_i\modulo p)ir^{i-1}$ and we claim that
\begin{align}
\Gcd{p,(a_1\modulo p+t)\modulo p}=1\Leftrightarrow\ex a\in\intintuptoex{p}:\ex b\in\Z_p:\Gcd{p,a}=1\land a_1=a-t+pb
\end{align}
for all $a_1\in\Z_p$. For ``$\Rightarrow$'' we set $a\ce(a_1+t)\modulo p$ and $b=\frac{a_1+t-(a_1+t)\modulo p}{p}$. Then clearly $a_1=a-t+pb$ and
\begin{align}
\Gcd{p,a}=\Gcd{p,(a_1+t)\modulo p}=\Gcd{p,(a_1\modulo p+t)\modulo p}=1.
\end{align}
For ``$\Leftarrow$'' we compute
\begin{align}
\Gcd{p,(a_1\modulo p+t)\modulo p}
&=\Gcd{p,((a-t+pb)\modulo p+t)\modulo p}
=\Gcd{p,a}=1.
\end{align}
\end{proof}

It is possible to generalize the above result to functions defined by polynomials over $\Q_p$. Clearly, before we can ask the question of whether such a polynomial is (weakly) $(p,r)$-suitable, we first need to check if it even defines a function on $\Z_p$ in the first place. Consider the example where $p=2$, $r=1$, $f(x)=x/2+1/2\in\Q_2[x]$, and $g(x)=x/2+1\in\Q_2[x]$. In this situation we have $f(n)\in\Z_2$ for all $n\in1+2\Z_2$ but $g(1)=3/2\notin\Z_2$. We define the following predicate on $\Q_p[x]$ which characterizes the decisive property that $f$ has but $g$ has not ($2\leq p\in\N$, $r\in\intintuptoex{p}$, $f\in\Q_p[x]$):
\begin{flalign}
\pushleft{\FPintegral{p}{r}(f)}&\Leftrightarrow f(r+p\Z_p)\subseteq\Z_p&\text{$f$ is \emph{$(p,r)$-integral}}\label{DFPintegral}
\end{flalign}
The following lemma provides an easy characterization of $(p,r)$-integral polynomials. It involves a generalization of the $p$-adic valuation $\nu_p$ to the case where $p$ is not a prime, which is discussed in the appendix.

\begin{lemma}
\label{LRatPolyIntegral}
Let $2\leq p\in\N$, $r\in\intintuptoex{p}$, $f=\sum_{i=0}^da_ix^i\in\Q_p[x]$ (note that for the whole theorem we define $0^0\ce1$), and
\begin{align}
K\ce-\min(\set{\nu_p(a_i)\mid i\in\intintuptoin{d}}\cup\set{0}).
\end{align}
Then,
\begin{align}
W\ce\intintuptoex{p^K}\cap(r+p\Z_p)
\end{align}
is a finite witness set for $f$ being $(p,r)$-integral, i.e. $\FPintegral{p}{r}(f)$ if and only if $f(W)\subseteq\Z_p$.
\end{lemma}

\begin{proof}
Let $n\in r+p\Z_p$ and $N\ce n\modulo p^K\in W$. By definition of $K$ we have $p^Kf\in\Z_p[x]$. Thus,
\begin{align}
f(n)\in\Z_p
\Leftrightarrow p^Kf(n)\in p^K\Z_p
\Leftrightarrow p^Kf(N)\in p^K\Z_p
\Leftrightarrow f(N)\in\Z_p.
\end{align}
\end{proof}

Now that we have characterized possible candidates for (weakly) $(p,r)$-suitable polynomials over $\Q_p$, we are ready to formulate the following generalization of Theorem~\ref{TPropPolyF}.

\begin{theorem}
\label{TPropRatPolyF}
Let $2\leq p\in\N$, $r\in\intintuptoex{p}$, $f:\Z_p\to\Z_p$, $g=\sum_{i=0}^da_ix^i\in\Q_p[x]$ (note that for the whole theorem we define $0^0\ce1$) $(p,r)$-integral with $f(n)=g(n)$ for all $n\in r+p\Z_p$, $k\in\Nz$,
\begin{align}
W_\ell&\ce\intintuptoex{p^\ell}\cap(r+p\Z_p)
\end{align}
for all $\ell\in\Nz$, and
\begin{align}
K&\ce-\min(\set{\nu_p(a_i)\mid i\in\intintuptoin{d}}\cup\set{0}).
\end{align}
Then, $W_{K+k}$ is a finite witness set for $f$ being (weakly) $(p,r)$-suitable at $k$, i.e.

\begin{theoremtable}
\theoremitem{(1)}&$\FPweaklysuitableat{p}{r}{k}(f)\Leftrightarrow$\tabularnewline
&$\quad\fa m,n\in W_{K+k}:m\equiv n\modulus p^k\Rightarrow(f-f\modulo p)(m)\equiv(f-f\modulo p)(n)\modulus p^k$\tabularnewline
\theoremitem{(2)}&$\FPsuitableat{p}{r}{k}(f)\Leftrightarrow$\tabularnewline
&$\quad\fa m,n\in W_{K+k}:m\equiv n\modulus p^k\Leftrightarrow(f-f\modulo p)(m)\equiv(f-f\modulo p)(n)\modulus p^k$.\tabularnewline[0.5\baselineskip]
\end{theoremtable}
Furthermore,

\begin{theoremtable}
\theoremitem{(3)}&$\fa\ell\in\intint{K+1}{\infty}:\FPweaklysuitableat{p}{r}{\ell}(f)\Rightarrow\FPweaklysuitableat{p}{r}{\ell+1}(f)$\tabularnewline
&In particular: $\FPweaklysuitableat{p}{r}{\intint{K+1}{\infty}}(f)\Leftrightarrow\FPweaklysuitableat{p}{r}{K+1}(f)$\tabularnewline
\theoremitem{(4)}&$\fa\ell\in\intint{K+2}{\infty}:(\FPweaklysuitableat{p}{r}{\ell}(f)\land\FPsuitableat{p}{r}{\ell+1}(f))\Rightarrow\FPsuitableat{p}{r}{\ell+2}(f)$\tabularnewline
&In particular: $\FPsuitableat{p}{r}{\intint{K+2}{\infty}}(f)\Leftrightarrow\FPsuitableat{p}{r}{\set{K+2,K+3}}(f)$.\tabularnewline[0.5\baselineskip]
\end{theoremtable}
In particular,

\begin{theoremtable}
\theoremitem{(1)}&$\FPweaklysuitable{p}{r}(f)\Leftrightarrow\FPweaklysuitableat{p}{r}{\intintuptoin{K+1}}(f)\Leftrightarrow$\tabularnewline
&$\quad\fa\ell\in\intintuptoin{K+1}:\fa m,n\in W_{K+\ell}:m\equiv n\modulus p^\ell\Rightarrow(f-f\modulo p)(m)\equiv(f-f\modulo p)(n)\modulus p^\ell$\tabularnewline
\theoremitem{(2)}&$\FPsuitable{p}{r}(f)\Leftrightarrow\FPsuitableat{p}{r}{\intintuptoin{K+3}}(f)\Leftrightarrow$\tabularnewline
&$\quad\fa\ell\in\intintuptoin{K+3}:\fa m,n\in W_{K+\ell}:m\equiv n\modulus p^\ell\Leftrightarrow(f-f\modulo p)(m)\equiv(f-f\modulo p)(n)\modulus p^\ell$.
\end{theoremtable}
\end{theorem}

\hfill

\begin{proof}
$ $\\
\proofitem[\qquad\qquad\qquad]{(1) and (2)}%
Let $m,n\in r+p\Z_p$, $M\ce m\modulo p^{K+k}\in W_{K+k}$, and $N\ce n\modulo p^{K+k}\in W_{K+k}$. Then,
\begin{align}
m\equiv n\modulus p^k
&\Leftrightarrow
M\equiv N\modulus p^k.
\end{align}
By definition of $K$ we have $p^Kg\in\Z_p[x]$. Thus,
\begin{align}
&(f-f\modulo p)(m)\equiv(f-f\modulo p)(n)\modulus p^k\\
&\quad\Leftrightarrow g(m)-g(m)\modulo p-(g(n)-g(n)\modulo p)\in p^k\Z_p\\
&\quad\Leftrightarrow p^Kg(m)-\left(p^Kg(m)\right)\modulo p^{K+1}-\left(p^Kg(n)-\left(p^Kg(n)\right)\modulo p^{K+1}\right)\in p^{K+k}\Z_p\\
&\quad\Leftrightarrow p^Kg(M)-\left(p^Kg(M)\right)\modulo p^{K+1}-\left(p^Kg(N)-\left(p^Kg(N)\right)\modulo p^{K+1}\right)\in p^{K+k}\Z_p\\
&\quad\Leftrightarrow g(M)-g(M)\modulo p-\left(g(N)-g(N)\modulo p\right)\in p^k\Z_p\\
&\quad\Leftrightarrow(f-f\modulo p)(M)\equiv(f-f\modulo p)(N)\modulus p^k
\end{align}
which completes the proof of (1) and (2).

\noindent
\proofitem{(3)}%
First we will prove
\begin{align}
\label{ERatPolyWCan}
m\equiv n\modulus{p^{K+1}}\Rightarrow f(m)\equiv f(n)\modulus{p}
\end{align}
for all $m,n\in r+p\Z_p$. For that let $a\in\Z_p$, such that $n=m+ap^{K+1}$. Then,
\begin{align}
f(m)-f(n)
&=\sum_{i=1}^da_i\left(m^i-\left(m+ap^{K+1}\right)^i\right)\\
&=\sum_{i=1}^da_i\left(m^i-\sum_{j=0}^i\binom{i}{j}m^j\left(ap^{K+1}\right)^{i-j}\right)\\
&=\sum_{i=1}^da_i\sum_{j=0}^{i-1}\binom{i}{j}m^ja^{i-j}p^{(K+1)(i-j)}\\
&=p\sum_{i=1}^dp^Ka_i\sum_{j=0}^{i-1}\binom{i}{j}m^ja^{i-j}p^{(K+1)(i-j-1)}\in p\Z_p
\end{align}
and hence $f(m)\equiv f(n)\modulus{p}$ as claimed.

Next we will show that
\begin{align}
\label{ERatPolyDiffConA}
&\frac{f\left(a+bvp^\ell\right)
-f\left(a+(b+1)vp^\ell\right)}{-vp^\ell}
-\frac{f\left(a+cvp^\ell\right)
-f\left(a+(c+1)vp^\ell\right)}{-vp^\ell}\in p\Z_p
\end{align}
for all $a,b,c,v\in\Z_p$. We have,
\bgroup
\allowdisplaybreaks
\begin{align}
&f\left(a+bvp^\ell\right)
-f\left(a+(b+1)vp^\ell\right)
-\left(f\left(a+cvp^\ell\right)
-f\left(a+(c+1)vp^\ell\right)\right)\\
&\quad
=\!\vphantom{a}\sum_{i=0}^da_i\left(a+bvp^\ell\right)^i
-\sum_{i=0}^da_i\left(a+(b+1)vp^\ell\right)^i-\vphantom{a}\\
&\nonumber
\quad
\phantom{\vphantom{a}=\!\vphantom{a}}
\sum_{i=0}^da_i\left(a+cvp^\ell\right)^i
+\sum_{i=0}^da_i\left(a+(c+1)vp^\ell\right)^i\\
&\quad
=\!\vphantom{a}\sum_{i=0}^da_i\sum_{j=0}^i\binom{i}{j}a^{i-j}b^jv^jp^{\ell j}
-\sum_{i=0}^da_i\sum_{j=0}^i\binom{i}{j}a^{i-j}(b+1)^jv^jp^{\ell j}-\vphantom{a}\\
&\nonumber
\quad
\phantom{\vphantom{a}=\!\vphantom{a}}
\sum_{i=0}^da_i\sum_{j=0}^i\binom{i}{j}a^{i-j}c^jv^jp^{\ell j}
+\sum_{i=0}^da_i\sum_{j=0}^i\binom{i}{j}a^{i-j}(c+1)^jv^jp^{\ell j}\\
&\quad
=\!\vphantom{a}\sum_{i=0}^da_i\sum_{j=0}^i\binom{i}{j}a^{i-j}\left(b^j-(b+1)^j-c^j+(c+1)^j\right)v^jp^{\ell j}\\
&\quad
=\!\vphantom{a}\sum_{i=2}^da_i\sum_{j=2}^i\binom{i}{j}a^{i-j}\left(b^j-(b+1)^j-c^j+(c+1)^j\right)v^jp^{\ell j}\\
&\quad
=v^2p^{2\ell-K}\sum_{i=2}^da_ip^K\sum_{j=2}^i\binom{i}{j}a^{i-j}\left(b^j-(b+1)^j-c^j+(c+1)^j\right)v^{j-2}p^{\ell(j-2)}\\
&\quad
\in v^2p^{2\ell-K}\Z_p
\end{align}
\egroup
and hence
\begin{align}
\label{ERatPolyDiffConB}
&\frac{f\left(a+bvp^\ell\right)
-f\left(a+(b+1)vp^\ell\right)}{-vp^\ell}
-\frac{f\left(a+cvp^\ell\right)
-f\left(a+(c+1)vp^\ell\right)}{-vp^\ell}\in vp^{\ell-K}\Z_p.
\end{align}
Since $\ell\geq K+1$, it follows that $vp^{\ell-K}\Z_p\subseteq p\Z_p$.

Using the preliminary results above we will now prove (3). From Eqn.~(\ref{ERatPolyWCan}) it follows that
\begin{align}
\label{ERatPolyAssA}
m\equiv n\modulus{p^\ell}&\Rightarrow f(m)\equiv f(n)\modulus{p^\ell}
\intertext{%
for all $m,n\in r+p\Z_p$ and by (1) and Eqn.~(\ref{ERatPolyWCan}) again we need to show that
}
m\equiv n\modulus{p^{\ell+1}}&\Rightarrow f(m)\equiv f(n)\modulus{p^{\ell+1}}
\end{align}
for all $m,n\in W_{K+\ell+1}$. Let $m,n\in W_{K+\ell+1}$ such that $m\equiv n\modulus{p^\ell}$. Our goal is to show that
\begin{align}
\label{ERatPolyGoalA}
\frac{f(m)-f(n)}{m-n}\in\Z_p,
\end{align}
because then $p^{\ell+1}$ divides $f(m)-f(n)=(m-n)\frac{f(m)-f(n)}{m-n}$ if $p^{\ell+1}$ divides $m-n$ (i.e. if $m\equiv n\modulus{p^{\ell+1}}$). Assume without loss of generality that $m<n$ (clearly, if $m=n$, then $f(m)\equiv f(n)\modulus{p^{\ell+1}}$) and let $a\in\N$ such that $n=m+ap^\ell$. We compute,
\begin{align}
\frac{f(m)-f(n)}{m-n}
&=\frac{f(m)-f\left(m+ap^\ell\right)}{-ap^\ell}\\
\label{ERatPolyAssB}
&=\frac{1}{a}\sum_{i=0}^{a-1}\frac{f\left(m+ip^\ell\right)-f\left(m+(i+1)p^\ell\right)}{-p^\ell}.
\intertext{%
Since
}
&m+ip^\ell\equiv m+(i+1)p^\ell\modulus{p^\ell}
\intertext{%
for all $i\in\intintuptoex{a}$, it follows from Eqn.~(\ref{ERatPolyAssA}) that
}
&f\left(m+ip^\ell\right)\equiv f\left(m+(i+1)p^\ell\right)\modulus{p^\ell}
\intertext{%
and hence
}
&\frac{f\left(m+ip^\ell\right)-f\left(m+(i+1)p^\ell\right)}{-p^\ell}\in\Z_p
\intertext{%
for all $i\in\intintuptoex{a}$. Thus all of the $a$ summands of Eqn.~(\ref{ERatPolyAssB}) are $p$-adic integers and by Eqn.~(\ref{ERatPolyDiffConA}) they are pairwise congruent modulo $p$. If $\gcd(p,a)=1$, it would thus follow that $\frac{f(m)-f(n)}{m-n}\in\Z_p$ and we are done. If $p$ and $a$ are not coprime, let $q_0\in\intintuptoin{p}$ be a common divisor of $p$ and $a$ and let $b_0\in\N$ such that $a=b_0q_0$. Then,
}
\frac{f(m)-f(n)}{m-n}
&=\frac{1}{b_0q_0}\sum_{i=0}^{b_0q_0-1}\frac{f\left(m+ip^\ell\right)-f\left(m+(i+1)p^\ell\right)}{-p^\ell}\\
&=\frac{1}{b_0}\sum_{j=0}^{b_0-1}\frac{1}{q_0}\sum_{i=0}^{q_0-1}\frac{f\left(m+(jq_0+i)p^\ell\right)-f\left(m+(jq_0+i+1)p^\ell\right)}{-p^\ell}
\intertext{%
where all summands of the inner sum are $p$-adic integers that are pairwise congruent modulo $p$ by Eqn.~(\ref{ERatPolyDiffConA}). Thus, the whole inner sum is divisible by $q_0$ and we get
}
\frac{f(m)-f(n)}{m-n}
&=\frac{1}{b_0}\sum_{i=0}^{b_0-1}\frac{f\left(m+iq_0p^\ell\right)-f\left(m+(i+1)q_0p^\ell\right)}{-q_0p^\ell}
\intertext{%
where all of the summands are again $p$-adic integers. As before, if $\gcd(p,b_0)=1$, we are done, and if $p$ and $b_0$ are not coprime, we let $q_1\in\intintuptoin{p}$ be a common divisor of $p$ and $b_0$, and $b_1\in\N$ such that $b_0=b_1q_1$. Then,
}
\frac{f(m)-f(n)}{m-n}
&=\frac{1}{b_1q_1}\sum_{i=0}^{b_1q_1-1}\frac{f\left(m+iq_0p^\ell\right)-f\left(m+(i+1)q_0p^\ell\right)}{-q_0p^\ell}\\
&=\frac{1}{b_1}\sum_{j=0}^{b_1-1}\frac{1}{q_1}\sum_{i=0}^{q_1-1}\frac{f\left(m+(jq_1+i)q_0p^\ell\right)-f\left(m+(jq_1+i+1)q_0p^\ell\right)}{-q_0p^\ell}.
\intertext{%
where, again, all summands of the inner sum are $p$-adic integers that are pairwise congruent modulo $p$ by Eqn.~(\ref{ERatPolyDiffConA}). Thus, the whole inner sum is  divisible by $q_1$ and we get
}
\frac{f(m)-f(n)}{m-n}
&=\frac{1}{b_1}\sum_{i=0}^{b_1-1}\frac{f\left(m+iq_1q_0p^\ell\right)-f\left(m+(i+1)q_1q_0p^\ell\right)}{-q_1q_0p^\ell}
\intertext{%
where all of the summands are again $p$-adic integers. Continuing iteratively we find $u\in\Nz$ (since $a\in\N$ cannot be a zero divisor), $q_0,\ldots,q_{u-1}\in\intintuptoin{p}$, and $b_0,\ldots,b_{u-1}\in\N$ such that $a=b_0q_0$, $b_i=b_{i+1}q_{i+1}$ for all $i\in\intintuptoex{u-1}$, $q_0$ is a common divisor of $p$ and $a$, $q_{i+1}$ is a common divisor of $p$ and $b_i$ for all $i\in\intintuptoex{u-1}$. Furthermore,
}
\frac{f(m)-f(n)}{m-n}
&=\frac{1}{b_{u-1}}\sum_{i=0}^{b_{u-1}-1}\frac{f\left(m+iq_{u-1}\cdots q_0p^\ell\right)-f\left(m+(i+1)q_{u-1}\cdots q_0p^\ell\right)}{-q_{u-1}\cdots q_0p^\ell},
\end{align}
with all summands being $p$-adic integers, and $\gcd(p,b_{u-1})=1$. But then $\frac{f(m)-f(n)}{m-n}\in\Z_p$, which completes the proof of (3).

\noindent
\proofitem{(4)}%
From Eqn.~(\ref{ERatPolyWCan}) it follows that
\begin{align}
m\equiv n\modulus{p^\ell}&\Rightarrow f(m)\equiv f(n)\modulus{p^\ell}\\
m\equiv n\modulus{p^{\ell+1}}&\Leftrightarrow f(m)\equiv f(n)\modulus{p^{\ell+1}}
\intertext{%
for all $m,n\in r+p\Z_p$ and by (2), (3), and Eqn.~(\ref{ERatPolyWCan}) we need to show that
}
f(m)\equiv f(n)\modulus{p^{\ell+2}}&\Rightarrow m\equiv n\modulus{p^{\ell+2}}
\end{align}
for all $m,n\in W_{K+\ell+2}$. Let $m,n\in W_{K+\ell+2}$ such that $f(m)\equiv f(n)\modulus{p^{\ell+1}}$. Then $m\equiv n\modulus{p^{\ell+1}}$ and we have $\frac{f(m)-f(n)}{m-n}\in\Z_p$ by Eqn.~(\ref{ERatPolyGoalA}). Our goal is to show that
\begin{align}
\label{ERatPolyGoalB}
\gcd\left(p,\frac{f(m)-f(n)}{m-n}\modulo p\right)=1
\end{align}
because then (cf. Eqn.~(\ref{EDiv})) $p^{\ell+2}$ divides $m-n$ if $p^{\ell+2}$ divides $f(m)-f(n)=(m-n)\frac{f(m)-f(n)}{m-n}$ (i.e. if $f(m)\equiv f(n)\modulus{p^{\ell+2}}$). Assume without loss of generality that $m<n$ (clearly, if $m=n$, then $m\equiv n\modulus{p^{\ell+2}}$) and let $a\in\N$ such that $n=m+ap^{\ell+1}$.

As a preliminary step we will prove that
\begin{align}
\label{ERatPolyAssC}
\gcd\left(p,b_i\modulo p\right)=1,
\end{align}
for all $i\in\Z$, where
\begin{align}
b_i\ce\frac{f\left(m+ip^\ell\right)-f\left(m+(i+1)p^\ell\right)}{-p^\ell}
\end{align}
which is in $\Z_p$ by Eqn.~(\ref{ERatPolyGoalA}) (here we need $\FPweaklysuitableat{p}{r}{\ell}(f)$). For that assume to the contrary that there is a divisor $q\in\intint{2}{p}$ of $p$ such that
\begin{align}
\ex i\in\Z:\gcd\left(p,b_i\modulo p\right)=q.
\end{align}
By Eqn.~(\ref{ERatPolyDiffConA}) the least significant digits of all $b_i$ coincide, say  $b_i\modulo p=s\in\intintuptoex{p}$ for all $i\in\Z$, which implies,
\begin{align}
\fa i\in\Z:\gcd\left(p,b_i\modulo p\right)=\gcd\left(p,s\right)=q.
\end{align}
We compute
\begin{align}
\frac{f\left(m\right)-f\left(m+p/qp^\ell\right)}{-p^\ell}
&=\sum_{i=0}^{p/q-1}b_i\\
&=p\sum_{i=0}^{p/q-1}\frac{b_i-s+s}{p}\\
&=p\left(\frac{s}{q}+\sum_{i=0}^{p/q-1}\frac{b_i-b_i\modulo p}{p}\right)\in p\Z_p
\end{align}
which implies that $f\left(m\right)\equiv f\left(m+p/qp^\ell\right)\modulus{p^{\ell+1}}$. But then $m\equiv m+p/qp^\ell\modulus{p^{\ell+1}}$ by $\FPsuitableat{p}{r}{\ell+1}(f)$ which is a contradiction, since $q$ divides $p$ and $q\geq 2$. This completes the proof of Eqn.~(\ref{ERatPolyAssC}).

Furthermore, we need the following general fact:
\begin{align}
\label{ERatPolyAssD}
\fa r\in\intintuptoex{p^2}:\fa a\in\N\text{, }a\mid p\text{ (in $\Z_p$) }:\fa b_0,\ldots,b_{a-1}\in r+p^2\Z_p:\frac{1}{a}\sum_{i=0}^{a-1}b_i\equiv r\modulus{p}.
\end{align}
In order to prove it let $c_i\in\Z_p$ such that $b_i=r+p^2c_i$ for all $i\in\intintuptoex{a}$. Then,
\begin{align}
\frac{1}{a}\sum_{i=0}^{a-1}b_i
&=\frac{1}{a}\sum_{i=0}^{a-1}\left(r+p^2c_i\right)
=r+p\frac{p}{a}\sum_{i=0}^{a-1}c_i
\equiv r\modulus{p}.
\end{align}

We continue with the proof of Eqn.~(\ref{ERatPolyGoalB}) and compute
\begin{align}
\frac{f(m)-f(n)}{m-n}
&=\frac{f(m)-f\left(m+ap^{\ell+1}\right)}{-ap^{\ell+1}}\\
\label{ERatPolyAssE}
&=\frac{1}{a}\sum_{i=0}^{a-1}\frac{f\left(m+ip^\ell\right)-f\left(m+(i+1)p^\ell\right)}{-p^\ell}.\\
\intertext{%
By Eqn.~(\ref{ERatPolyAssC}) it follows that the least significant digits of all summands of Eqn.~(\ref{ERatPolyAssE}) are coprime to $p$. Furthermore, they are all congruent modulo $p^2$ by Eqn.~(\ref{ERatPolyDiffConB}), since $\ell\geq K+2$. If $\gcd(p,a)=1$ (in which case $a$ divides $p$ in $\Z_p$), it would thus follow that $\gcd\left(p,\frac{f(m)-f(n)}{m-n}\modulo p\right)=1$ due to Eqn.~(\ref{ERatPolyAssD}). If $p$ and $a$ are not coprime, let $q_0\in\intintuptoin{p}$ be a common divisor of $p$ and $a$ and let $b_0\in\N$ such that $a=b_0q_0$. Then,
}
\frac{f(m)-f(n)}{m-n}
&=\frac{1}{b_0q_0}\sum_{i=0}^{b_0q_0-1}\frac{f\left(m+ip^\ell\right)-f\left(m+(i+1)p^\ell\right)}{-p^\ell}\\
&=\frac{1}{b_0}\sum_{j=0}^{b_0-1}\frac{1}{q_0}\sum_{i=0}^{q_0-1}\frac{f\left(m+(jq_0+i)p^\ell\right)-f\left(m+(jq_0+i+1)p^\ell\right)}{-p^\ell}
\intertext{%
where all summands of the inner sum are $p$-adic integers that are pairwise congruent modulo $p^2$ by Eqn.~(\ref{ERatPolyDiffConB}), while their least significant digits are coprime to $p$. Thus, the least significant digits of the outer summands (i.e. the inner sum divided by $q_0$) are coprime to $p$ by Eqn.~(\ref{ERatPolyAssD}), since $q_0$ divides $p$ and we get
}
\frac{f(m)-f(n)}{m-n}
&=\frac{1}{b_0}\sum_{i=0}^{b_0-1}\frac{f\left(m+iq_0p^\ell\right)-f\left(m+(i+1)q_0p^\ell\right)}{-q_0p^\ell}
\intertext{%
where the least significant digits of the summands are coprime to $p$. As before, if $\gcd(p,b_0)=1$, we are done by Eqn.~(\ref{ERatPolyAssD}), and if $p$ and $b_0$ are not coprime, we let $q_1\in\intintuptoin{p}$ be a common divisor of $p$ and $b_0$, and $b_1\in\N$ such that $b_0=b_1q_1$. Then,
}
\frac{f(m)-f(n)}{m-n}
&=\frac{1}{b_1q_1}\sum_{i=0}^{b_1q_1-1}\frac{f\left(m+iq_0p^\ell\right)-f\left(m+(i+1)q_0p^\ell\right)}{-q_0p^\ell}\\
&=\frac{1}{b_1}\sum_{j=0}^{b_1-1}\frac{1}{q_1}\sum_{i=0}^{q_1-1}\frac{f\left(m+(jq_1+i)q_0p^\ell\right)-f\left(m+(jq_1+i+1)q_0p^\ell\right)}{-q_0p^\ell}.
\intertext{%
where, again, all summands of the inner sum are $p$-adic integers that are pairwise congruent modulo $p^2$ by Eqn.~(\ref{ERatPolyDiffConB}), and their least significant digits are coprime to $p$. Thus, the least significant digits of the outer summands (i.e. the inner sum divided by $q_1$) are coprime to $p$ by Eqn.~(\ref{ERatPolyAssD}), since $q_1$ divides $p$ and we get
}
\frac{f(m)-f(n)}{m-n}
&=\frac{1}{b_1}\sum_{i=0}^{b_1-1}\frac{f\left(m+iq_1q_0p^\ell\right)-f\left(m+(i+1)q_1q_0p^\ell\right)}{-q_1q_0p^\ell}
\intertext{%
where the least significant digits of the summands are coprime to $p$. Continuing iteratively we find $u\in\Nz$ (since $a\in\N$ cannot be a zero divisor), $q_0,\ldots,q_{u-1}\in\intintuptoin{p}$, and $b_0,\ldots,b_{u-1}\in\N$ such that $a=b_0q_0$, $b_i=b_{i+1}q_{i+1}$ for all $i\in\intintuptoex{u-1}$, $q_0$ is a common divisor of $p$ and $a$, $q_{i+1}$ is a common divisor of $p$ and $b_i$ for all $i\in\intintuptoex{u-1}$,
}
\frac{f(m)-f(n)}{m-n}
&=\frac{1}{b_{u-1}}\sum_{i=0}^{b_{u-1}-1}\frac{f\left(m+iq_{u-1}\cdots q_0p^\ell\right)-f\left(m+(i+1)q_{u-1}\cdots q_0p^\ell\right)}{-q_{u-1}\cdots q_0p^\ell}
\end{align}
with the least significant digits of the summands being coprime to $p$, and $\gcd(p,b_{u-1})=1$. But then $\gcd\left(p,\frac{f(m)-f(n)}{m-n}\modulo p\right)=1$ by Eqn.~(\ref{ERatPolyAssD}) which completes the proof of (4).
\end{proof}

The difference it makes to go from polynomials in $\Z_p[x]$ to $(p,r)$-integral polynomials in $\Q_p[x]$ in the context of $p$-adic systems is quite remarkable. While all polynomials in $\Z_p[x]$ are $(p,r)$-integral (trivially) and weakly $(p,r)$-suitable (Theorem~\ref{TPropPolyF}) for all $2\leq p\in\N$ and all $r\in\intintuptoex{p}$, both need to be checked algorithmically for polynomials in $\Q_p[x]$ (Lemma~\ref{LRatPolyIntegral} and Theorem~\ref{TPropRatPolyF}). $f(x)=1/2x$ is an easy example of a polynomial that is $(2,0)$-integral but not weakly $(2,0)$-suitable at $2$ ($f(0)=0\not\equiv2=f(2^2)\modulus{2^2}$). Furthermore, it is very easy to check whether  a polynomial $f\in\Z_p[x]$ is $(p,r)$-suitable, as this only depends on the derivative of $f$ in $r$ (Theorem~\ref{TPropPolyF}). For polynomials in $\Q_p[x]$ however, $(p,r)$-suitability needs to be checked algorithmically (Theorem~\ref{TPropRatPolyF}). As an example consider $f(x)=1/4x^3+x^2+x$ which is $(2,0)$-integral, weakly $(2,0)$-suitable, but not $(2,0)$-suitable at $2$ ($f(0)=0\equiv2^3=f(2)\modulus{2^2}$) despite the fact that $f'(0)=1$ is coprime to $2$. $f(x)=1/8x^3+x^2+x$ defines a polynomial that is not even weakly $(2,0)$-suitable at $2$ ($f(2)=7\not\equiv69=f(2+2^2)\modulus{2^2}$) and yet again $f'(0)=1$. Another difference is that for any $p$-adic system $\FFf{F}$ defined by polynomials in $\Z_p[x]$ there exits a $p$-adic system $\FFf{G}$ defined by polynomials in $\Z_p[x]$ that is a weak canonical form of $\FFf{F}$ (e.g. $(x,3x+1)$ which is a weak canonical form of $(x+1,3x+2)$). For $p$-adic systems defined by polynomials in $\Q_p[x]$ this is no longer the case. Indeed, if $f(x)=1/16x^4+x^3+1/2x^2+x+1$, then $f$ is $(2,0)$-integral (even $(2,0)$-suitable) but $f(0)=1\not\equiv14=f(2)\modulus{2}$, so the $2$-adic system $\FFf{F}\ce(f(x),x-1)$ does not have a weak canonical form that can be expressed with polynomials (a weak canonical form of $\FFf{F}$ is given by $((f-f\modulo2)(x),x-1)$ but $f-f\modulo2$ is not a polynomial function). Further indications of the much more erratic behavior of polynomials in $Q_p[x]$ are given by the following examples. Let
\begin{align}
f(x)&=-3/512x^7+1/128x^5+x^4+1/8x^2+1\\
g(x)&=1/512x^7+1/128x^5+x^4+1/8x^2+1\\
h(x)&=1/32x^{11}+1/2x^9+1/16x^8+1/16x^7+\vphantom{a}\\
\nonumber
&\phantom{\vphantom{a}=\vphantom{a}}1/8x^6+1/16x^5+1/32x^4+1/8x^3-1/8x^2+x+1
\end{align}
and $k\in\N$. Then $f$, $g$, and $h$ are $(2,0)$-integral and we have
\begin{align}
\FPweaklysuitableat{2}{0}{k}(f)&\Leftrightarrow k\in\set{1,3}\\
\FPsuitableat{2}{0}{k}(f)&\Leftrightarrow k\in\set{1}\\
\FPweaklysuitableat{2}{0}{k}(g)&\Leftrightarrow k\in\set{1}\cup\intint{4}{\infty}\\
\FPsuitableat{2}{0}{k}(g)&\Leftrightarrow k\in\set{1}\\
\FPweaklysuitableat{2}{0}{k}(h)&\Leftrightarrow k\in\set{1}\cup\intint{3}{\infty}\\
\FPsuitableat{2}{0}{k}(h)&\Leftrightarrow k\in\set{1,3}.
\end{align}

Next, we will fix our notions in relation to $p$-adic systems defined by polynomials. We define the following predicates on $\SoFibredFunctions{p}$ ($2\leq p\in\N$, $\FFf{F}\in\SoFibredFunctions{p}$, $A\subseteq\Q_p$, $D\subseteq\Nz\cup\set{-\infty}$):
\begin{flalign}
\pushleft{\FFPpolynomialcoefficientsdegree{A}{D}(\FFf{F})}&\Leftrightarrow\FFPdomain{\Z_p}(\FFf{F})&\hspace{-10em}\text{$\FFf{F}$ is \emph{$A$-polynomial with degree in $D$ or, if $D=\set{d}$,}}\label{DFFPpolynomialcoefficientsdegree}\\
\nonumber
&\phantom{\vphantom{a}\Leftrightarrow\vphantom{a}}\fa r\in\intintuptoex{p}:\FFf{F}[r]\vert_{r+p\Z_p}\in A[x]&\text{$\FFf{F}$ is \emph{$A$-polynomial of degree $d$}}\\
\nonumber
&\phantom{\vphantom{a}\Leftrightarrow\fa r\in\intintuptoex{p}:\vphantom{a}}\deg(\FFf{F}[r]\vert_{r+p\Z_p})\in D\\
\pushleft{\FFPpolynomialcoefficients{A}(\FFf{F})}&\Leftrightarrow\FFPpolynomialcoefficientsdegree{A}{\Nz\cup\set{-\infty}}(\FFf{F})&\text{$\FFf{F}$ is \emph{$A$-polynomial}}\label{DFFPpolynomialcoefficients}\\
\pushleft{\FFPpolynomial(\FFf{F})}&\Leftrightarrow\FFPpolynomialcoefficientsdegree{\Q_p}{\Nz\cup\set{-\infty}}(\FFf{F})&\text{$\FFf{F}$ is \emph{polynomial}}\label{DFFPpolynomial}\\
\pushleft{\FFPlinearpolynomialcoefficients{A}(\FFf{F})}&\Leftrightarrow\FFPpolynomialcoefficientsdegree{A}{\set{-\infty,0,1}}(\FFf{F})&\text{$\FFf{F}$ is \emph{$A$-linear-polynomial}}\label{DFFPlinearpolynomialcoefficients}\\
\pushleft{\FFPlinearpolynomial(\FFf{F})}&\Leftrightarrow\FFPpolynomialcoefficientsdegree{\Q_p}{\set{-\infty,0,1}}(\FFf{F})&\text{$\FFf{F}$ is \emph{linear-polynomial}}\label{DFFPlinearpolynomial}
\end{flalign}

\noindent
Using our new predicates it is easy to formulate the following corollary to Theorem~\ref{TPropPolyF}.

\begin{corollary}
\label{CPropPolyF}
Let $2\leq p\in\N$ and $\FFf{F}\in\SoFibredFunctions{p}(\FFPpolynomialcoefficients{\Z_p})$. Then,

\begin{theoremtable}
\theoremitem{(1)}&$\FFPweakblockS(\FFf{F})$\tabularnewline
\theoremitem{(2)}&$\FFPblock(\FFf{F})\Leftrightarrow\ex k\in\intint{2}{\infty}:\FFPblockat{k}(\FFf{F})$.
\end{theoremtable}
\end{corollary}

\begin{proof}
Follows directly from Theorem~\ref{TPropPolyF}.
\end{proof}

\noindent
A similar corollary to Theorem~\ref{TPropRatPolyF} can also be formulated.

\begin{corollary}
\label{CPropRatPolyF}
Let $2\leq p\in\N$, $\FFf{F}\in\SoFibredFunctions{p}(\FFPpolynomial)$, and $K\in\Nz$ such that $-K$ is the minimum of $0$ and the $p$-adic valuations of all coefficients of the polynomials $\FFf{F}[r]\vert_{r+p\Z_p}$, $r\in\intintuptoex{p}$. Then,

\begin{theoremtable}
\theoremitem{(1)}&$\FFPweakblockS(\FFf{F})\Leftrightarrow\FFPweakblockSat{\intintuptoin{K+1}}(\FFf{F})$\tabularnewline
\theoremitem{(2)}&$\FFPblock(\FFf{F})\Leftrightarrow\FFPblockat{\intintuptoin{K+3}}(\FFf{F})$.
\end{theoremtable}
\end{corollary}

\begin{proof}
Follows directly from Theorem~\ref{TPropRatPolyF}.
\end{proof}

Theorem~\ref{TPropPolyF} and Corollary~\ref{CPropPolyF} have a remarkable consequence: every $\Z_p$-polynomial $p$-fibred function has the weak block property and ``almost all'' of them also have the block property in the sense that if $P_0,\ldots,P_{p-1}\in\Z_p[x]$ are arbitrary, then there are $a_0,\ldots,a_{p-1}\in\intintuptoex{p}$ such that $(P_0(x)+a_0x,\ldots,P_{p-1}(x)+a_{p-1}x)$ has the block property. In other words: every $p$-fibred function defined by polynomials over $\Z_p$ can be turned into a $p$-adic system by only modifying the linear coefficients! By Theorem~\ref{TPropPolyF} all of the following $p$-fibred functions are examples of $p$-adic systems:

\begin{theoremtable}
$\bullet$&$(x)^p=(x,\ldots,x)\in\SoSystems{p}$ where $2\leq p\in\N$ (standard base $p$)\tabularnewline
$\bullet$&$(x,3x+1)\in\SoSystems{2}$ (Collatz)\tabularnewline
$\bullet$&$(7x^3-4x^2+x-6,3x^7-x+1,x^2+6x+2)\in\SoSystems{3}$\tabularnewline
$\bullet$&$(\frac{32}{7}x^2+\frac{5}{3}x-4,\frac{13}{11}x+5,\frac{1}{17}x+2,3x^2+\frac{7}{19}x-\frac{14}{5})\in\SoSystems{4}$\tabularnewline
$\bullet$&$(\I x^2+x,5\I x^4-2+7,x+3,-9x^3+12x+7,-5\I x^2+x+1)\in\SoSystems{5}$\tabularnewline
&where $\I^2=-1$, i.e. $\I\in\set{\ldots2431212,\ldots2013233}\subseteq\Z_5$\tabularnewline
$\bullet$&$\left(\prod_{i=0}^{p-1}(x-i)\right)^p\in\SoSystems{p}$ where $p\in\P$.\tabularnewline[0.5\baselineskip]
\end{theoremtable}

\noindent
Theorem~\ref{TPropRatPolyF} provides the following additional examples of $p$-adic systems:

\begin{theoremtable}
$\bullet$&$(\frac{17}{4}x^6+\frac{37}{16}x^5-107x^4-\frac{15}{4}x^3+78x^2+3x-2,-25x^6+\frac{7}{4}x^5-\frac{49}{2}x^4-\frac{21}{2}x^3+5x^2+\frac{79}{4}x+\frac{19}{2})\in\SoSystems{2}$\tabularnewline
$\bullet$&$(-\frac{23}{27}x^3-11x^2+x-20,-28x^4+\frac{7}{9}x^3+\frac{29}{3}x^2+\frac{4}{3}x+\frac{11}{9},-2x^4-\frac{29}{9}x^3-\frac{5}{3}x^2-\frac{11}{3}x-\frac{2}{9})\in\SoSystems{3}$.\tabularnewline[0.5\baselineskip]
\end{theoremtable}

A direct consequence of the fact that every $\Z_p$-polynomial $p$-fibred function has the weak block property is, that whenever we extend a $p$-fibred function $\FFf{F}$ with domain $\Z$ that is defined by polynomial functions with integer coefficients (such as $\FFf{F}_p$ or $\FFf{F}_C$) to $\Z_p$ by simply changing the domain from $\Z$ to $\Z_p$ while keeping the polynomials fixed (cf. Eqn.~(\ref{ECollatzZ}) and Eqn.~(\ref{ECollatzZtwo})), the $p$-digit table of the new extended $p$-fibred function will coincide with the unique extension of the $p$-digit table of $\FFf{F}$ as given by Lemma~\ref{LDTExt}.

\begin{corollary}
\label{CPolySDTExt}
Let $2\leq p\in\N$, $\FFf{G}\in\SoFibredFunctions{p}(\FFPpolynomialcoefficients{\Z})$, $\FFf{F}\ce\FFf{G}\vert_\Z$, and $\STf{E}\in\SoDigitTables{p}(\STPdomain{\Z_p},\DTPweakblock)$ such that $\STf{E}\vert_\Z=\FFFDT(\FFf{F})$ (cf. Lemma~\ref{LDTExt} and note that $\FFPweakblock{\FFf(F)}$ by Corollary~\ref{CPropPolyF}~(1)). Then, $\STf{E}=\FFFDT(\FFf{G})$.
\end{corollary}

\begin{proof}
Follows directly from Corollary~\ref{CFFwbDTExt} and Corollary~\ref{CPropPolyF}~(1).
\end{proof}

It is possible to reduce the degrees of the polynomials defining a $\Z_p$-polynomial $p$-fibred function $\FFf{F}$ to $k-1$ while keeping $\FFFDT(\FFf{F})\llbracket\intintuptoex{k}\rrbracket$ constant. In order to prove this we need the following lemma which utilizes the well-known notion of Vandermonde matrices.

\begin{lemma}
\label{LReduceDegree}
Let $2\leq p\in\N$, $r\in\intintuptoex{p}$, $f\in\Z_p[x]$, and $k\in\N$. Furthermore, let
\begin{align}
A&\ce\left((r+ip)^j\right)_{i,j\in\intintuptoex{k}}=
\left(\begin{matrix}
(r+(\mathrlap{0}\phantom{k-1})p)^0&\ldots&(r+(\mathrlap{0}\phantom{k-1})p)^{k-1}\\
\vdots&&\vdots\\
(r+(k-1)p)^0&\ldots&(r+(k-1)p)^{k-1}
\end{matrix}\right)
\in\N^{k\times k}\\
b&\ce
\left(\begin{matrix}
f(r+(\mathrlap{0}\phantom{k-1})p)\\
\vdots\\
f(r+(k-1)p)
\end{matrix}\right)\in\Z_p^k
\end{align}
(note that for the whole theorem we define $0^0\ce1$). Then $A$ is invertible, $A^{-1}\cdot b\in\Z_p^k$, and
\begin{align}
g(x)\ce\left((A^{-1}\cdot b)\modulo p^k\right)\cdot(1,x,\ldots,x^{k-1})
\end{align}
satisfies $g(x)\in\intintuptoex{p^k}[x]$, $\deg(g)\leq k-1$, and $f(n)\equiv g(n)\modulus{p^k}$ for all $n\in r+p\Z_p$.
\end{lemma}

\begin{proof}
We consider the following system of $k$ equations in $a_0,\ldots,a_{k-1}\in\Q_p$
\begin{align}
\begin{matrix*}[l]
a_0(r+(\mathrlap{0}\phantom{k-1})p)^0&+\cdots+&a_{k-1}(r+(\mathrlap{0}\phantom{k-1})p)^{k-1}&=&f(r+(\mathrlap{0}\phantom{k-1})p)\\
&&\qquad\vdots\\
a_0(r+(k-1)p)^0&+\cdots+&a_{k-1}(r+(k-1)p)^{k-1}&=&f(r+(k-1)p)
\end{matrix*}
\end{align}
which is solved by $(a_0,\ldots,a_{k-1})=A^{-1}\cdot b\in\Q_p^k$. It can be easily verified by induction on $k$ that
\begin{align}
\det(A)=\prod_{i=0}^{k-1}i!p^i\neq0
\end{align}
which implies that $A^{-1}$ does indeed exist. To prove that $A^{-1}\cdot b\in\Z_p^k$ let $u,v\in\Z_p[x]$ with $\deg(v)<k$ such that
\begin{align}
f(x)=\left(\prod_{i=0}^{k-1}(x-(r+ip))\right)u(x)+v(x).
\end{align}
Then $v(r+ip)=f(r+ip)$ for all $i\in\intintuptoex{k}$, which implies that the polynomial $v$ of degree less than $k$ interpolates the $k$ points $(r+ip,f(r+ip))$, $i\in\intintuptoex{k}$ and thus is uniquely defined (within $\Q_p[x]$). In the same way the polynomial $(A^{-1}\cdot b)\cdot(1,x,\ldots,x^{k-1})$ is of degree less than $k$ and interpolates the same points, which implies that $(A^{-1}\cdot b)\cdot(1,x,\ldots,x^{k-1})=v(x)\in\Z_p[x]$ and hence $A^{-1}\cdot b\in\Z_p^k$.

We are left to show that $f(n)\equiv g(n)\modulus{p^k}$ for all $n\in r+p\Z_p$ and compute
\begin{align}
f(r+ap)-v(r+ap)
&=\left(\prod_{i=0}^{k-1}(r+ap-(r+ip))\right)u(r+ap)\\
&=p^k\left(\prod_{i=0}^{k-1}(a-i)\right)u(r+ap)\\
&\in p^k\Z_p
\end{align}
for all $a\in\Z_p$. Hence $f(n)\equiv v(n)\equiv g(n)\modulus{p^k}$ for all $n\in r+p\Z_p$.
\end{proof}

\ignore{
\begin{lemma}
\label{LReduceDegreeAll}
Let $2\leq p\in\N$, $f\in\Z_p[x]$, and $k\in\N$. Furthermore, let
\begin{align}
A&\ce\left(i^j\right)_{i,j\in\intintuptoex{pk}}=
\left(\begin{matrix}
0^0&\ldots&0^{pk-1}\\
\vdots&&\vdots\\
(pk-1)^0&\ldots&(pk-1)^{pk-1}
\end{matrix}\right)
\in\N^{(pk)\times(pk)}\\
b&\ce
\left(\begin{matrix}
f(0)\\
\vdots\\
f(pk-1)
\end{matrix}\right)\in\Z_p^{pk}
\end{align}
(note that for the whole theorem we define $0^0\ce1$). Then $A$ is invertible, $A^{-1}\cdot b\in\Z_p^{pk}$, and
\begin{align}
g(x)\ce\left((A^{-1}\cdot b)\modulo p^k\right)\cdot(1,x,\ldots,x^{pk})
\end{align}
satisfies $g(x)\in\intintuptoex{p^k}[x]$, $\deg(g)\leq pk-1$, and $f(n)\equiv g(n)\modulus{p^k}$ for all $n\in\Z_p$.
\end{lemma}
}

\begin{theorem}
\label{TReduceDegree}
For arbitrary $2\leq p\in\N$, $r\in\intintuptoex{p}$, $f\in\Z_p[x]$, and $k\in\N$ let $g_{p,r,f,k}\in\intintuptoex{p^k}[x]$ denote the polynomial $g(x)$ in Lemma~\ref{LReduceDegree} for the given parameters $p,r,f,k$. Let $2\leq p\in\N$, $\FFf{F}\in\SoFibredFunctions{p}(\FFPdomain{\Z_p},\FFPpolynomialcoefficients{\Z_p})$, and $k\in\N$. Furthermore, let $\FFf{G}\in\SoFibredFunctions{p}(\FFPdomain{\Z_p},\FFPpolynomial)$ such that $\FFf{G}[r](x)=g_{p,r,F[r],k}(x)$ for all $r\in\intintuptoex{p}$. Then $\FFPpolynomialcoefficientsdegree{\intintuptoex{p^k}}{\intintuptoex{k}}(\FFf{G})$ and $\FFFDT(\FFf{F})\llbracket\intintuptoex{k}\rrbracket=\FFFDT(\FFf{G})\llbracket\intintuptoex{k}\rrbracket$.
\end{theorem}

\begin{proof}
Follows directly from Theorem~\ref{TPropPolyF}~(1) ($\FFf{F}[r]$ is weakly $(p,r)$-suitable for all $r\in\intintuptoex{p}$), Theorem~\ref{TFuncSuit}~(1) (therefore $\FFf{F}$ has the weak block property), Corollary~\ref{CComputeDT}~(2) (thus $\FFFDT(\FFf{F})\llbracket\intintuptoex{k}\rrbracket=\FFFDT((\FFf{F}[0]\modulo p^k,\ldots,\FFf{F}[p-1]\modulo p^k))\llbracket\intintuptoex{k}\rrbracket$), and Lemma~\ref{LReduceDegree} (thus $\FFFDT(\FFf{F})\llbracket\intintuptoex{k}\rrbracket=\FFFDT(\FFf{G})\llbracket\intintuptoex{k}\rrbracket$).
\end{proof}

\noindent
As an example consider the $2$-adic system $\FFf{F}\ce(-36x^9-67x^8-47x^7-35x^6-13x^5+79x^4-40x^3+95x^2+75x+4,15x^9-x^8-58x^7-92x^6-68x^5+10x^4-54x^3+98x^2-10x+48)$ and $k\ce4$. By Theorem~\ref{TReduceDegree} the $2$-adic system $\FFf{G}\ce(8x^3+11x^2+11x+4,14x^2+x+1)$ satisfies  $\FFFDT(\FFf{F})\llbracket\intintuptoex{4}\rrbracket=\FFFDT(\FFf{G})\llbracket\intintuptoex{4}\rrbracket$ and this is indeed the case.

\myparagraphtoc{$p$-adic systems defined by rational functions.}
Using the results from the previous subsection and Theorem~\ref{TSuitFProd}, we can identify many rational functions that are $(p,r)$-suitable.

\begin{theorem}
\label{TPropRatF}
Let $2\leq p\in\N$, $r\in\intintuptoex{p}$, $f:\Z_p\to\Z_p$, $g,h\in\Z_p[x]$ such that $\gcd(p,h(r)\modulo p)=1$ and $f(n)=g(n)/h(n)$ for all $n\in r+p\Z_p$, and $k\in\Nz$. Then,

\begin{theoremtable}
\theoremitem{(1)}&$\FPweaklysuitableat{p}{r}{k}(f)$\tabularnewline
\theoremitem{(2)}&$k\leq1\Rightarrow\FPsuitableat{p}{r}{k}(f)$\tabularnewline
\theoremitem{(3)}&$k\geq2\Rightarrow(\FPsuitableat{p}{r}{k}(f)\Leftrightarrow\FPsuitableat{p}{r}{k}(g-(f(r)\modulo p)h))$.
\end{theoremtable}
\end{theorem}

\noindent
Note that $g-(f(r)\modulo p)h$ in (3) is a polynomial in $\Z_p[x]$, which implies that we can use Theorem~\ref{TPropPolyF}~(3) to check whether it is $(p,r)$-suitable at $k$.

\begin{proof}[Proof of Theorem~\ref{TPropRatF}]
(2) is clearly true for any function $f$. In order to prove (1) and (3) we first observe that
\begin{align}
f(n)
&=\frac{g(n)}{h(n)}
=\frac{g(n)-(f(r)\modulo p)h(n)}{h(n)}+f(r)\modulo p
\end{align}
for all $n\in r+p\Z_p$. Thus, $f$ is (weakly) $(p,r)$-suitable if and only if
\begin{align}
\hat{f}:\Z_p&\to\Z_p\\
\nonumber
n&\mapsto
\begin{cases}
\frac{g(n)-(f(r)\modulo p)h(n)}{h(n)}&\text{if }n\in r+p\Z_p\\
f(n)&\text{if }n\notin r+p\Z_p
\end{cases}
\end{align}
is (weakly) $(p,r)$-suitable. Furthermore,
\begin{align}
(g(n)h(r)-g(r)h(n))\modulo p
&=0
\end{align}
for all $n\in r+p\Z_p$, and hence
\begin{align}
\hat{f}(n)\modulo p
&=(f(n)-f(r))\modulo p
=\frac{g(n)h(r)-g(r)h(n)}{h(n)h(r)}\modulo p
=0
\end{align}
for all $n\in r+p\Z_p$, since $\gcd(p,h(r)\modulo p)=1$. Thus $\hat{f}$ (a quotient of two functions) satisfies all conditions of Theorem~\ref{TSuitFProd} and (1) and (3) follow from the ``In particular'' part and from Theorem~\ref{TPropPolyF} (in particular one needs that all polynomials are weakly $(p,r)$-suitable and also, for $k\geq2$, that a polynomial is $(p,r)$-suitable at $k$ if and only if it is $(p,r)$-suitable at $\intintuptoin{k}$).
\end{proof}

\noindent
Since the question of (weak) suitability of the rational functions treated in Theorem~\ref{TPropRatF} reduces to a questions of (weak) suitability of polynomial functions, as a result we get the following analogue of Corollary~\ref{CPropPolyF}.

\begin{corollary}
\label{CPropRatF}
Let $2\leq p\in\N$ and $\FFf{F}\in\SoFibredFunctions{p}(\FFPdomain{\Z_p})$ such that for all $r\in\intintuptoex{p}$ there are $g_r,h_r\in\Z_p[x]$ with $\gcd(p,h_r(r)\modulo p)=1$ and $\FFf{F}[r](n)=g_r(n)/h_r(n)$ for all $n\in r+p\Z_p$ (cf. the assumptions of Theorem~\ref{TPropRatF}). Then,

\begin{theoremtable}
\theoremitem{(1)}&$\FFPweakblockS(\FFf{F})$\tabularnewline
\theoremitem{(2)}&$\FFPblock(\FFf{F})\Leftrightarrow\ex k\in\intint{2}{\infty}:\FFPblockat{k}(\FFf{F})$.
\end{theoremtable}
\end{corollary}

\begin{proof}
Follows directly from Corollary~\ref{CPropPolyF} and Theorem~\ref{TPropRatF}.
\end{proof}

Given Theorem~\ref{TPropRatF} all of the following $p$-fibred functions are examples of $p$-adic systems:

\begin{theoremtable}
$\bullet$&$\left(\frac{1}{3x+1},\frac{1}{x}\right)\in\SoSystems{2}$ (inverse Collatz)\tabularnewline
$\bullet$&$\left(\frac{x^3-3x^2+7x-1}{5x^2-3},\frac{6x^5+13x^4-9x+4}{19x^3-3x^2+1}\right)\in\SoSystems{2}$\tabularnewline
$\bullet$&$\left(\frac{3x^3-x}{x-5},\frac{1}{7x-2},\frac{2x-9}{x-1},\frac{5x-7}{x^2-2},\frac{-2x+1}{x-3},\frac{5x^2+x}{x^2-6}\right)\in\SoSystems{6}$.\tabularnewline[0.5\baselineskip]
\end{theoremtable}

One open problem is the characterization of all rational functions on $\Z_p$ that are (weakly) $(p,r)$ suitable, i.e. dropping the assumption $\gcd(p,h(r)\modulo p)=1$ in Theorem~\ref{TPropRatF}. This would also be a generalization of Theorem~\ref{TPropRatPolyF} which treats such rational functions requiring in return the function in the denominator to be constant. It appears likely that there is an analogue of Theorem~\ref{TPropRatPolyF} for this general situation, i.e. that one can find a finite witness set to check (weak) $(p,r)$-suitability at $k$ and that it suffices to check (weak) $(p,r)$-suitability at $k$ for finitely many $k$ to get full (weak) $(p,r)$-suitability. An educated guess for the general situation might be given by the following conjecture (cf. Theorem~\ref{TPropRatPolyF}).

\begin{conjecture}
\label{ConPropRatPolyF}
Let $2\leq p\in\N$, $r\in\intintuptoex{p}$, $f:\Z_p\to\Z_p$, $g,h\in\Z_p[x]$ such that $f(n)=g(n)/h(n)$ for all $n\in r+p\Z_p$, $k\in\Nz$,
\begin{align}
W_\ell&\ce\intintuptoex{p^\ell}\cap(r+p\Z_p)
\end{align}
for all $\ell\in\Nz$, and
\begin{align}
K&\ce\max(\set{\nu_p(h(n))\mid n\in r+p\Z_p})\in\Nz
\end{align}
(note that this maximum exists and is in $\Nz$ because otherwise $h$ would have a root in $r+p\Z_p$ which contradicts the assumption $f(n)=g(n)/h(n)$ for $n\in r+p\Z_p$).

\noindent
Then, $W_{K+k}$ is a finite witness set for $f$ being (weakly) $(p,r)$-suitable at $k$, i.e.

\begin{theoremtable}
\theoremitem{(1)}&$\FPweaklysuitableat{p}{r}{k}(f)\Leftrightarrow$\tabularnewline
&$\quad\fa m,n\in W_{K+k}:m\equiv n\modulus p^k\Rightarrow(f-f\modulo p)(m)\equiv(f-f\modulo p)(n)\modulus p^k$\tabularnewline
\theoremitem{(2)}&$\FPsuitableat{p}{r}{k}(f)\Leftrightarrow$\tabularnewline
&$\quad\fa m,n\in W_{K+k}:m\equiv n\modulus p^k\Leftrightarrow(f-f\modulo p)(m)\equiv(f-f\modulo p)(n)\modulus p^k$.\tabularnewline[0.5\baselineskip]
\end{theoremtable}
Furthermore,

\begin{theoremtable}
\theoremitem{(3)}&$\fa\ell\in\intint{K+1}{\infty}:\FPweaklysuitableat{p}{r}{\ell}(f)\Rightarrow\FPweaklysuitableat{p}{r}{\ell+1}(f)$\tabularnewline
&In particular: $\FPweaklysuitableat{p}{r}{\intint{K+1}{\infty}}(f)\Leftrightarrow\FPweaklysuitableat{p}{r}{K+1}(f)$\tabularnewline
\theoremitem{(4)}&$\fa\ell\in\intint{K+2}{\infty}:(\FPweaklysuitableat{p}{r}{\ell}(f)\land\FPsuitableat{p}{r}{\ell+1}(f))\Rightarrow\FPsuitableat{p}{r}{\ell+2}(f)$\tabularnewline
&In particular: $\FPsuitableat{p}{r}{\intint{K+2}{\infty}}(f)\Leftrightarrow\FPsuitableat{p}{r}{\set{K+2,K+3}}(f)$.\tabularnewline[0.5\baselineskip]
\end{theoremtable}
In particular,

\hfill

\begin{theoremtable}
\theoremitem{(1)}&$\FPweaklysuitable{p}{r}(f)\Leftrightarrow\FPweaklysuitableat{p}{r}{\intintuptoin{K+1}}(f)\Leftrightarrow$\tabularnewline
&$\quad\fa\ell\in\intintuptoin{K+1}:\fa m,n\in W_{K+\ell}:m\equiv n\modulus p^\ell\Rightarrow(f-f\modulo p)(m)\equiv(f-f\modulo p)(n)\modulus p^\ell$\tabularnewline
\theoremitem{(2)}&$\FPsuitable{p}{r}(f)\Leftrightarrow\FPsuitableat{p}{r}{\intintuptoin{K+3}}(f)\Leftrightarrow$\tabularnewline
&$\quad\fa\ell\in\intintuptoin{K+3}:\fa m,n\in W_{K+\ell}:m\equiv n\modulus p^\ell\Leftrightarrow(f-f\modulo p)(m)\equiv(f-f\modulo p)(n)\modulus p^\ell$.
\end{theoremtable}
\end{conjecture}

\noindent
Unfortunately, a general proof could not be given by now and will be subject of future work.

\section{Generalizing Hensel's Lemma using $p$-adic systems}
\label{SHensel}
In this section we use $p$-adic systems to give a very easy and accessible proof of a surprisingly radical generalization of a famous lemma due to Kurt Hensel which has several equivalent formulations one of which reads as follows \cite{Hensel:1904}.

\begin{lemma}[Hensel's Lemma]
\label{LHensel}
Let $p$ be a prime, $r\in\intintuptoex{p}$, and $f\in\Z_p[x]$ such that $f(r)\modulo p=0$ and $f'(r)\modulo p\neq0$. Then, $f$ has a unique root in $r+p\Z_p$.
\end{lemma}

\noindent
The generalization we will prove below is given by the following theorem.

\begin{theorem}
\label{TSuitFUnRoot}
Let $2\leq p\in\N$, $r\in\intintuptoex{p}$, and $f\in(\Z_p)^{\Z_p}$ such that $f(r+p\Z_p)\subseteq p\Z_p$ and $\FPsuitable{p}{r}(f)$. Then, $f$ has a unique root in $r+p\Z_p$.
\end{theorem}

\noindent
A comparison of the differing assumptions of Hensel's Lemma and the above theorem shows:

\begin{theoremtable}[rlX]
\theoremitem{(1)}&$p\in\P$&$\;\rightarrow\;\;2\leq p\in\N$\tabularnewline
\theoremitem{(2)}&$f\in\Z_p[x]$&$\;\rightarrow\;\; f\in(\Z_p)^{\Z_p}$\tabularnewline
\theoremitem{(3)}&$f(r)\modulo p=0$&$\;\rightarrow\;\; f(r+p\Z_p)\subseteq p\Z_p$\quad($\Leftrightarrow f(r)\modulo p=0$ if $f\in\Z_p[x]$)\tabularnewline
\theoremitem{(4)}&$f'(r)\modulo p\neq0$&$\;\rightarrow\;\;\FPsuitable{p}{r}(f)$\quad($\Leftrightarrow f'(r)\modulo p\neq0$ if $p\in\P$ and $f\in\Z_p[x]$ by Theorem~\ref{TPropPolyF}~(3)).\tabularnewline[0.5\baselineskip]
\end{theoremtable}

\noindent
The generalizing aspect of Theorem~\ref{TSuitFUnRoot} is entailed by (1) and (2), while (3) and (4) are necessary adaptations of the conditions on $f$. It can be seen that the theorem's conditions in (3) and (4) are equivalent to their corresponding versions in Hensel's Lemma if one assumes that $p$ is prime and $f$ is a polynomial function with coefficients in $\Z_p$. The most remarkable difference is probably given by (2), where the assumption that $f$ is a polynomial function is dropped and arbitrary functions on $\Z_p$ are allowed. It should be noted that the idea of generalizing Hensel's Lemma by dropping the assumption that $f$ is a polynomial function, is not new. A related result can be found in \cite{YurovaAxelssonKhrennikov:2016} which considers general functions on $\Z_p$, but only considers the case where $p$ is prime. The further generalization which is given in the following subsection appears to be completely new, however. In order to prove Theorem~\ref{TSuitFUnRoot} we need the following lemma.

\begin{lemma}
\label{LShiftedSuitF}
Let $2\leq p\in\N$, $r\in\intintuptoex{p}$, $A\subseteq\Z_p$, $a,b\in\Z_p$, $f,g,h:A\to\Z_p$ such that
\begin{align}
g(n)&=f(n)+an+b\\
h(n)&=f(n)+apn+b
\end{align}
for all $n\in A$, and $k\in\Nz$. Then,

\begin{theoremtable}
\theoremitem{(1)}&$\FPweaklysuitableat{p}{r}{k}(f)\Rightarrow\FPweaklysuitableat{p}{r}{k}(g)$\tabularnewline
\theoremitem{(2)}&$\FPsuitableat{p}{r}{\intintuptoin{k}}(f)\Rightarrow\FPsuitableat{p}{r}{\intintuptoin{k}}(h)$.
\end{theoremtable}
\end{lemma}

\begin{proof}
The statements are clearly true if $a=0$, hence we may assume without loss of generality that $b=0$ and $f((r+p\Z_p)\cap A)\subseteq p\Z_p$.

\noindent
\proofitem{(1)}%
Let $m,n\in(r+p\Z_p)\cap A$. Then,
\begin{align}
m\equiv n\modulus p^k
&\Rightarrow m\equiv n\modulus p^k\land f(m)\equiv f(n)\modulus p^k\\
&\Rightarrow f(m)+am\equiv f(n)+an\modulus p^k\\
&\Rightarrow g(m)\equiv g(n)\modulus p^k.
\end{align}

\noindent
\proofitem{(2)}%
Let $m,n\in(r+p\Z_p)\cap A$. Then,
\begin{align}
m\equiv n\modulus p^k
&\Rightarrow h(m)\equiv h(n)\modulus p^k
\end{align}
by (1) and we are left to show that
\begin{align}
h(m)\equiv h(n)\modulus p^k
&\Rightarrow m\equiv n\modulus p^k.
\end{align}
Assume to the contrary that $h(m)\equiv h(n)\modulus{p^k}$ and $m\not\equiv n\modulus{p^k}$. Let
\begin{align}
\ell&\ce\max\set{i\in\Nz:m\equiv n\modulus{p^i}}.
\intertext{Then $\ell<k$ and}
\ell&=\max\set{i\in\Nz:f(m)\equiv f(n)\modulus{p^i}}
\end{align}
since $\FPsuitableat{p}{r}{\intintuptoin{k}}(f)$.
Let $c,d,e\in\Z_p$ such that $h(m)=h(n)+cp^k$, $m=n+dp^\ell$ and $f(m)=f(n)+ep^\ell$.
Then,
\begin{align}
f(m)+apm=f(n)+apn+cp^k
&\Leftrightarrow f(n)+ep^\ell+ap(n+dp^\ell)=f(n)+apn+cp^k\\
&\Leftrightarrow ep^\ell+adp^{\ell+1}=cp^k\\
&\Leftrightarrow e+adp=cp^{k-\ell}\\
&\Leftrightarrow e=p(cp^{k-\ell-1}-ad)
\end{align}
where $cp^{k-\ell-1}-ad\in\Z_p$. Hence, $f(m)\equiv f(n)\modulus{p^{\ell+1}}$ which is a contradiction.
\end{proof}

\noindent
Note that the stronger version
\begin{align}
\FPsuitableat{p}{r}{k}(f)\Rightarrow\FPsuitableat{p}{r}{k}(h)
\end{align}
of (2) is not true in general as the following example shows.

\begin{example}
\label{EShiftedSuitF}
Let $f:\Z_2\to\Z_2$ such that $f(n)=(n\modulo8=2\;?\;4:(n\modulo8=4\;?\;2:n\modulo8))$, $a=1$, $b=0$, and $h:\Z_2\to\Z_2$ with $h(n)=f(n)+2an+b=f(n)+2n$ for all $n\in\N_2$. Then, $\FPsuitableat{2}{0}{3}(f)$ but $\neg\FPsuitableat{2}{0}{3}(h)$, since $0\not\equiv2\modulus{2^3}$ but $h(0)=0\equiv8=h(2)\modulus{2^3}$. Note that $\neg\FPsuitableat{2}{0}{2}(f)$ since $0\not\equiv2\modulus{2^2}$ but $f(0)=0\equiv4=f(2)\modulus{2^2}$, which is what makes the counter-example possible.
\end{example}

\begin{proof}[Proof of Theorem~\ref{TSuitFUnRoot}]
Let $g\in(\Z_p)^{\Z_p}$ with $g(n)=f(n)+pn$ for all $n\in\Z_p$. Then, $\FPsuitable{p}{r}(g)$ by Lemma~\ref{LShiftedSuitF}~(2) and also $g(r+p\Z_p)\subseteq p\Z_p$. Let $\id_{\Z_p}$ denote the identity function on $\Z_p$ and $\FFf{F}\ce(\id_{\Z_p})^r\cdot(g)\cdot(\id_{\Z_p})^{p-r-1}$. Then, $\FFf{F}\in\SoSystems{p}$ by Theorem~\ref{TFuncSuit}~(2) and thus there is a unique $n\in\Z_p$ such that $\FFFDT(\FFf{F})[n]=(r)^\infty$ by Lemma~\ref{LInfBlDTCompl}. It follows that $n\in r+p\Z_p$ and $\FFFDT(\FFf{F})[n]=\FFFDT(\FFf{F})[\FFf{F}(n)]$. Thus, $n=\FFf{F}(n)=(g(n)-g(n)\modulo p)/p=g(n)/p$ by Lemma~\ref{LInfBlDTEq}, i.e. $f(n)=g(n)-pn=0$. If $m\in r+p\Z_p$ and $f(m)=0$, then $\FFf{F}(m)=g(m)/p=m$ and hence $\FFFDT(\FFf{F})[m]=(r)^\infty$. Thus $m=n$, again by Lemma~\ref{LInfBlDTEq}.
\end{proof}

An important application of Hensel's Lemma lies in proving that certain real or complex numbers that are defined by polynomial equations (such as $\sqrt{2}$ or $\I$) have counterparts within $\Z_p$ for some $2\leq p\in\N$. Examples are given below. Note that the polynomial function $f$ always satisfies the conditions of Hensel's Lemma (respectively Theorem~\ref{TSuitFUnRoot}).

\begin{example}
\label{EHensel}
$ $

\begin{theoremtable}
\theoremitem{1)}&$p=2$, $r=1$, $f(x)=x^2-x-4\;\;\rightarrow\;\;\pm\sqrt{17}\in\Z_2$\tabularnewline
\theoremitem{2)}&$p=3$, $r\in\set{1,2}$, $f(x)=x^2+2\;\;\rightarrow\;\;\pm\sqrt{-2}\in\Z_3$\tabularnewline
\theoremitem{3)}&$p=5$, $r\in\set{2,3}$, $f(x)=x^2+1\;\;\rightarrow\;\;\pm\I\in\Z_5$\tabularnewline
\theoremitem{4)}&$p=7$, $r\in\set{3,4}$, $f(x)=x^2-2\;\;\rightarrow\;\;\pm\sqrt{2}\in\Z_7$\tabularnewline
\theoremitem{5)}&$p\in\P$, $r\in\intint{1}{p-1}$, $f(x)=x^{p-1}-1\;\;\rightarrow\;\;\abs{\set{x\in\Z_p\mid x^{p-1}=1}}=p-1$.
\end{theoremtable}
\end{example}

\noindent
It is conceivable that the generalization of Hensel's Lemma given by Theorem~\ref{TSuitFUnRoot} allows for similar deductions when applied to other classes of functions. Exploring such possibilities could be the subject of future research.

As we have seen above, $(p,r)$-suitable functions that map $r+p\Z_p$ to $p\Z_p$ have a unique root in $r+p\Z_p$. The following theorem examines the converse direction for polynomial functions.

\begin{theorem}
\label{TSuitFUnRootEq}
Let $2\leq p\in\N$, $r\in\intintuptoex{p}$, and $f\in\Z_p[x]$ such that $f(r)\modulo p=0$. Then, $\FPsuitable{p}{r}(f)$ if and only if $f$ has a unique root $a$ in $r+p\Z_p$ and $\Gcd{p,g(r)\modulo p}=1$, where $g\in\Z_p[x]$ such that $f(x)=(x-a)g(x)$ for all $x\in\Z_p$.
\end{theorem}

\begin{proof}
First we note that $\FPsuitable{p}{r}(f)\Leftrightarrow\Gcd{p,f'(r)\modulo p}=1$ by Theorem~\ref{TPropPolyF}~(3).

Next we prove ``$\Rightarrow$''. $f$ has a unique root $a$ in $r+p\Z_p$ by Corollary~\ref{TSuitFUnRoot}. Let $g\in\Z_p[x]$ such that $f(x)=(x-a)g(x)$ for all $x\in\Z_p$. Then,
\begin{align}
f'(x)=g(x)+(x-a)g'(x)
\end{align}
and hence
\begin{align}
1=\Gcd{p,f'(r)\modulo p}=\Gcd{p,(g(r)+(r-a)g'(r))\modulo p}=\Gcd{p,g(r)\modulo p}.
\end{align}
To prove ``$\Leftarrow$'' we analogously compute
\begin{align}
1=\Gcd{p,g(r)\modulo p}=\Gcd{p,(g(r)+(r-a)g'(r))\modulo p}=\Gcd{p,f'(r)\modulo p}.
\end{align}
\end{proof}

\noindent
Even if $p\in\P$, the condition $\Gcd{p,g(r)\modulo p}=1$ in the above theorem cannot be dropped which is shown by the following example.

\begin{example}
\label{ESuitFUnRootEq}
If $f(x)\ce(x-1)(x^2+1)$, then $\set{x\in1+2\Z_2\mid f(x)=0}=\set{1}$ but $\neg\FPsuitable{2}{1}(f)$ (note that $g(x)=x^2+1$ here).
\end{example}

\myparagraphtoc{$p$-fibred rational functions and a further generalization of Hensel's Lemma.}
In the remaining part of this section we investigate another generalization of Hensel's Lemma which involves several $(p,r)$-suitable functions at once and allows for arbitrary compositions while keeping a unique root in a given residue class. Before we can formulate the results, we need to extend our framework by a new notion closely related to $p$-fibred functions which we formally defined in Section~\ref{SNotDef}. Let $2\leq p\in\N$. The elements of any set $\CoSequences(\SPboundedby{(\Q_p)^A},\SPlength{p})$ where $A\subseteq\Q_p$, are called \emph{$p$-fibred rational functions}, i.e. a $p$-fibred rational function $\FRFf{R}=(\FRFf{R}[0],\ldots,\FRFf{R}[p-1])$ is a $p$-tuple of functions $\FRFf{R}[r]:A\to\Q_p$, $r\in\intintuptoex{p}$, on some fixed subset $A$ of the $p$-adic numbers. The \emph{set of all $p$-fibred rational functions} shall be denoted by \label{DSoFibredRationalFunctions}$\SoFibredRationalFunctions{p}$. For any $p$-fibred rational function $\FRFf{R}\in\CoSequences(\SPboundedby{(\Q_p)^A},\SPlength{p})$ we set \label{DFRFFdomain}$\FRFFdomain(\FRFf{R})\ce A$, the \emph{domain of $\FRFf{R}$}. For any subset $A$ of $\FRFFdomain(\FRFf{R})$ we define \label{DFRFrestriction}$\FRFf{R}\vert_A\ce(\FRFf{R}[0]\vert_A,\ldots,\FRFf{R}[p-1]\vert_A)$, the \emph{restriction of $\FRFf{R}$ to $A$}.

\noindent
We define the following predicates on $\SoFibredRationalFunctions{p}$ ($\FRFf{R}\in\SoFibredRationalFunctions{p}$, $A$ set):
\begin{flalign}
\pushleft{\FRFPdomain{A}(\FRFf{R})}&\Leftrightarrow\FRFFdomain(\FRFf{R})=A&\text{$\FRFf{R}$ has \emph{domain $A$}}\label{DFRFPdomain}\\
\pushleft{\FRFPboundedby{A}(\FRFf{R})}&\Leftrightarrow\fa r\in\intintuptoex{p}:\FRFf{R}[r](\FRFFdomain(\FRFf{R}))/p\subseteq A&\text{$\FRFf{R}$ is \emph{$A$-bounded}}\label{DFRFPboundedby}\\
\pushleft{\FRFPclosed(\FRFf{R})}&\Leftrightarrow\FRFPboundedby{\FRFFdomain(\FRFf{R})}(\FRFf{R})&\text{$\FRFf{R}$ is \emph{closed}}\label{DFRFPclosed}\\
\pushleft{\FRFPintegral(\FRFf{R})}&\Leftrightarrow\FRFFdomain(\FRFf{R})\subseteq\Z_p\land\FRFf{R}[r]((r+p\Z_p)\cap\FRFFdomain(\FRFf{R}))\subseteq\Z_p&\text{$\FRFf{R}$ is \emph{integral}}\label{DFRFPintegral}
\end{flalign}

\noindent
For any closed $p$-fibred rational function $\FRFf{R}$ and for any $\Sf{D}\in\CoSequences(\SPboundedby{\intintuptoex{p}},\SPfinite)$ we define
\begin{align}
\label{EFibredRFunction}
\FRFF{\FRFf{R}}{\Sf{D}}:\FRFFdomain(\FRFf{R})&\to\Q_p\\
\nonumber
x&\mapsto
\begin{cases}
x&\text{if }\abs{\Sf{D}}=0\\
\frac{\FRFf{R}[\Sf{D}[\abs{\Sf{D}}-1]]\left(\FRFF{\FRFf{R}}{\Sf{D}[\intintuptoex{\abs{\Sf{D}}-1}]}(x)\right)}{p}&\text{if }\abs{\Sf{D}}>0
\end{cases}
\end{align}
and for any $\Sf{D}\in\CoSequences(\SPboundedby{\intintuptoex{p}})$ we call$\FRFFST{\Sf{D}}(\FRFf{R})\ce\big(\big(\FRFF{\FRFf{R}}{\Sf{D[\intintuptoex{k}]}}(n)\big)_{k\in\intintuptoin{\abs{\Sf{D}}}}\big)_{n\in\FRFFdomain(\FRFf{R})}\in\CoSequenceTables(\STPdomain{\FRFFdomain(\FRFf{R})},\STPlength{\abs{\Sf{D}}+1})$ the  \label{DFRFFST}\emph{$\FRFf{R}$-sequence table with respect to $\Sf{D}$}. For $n\in\FRFFdomain(\FRFf{R})$ the \label{DFRFFSTentry}\emph{$\FRFf{R}$-sequence of $n$ with respect to $\Sf{D}$} is given by the $\FRFFST{\Sf{D}}(\FRFf{R})$-sequence of $n$.

\noindent
We define the following additional predicates on $\SoFibredRationalFunctions{p}$ ($\FRFf{R}\in\SoFibredRationalFunctions{p}$, $A\subseteq\Q_p$, $D\subseteq\Nz\cup\set{-\infty}$):
\begin{flalign}
\pushleft{\FRFPpolynomialcoefficientsdegree{A}{D}(\FRFf{R})}&\Leftrightarrow\FRFPdomain{\Q_p}(\FRFf{R})&\hspace{-10em}\text{$\FRFf{R}$ is \emph{$A$-polynomial with degree in $D$ or, if $D=\set{d}$,}}\label{DFRFPpolynomialcoefficientsdegree}\\
\nonumber
&\phantom{\vphantom{a}\Leftrightarrow\vphantom{a}}\fa r\in\intintuptoex{p}:\FRFf{R}[r](x)\in A[x]&\text{$\FRFf{R}$ is \emph{$A$-polynomial of degree $d$}}\\
\nonumber
&\phantom{\vphantom{a}\Leftrightarrow\fa r\in\intintuptoex{p}:\vphantom{a}}\deg(\FRFf{R}[r](x))\in D\\
\pushleft{\FRFPpolynomialcoefficients{A}(\FRFf{R})}&\Leftrightarrow\FRFPpolynomialcoefficientsdegree{A}{\Nz\cup\set{-\infty}}(\FRFf{R})&\text{$\FRFf{R}$ is \emph{$A$-polynomial}}\label{DFRFPpolynomialcoefficients}\\
\pushleft{\FRFPpolynomial(\FRFf{R})}&\Leftrightarrow\FRFPpolynomialcoefficientsdegree{\Q_p}{\Nz\cup\set{-\infty}}(\FRFf{R})&\text{$\FRFf{R}$ is \emph{polynomial}}\label{DFRFPpolynomial}\\
\pushleft{\FRFPlinearpolynomialcoefficients{A}(\FRFf{R})}&\Leftrightarrow\FRFPpolynomialcoefficientsdegree{A}{\set{-\infty,0,1}}(\FRFf{R})&\text{$\FRFf{R}$ is \emph{$A$-linear-polynomial}}\label{DFRFPlinearpolynomialcoefficients}\\
\pushleft{\FRFPlinearpolynomial(\FRFf{R})}&\Leftrightarrow\FRFPpolynomialcoefficientsdegree{\Q_p}{\set{-\infty,0,1}}(\FRFf{R})&\text{$\FRFf{R}$ is \emph{linear-polynomial}}\label{DFRFPlinearpolynomial}
\end{flalign}

\noindent
Furthermore, we establish a relation between $\SoFibredRationalFunctions{p}(\FRFPintegral)$ and $\SoFibredFunctions{p}$ by the bijection:
\begin{align}
\label{DFRFFint}
\FRFFint:\SoFibredRationalFunctions{p}(\FRFPintegral)&\to\SoFibredFunctions{p}\\
\nonumber
\FRFf{R}&\mapsto(\FRFf{R}[0]:\FRFFdomain(\FRFf{R})\to\Z_p,\ldots,\FRFf{R}[p-1]:\FRFFdomain(\FRFf{R})\to\Z_p)
\end{align}
i.e. ``$\FRFFint$'' changes the codomains of all $\FRFf{R}[r]$, $r\in\intintuptoex{p}$, from $\Q_p$ to $\Z_p$. Note that we cannot simply identify elements of $\SoFibredRationalFunctions{p}(\FRFPintegral)$ and $\SoFibredFunctions{p}$ by $\FRFFint$ since this would introduce ambiguities for the predicates ``$\FRFPboundedby{A}$'' and ``$\FRFPclosed$'' which are defined differently for integral $p$-fibred rational functions and $p$-fibred functions and do not coincide in general.

The essential difference between $p$-fibred functions and $p$-fibred rational functions lies in the way they are considered to be functions on their respective domains (cf. Eqn.~(\ref{EFibredFunction}) and Eqn.~(\ref{EFibredRFunction})). While a $p$-fibred function $\FFf{F}$ gets the information on which of the functions $\FFf{F}[0],\ldots,\FFf{F}[p-1]$ to apply to some $n\in\FFFdomain(\FFf{F})$ from the residue class of $n$ modulo $p$, a $p$-fibred rational function gets this information from the supplied ``$p$-digit sequence'' $\Sf{D}$. This difference also influences the definitions of the respective ``$\FRFPboundedby{A}$'' and ``$\FRFPclosed$'' predicates. In both cases, the purpose of the ``$\FRFPclosed$'' predicate is to guarantee that application of a $p$-fibred function or $p$-fibred rational function in the sense of Eqn.~(\ref{EFibredFunction}) or Eqn.~(\ref{EFibredRFunction}) respectively, yields elements of the domain again, rendering iterative application possible. Whenever a closed $p$-fibred rational function $\FRFf{R}$ ``extends'' a closed $p$-fibred function $\FFf{F}$ in weak canonical form (which will be the main purpose of $p$-fibred rational functions), the different ways of applying them as functions coincide, however, if the supplied $p$-digit sequence $\Sf{D}$ is the $\FFf{F}$-digit expansion of the argument $n$, as the following lemma shows.

\begin{lemma}
\label{LFRF}
Let $2\leq p\in\N$, $\FRFf{R}\in\SoFibredRationalFunctions{p}(\FRFPclosed)$ such that $\FRFPintegral(\FRFf{R}\vert_{\Z_p\cap\FRFFdomain(\FRFf{R})})$ and such that $\FFf{F}\ce\FRFFint(\FRFf{R}\vert_{\Z_p\cap\FRFFdomain(\FRFf{R})})\in\SoFibredFunctions{p}(\FFPclosed,\FFPweakcanonicalform)$. Furthermore, let $n\in\Z_p\cap\FRFFdomain(\FRFf{R})$, $\Sf{D}\ce\FFFDT(\FFf{F})[n]$, and $k\in\Nz$. Then,
\begin{align}
\FFf{F}^k(n)=\FRFF{\FRFf{R}}{\Sf{D}[\intintuptoex{k}]}(n).
\end{align}
In particular: $\FFFST(\FFf{F})[n]=\FRFFST{\Sf{D}}(\FRFf{R})[n]$.
\end{lemma}

\noindent
As an example consider the $2$-fibred rational function $\FRFf{R}\ce(x,3x+1)\in\SoFibredRationalFunctions{2}(\FRFPdomain{\Q_2},\FRFPclosed)$ and the corresponding $2$-fibred function $\FFf{F}\ce\FRFFint(\FRFf{R}\vert_{\Z_p})\in\SoFibredFunctions{2}(\FFPdomain{\Z_2},\FFPclosed,\FFPweakcanonicalform)$ (note that every $p$-tuple of polynomials in $\Z_p[x]$ may be regarded as an element of both $\SoFibredRationalFunctions{p}(\FRFPdomain{\Q_p},\FRFPclosed)$ and $\SoFibredFunctions{p}(\FFPdomain{\Z_p},\FFPclosed)$, as the restriction to $\Z_p$ of every $p$-fibred rational function defined by polynomials in $\Z_p[x]$ is always integral). Then, $\FFFST(\FFf{F})[17]=(17,26,13,20,10,5,8,4)\cdot(2,1)^\infty$ and $\Sf{D}\ce\FFFDT(\FFf{F})[17]=(1,0,1,0,0,1,0,0)\cdot(0,1)^\infty$. Furthermore, $\FFf{F}^6(17)=8=\FRFF{\FRFf{R}}{\Sf{D}[\intintuptoex{6}]}(17)$ and $\FRFFST{\Sf{D}}(\FRFf{R})[17]=(17,26,13,20,10,5,8,4)\cdot(2,1)^\infty$.

\begin{proof}[Proof of Lemma~\ref{LFRF}]
We prove the statement by induction on $k$. If $k=0$, then
\begin{align}
\FFf{F}^0(n)=n=\FRFF{\FRFf{R}}{()}(n)=\FRFF{\FRFf{R}}{\Sf{D}[\intintuptoex{0}]}(n).
\end{align}
Now we assume $\FFf{F}^k(n)=\FRFF{\FRFf{R}}{\Sf{D}[\intintuptoex{k}]}(n)$ and compute
\begin{align}
\FFf{F}^{k+1}(n)
&=\FFf{F}(\FFf{F}^k(n))
=\frac{\FFf{F}[\FFf{F}^k(n)\modulo p](\FFf{F}^k(n))}{p}\;(\Leftarrow\FFPweakcanonicalform(\FFf{F}))\\
(\FFf{F}^k(n)\in\Z_p)\;&=\frac{\FRFf{R}[\FFf{F}^k(n)\modulo p](\FFf{F}^k(n))}{p}
=\frac{\FFf{R}[\Sf{D}[k]](\FRFF{\FRFf{R}}{\Sf{D}[\intintuptoex{k}]}(n))}{p}
=\FRFF{\FRFf{R}}{\Sf{D}[\intintuptoex{k+1}]}(n).
\end{align}
\end{proof}

\noindent
The above lemma is basically ``built into the DNA'' of $p$-fibred rational functions: the way a $p$-fibred rational function is applied as a regular function is defined in the very way that would make the lemma true. However, under additional assumptions (namely that $\FFf{F}$ has the block property and that for all $r\in\intintuptoex{p}$ the only elements of $\Q_p$ that are mapped into $p\Z_p$ by $\FRFf{R}[r]$ are those in $r+p\Z_p$), a much stronger version of the above lemma can be proven that is not built in by design. Before we state the new lemma, we reconsider the above example which happens to satisfy the additional assumptions: if $\FRFf{R}\ce(x,3x+1)\in\SoFibredRationalFunctions{2}(\FRFPdomain{\Q_2},\FRFPclosed)$ and $\FFf{F}\ce\FRFFint(\FRFf{R}\vert_{\Z_p})\in\SoFibredFunctions{2}(\FFPdomain{\Z_2},\FFPweakcanonicalform)$ then $\FFPblock(\FFf{F})$ and $\FRFf{R}[r](\Q_2\setminus(r+2\Z_2))\cap2\Z_2=\emptyset$ for all $r\in\set{0,1}$. Now, if $n$ is any $2$-adic integer, say $n\ce17$, and $\Sf{D}\ce\FFFDT(\FFf{F})[n][\intintuptoex{k}]$ is the initial part of length $k\in\Nz$ of the $\FFf{F}$-digit expansion of $n$, say $k=6$ and hence $\Sf{D}=(1,0,1,0,0,1)$, we know from the above lemma that $\FRFF{\FRFf{R}}{\Sf{D}}[n]$, the result of applying the functions $\frac{x}{2}$ and $\frac{3x+1}{2}$ iteratively to $n$ in the order given by $\Sf{D}$, is equal to $\FFf{F}^k(n)$, which is $8$ in our case. In principle, the order given by $\Sf{D}$ might not be the only one that would result in $17$ being mapped to $8$: could there be another $\Sf{E}\in\CoSequences(\SPboundedby{\intintuptoex{2}},\SPfinite)$ such that $\FRFF{\FRFf{R}}{\Sf{E}}[n]$ is also equal to $\FFf{F}^k(n)$, i.e. $\FRFF{\FRFf{R}}{\Sf{E}}[17]=8$? The answer, which is ``no'', is given by the following lemma, which basically states that if the mentioned additional assumptions hold, the only way a $p$-adic integer $n$ can be mapped to another $p$-adic integer $\FRFF{\FRFf{R}}{\Sf{D}}(n)$ by $\FRFF{\FRFf{R}}{\Sf{D}}$, is that $\FRFF{\FRFf{R}}{\Sf{D}}(n)$ is contained in the $\FFf{F}$-sequence of $n$, say the entry with index $k\in\Nz$, and $\Sf{D}$ is equal to the initial part of length $k$ of the $\FFf{F}$-digit expansion of $n$. Before we state the lemma, we define the following predicate on the set $(\Q_p)^A$ of all mappings from $A$ to $\Q_p$, where $2\leq p\in\N$ and $A\subseteq\Q_p$ ($f\in(\Q_p)^A$, $r\in\intintuptoex{p}$):
\begin{flalign}
\pushleft{\FPavoiding{p}{r}(f)}&\Leftrightarrow f((\Q_p\setminus(r+p\Z_p))\cap A)\cap p\Z_p=\emptyset&\text{$f$ is \emph{$(p,r)$-avoiding}}\label{DFPavoiding}
\end{flalign}

\noindent
Note that if $r+p\Z_p\subseteq A$, $f(r+p\Z_p)\subseteq p\Z_p$, and $\FPsuitable{p}{r}(f)$ (e.g. if $f\vert_{\Z_p}$ is the entry with index $r$ of a $p$-adic system in weak canonical form, cf. Theorem~\ref{TFuncSuit}~(2)), then $f(r+p\Z_p)=p\Z_p$ by Lemma~\ref{LSuitFProp}~(2). In this case the condition for $f$ to be $(p,r)$-avoiding can also be written as $f((\Q_p\setminus(r+p\Z_p))\cap A)\cap f(r+p\Z_p)=\emptyset$ or even as $f(\Q_p\setminus(r+p\Z_p))\cap f(r+p\Z_p)=\emptyset$ if $A=\Q_p$. Additionally, we define the following predicate on $\SoFibredRationalFunctions{p}$ ($\FRFf{R}\in\SoFibredRationalFunctions{p}$):
\begin{flalign}
\pushleft{\FRFPavoiding(\FRFf{R})}&\Leftrightarrow\fa r\in\intintuptoex{p}:\FPavoiding{p}{r}(\FRFf{R}[r])&\text{$\FRFf{R}$ is \emph{avoiding}}\label{DFRFPavoiding}
\end{flalign}
A $p$-fibred function is said to be \emph{avoiding}\label{DFFPavoiding} if the correspondig $p$-fibred rational function is avoiding. Using the new predicates we now formulate the lemma.

\begin{lemma}
\label{LEqSEqD}
\quad Let $2\leq p\in\N$ and $\FRFf{R}\in\SoFibredRationalFunctions{p}(\FRFPclosed,\FRFPavoiding(\FRFf{R}))$ with $\Z_p\subseteq\FRFFdomain(\FRFf{R})$ such that $\FRFPintegral(\FRFf{R}\vert_{\Z_p})$, and $\FFf{F}\ce\FRFFint(\FRFf{R}\vert_{\Z_p})\in\SoSystems{p}(\FFPweakcanonicalform)$. Furthermore, let $n\in\Z_p$, $k\in\Nz$, and $\Sf{D}\in\CoSequences(\SPboundedby{\intintuptoex{p}},\SPlength{k})$ such that $\FRFF{\FRFf{R}}{\Sf{D}}(n)\in\Z_p$. Then, $\FFFST(\FFf{F})[n][\intintuptoin{k}]=\FRFFST{\Sf{D}}(\FRFf{R})[n]$ and $\FFFDT(\FFf{F})[n][\intintuptoex{k}]=\Sf{D}$. In particular. $\FFf{F}^k(n)=\FRFF{\FRFf{R}}{\Sf{D}}(n)$.
\end{lemma}

\begin{proof}
Let $m$ be the unique element of $\Z_p$ such that $\FFFDT(\FFf{F})[m]=\Sf{D}\cdot\FFFDT(\FFf{F})[\FRFF{\FRFf{R}}{\Sf{D}}(n)]$ (cf. Lemma~\ref{LInfBlDTCompl}). Then,
\begin{align}
\FFFDT(\FFf{F})[\FFf{F}^{k}(m)]
=\FFFDT(\FFf{F})[m][k,\infty]
=\FFFDT(\FFf{F})[\FRFF{\FRFf{R}}{\Sf{D}}(n)]
\end{align}
and hence
\begin{align}
\FFf{F}^{k}(m)=\FRFF{\FRFf{R}}{\Sf{D}}(n)
\end{align}
by Lemma~\ref{LInfBlDTEq}. Also,
\begin{align}
\FFFDT(\FFf{F})[m][\intintuptoex{k}]=\Sf{D}.
\end{align}

We claim that $\FFFST(\FFf{F})[m][k-\ell]=\FRFFST{\Sf{D}}(\FRFf{R})[n][k-\ell]$ for all $\ell\in\intintuptoin{k}$ (which in particular implies that $m=n$) and proceed by induction on $\ell$. If $\ell=0$, then
\begin{align}
\FFFST(\FFf{F})[m][k-\ell]
&=\FFFST(\FFf{F})[m][k]
=\FFf{F}^{k}(m)
=\FRFF{\FRFf{R}}{\Sf{D}}(n)
=\FRFFST{\Sf{D}}(\FRFf{R})[n][k]
=\FRFFST{\Sf{D}}(\FRFf{R})[n][k-\ell].
\end{align}
Now assume that $\FFFST(\FFf{F})[m][k-\ell]=\FRFFST{\Sf{D}}(\FRFf{R})[n][k-\ell]$ for some $\ell\in\intintuptoex{k}$. Then,
\begin{align}
\FRFf{R}[\Sf{D}[k-(\ell+1)]](\FRFFST{\Sf{D}}(\FRFf{R})[n][k-(\ell+1)])
&=\FRFf{R}[\Sf{D}[k-(\ell+1)]](\FRFF{\FRFf{R}}{\Sf{D[\intintuptoex{k-(\ell+1)}}}(n))\\
&=p\FRFF{\FRFf{R}}{\Sf{D}[\intintuptoex{k-\ell}]}(n)\\
&=p\FRFFST{\Sf{D}}(\FRFf{R})[n][k-\ell]\\
&=p\FFFST(\FFf{F})[m][k-\ell]\in p\Z_p.
\end{align}
Thus, it follows from $\FRFf{R}[\Sf{D}[k-(\ell+1)]](\Q_p\setminus(\Sf{D}[k-(\ell+1)]+p\Z_p))\cap p\Z_p=\emptyset$ that
\begin{align}
\FRFFST{\Sf{D}}(\FRFf{R})[n][k-(\ell+1)]\in\Sf{D}[k-(\ell+1)]+p\Z_p.
\end{align}
In addition, 
\begin{align}
\FFFST(\FFf{F})[m][k-(\ell+1)]\modulo p&=\FFFDT(\FFf{F})[m][k-(\ell+1)]=\FFFDT(\FFf{F})[m][\intintuptoex{k}][k-(\ell+1)]=\Sf{D}[k-(\ell+1)]
\end{align}
and hence
\begin{align}
\FFFST(\FFf{F})[m][k-(\ell+1)]\in\Sf{D}[k-(\ell+1)]+p\Z_p.
\end{align}
But then,
\begin{align}
\FFf{F}[\Sf{D}[k-(\ell+1)]](\FRFFST{\Sf{D}}(\FRFf{R})[n][k-(\ell+1)])
&=\FRFf{R}[\Sf{D}[k-(\ell+1)]](\FRFFST{\Sf{D}}(\FRFf{R})[n][k-(\ell+1)])\\
&=p\FFFST(\FFf{F})[m][k-\ell]\\
&=p\FFf{F}^{k-\ell}(m)\\
&=p\FFf{F}(\FFf{F}^{k-(\ell+1)}(m))\\
&=p\FFf{F}(\FFFST(\FFf{F})[m][k-(\ell+1)])\\
&=\FFf{F}[\Sf{D}[k-(\ell+1)]](\FFFST(\FFf{F})[m][k-(\ell+1)])
\end{align}
and hence Theorem~\ref{TFuncSuit}~(2) and Lemma~\ref{LSuitFProp}~(1) imply that
\begin{align}
\FFFST(\FFf{F})[m][k-(\ell+1)]
&=\FRFFST{\Sf{D}}(\FRFf{R})[n][k-(\ell+1)].
\end{align}
\end{proof}

One could assume that if the entries of $\FRFf{R}$ are polynomials, the assumptions of the lemma (i.e. that $\FRFf{R}$ is avoiding) could be loosened. The following examples show that this is not the case.

\begin{example}
\label{EEqSEqD}
$ $

\begin{theoremtable}
$\bullet$&Let $\FRFf{R}\ce(2x^2-5x+2,x+3)\in\SoFibredRationalFunctions{2}(\FRFPdomain{\Q_2},\FRFPclosed)$ and $\FFf{F}\ce\FRFFint(\FRFf{R}\vert_{\Z_p})\in\SoFibredFunctions{2}(\FFPdomain{\Z_2},\FFPweakcanonicalform,\FFPblock)$. Then, $\FRFFST{(1,0)}(\FRFf{R})[-2]=(-2,\frac{1}{2},0)$, but $\FFFST(\FFf{F})[-2][\intintuptoin{2}]=(-2,10,76)$ and $\FFFDT(\FFf{F})[-2][\intintuptoex{2}]=(0,0)$ (note: $\FFFST(\FFf{F})[1][\intintuptoin{2}]=(1,2,0)$ and $\FFFDT(\FFf{F})[1][\intintuptoex{2}]=(1,0)$).\tabularnewline
$\bullet$&Let $\FRFf{R}\ce(x^2+x,x^2-x)\in\SoFibredRationalFunctions{2}(\FRFPdomain{\Q_2},\FRFPclosed)$ and $\FFf{F}\ce\FRFFint(\FRFf{R}\vert_{\Z_p})\in\SoFibredFunctions{2}(\FFPdomain{\Z_2},\FFPweakcanonicalform,\FFPblock)$. Then, $\FRFFST{(0,1)}(\FRFf{R})[1]=(1,1,0)$, but $\FFFST(\FFf{F})[1][\intintuptoin{2}]=(1,0,0)$ and $\FFFDT(\FFf{F})[1][\intintuptoex{2}]=(1,0)$. Therefore, even for polynomials the condition $\FRFf{R}[r](\Q_p\setminus(r+p\Z_p))\cap p\Z_p=\emptyset$ in Lemma~\ref{LEqSEqD} is necessary and cannot even be reduced to $\FRFf{R}[r](\Q_p\setminus\Z_p)\cap p\Z_p=\emptyset$.
\end{theoremtable}
\end{example}

Using the above lemma we are able to prove the promised generalization of Theorem~\ref{TSuitFUnRoot} which is a further generalization of Hensel's Lemma.

\begin{theorem}
\label{TAvoidFRFUnFP}
Let $2\leq p\in\N$ and $\FRFf{R}\in\SoFibredRationalFunctions{p}(\FRFPclosed,\FRFPavoiding(\FRFf{R}))$ with $\Z_p\subseteq\FRFFdomain(\FRFf{R})$ such that $\FRFPintegral(\FRFf{R}\vert_{\Z_p})$, and $\FFf{F}\ce\FRFFint(\FRFf{R}\vert_{\Z_p})\in\SoSystems{p}(\FFPweakcanonicalform)$. Furthermore, let $\Sf{D}\in\CoSequences(\SPboundedby{\intintuptoex{p}},\SPlength{\N})$. Then, $\FRFF{\FRFf{R}}{\Sf{D}}$ has a unique fixed point in $\Z_p$.
\end{theorem}

\noindent
A comparison of the assumptions of Theorem~\ref{TSuitFUnRoot} and Theorem~\ref{TAvoidFRFUnFP} shows that the latter is not a pure generalization of the first, as the $p$-fibred rational function $\FRFf{R}$ in Theorem~\ref{TAvoidFRFUnFP} is required to be avoiding, whereas the function $f$ in Theorem~\ref{TSuitFUnRoot} is not required to be $(p,r)$-avoiding. However, if this is the case, the unique root of Theorem~\ref{TSuitFUnRoot} corresponds to the unique fixed point in Theorem~\ref{TAvoidFRFUnFP} if $\abs{\Sf{D}}=1$ (cf. Lemma~\ref{LShiftedSuitF}). What Theorem~\ref{TAvoidFRFUnFP} generalizes is thus the length of $\Sf{D}$.

\begin{proof}[Proof of Theorem~\ref{TAvoidFRFUnFP}]
Let $k\ce\abs{\Sf{D}}$ and $n$ be the unique element of $\Z_p$ such that $\FFFDT(\FFf{F})[n]=\Sf{D}^\infty$ (cf. Lemma~\ref{LInfBlDTCompl}). Then, $\FRFF{\FRFf{R}}{\Sf{D}}(n)=\FFf{F}^{k}(n)$ by Lemma~\ref{LFRF}. In addition,
\begin{align}
\FFFDT(\FFf{F})[\FFf{F}^{k}(n)]=\FFFDT(\FFf{F})[n][k,\infty]=\Sf{D}^\infty[k,\infty]=\Sf{D}^\infty=\FFFDT(\FFf{F})[n]
\end{align}
and hence $\FFf{F}^{k}(n)=n$ by Lemma~\ref{LInfBlDTEq}. Therefore, $n$ is a fixed point of $\FRFF{\FRFf{R}}{\Sf{D}}$.

If $m\in\Z_p$ is a fixed point of $\FRFF{\FRFf{R}}{\Sf{D}}$, then $\FRFF{\FRFf{R}}{\Sf{D}}(m)=m\in\Z_p$ and Lemma~\ref{LEqSEqD} implies that $\FFFST(\FFf{F})[m][\intintuptoin{k}]=\FRFFST{\Sf{D}}(\FRFf{R})[m]$ and $\FFFDT(\FFf{F})[m][\intintuptoex{k}]=\Sf{D}$. In particular, $\FFf{F}^{k}(m)=\FRFF{\FRFf{R}}{\Sf{D}}(m)=m$ and therefore $\FFFDT(\FFf{F})[m]=\Sf{D}^\infty$. But then Lemma~\ref{LInfBlDTEq} implies that $m=n$.
\end{proof}

Just as Theorem~\ref{TPropPolyF} characterizes polynomials in $\Z_p[x]$ that are $(p,r)$-suitable (i.e. suitable to be the building blocks of $p$-adic systems), we will now try to characterize polynomials that are $(p,r)$-avoiding (i.e. suitable to be the building blocks of avoiding $p$-fibred rational functions). For this purpose we need an analogue of the $p$-adic valuation $\nu_p$ for the case where $p$ is not a prime. These and other technicalities of $p$-adic numbers are discussed in the appendix. Using the definitions and facts presented there, we are able to characterize $(p,r)$-avoiding polynomials.

\begin{theorem}
\label{TCharacAvPolyF}
Let $2\leq p\in\N$, $f=\sum_{i=0}^da_ix^i\in\Z_p[x]$ (note that for the whole theorem we define $0^0\ce1$ and assume $d=0$, $a_0=0$ if $f=0$), and $r\in\intintuptoex{p}$. Furthermore, let $s\in\N$, $q_1,\ldots,q_s\in\P$ the distinct prime factors of $p$, and for all $q\in\set{q_1,\ldots,q_s}$ let $t_q\in\intintuptoin{d+1}$ and $i_{q,1},\ldots,i_{q,t_q}\in\intintuptoin{d}$ be the longest possible strictly increasing sequence of indices satisfying $\nu_q(a_{i_{q,j}})\neq\infty$ for all $j\in\intint{1}{t_q}$. We set
\begin{align}
K&\ce\max\bigg(\bigg\{\frac{\nu_q(a_{i_{q,j+1}})-\nu_q(a_{i_{q,j}})+i_{q,j+1}(\nu_q(p)-1)}{(i_{q,j+1}-i_{q,j})\nu_q(p)}\;\bigg\vert\;q\in\set{q_1,\ldots,q_s}\land j\in\intint{1}{t_q-1}\bigg\}\cup\\
\nonumber
&\phantom{\vphantom{a}\ce\max\bigg(}\bigg\{\frac{\nu_q(a_{i_{q,t_q}})-\nu_q(p)+i_{q,t_q}(\nu_q(p)-1)}{i_{q,t_q}\nu_q(p)}\;\bigg\vert\;q\in\set{q_1,\ldots,q_s}\land t_q\geq1\land i_{q,t_q}\geq1\bigg\}\cup\set{1}\bigg).
\end{align}
Then,
\begin{align}
W\ce\left(\intintuptoex{p}\setminus\set{r}\right)\cup\set{a/p^k\mid k\in\intint{1}{K}\land a\in\intintuptoex{p^{dk+1}}}
\end{align}
is a finite witness set for $f$ being $(p,r)$-avoiding, i.e. $\FPavoiding{p}{r}(f)$ if and only if $f(W)\cap p\Z_p=\emptyset$.
\end{theorem}

\begin{proof}
We need to show
\begin{align}
\ex n\in\Q_p\setminus(r+p\Z_p):f(n)\in p\Z_p\quad\Rightarrow\quad\ex w\in W:f(w)\in p\Z_p.
\end{align}
Clearly, if $f(n)\in p\Z_p$ for some $n\in\Z_p\setminus(r+p\Z_p)$, then $w=n\modulo p\in W$ and $f(w)\in p\Z_p$. We are thus left to show that
\begin{align}
\ex n\in\Q_p\setminus\Z_p:f(n)\in p\Z_p\quad\Rightarrow\quad\ex w\in W:f(w)\in p\Z_p.
\end{align}
For that let $n\in\Q_p\setminus\Z_p$, $\ell\ce-\nu_p(n)\in\N$, and $m\ce np^\ell\in\Z_p$ and assume that $f(n)\in p\Z_p$. Then,
\begin{align}
f(n)
&=f(m/p^\ell)
=\sum_{i=0}^da_i\left(m/p^\ell\right)^i
=\sum_{i=0}^da_im^i/p^{i\ell}
=1/p^{d\ell}\sum_{i=0}^da_im^ip^{(d-i)\ell}
\end{align}
and thus
\begin{align}
f(n)p^{d\ell}=\sum_{i=0}^da_im^ip^{(d-i)\ell}\in p^{d\ell+1}\Z_p.
\end{align}
We may thus reformulate our goal as
\begin{align}
\ex v\in\Q_p\setminus\Z_p:-\nu_p(v)\leq K\land f(v)\in p\Z_p,
\end{align}
because for any such $v$ it follows analogously that
\begin{align}
f(v)p^{dk}=\sum_{i=0}^da_iu^ip^{(d-i)k}\in p^{dk+1}\Z_p,
\end{align}
where $k\ce-\nu_p(v)$ and $u\ce vp^k$, and $w\ce(u\modulo p^{dk+1})/p^k\in W$ thus satisfies $f(w)\in p\Z_p$. To prove this new goal we distinguish three cases:\\[\baselineskip]
\textbf{Case 1:} $\ex q\in\set{q_1,\ldots,q_s}:t_q=0$.\\
Let $v\ce\varphi_p^{-1}((0,\ldots,0,1/p,0,\ldots,0))\in\Q_p\setminus\Z_p$ where the ``$1/p$'' is at the $j$-th position if $q=q_j$. Then $k=1\leq K$ and $f(v)=0\in p\Z_p$.\\[\baselineskip]
\textbf{Case 2:} $\fa q\in\set{q_1,\ldots,q_s}:t_q\geq1$\\
\phantom{\textbf{Case 2:} }$\fa q\in\set{q_1,\ldots,q_s}:\nu_q(m)\neq\infty\Rightarrow i_{q,t_q}=0$.\\
In this case as a result we get
\begin{align}
\nu_q\left(a_im^ip^{(d-i)\ell}\right)
=\nu_q(a_i)+i\nu_q(m)+(d-i)\ell\nu_q(p)
=\infty
\end{align}
for all $q\in\set{q_1,\ldots, q_s}$ and $i\in\intint{1}{d}$, and hence $f(mp^c)=a_0\in p\Z_p$ for all $c\in\Z$. Setting $v\ce m/p$ thus implies $k=1\leq K$ and $f(v)\in p\Z_p$.\\[\baselineskip]
\textbf{Case 3:} $\fa q\in\set{q_1,\ldots,q_s}:t_q\geq1$\\
\phantom{\textbf{Case 3:} }$\ex q\in\set{q_1,\ldots,q_s}:\nu_q(m)\neq\infty\land i_{q,t_q}\geq1$.\\
We may assume without loss of generality that there is no $a\in\Z_p\setminus(r+p\Z_p)$ such that $f(a)\in p\Z_p$, because we already treated this case at the beginning of the proof. But then it follows that
\begin{align}
\ex q\in\set{q_1,\ldots,q_s}:\nu_q(m)<\ell\nu_q(p)\land i_{q,t_q}\geq1,
\end{align}
because otherwise we could find an $a\in\Z_p\setminus(r+p\Z_p)$ with $f(a)\in p\Z_p$ in the following way: set $b\ce\varphi_p^{-1}((b_1,\ldots,b_s))$ and $c\ce\varphi_p^{-1}((1-b_1,\ldots,1-b_s))$ with $b_j=(i_{q_j,t_{q_j}}=0\;?\;0,1)$ for all $j\in\intint{1}{s}$. Then $c\neq0$ (otherwise $\nu_q(m)\geq\ell\nu_q(p)\geq\nu_q(p)$ for all $q$ in $\set{q_1,\ldots,q_s}$, hence $m\in p\Z_p$ which is a contradiction), $bn$ and $bn+c$ are in $\Z_p$; they are not congruent modulo $p$, and both satisfy $f(bn)=f(bn+c)=f(n)\in p\Z_p$. Thus $a=bn$ or $a=bn+c$ contradicts our assumption that no such $a$ exists.

Now let $v\ce bn\in\Q_p\setminus\Z_p$, $k\ce-\nu_p(v)$, and $u\ce vp^k$. Then $f(v)=f(n)\in p\Z_p$ and we claim that $k\leq K$ which would complete the proof. Assume to the contrary that $k>K$ and let $q\in\set{q_1,\ldots,q_s}$ such that $\nu_q(u)\in\intintuptoex{\nu_q(p)}$ (such a $q$ exists because $\nu_p(u)=0$). Then $i_{q_,t_q}\geq1$ (otherwise $\nu_q(u)=\infty$ by $u=bnp^k$ and by the definition of $b$) and
\begin{align}
\nu_q\left(a_iu^ip^{(d-i)k}\right)
&=\nu_q(a_i)+i\nu_q(k)+(d-i)k\nu_q(p)\\
&\in\nu_q(a_i)+(d-i)k\nu_q(p)+i\intintuptoex{\nu_q(p)}
\end{align}
for all $i\in\intintuptoin{d}$. By assumption we have
\begin{align}
\frac{\nu_q(a_{i_{q,j+1}})-\nu_q(a_{i_{q,j}})+i_{q,j+1}(\nu_q(p)-1)}{(i_{q,j+1}-i_{q,j})\nu_q(p)}\leq K<k
\end{align}
or equivalently
\begin{align}
\nu_q(a_{i_{q,j+1}})+(d-i_{q,j+1})k\nu_q(p)+i_{q,j+1}(\nu_q(p)-1)<\nu_q(a_{i_{q,j}})+(d-i_{q,j})k\nu_q(p)
\end{align}
for all $j\in\intint{1}{t_q-1}$ and hence
\begin{align}
\nu_q\left(a_{i_{q,j_1}}u^{i_{q,j_1}}p^{(d-i_{q,j_1})k}\right)\neq\nu_q\left(a_{i_{q,j_2}}u^{i_{q,j_2}}p^{(d-i_{q,j_2})k}\right)
\end{align}
for all distinct $j_1,j_2\in\intint{1}{t_q}$. Furthermore, we have, again by assumption and by $t_q\geq 1$ and $i_{q_,t_q}\geq1$,
\begin{align}
\frac{\nu_q(a_{i_{q,t_q}})-\nu_q(p)+i_{q,t_q}(\nu_q(p)-1)}{i_{q,t_q}\nu_q(p)}\leq K<k
\end{align}
or equivalently
\begin{align}
\nu_q(a_{i_{q,t_q}})+(d-i_{q,t_q})k\nu_q(p)+i_{q,t_q}(\nu_q(p)-1)<(dk+1)\nu_q(p).
\end{align}
Thus we get
\begin{align}
\nu_q\left(f(v)p^{dk}\right)&=\inf\Big\{\nu_q\left(a_{i_{q,j}}u^{i_{q,j}}p^{(d-i_{q,j})k}\right)\;\Big\vert\;j\in\intint{1}{t_q}\Big\}\\
&=\nu_q\left(a_{i_{q,t_q}}u^{i_{q,t_q}}p^{(d-i_{q,t_q})k}\right)\\
&\leq\nu_q(a_{i_{q,t_q}})+(d-i_{q,t_q})k\nu_q(p)+i_{q,t_q}(\nu_q(p)-1)\\
&<(dk+1)\nu_q(p)
\end{align}
which contradicts $f(v)p^{dk}\in p^{dk+1}\Z_p$.
\end{proof}

\noindent
As an example consider the following polynomials: $f_0(x)\ce7x^3-4x^2+x-6$, $f_1(x)\ce3x^7-x+1$, $f_2(x)\ce5x^4+4x-1$. By Theorem~\ref{TCharacAvPolyF} all of the $f_r$ are $(3,r)$-avoiding, and by Theorem~\ref{TPropPolyF} they are $(3,r)$-suitable. Thus $\FRFf{R}\ce(f_0,f_1,f_2)\in\SoFibredRationalFunctions{3}$ is avoiding by definition and $\FFf{F}\ce\FRFFint(\FRFf{R}\vert_{\Z_p})=(f_0,f_1,f_2)\in\SoFibredFunctions{3}$ has the block property by Theorem~\ref{TFuncSuit}~(2). In addition, $\FRFf{R}$ is clearly closed and $\FFf{F}$ is in weak canonical form. Thus the conditions of Theorem~\ref{TAvoidFRFUnFP} are met and $\FRFF{\FRFf{R}}{\Sf{D}}$ has a unique fixed point in $\Z_p$ for every $\Sf{D}\in\CoSequences(\SPboundedby{\intintuptoex{p}},\SPlength{\N})$. Stripping off the machinery of $p$-fibred functions this implies that for every polynomial of the form $g_{r_1}\circ\ldots\circ g_{r_\ell}$, where $g_r\ce f_r/3$ for every $r\in\set{0,1,2}$, $\ell\in\N$, and $r_1,\ldots,r_\ell\in\set{0,1,2}$, there is a unique $n\in\Z_p$ satisfying $g_{r_1}\circ\ldots\circ g_{r_\ell}(n)=n$.

\section{Periodic and ultimately periodic digit expansions}
\label{SPeriodicExpansions}
One problem about $p$-adic systems we are particularly interested in concerns the characterization of the sets of all $p$-adic integers which have periodic or ultimately periodic digit expansions with respect to a given $p$-adic system. It is a natural generalization of the very specific question asked in the Collatz conjecture which claims that all natural numbers have periodic digit expansions with period $(1,0)$. If one could characterize the set of all $2$-adic integers having ultimately periodic $\FFf{F}_C$-digit expansions, one would probably find what numerous experiments suggest, namely that it is exactly the set $\Q\cap\Z_2$ of rational numbers with odd denominators. By Corollary~\ref{CInitPer} this would automatically prove that orbits of natural numbers under $\FFf{F}_C$ cannot diverge, which would be a significant step in proving the Collatz conjecture.

For $2\leq p\in\N$, $\FFf{F}\in\SoSystems{p}$, and $n\in\Z_p$ we say that $n$ is a \emph{periodic}, \emph{ultimately periodic}, or \emph{aperiodic point} of $\FFf{F}$ if $\FFFDT(\FFf{F})[n]$ is periodic, ultimately periodic, or aperiodic, respectively. Note that by Corollary~\ref{CInitPer} we could replace $\FFFDT(\FFf{F})[n]$ by $\FFFST(\FFf{F})[n]$ to get equivalent definitions. We define the sets
\begin{align}
\label{DFFFSoPP}
\FFFSoPP(\FFf{F})&\ce\set{n\in\Z_p\mid\SPperiodic(\FFFDT(\FFf{F})[n])}\\
\label{DFFFSoUPP}
\FFFSoUPP(\FFf{F})&\ce\set{n\in\Z_p\mid\SPultimatelyperiodic(\FFFDT(\FFf{F})[n])}\\
\label{DFFFSoAPP}
\FFFSoAPP(\FFf{F})&\ce\set{n\in\Z_p\mid\SPaperiodic(\FFFDT(\FFf{F})[n])}
\end{align}
of periodic, ultimately periodic, and aperiodic points of $\FFf{F}$. Furthermore, for every logical expression $E$ in the three unknown sets $\FFFSoPP$, $\FFFSoUPP$, and $\FFFSoAPP$ we define the following predicate on $\SoSystems{p}$ ($2\leq p\in\N$, $\FFf{F}\in\SoSystems{p}$):
\begin{flalign}
\pushleft{\FFPsatisfies{E}(\FFf{F})}&\Leftrightarrow E[\FFFSoPP\to\FFFSoPP(\FFf{F}),\FFFSoUPP\to\FFFSoUPP(\FFf{F}),\FFFSoAPP\to\FFFSoAPP(\FFf{F})]&\text{$\FFf{F}$ \emph{satisfies $E$}}\label{DFFPsatisfies}
\end{flalign}
where $E[\FFFSoPP\to\FFFSoPP(\FFf{F}),\FFFSoUPP\to\FFFSoUPP(\FFf{F}),\FFFSoAPP\to\FFFSoAPP(\FFf{F})]$ is the logical expression obtained by substituting $\FFFSoPP$ with $\FFFSoPP(\FFf{F})$, $\FFFSoUPP$ with $\FFFSoUPP(\FFf{F})$, and $\FFFSoAPP$ with $\FFFSoAPP(\FFf{F})$ in $E$, e.g. $[\FFFSoUPP=\Q\cap\Z_2](\FFf{F}_C)\Leftrightarrow\FFFSoUPP(\FFf{F}_C)=\Q\cap\Z_2$.

\noindent
Additionally, we define the following predicates on $\SoSystems{p}$ ($2\leq p\in\N$, $\FFf{F}\in\SoSystems{p}$, $A$ set):
\begin{flalign}
\pushleft{\FFPperiodicon{A}(\FFf{F})}&\Leftrightarrow[\FFFSoPP=A](\FFf{F})&\text{$\FFf{F}$ is \emph{periodic on $A$}}\label{DFFPperiodicon}\\
\pushleft{\FFPultimatelyperiodicon{A}(\FFf{F})}&\Leftrightarrow[\FFFSoUPP=A](\FFf{F})&\text{$\FFf{F}$ is \emph{ultimately periodic on $A$}}\label{DFFPultimatelyperiodicon}\\
\pushleft{\FFPaperiodicon{A}(\FFf{F})}&\Leftrightarrow[\FFFSoAPP=A](\FFf{F})&\text{$\FFf{F}$ is \emph{aperiodic on $A$}}\label{DFFPaperiodicon}
\end{flalign}

Using $\pi_{\FFf{F},\FFf{G}}$ we can characterize when two $p$-adic systems $\FFf{F}$ and $\FFf{G}$ are periodic, ultimately periodic, or aperiodic on the same sets.

\begin{lemma}
\label{LPeriodicOn}
Let $2\leq p\in\N$, $\FFf{F},\FFf{G}\in\SoSystems{p}$, $B\subseteq\CoSequences(\SPboundedby{\intintuptoex{p}},\neg\SPfinite)$, $A_{\FFf{F}}\ce\set{n\in\Z_p\mid\FFFDT(\FFf{F})[n]\in B}$, and $A_{\FFf{G}}\ce\set{n\in\Z_p\mid\FFFDT(\FFf{G})[n]\in B}$. Then,
\begin{align}
A_{\FFf{F}}=A_{\FFf{G}}\Leftrightarrow\pi_{\FFf{F},\FFf{G}}(A_\FFf{F})=A_\FFf{F}.
\end{align}
In particular,

\begin{theoremtable}
\theoremitem{(1)}&$\FFFSoPP(\FFf{F})=\FFFSoPP(\FFf{G})\Leftrightarrow\pi_{\FFf{F},\FFf{G}}(\FFFSoPP(\FFf{F}))=\FFFSoPP(\FFf{F})$\tabularnewline
\theoremitem{(2)}&$\FFFSoUPP(\FFf{F})=\FFFSoUPP(\FFf{G})\Leftrightarrow\pi_{\FFf{F},\FFf{G}}(\FFFSoUPP(\FFf{F}))=\FFFSoUPP(\FFf{F})$\tabularnewline
\theoremitem{(3)}&$\FFFSoAPP(\FFf{F})=\FFFSoAPP(\FFf{G})\Leftrightarrow\pi_{\FFf{F},\FFf{G}}(\FFFSoAPP(\FFf{F}))=\FFFSoAPP(\FFf{F})$.
\end{theoremtable}
\end{lemma}

\begin{proof}
First we observe that $A_{\FFf{F}}=\psi_{\FFf{F}}^{-1}(B)$ and $A_{\FFf{G}}=\psi_{\FFf{G}}^{-1}(B)$ by definition of $\psi_{\FFf{F}}$ and $\psi_{\FFf{G}}$. Since $\psi_{\FFf{F}}$ and $\psi_{\FFf{G}}$ are bijective it follows that
\begin{align}
A_{\FFf{F}}=A_{\FFf{G}}
&\Leftrightarrow
\psi_{\FFf{G}}(A_{\FFf{F}})=\psi_{\FFf{G}}(A_{\FFf{G}})=\psi_{\FFf{G}}(\psi_{\FFf{G}}^{-1}(B))=B=\psi_{\FFf{F}}(\psi_{\FFf{F}}^{-1}(B))=\psi_{\FFf{F}}(A_{\FFf{F}})\\
&\Leftrightarrow
A_{\FFf{F}}=\psi_{\FFf{G}}^{-1}(\psi_{\FFf{G}}(A_{\FFf{F}}))=\psi_{\FFf{G}}^{-1}(\psi_{\FFf{F}}(A_{\FFf{F}}))=\pi_{\FFf{F},\FFf{G}}(A_\FFf{F}).
\end{align}

The ''In particular'' part follows directly if one sets $B\ce\set{\Sf{S}\in\CoSequences(\SPboundedby{\intintuptoex{p}},\neg\SPfinite)\mid\SPperiodic(\Sf{S})}$, $B\ce\set{\Sf{S}\in\CoSequences(\SPboundedby{\intintuptoex{p}},\neg\SPfinite)\mid\SPultimatelyperiodic(\Sf{S})}$, or $B\ce\set{\Sf{S}\in\CoSequences(\SPboundedby{\intintuptoex{p}},\neg\SPfinite)\mid\SPaperiodic(\Sf{S})}$ respectively.
\end{proof}

\noindent
In the upcoming section we will prove direct formulas for $\pi_{\FFf{F},\FFf{G}}(n)$, $n\in\Z_p$ for several combinations of $p$-adic systems $\FFf{F}$ and $\FFf{G}$. For example, Theorem~\ref{TConstIrr} implies that $\pi_{(x,3x+1),(x,3x+c)}(n)=cn$ for all odd integers $c$. Since the mapping $n\mapsto cn$ maps $\Q\cap\Z_2$ to itself, it follows from the ``In particular'' part of Lemma~\ref{LPeriodicOn} that $(x,3x+1)$ is ultimately periodic on $\Q\cap\Z_2$ if and only if $(x,3x+c)$ is ultimately periodic on $\Q\cap\Z_2$. It thus suffices to prove that $(x,3x+c)$ is ultimately periodic on $\Q\cap\Z_2$ for any $c$ to automatically prove it for all $c$. Furthermore, if a $2$-adic system $\FFf{G}$ could be found which can be proven to be ultimately periodic on $\Q\cap\Z_2$ and for which $\pi_{(x,3x+1),\FFf{G}}(n)$ can be directly computed by a formula, such that $\pi_{(x,3x+1),\FFf{G}}(\Q\cap\Z_2)=\Q\cap\Z_2$ can be demonstrated, it would again follow from the ``In particular'' part of Lemma~\ref{LPeriodicOn} that $(x,3x+1)$ is ultimately periodic on $\Q\cap\Z_2$. Unfortunately, such a $\FFf{G}$ could not be found so far, but further investigations in this direction are made in the upcoming section.

\myparagraphtoc{Contractive and expansive $p$-adic systems.}
In the special cases where a $p$-adic system $\FFf{F}$ is either \emph{contractive} or \emph{expansive} we can easily characterize $\FFFSoPP(\FFf{F})\cap\Q$ and $\FFFSoUPP(\FFf{F})\cap\Q$ respectively.

\noindent
We define the following predicates on $\SoSystems{p}$ ($2\leq p\in\N$, $\FFf{F}\in\SoSystems{p}$):
\bgroup
\allowdisplaybreaks
\begin{flalign}
\pushleft{\FFPcontractive(\FFf{F})}&\Leftrightarrow\FFf{F}(\Q\cap\Z_p)\subseteq\Q\cap\Z_p&\hspace{-3em}\text{$\FFf{F}$ is \emph{contractive}}\label{DFFPcontractive}\\
\nonumber
&\phantom{\vphantom{a}\Leftrightarrow\vphantom{a}}\ex 0\leq M\in\R:\fa n\in\Q\cap\Z_p:\abs{n}>M\Rightarrow\abs{\FFf{F}(n)}<\abs{n}\\
\pushleft{\FFPexpansive(\FFf{F})}&\Leftrightarrow\FFf{F}(\Q\cap\Z_p)\subseteq\Q\cap\Z_p&\hspace{-2em}\text{$\FFf{F}$ is \emph{expansive}}\label{DFFPexpansive}\\
\nonumber
&\phantom{\vphantom{a}\Leftrightarrow\vphantom{a}}\ex 0\leq M\in\R:\fa n\in\Q\cap\Z_p:\abs{n}>M\Rightarrow\abs{\FFf{F}(n)}>\abs{n}\\
\pushleft{\mathrlap{\FFPmixed(\FFf{F})}}&\Leftrightarrow\FFf{F}(\Q\cap\Z_p)\subseteq\Q\cap\Z_p&\hspace{-4em}\text{$\FFf{F}$ is \emph{of mixed type}}\label{DFFPmixed}\\
\nonumber
&\phantom{\vphantom{a}\Leftrightarrow\vphantom{a}}\neg\FFPcontractive(\FFf{F})\land\neg\FFPexpansive(\FFf{F})\\
\pushleft{\FFPcontractsdenominators(\FFf{F})}&\Leftrightarrow\fa a/b\in\Q\cap\Z_p\text{ with }(a,b)\in\Z\times\N\text{ coprime}:&\hspace{-7em}\text{$\FFf{F}$ \emph{contracts denominators}}\label{DFFPcontractsdenominators}\\
\nonumber
&\phantom{\vphantom{a}\Leftrightarrow\vphantom{a}}\quad\ex(c,d)\in\Z\times\N\text{ coprime}:\FFf{F}(a/b)=c/d\land d\leq b\\
\pushleft{\FFPexpandsdenominators(\FFf{F})}&\Leftrightarrow\fa a/b\in\Q\cap\Z_p\text{ with }(a,b)\in\Z\times\N\text{ coprime}:&\hspace{-7em}\text{$\FFf{F}$ \emph{expands denominators}}\label{DFFPexpandsdenominators}\\
\nonumber
&\phantom{\vphantom{a}\Leftrightarrow\vphantom{a}}\quad\ex(c,d)\in\Z\times\N\text{ coprime}:\FFf{F}(a/b)=c/d\land d>b\\
\pushleft{\FFPmixesdenominators(\FFf{F})}&\Leftrightarrow\FFf{F}(\Q\cap\Z_p)\subseteq\Q\cap\Z_p&\hspace{-7em}\text{$\FFf{F}$ \emph{mixes denominators}}\label{DFFPmixesdenominators}\\
\nonumber
&\phantom{\vphantom{a}\Leftrightarrow\vphantom{a}}\neg\FFPcontractsdenominators(\FFf{F})\land\neg\FFPexpandsdenominators(\FFf{F})
\end{flalign}
\egroup

\noindent
Note that $\FFPcontractsdenominators(\FFf{F})$ and $\FFPexpandsdenominators(\FFf{F})$ both imply $\FFf{F}(\Q\cap\Z_p)\subseteq\Q\cap\Z_p$.

If a $p$-fibred function is contractive or expansive, this has the following consequences for periodic and ultimately periodic digit expansions.

\begin{lemma}
\label{LContrExp}
Let $2\leq p\in\N$, $\FFf{F}\in\SoSystems{p}$, $0\leq M\in\R$ as in the definition above, and $A\ce\set{n\in\Q\cap\Z_p\mid\abs{n}\leq M}$. Then,

\begin{theoremtable}[rlX]
\theoremitem{(1)}&$\FFPcontractive(\FFf{F})$&$\hspace{-3.pt}\quad\Rightarrow\quad\FFFSoPP(\FFf{F})\cap\Q\subseteq\bigcup_{k=0}^\infty\FFf{F}^k(A)$\tabularnewline
&&\leavevmode\phantom{$\hspace{-3.pt}\quad\Rightarrow\quad\vphantom{a}$}In particular: Every period of $\FFf{F}$ that contains a rational number\tabularnewline
&&\leavevmode\phantom{$\hspace{-3.pt}\quad\Rightarrow\quad\vphantom{a}$In particular: }also contains an element of $A$\tabularnewline
\theoremitem{(2)}&$\FFPcontractive(\FFf{F})\land\FFPcontractsdenominators(\FFf{F})$&$\hspace{-3.pt}\quad\Rightarrow\quad\FFFSoUPP(\FFf{F})\cap\Q=\Q\cap\Z_p$\tabularnewline
&&\leavevmode\phantom{$\hspace{-3.pt}\quad\Rightarrow\quad\vphantom{a}$}In particular: $\FFFSoUPP(\FFf{F})\subseteq\Q\Rightarrow\FFPultimatelyperiodicon{\Q\cap\Z_p}(\FFf{F})$\tabularnewline
\theoremitem{(3)}&$\FFPexpansive(\FFf{F})$&$\hspace{-3.pt}\quad\Rightarrow\quad\FFFSoUPP(\FFf{F})\cap\Q\subseteq A$\tabularnewline
&&\leavevmode\phantom{$\hspace{-3.pt}\quad\Rightarrow\quad\vphantom{a}$}In particular: $\neg\FFPultimatelyperiodicon{\Q\cap\Z_p}(\FFf{F})$\tabularnewline
\theoremitem{(4)}&$\FFPexpandsdenominators(\FFf{F})$&$\hspace{-3.pt}\quad\Rightarrow\quad\FFFSoUPP(\FFf{F})\cap\Q=\emptyset$\tabularnewline
&&\leavevmode\phantom{$\hspace{-3.pt}\quad\Rightarrow\quad\vphantom{a}$}In particular: $\neg\FFPultimatelyperiodicon{\Q\cap\Z_p}(\FFf{F})$.
\end{theoremtable}
\end{lemma}

\begin{proof}
$ $\\
\proofitem{(1)}%
Let $n\in\FFFSoPP(\FFf{F})\cap\Q$. Then, $\abs{\FFf{F}^{\ell+1}(n)}\geq\abs{\FFf{F}^\ell(n)}$ for some $\ell\in\Nz$ (otherwise $n$ cannot be a periodic point of $\FFf{F}$) and thus $\abs{\FFf{F}^\ell(n)}\leq M$ by definition of $M$ which proves the ``In particular'' part. Also, $n=\FFf{F}^k(\FFf{F}^\ell(n))$ for some $k\in\Nz$ and thus we get $n\in\bigcup_{k=0}^\infty\FFf{F}^k(A)$ as claimed.

\noindent
\proofitem{(2)}%
Let $n=a/b\in\Q\cap\Z_p$ with $(a,b)\in\Z\times\N$ coprime and
\begin{align}
B\ce\set{c/d\in\Q\mid(c,d)\in\Z\times\N\land d\leq b}.
\end{align}
Then, $\FFPcontractive(\FFf{F})$ and $\FFPcontractsdenominators(\FFf{F})$ imply that $\set{\FFf{F}^k(n)\mid k\in\Nz}$ is contained in $\bigcup_{k=0}^\infty\FFf{F}^k\left(A\cap B\right)$ which is a finite set. Thus $n$ is an ultimately periodic point of $\FFf{F}$.

\noindent
(3) and (4) follow directly from the definitions.
\end{proof}

\noindent
An important class of $p$-adic systems that are expansive (and thus cannot be ultimately periodic on $\Q\cap\Z_p$) is given by $(\Q\cap\Z_p)$-polynomial $p$-adic systems where each polynomial is either of degree $2$ or higher or has a linear coefficient greater than $p$ in absolute value.

\begin{theorem}
\label{TPolyExp}
Let $2\leq p\in\N$ and $\FFf{F},\FFf{G}\in\SoSystems{p}(\FRFPpolynomialcoefficients{\Q\cap\Z_p})$ such that $\FFf{F}[r]$ either is of degree $2$ or higher or $\FFf{F}[r]=a_r+b_rx$ with $\abs{b_r}>p$ for all $r\in\intintuptoex{p}$. Then,

\begin{theoremtable}
\theoremitem{(1)}&$\FFPexpansive(\FFf{F})$\tabularnewline
&In particular: $\neg\FFPultimatelyperiodicon{\Q\cap\Z_p}(\FFf{F})$\tabularnewline
\theoremitem{(2)}&$\FFPcontractive(\FFf{G})\Rightarrow\FRFPlinearpolynomial(\FFf{G})$.
\end{theoremtable}
\end{theorem}

\noindent
Note that by Theorem~\ref{TFuncSuit}~(2) and Theorem~\ref{TPropPolyF}~(3) we get $\SoSystems{p}(\FRFPpolynomialcoefficients{\Q\cap\Z_p})=\SoSystems{p}(\FRFPpolynomialcoefficientsdegree{\Q\cap\Z_p}{\intint{1}{\infty}})$.

\begin{proof}[Proof of Theorem~\ref{TPolyExp}]
$ $\\
\proofitem{(1)}%
For every $r\in\intintuptoex{p}$, $\FFf{F}[r]$ is a polynomial function (on $r+p\Z_p$) with coefficients in $\Q\cap\Z_p$. Thus we clearly have $\FFf{F}(\Q\cap\Z_p)\subseteq\Q\cap\Z_p$ and for all $r\in\intintuptoex{p}$ there is an $0\leq M_r\in\R$ such that
\begin{align}
\abs{\FFf{F}[r](n)-\FFf{F}[r](n)\modulo p}>p\abs{n}
\end{align}
for all $n\in\Q\cap(r+p\Z_p)$ with $\abs{n}>M_r$ (note that $\FFf{F}[r](n)\modulo p\in\intintuptoex{p}$) by the additional assumptions. Let $M\ce\max\set{M_0,\ldots,M_{p-1}}$, $n\in\Q\cap\Z_p$ with $\abs{n}>M$, and $r\ce n\modulo p$. Then,
\begin{align}
\abs{\FFf{F}(n)}&=\abs{\frac{\FFf{F}[r](n)-\FFf{F}[r](n)\modulo p}{p}}>\abs{n}
\end{align}
which implies that $\FFf{F}$ is expansive.

\noindent
\proofitem{(2)}%
Clearly, if $\FFf{G}[r]$ is of degree $2$ or higher for at least one $r\in\intintuptoex{p}$, then $\FFf{G}$ cannot be contractive.
\end{proof}

\section{Linear-polynomial $p$-adic systems and the question of ultimate periodicity}
\label{SLinPoly}
In the following we investigate linear-polynomial $p$-adic systems such as $\FFf{F}_p=(x)^p$ (standard base $p$) or $\FFf{F}_C=(x,3x+1)$ (Collatz).

\myparagraphtoc{Basic facts.}
We begin by applying results from previous sections to characterize linear polynomials with special properties.

\begin{lemma}
\label{LLinPolySuit}
Let $2\leq p\in\N$, $r\in\intintuptoex{p}$, $f=a+bx\in\Z_p[x]$, and $k\in\Nz$. Then,

\begin{theoremtable}
\theoremitem{(1)}&$\FPweaklysuitable{p}{r}(f)$\tabularnewline
\theoremitem{(2)}&$\FPsuitable{p}{r}(f)\Leftrightarrow\Gcd{p,b\modulo p}=1$\tabularnewline
\theoremitem{(3)}&$\FPsuitable{p}{r}(f)\Rightarrow(\FPavoiding{p}{r}(f)\Leftrightarrow f(r)\in p\Z_p)$.
\end{theoremtable}
\end{lemma}

\noindent
Note that by (3) and Theorem~\ref{TFuncSuit}~(2) it follows that every linear-polynomial $p$-adic system in weak canonical form is avoiding when interpreted as the $p$-fibred rational function defined by the same polynomials.

\begin{proof}[Proof of Lemma~\ref{LLinPolySuit}]
(1) and (2) follow directly from Theorem~\ref{TPropPolyF}. For the proof of (3) we assume  $\FPsuitable{p}{r}(f)$ which is equivalent to $\Gcd{p,b\modulo p}=1$ by (2). Thus, $f(\Q_p\setminus\Z_p)\cap p\Z_p=\emptyset$ and $f(\intintuptoex{p})\modulo p=\intintuptoex{p}$. But then
\begin{align}
\FPavoiding{p}{r}(f)\Leftrightarrow f(\intintuptoex{p}\setminus\set{r})\cap p\Z_p=\emptyset\Leftrightarrow f(r)\in p\Z_p
\end{align}
by Theorem~\ref{TCharacAvPolyF}.
\end{proof}

\myparagraphtoc{A formula for $\FRFF{\FRFf{R}}{\Sf{D}}$.}
Next we give a direct formula for $\FRFF{\FRFf{R}}{\Sf{D}}$ if $\FRFf{R}$ is a linear-polynomial $p$-fibred rational function. Note that for the special case $\FRFf{R}=(x,3x+1)$ this formula can be found in many publications on the original Collatz conjecture, such as \cite{BoehmSontacchi:1978,Seifert:1988,Amigo:2000,Amigo:2006}. For any sequence $\Sf{S}$ and any $a$ we denote by $\SFposition{\Sf{S}}{a}$ the increasing sequence of all indices $i\in\intintuptoex{\abs{\Sf{S}}}$ for which $\Sf{S}[i]=a$ and set $\SFcount{\Sf{S}}{a}\ce\abs{\SFposition{\Sf{S}}{a}}$. Furthermore, for $\FFf{F}\ce(a_0+b_0x,\ldots,a_{p-1}+b_{p-1}x)\in\SoFibredFunctions{p}(\FFPlinearpolynomial)\cup\SoFibredRationalFunctions{p}(\FFPlinearpolynomial)$ and $\Sf{D}\in\CoSequences(\SPboundedby{\intintuptoex{p}},\SPfinite)$ we define
\begin{align}
\label{DFRFlinA}
\FRFlinA{\FFf{F}}{\Sf{D}}
&\ce\sum_{i=0}^{\abs{\Sf{D}}-1}a_{\Sf{D}[i]}p^i\prod_{j=i+1}^{\abs{\Sf{D}}-1}b_{\Sf{D}[j]}\\
\label{DFRFlinB}
\FRFlinB{\FFf{F}}{\Sf{D}}
&\ce\prod_{i=0}^{\abs{\Sf{D}}-1}b_{\Sf{D}[i]}.
\intertext{By collecting the ``$a_i$''s and ``$b_i$''s it follows that}
\FRFlinA{\FFf{F}}{\Sf{D}}
&=\sum_{r=0}^{p-1}a_r\sum_{i=0}^{\SFcount{\Sf{D}}{r}-1}p^{\SFposition{\Sf{D}}{r}[i]}\prod_{s=0}^{p-1}b_s^{\SFcount{\Sf{D}[\SFposition{\Sf{D}}{r}[i]+1,\abs{\Sf{D}}-1]}{s}}\\
\FRFlinB{\FFf{F}}{\Sf{D}}
&=\prod_{r=0}^{p-1}b_r^{\SFcount{\Sf{D}}{r}}.
\end{align}

\begin{theorem}
\label{TLinDirForm}
Let $2\leq p\in\N$, $\FRFf{R}\in\SoFibredRationalFunctions{p}(\FRFPlinearpolynomial)$, $\Sf{D}\in\CoSequences(\SPboundedby{\intintuptoex{p}},\SPfinite)$, and $n\in\Q_p$. Then,
\begin{align}
\FRFF{\FRFf{R}}{\Sf{D}}(n)&=\frac{\FRFlinA{\FRFf{R}}{\Sf{D}}+n\FRFlinB{\FRFf{R}}{\Sf{D}}}{p^{\abs{\Sf{D}}}}.
\end{align}
In particular: $\FRFF{\FRFf{R}}{\Sf{D}}(x)\in\Q_p[x]$ is a linear polynomial.
\end{theorem}

\begin{proof}
Let $a_0,b_0,\ldots,a_{p-1},b_{p-1}\in\Q_p$ such that $\FRFf{R}=(a_0+b_0x,\ldots,a_{p-1}+b_{p-1}x)$. We prove the formula by induction on the length of $\Sf{D}$. If $\abs{\Sf{D}}=0$, then $\FRFF{\FRFf{R}}{\Sf{D}}(n)=n$ as claimed. Now assume that the formula is true for $\Sf{D}$ and let $\Sf{E}\ce\Sf{D}\cdot(e)\in\CoSequences(\SPboundedby{\intintuptoex{p}},\SPfinite)$ where $e\in\intintuptoex{p}$. We compute
\bgroup
\allowdisplaybreaks
\begin{align}
\FRFF{\FRFf{R}}{\Sf{E}}(n)
&=\frac{\FRFf{R}[\Sf{E}[\abs{\Sf{E}}-1]]\left(\FRFF{\FRFf{R}}{\Sf{E}[\intintuptoex{\abs{\Sf{E}}-1}]}(n)\right)}{p}
=\frac{\FRFf{R}[e]\left(\FRFF{\FRFf{R}}{\Sf{D}}(n)\right)}{p}\\
&=\left(a_e+b_e\left(n\cdot\prod_{i=0}^{\abs{\Sf{D}}-1}b_{\Sf{D}[i]}+\sum_{i=0}^{\abs{\Sf{D}}-1}a_{\Sf{D}[i]}p^i\prod_{j=i+1}^{\abs{\Sf{D}}-1}b_{\Sf{D}[j]}\right)\frac{1}{p^{\abs{\Sf{D}}}}\right)\frac{1}{p}\\
&=\left(a_{\Sf{E}[\abs{\Sf{E}}-1]}+b_{\Sf{E}[\abs{\Sf{E}}-1]}\left(n\cdot\prod_{i=0}^{\abs{\Sf{E}}-2}b_{\Sf{E}[i]}+\sum_{i=0}^{\abs{\Sf{E}}-2}a_{\Sf{E}[i]}p^i\prod_{j=i+1}^{\abs{\Sf{E}}-2}b_{\Sf{E}[j]}\right)\frac{1}{p^{\abs{\Sf{E}}-1}}\right)\frac{1}{p}\\
&=\left(n\cdot\prod_{i=0}^{\abs{\Sf{E}}-1}b_{\Sf{E}[i]}+a_{\Sf{E}[\abs{\Sf{E}}-1]}p^{\abs{\Sf{E}}-1}+\sum_{i=0}^{\abs{\Sf{E}}-2}a_{\Sf{E}[i]}p^i\prod_{j=i+1}^{\abs{\Sf{E}}-1}b_{\Sf{E}[j]}\right)\frac{1}{p^{\abs{\Sf{E}}}}\\
&=\left(n\cdot\prod_{i=0}^{\abs{\Sf{E}}-1}b_{\Sf{E}[i]}+\sum_{i=0}^{\abs{\Sf{E}}-1}a_{\Sf{E}[i]}p^i\prod_{j=i+1}^{\abs{\Sf{E}}-1}b_{\Sf{E}[j]}\right)\frac{1}{p^{\abs{\Sf{E}}}}.
\end{align}
\egroup
\end{proof}

\myparagraphtoc{Find the number that has a given expansion.}
Theorem~\ref{TLinDirForm} has several consequences, one of them being that all ultimately periodic points of $(\Q\cap\Z_p)$-linear-polynomial $p$-adic systems are rational numbers. This follows from a general formula for the unique $p$-adic integer that has a given ultimately periodic digit expansion with respect to a given linear-polynomial $p$-adic system.

\begin{corollary}
\label{CPerFormula}
Let $2\leq p\in\N$, $\FFf{F}\in\SoSystems{p}(\FFPweakcanonicalform,\FFPlinearpolynomial)$, $\Sf{D}\in\CoSequences(\SPboundedby{\intintuptoex{p}},\SPultimatelyperiodic)$, and $n$ the unique element of $\Z_p$ satisfying $\FFFDT(\FFf{F})[n]=\Sf{D}$ (cf. Lemma~\ref{LInfBlDTCompl}). Then,
\begin{align}
n&=\left(\frac{\FRFlinA{\FFf{F}}{\SFperiodic{\Sf{D}}}}{p^{\abs{\SFperiodic{\Sf{D}}}}-\FRFlinB{\FFf{F}}{\SFperiodic{\Sf{D}}}}p^{\abs{\SFinitial{\Sf{D}}}}-\FRFlinA{\FFf{F}}{\SFinitial{\Sf{D}}}\right)\frac{1}{\FRFlinB{\FFf{F}}{\SFinitial{\Sf{D}}}}.
\end{align}
In particular,
\begin{align}
\label{CPerFormulaIPa}
\FFFSoPP(\FFf{F})&=\set{\frac{\FRFlinA{\FFf{F}}{\Sf{D}_P}}{p^{\abs{\Sf{D}_P}}-\FRFlinB{\FFf{F}}{\Sf{D}_P}}\mid\Sf{D}_P\in\CoSequences(\SPboundedby{\intintuptoex{p}},\SPfinite)}\\
\label{CPerFormulaIPb}
\FFFSoUPP(\FFf{F})&=\set{\left(\frac{\FRFlinA{\FFf{F}}{\Sf{D}_P}}{p^{\abs{\Sf{D}_P}}-\FRFlinB{\FFf{F}}{\Sf{D}_P}}p^{\abs{\Sf{D}_I}}-\FRFlinA{\FFf{F}}{\Sf{D}_I}\right)\frac{1}{\FRFlinB{\FFf{F}}{\Sf{D}_I}}\mid\Sf{D}_I,\Sf{D}_P\in\CoSequences(\SPboundedby{\intintuptoex{p}},\SPfinite)}
\end{align}
and if $\FFPlinearpolynomialcoefficients{\Q\cap\Z_p}(\FFf{F})$, then $[\FFFSoUPP\subseteq\Q\cap\Z_p](\FFf{F})$.
\end{corollary}

\begin{proof}
Let $a_0,b_0,\ldots,a_{p-1},b_{p-1}\in\Z_p$ and $\FRFf{R}\ce(a_0+b_0x,\ldots,a_{p-1}+b_{p-1}x)\in\SoFibredRationalFunctions{p}(\FRFPlinearpolynomial)$. Then, $\FFf{F}=\FRFFint(\FRFf{R}\vert_{\Z_p})$ and thus
\begin{align}
\FFf{F}^k(n)=\FRFF{\FRFf{R}}{\Sf{D}[\intintuptoex{k}]}(n)
\end{align}
for all $k\in\Nz$ by Lemma~\ref{LFRF}. Furthermore, let $m\ce\FFf{F}^{\abs{\SFinitial{\Sf{D}}}}(n)$. Then, $\FFFDT(\FFf{F})[m]=\SFperiodic{\Sf{D}}^\infty$ and hence
\begin{align}
\FFf{F}^k(m)=\FRFF{\FRFf{R}}{(\SFperiodic{\Sf{D}}^\infty)[\intintuptoex{k}]}(m)
\end{align}
for all $k\in\Nz$, again by Lemma~\ref{LFRF}. Thus, Corollary~\ref{CInitPer} and Theorem~\ref{TLinDirForm} imply
\begin{align}
m
&=\FFf{F}^{\abs{\SFinitial{\Sf{D}}}}(n)
=\FRFF{\FRFf{R}}{\Sf{D}[\intintuptoex{\abs{\SFinitial{\Sf{D}}}}]}(n)
=\FRFF{\FRFf{R}}{\SFinitial{\Sf{D}}}(n)
=\frac{\FRFlinA{\FRFf{R}}{\SFinitial{\Sf{D}}}+n\FRFlinB{\FRFf{R}}{\SFinitial{\Sf{D}}}}{p^{\abs{\SFinitial{\Sf{D}}}}}\\
m
&=\FFf{F}^{\abs{\SFperiodic{\Sf{D}}}}(m)
=\FRFF{\FRFf{R}}{(\SFperiodic{\Sf{D}}^\infty)[\intintuptoex{\abs{\SFperiodic{\Sf{D}}}}]}(m)
=\FRFF{\FRFf{R}}{\SFperiodic{\Sf{D}}}(m)
=\frac{\FRFlinA{\FRFf{R}}{\SFperiodic{\Sf{D}}}+m\FRFlinB{\FRFf{R}}{\SFperiodic{\Sf{D}}}}{p^{\abs{\SFperiodic{\Sf{D}}}}}.
\end{align}
Solving the second equation for $m$ and plugging in the result into the first equation yields
\begin{align}
m&=\frac{\FRFlinA{\FFf{F}}{\SFperiodic{\Sf{D}}}}{p^{\abs{\SFperiodic{\Sf{D}}}}-\FRFlinB{\FFf{F}}{\SFperiodic{\Sf{D}}}}\\
n&=\left(\frac{\FRFlinA{\FFf{F}}{\SFperiodic{\Sf{D}}}}{p^{\abs{\SFperiodic{\Sf{D}}}}-\FRFlinB{\FFf{F}}{\SFperiodic{\Sf{D}}}}p^{\abs{\SFinitial{\Sf{D}}}}-\FRFlinA{\FFf{F}}{\SFinitial{\Sf{D}}}\right)\frac{1}{\FRFlinB{\FFf{F}}{\SFinitial{\Sf{D}}}}.
\end{align}

The  inclusion ``$\subseteq$'' in Eqn.~(\ref{CPerFormulaIPa}) and Eqn.~(\ref{CPerFormulaIPb}) of the ``In particular'' part follows directly from what we just proved. For the other inclusion let $\Sf{D}_I,\Sf{D}_P\in\CoSequences(\SPboundedby{\intintuptoex{p}},\SPfinite)$ and set
\begin{align}
m&\ce\frac{\FRFlinA{\FFf{F}}{\Sf{D}_P}}{p^{\abs{\Sf{D}_P}}-\FRFlinB{\FFf{F}}{\Sf{D}_P}}\in\Z_p\\
n&\ce\left(\frac{\FRFlinA{\FFf{F}}{\Sf{D}_P}}{p^{\abs{\Sf{D}_P}}-\FRFlinB{\FFf{F}}{\Sf{D}_P}}p^{\abs{\Sf{D}_I}}-\FRFlinA{\FFf{F}}{\Sf{D}_I}\right)\frac{1}{\FRFlinB{\FFf{F}}{\Sf{D}_I}}\in\Z_p.
\end{align}
Then $\FRFF{\FRFf{R}}{\Sf{D}_I}(n)=m\in\Z_p$ and $\FRFF{\FRFf{R}}{\Sf{D}_P}(m)=m\in\Z_p$, and since $\FRFf{R}$ is avoiding (cf. the remark after Lemma~\ref{LLinPolySuit}), we get $\FFf{F}^{\abs{\Sf{D}_I}}(n)=\FRFF{\FRFf{R}}{\Sf{D}_I}(n)=m$, and $\FFf{F}^{\abs{\Sf{D}_P}}(m)=\FRFF{\FRFf{R}}{\Sf{D}_P}(m)=m$ by Lemma~\ref{LEqSEqD}. Thus, $m\in\FFFSoPP(\FFf{F})$ and $n\in\FFFSoUPP(\FFf{F})$.

Clearly, if $\FFPlinearpolynomialcoefficients{\Q\cap\Z_p}(\FFf{F})$, i.e. $a_0,b_0,\ldots,a_{p-1},b_{p-1}\in\Q\cap\Z_p$, then
\begin{align}
\FRFlinA{\FFf{F}}{\SFinitial{\Sf{D}}},\FRFlinA{\FFf{F}}{\SFperiodic{\Sf{D}}},\FRFlinB{\FFf{F}}{\SFinitial{\Sf{D}}},\FRFlinB{\FFf{F}}{\SFperiodic{\Sf{D}}}\in\Q\cap\Z_p
\end{align}
and hence $n\in\Q\cap\Z_p$.
\end{proof}

\noindent
Another consequence of Theorem~\ref{TLinDirForm} is a complete characterization of those $p$-adic integers whose digit-expansions with respect to a given linear-polynomial $p$-adic system have a given beginning.

\begin{corollary}
\label{CBegFormula}
Let $2\leq p\in\N$, $\FFf{F}\in\SoSystems{p}(\FFPweakcanonicalform,\FFPlinearpolynomial)$, $\Sf{D}\in\CoSequences(\SPboundedby{\intintuptoex{p}},\SPfinite)$, and $n$ the unique element of $\intintuptoex{p^{\abs{\Sf{D}}}}$ satisfying $\FFFDT(\FFf{F})[n][\intintuptoex{\abs{\Sf{D}}}]=\Sf{D}$ (cf. Lemma~\ref{LFinBlDTCompl}). Then,
\begin{align}
n&=\left(b\FRFlinA{\FFf{F}}{\Sf{D}}\right)\modulo p^{\abs{\Sf{D}}}
\end{align}
where $b\in\Z$ such that
\begin{align}
p^{\abs{\Sf{D}}}a-\left(\FRFlinB{\FFf{F}}{\Sf{D}}\modulo p^{\abs{\Sf{D}}}\right)b=1
\end{align}
for some $a\in\Z$ (find $b$ with extended Euclidean algorithm). In particular,
\begin{align}
\set{m\in\Z_p\mid\FFFDT(\FFf{F})[m][\intintuptoex{\abs{\Sf{D}}}]=\Sf{D}}=n+p^{\abs{\Sf{D}}}\Z_p
\end{align}
by $\FFPblock(\FFf{F})$.
\end{corollary}

\begin{proof}
Let $A\ce\FRFlinA{\FFf{F}}{\Sf{D}}$ and $B\ce\FRFlinB{\FFf{F}}{\Sf{D}}$. If we interpret $\FFf{F}$ as the $p$-fibred rational function that is defined by the same polynomials as the $p$-adic system $\FFf{F}$, we have $\FRFPavoiding(\FFf{F})$ according to the remark after Lemma~\ref{LLinPolySuit}. Thus, Lemma~\ref{LEqSEqD} and Theorem~\ref{TLinDirForm} imply that
\begin{align}
\FFFDT(\FFf{F})[m][\intintuptoex{\abs{\Sf{D}}}]=\Sf{D}
&\Leftrightarrow
\frac{A+mB}{p^{\abs{\Sf{D}}}}=\FRFF{\FFf{F}}{\Sf{D}}(m)\in\Z_p\\
&\Leftrightarrow
A+mB\in p^{\abs{\Sf{D}}}\Z_p\\
&\Leftrightarrow
A\modulo p^{\abs{\Sf{D}}}+\left(m\modulo p^{\abs{\Sf{D}}}\right)\left(B\modulo p^{\abs{\Sf{D}}}\right)\in p^{\abs{\Sf{D}}}\Z\\
&\Leftrightarrow
\ex a\in\Z:A\modulo p^{\abs{\Sf{D}}}+\left(m\modulo p^{\abs{\Sf{D}}}\right)\left(B\modulo p^{\abs{\Sf{D}}}\right)=ap^{\abs{\Sf{D}}}\\
&\Leftrightarrow
\ex a\in\Z:p^{\abs{\Sf{D}}}a-\left(B\modulo p^{\abs{\Sf{D}}}\right)\left(m\modulo p^{\abs{\Sf{D}}}\right)=A\modulo p^{\abs{\Sf{D}}}
\end{align}
for all $m\in\Z_p$. Since $\gcd\left(p^{\abs{\Sf{D}}},B\modulo p^{\abs{\Sf{D}}}\right)=1$ by Lemma~\ref{LLinPolySuit}~(2), the equation
\begin{align}
p^{\abs{\Sf{D}}}a-\left(B\modulo p^{\abs{\Sf{D}}}\right)b=1
\end{align}
has a solution $a,b\in\Z$ by B\'{e}zout's Lemma. For any such $a,b\in\Z$ we get
\begin{align}
&p^{\abs{\Sf{D}}}a\left(A\modulo p^{\abs{\Sf{D}}}\right)-\left(B\modulo p^{\abs{\Sf{D}}}\right)b\left(A\modulo p^{\abs{\Sf{D}}}\right)=A\modulo p^{\abs{\Sf{D}}}\\
&\Leftrightarrow p^{\abs{\Sf{D}}}a\left(A\modulo p^{\abs{\Sf{D}}}\right)-\left(B\modulo p^{\abs{\Sf{D}}}\right)\left(p^{\abs{\Sf{D}}}\frac{b\left(A\modulo p^{\abs{\Sf{D}}}\right)-\left(bA\right)\modulo p^{\abs{\Sf{D}}}}{p^{\abs{\Sf{D}}}}+\left(bA\right)\modulo p^{\abs{\Sf{D}}}\right)=A\modulo p^{\abs{\Sf{D}}}\\
&\Leftrightarrow p^{\abs{\Sf{D}}}C-\left(B\modulo p^{\abs{\Sf{D}}}\right)\left(\left(bA\right)\modulo p^{\abs{\Sf{D}}}\right)=A\modulo p^{\abs{\Sf{D}}}
\end{align}
where
\begin{align}
C&\ce a\left(A\modulo p^{\abs{\Sf{D}}}\right)-\left(B\modulo p^{\abs{\Sf{D}}}\right)\frac{b\left(A\modulo p^{\abs{\Sf{D}}}\right)-\left(bA\right)\modulo p^{\abs{\Sf{D}}}}{p^{\abs{\Sf{D}}}}\in\Z.
\end{align}
Thus,
\begin{align}
\FFFDT(\FFf{F})[bA][\intintuptoex{\abs{\Sf{D}}}]=\Sf{D}
\end{align}
and consequently $n=(bA)\modulo p^{\abs{\Sf{D}}}$.
\end{proof}

\myparagraphtoc{Inverse problem: given a number and expansion, find a system.}
Corollary~\ref{CPerFormula} allows to compute the unique $p$-adic integer that has a given ultimately periodic $\FFf{F}$-digit expansion for a given linear-polynomial $p$-adic system $\FFf{F}$ in weak canonical form. In the other direction one might try to find one or even all linear-polynomial $p$-adic systems for which a given $p$-adic integer has a given ultimately periodic digit expansion. The characterization of these $p$-adic systems is given in the following corollary which is another consequence of Theorem~\ref{TLinDirForm}.

\begin{corollary}
\label{CInvSyst}
Let $2\leq p\in\N$, $\Sf{D}\in\CoSequences(\SPboundedby{\intintuptoex{p}},\SPultimatelyperiodic)$, $r\ce\Sf{D}[0]$, $n\in r+p\Z_p$, and $\FFf{F}=(a_0+b_0x,\ldots,a_{p-1}+b_{p-1}x)\in\SoSystems{p}(\FFPweakcanonicalform,\FFPlinearpolynomial)$. Furthermore, let
\begin{align}
k&\ce\abs{\SFinitial{\Sf{D}}}\\
\ell&\ce\abs{\SFperiodic{\Sf{D}}}\\
c_s&\ce\sum_{i=0}^{\SFcount{\SFinitial{\Sf{D}}}{s}-1}p^{\SFposition{\SFinitial{\Sf{D}}}{s}[i]}\prod_{t=0}^{p-1}b_t^{\SFcount{\SFinitial{\Sf{D}}[\SFposition{\SFinitial{\Sf{D}}}{s}[i]+1,k-1]}{t}},\;s\in\intintuptoex{p}\\
d_s&\ce\sum_{i=0}^{\SFcount{\SFperiodic{\Sf{D}}}{s}-1}p^{\SFposition{\SFperiodic{\Sf{D}}}{s}[i]}\prod_{t=0}^{p-1}b_t^{\SFcount{\SFperiodic{\Sf{D}}[\SFposition{\SFperiodic{\Sf{D}}}{s}[i]+1,\ell-1]}{t}},\;s\in\intintuptoex{p}\\
c&\ce\FRFlinB{\FFf{F}}{\SFinitial{\Sf{D}}}\\
d&\ce\FRFlinB{\FFf{F}}{\SFperiodic{\Sf{D}}}.
\end{align}
If $a_0,\ldots,a_{r-1},a_{r+1},\ldots,a_{p-1},b_0,\ldots,b_{p-1}$ are fixed, then there is a unique choice for $a_r$ such that $\FFFDT(\FFf{F})[n]=\Sf{D}$. This choice is given by
\begin{align}
\label{EInvSystA}
a_r=\left(cn-\sum_{s\in\intintuptoex{p}\setminus\set{r}}a_s\left(\frac{d_sp^k}{p^\ell-d}-c_s\right)\right)\Bigg/\left(\frac{d_rp^k}{p^\ell-d}-c_r\right).
\end{align}
\end{corollary}

\begin{proof}
By rearranging the formula for $n$ in Corollary~\ref{CPerFormula} it follows that if $a_r$ can be chosen such that $\FFFDT(\FFf{F})[n]=\Sf{D}$ holds, $a_r$ must satisfy Eqn.~(\ref{EInvSystA}). We are thus left to show that $a_r$ as given in Eqn.~(\ref{EInvSystA}) is an element of $-rb_r+p\Z_p$ and thus defines a valid linear-polynomial $p$-adic system in weak canonical form (within $\FFf{F}$). This would follow in particular if we could demonstrate that
\begin{align}
\label{EInvSystB}
\Gcd{p,\left(\frac{d_rp^k}{p^\ell-d}-c_r\right)\modulo p}=1
\end{align}
and
\begin{align}
\label{EInvSystC}
cn-\sum_{s\in\intintuptoex{p}\setminus\set{r}}a_s\left(\frac{d_sp^k}{p^\ell-d}-c_s\right)\in-rb_r\left(\frac{d_rp^k}{p^\ell-d}-c_r\right)+p\Z_p
\end{align}
which is what we will do.

We begin by showing Eqn.~(\ref{EInvSystB}) and use the fact $\Sf{D}[0]=r$ and Lemma~\ref{LLinPolySuit}~(2) to compute
\begin{align}
&\Gcd{p,\left(\frac{d_rp^k}{p^\ell-d}-c_r\right)\modulo p}\\
&\quad=\Gcd{p,\left(k=0\;?\;\frac{d_r}{p^\ell-d}:-c_r\right)\modulo p}\\
&\quad=\Gcd{p,\left(k=0\;?\;\frac{\prod_{t=0}^{p-1}b_t^{\SFcount{\SFperiodic{\Sf{D}}[1,\ell-1]}{t}}}{p^\ell-\prod_{t=0}^{p-1}b_t^{\SFcount{\SFperiodic{\Sf{D}}}{t}}}:-\prod_{t=0}^{p-1}b_t^{\SFcount{\SFinitial{\Sf{D}}[1,k-1]}{t}}\right)\modulo p}\\
&\quad=1.
\end{align}

To show Eqn.~(\ref{EInvSystC}) let $m$ be the unique element of $r+p\Z_p$ satisfying $\FFFDT(\FFf{G})[m]=\Sf{D}$ (cf. Lemma~\ref{LInfBlDTCompl}) where $\FFf{G}\in\SoSystems{p}(\FFPweakcanonicalform,\FFPlinearpolynomial)$ is given by
\begin{align}
\FFf{G}\ce(a_0+b_0x,\ldots,a_{r-1}+b_{r-1}x,-rb_r+b_rx,a_{r+1}+b_{r+1}x,\ldots,a_{p-1}+b_{p-1}x).
\end{align}
Then by Corollary~\ref{CPerFormula} we get
\begin{align}
&m=\frac{1}{c}\left(-rb_r\left(\frac{d_rp^k}{p^\ell-d}-c_r\right)+\sum_{s\in\intintuptoex{p}\setminus\set{r}}a_s\left(\frac{d_sp^k}{p^\ell-d}-c_s\right)\right)
\end{align}
which implies
\begin{align}
\left(cn-\sum_{s\in\intintuptoex{p}\setminus\set{r}}a_s\left(\frac{d_sp^k}{p^\ell-d}-c_s\right)\right)\modulo p
&=\left(cm-\sum_{s\in\intintuptoex{p}\setminus\set{r}}a_s\left(\frac{d_sp^k}{p^\ell-d}-c_s\right)\right)\modulo p\\
&=\left(-rb_r\left(\frac{d_rp^k}{p^\ell-d}-c_r\right)\right)\modulo p.
\end{align}
\end{proof}

\noindent
The above corollary indeed characterizes all weak canonical linear-polynomial $p$-adic systems for which a given $p$-adic integer $n$ has a given ultimately periodic digit expansion $\Sf{D}$: simply pick all coefficients at random apart from the constant coefficient of $\FFf{F}[n\modulo p]$ which is fixed uniquely by the other coefficients and can be computed by the formula given in the corollary. A further corollary of Theorem~\ref{TLinDirForm} which will answer a similar question is proven next.

\begin{corollary}
\label{CInvSystNum}
Let $2\leq p\in\N$, $r\in\intintuptoex{p}$, $\Sf{D}\in\CoSequences(\SPboundedby{\intintuptoex{p}},\SPultimatelyperiodic)$, and $\FFf{F}=(a_0+b_0x,\ldots,a_{p-1}+b_{p-1}x)\in\SoSystems{p}(\FFPweakcanonicalform,\FFPlinearpolynomial)$. Furthermore, let $k,\ell,c_0,\ldots,c_{p-1},d_0,\ldots,d_{p-1},c,d$ as in Corollary~\ref{CInvSyst} and
\begin{align}
K&\ce(\SFcount{\Sf{D}}{r}\neq0\;?\;\SFposition{\Sf{D}}{r}[0]+1:-\infty)\\
A_1&\ce(-rb_r)\modulo p\\
B_1&\ce\frac{1}{c}\left(A_1\left(\frac{d_rp^k}{p^\ell-d}-c_r\right)
+\sum_{s\in\intintuptoex{p}\setminus\set{r}}a_s\left(\frac{d_sp^k}{p^\ell-d}-c_s\right)\right)\\
C_1&\ce\left(\SFcount{\Sf{D}}{r}\neq0\;?\;\frac{1}{cp^{K-1}}\left(\frac{d_rp^k}{p^\ell-d}-c_r\right):1\right)\\
A_2&\ce\left(\SFcount{\Sf{D}}{r}\neq0\;?\;A_1-\frac{B_1-B_1\modulo p^K}{C_1p^{K-1}}:A_1\right)\\
B_2&\ce\left(\SFcount{\Sf{D}}{r}\neq0\;?\;B_1\modulo p^K,B_1\right)\\
C_2&\ce\frac{1}{C_1}.
\end{align}
If $a_0,\ldots,a_{r-1},a_{r+1},\ldots,a_{p-1},b_0,\ldots,b_{p-1}$ are fixed, then

\begin{theoremtable}
\theoremitem{(1)}&$\set{(a_r,n)\in(\Z_p)^2\mid\FFFDT(\FFf{F})[n]=\Sf{D}}=\set{\left(A_1+mp,B_1+mC_1p^K\right)\mid m\in\Z_p}$\tabularnewline
&$\phantom{\set{(a_r,n)\in(\Z_p)^2\mid\FFFDT(\FFf{F})[n]=\Sf{D}}}=\set{\left(A_2+mC_2p,B_2+mp^K\right)\mid m\in\Z_p}$\tabularnewline
&where $A_1,B_1,C_1,A_2,B_2,C_2\in\Z_p$, $A_1\in\intintuptoex{p}$, $B_2\in\intintuptoex{p^K}$ if $\SFcount{\Sf{D}}{r}\neq0$, and \tabularnewline
&$\Gcd{p,C_1\modulo p}=\Gcd{p,C_2\modulo p}=1$\tabularnewline
\theoremitem{(2)}& If $\FFPlinearpolynomialcoefficients{\Q\cap\Z_p}(\FFf{F})$ and $K_1,L_1,K_2,L_2$ are the denominators of $B_1,C_1,A_2,C_2$ respectively then,\tabularnewline
&\quad$\set{(a_r,n)\in\Z^2\mid\FFFDT(\FFf{F})[n]=\Sf{D}}\neq\emptyset\Leftrightarrow K_1\mid L_1$\tabularnewline
&\quad$\phantom{\set{(a_r,n)\in\Z^2\mid\FFFDT(\FFf{F})[n]=\Sf{D}}\neq\emptyset}\Leftrightarrow K_2\mid L_2\land B_2\in\Z$\tabularnewline
&\quad In this case:\tabularnewline
&\quad$\set{(a_r,n)\in\Z^2\mid\FFFDT(\FFf{F})[n]=\Sf{D}}$\tabularnewline
&\qquad$=\set{\left(A_1+(S_1B_1+m)L_1p,B_1+(S_1B_1+m)L_1C_1p^K\right)\mid m\in\Z}$\tabularnewline
&\qquad$=\set{\left(A_2+(S_2A_2+m)L_2C_2p,B_2+(S_2A_2+m)L_2p^K\right)\mid m\in\Z}$\tabularnewline
&\quad where $S_1,S_2\in\Z$ such that\tabularnewline
&\quad$L_1R_1-L_1C_1p^KS_1=1$\tabularnewline
&\quad$L_2R_2-L_2C_2pS_2=1$\tabularnewline
&\quad for some $R_1,R_2\in\Z$ (find $S_1,S_2$ with extended Euclidean algorithm).
\end{theoremtable}
\end{corollary}

\begin{proof}
$ $\\
\proofitem{(1)}%
Clearly, $A_1,B_1,A_2,B_2\in\Z_p$, $A_1\in\intintuptoex{p}$, and $B_2\in\intintuptoex{p^K}$ if $\SFcount{\Sf{D}}{r}\neq0$. In addition, if $\SFcount{\Sf{D}}{r}=0$, then $C_1,C_2=1\in\Z_p$ and $\Gcd{p,C_1\modulo p}=\Gcd{p,C_2\modulo p}=\Gcd{p,1}=1$. If, however, $\SFcount{\Sf{D}}{r}\neq0$ then either $\SFposition{\Sf{D}}{r}[0]<\abs{\SFinitial{\Sf{D}}}$ or $\SFposition{\Sf{D}}{r}[0]\geq\abs{\SFinitial{\Sf{D}}}$. In the first case we get $K-1=\SFposition{\Sf{D}}{r}[0]=\SFposition{\SFinitial{\Sf{D}}}{r}[0]$,
\begin{align}
C_1&=\frac{1}{cp^{K-1}}\left(\frac{d_rp^k}{p^\ell-d}-c_r\right)\\
&=\frac{p^{\abs{\SFinitial{\Sf{D}}}}}{p^{\SFposition{\Sf{D}}{r}[0]}}\frac{d_r}{c(p^\ell-d)}-\frac{1}{c}\prod_{t=0}^{p-1}b_t^{\SFcount{\SFinitial{\Sf{D}}[\SFposition{\SFinitial{\Sf{D}}}{r}[0]+1,k-1]}{t}}-\vphantom{a}\\
\nonumber
&\phantom{\vphantom{a}=\vphantom{a}}\frac{1}{c}\sum_{i=1}^{\SFcount{\SFinitial{\Sf{D}}}{r}-1}\frac{p^{\SFposition{\SFinitial{\Sf{D}}}{r}[i]}}{p^{\SFposition{\SFinitial{\Sf{D}}}{r}[0]}}\prod_{t=0}^{p-1}b_t^{\SFcount{\SFinitial{\Sf{D}}[\SFposition{\SFinitial{\Sf{D}}}{r}[i]+1,k-1]}{t}}\\
&\in\Z_p
\end{align}
and
\begin{align}
\Gcd{p,C_1\modulo p}&=\Gcd{p,\left(\frac{1}{c}\prod_{t=0}^{p-1}b_t^{\SFcount{\SFinitial{\Sf{D}}[\SFposition{\SFinitial{\Sf{D}}}{r}[0]+1,k-1]}{t}}\right)\modulo p}\\
&=1
\end{align}
by Lemma~\ref{LLinPolySuit}~(2). In the second case we get $K-1=\SFposition{\Sf{D}}{r}[0]=\abs{\SFinitial{\Sf{D}}}+\SFposition{\SFperiodic{\Sf{D}}}{r}[0]$,
\begin{align}
C_1&=\frac{1}{cp^{K-1}}\left(\frac{d_rp^k}{p^\ell-d}-c_r\right)\\
&=\frac{1}{c(p^\ell-d)}\prod_{t=0}^{p-1}b_t^{\SFcount{\SFperiodic{\Sf{D}}[\SFposition{\SFperiodic{\Sf{D}}}{r}[0]+1,\ell-1]}{t}}+\vphantom{a}\\
\nonumber
&\phantom{\vphantom{a}=\vphantom{a}}\frac{1}{c(p^\ell-d)}\sum_{i=1}^{\SFcount{\SFperiodic{\Sf{D}}}{r}-1}\frac{p^{\SFposition{\SFperiodic{\Sf{D}}}{r}[i]}}{p^{\SFposition{\SFperiodic{\Sf{D}}}{r}[0]}}\prod_{t=0}^{p-1}b_t^{\SFcount{\SFperiodic{\Sf{D}}[\SFposition{\SFperiodic{\Sf{D}}}{r}[i]+1,\ell-1]}{t}}\\
&\in\Z_p
\end{align}
and
\begin{align}
\Gcd{p,C_1\modulo p}&=\Gcd{p,\left(\frac{1}{c(p^\ell-d)}\prod_{t=0}^{p-1}b_t^{\SFcount{\SFperiodic{\Sf{D}}[\SFposition{\SFperiodic{\Sf{D}}}{r}[0]+1,\ell-1]}{t}}\right)\modulo p}\\
&=1
\end{align}
again by Lemma~\ref{LLinPolySuit}~(2). We thus proved that, in any case, $C_1,C_2\in\Z_p$ and $\Gcd{p,C_1\modulo p}=\Gcd{p,C_2\modulo p}=1$.

If $a_r=A_1+mp=(-rb_r)\modulo p+mp$ for some $m\in\Z_p$ (which exactly covers all possible candidates for $a_r$ for $\FFf{F}$ to be in weak canonical form), then by Corollary~\ref{CPerFormula} there is a unique $n\in\Z_p$ satisfying $\FFFDT(\FFf{F})[n]=\Sf{D}$ which is given by
\begin{align}
n&=\left(\frac{\FRFlinA{\FFf{F}}{\SFperiodic{\Sf{D}}}}{p^{\abs{\SFperiodic{\Sf{D}}}}-\FRFlinB{\FFf{F}}{\SFperiodic{\Sf{D}}}}p^{\abs{\SFinitial{\Sf{D}}}}-\FRFlinA{\FFf{F}}{\SFinitial{\Sf{D}}}\right)\frac{1}{\FRFlinB{\FFf{F}}{\SFinitial{\Sf{D}}}}\\
&=\frac{1}{c}\sum_{r=0}^{p-1}a_r\left(\frac{d_rp^k}{p^\ell-d}-c_r\right)\\
&=\frac{1}{c}\left((A_1+mp)\left(\frac{d_rp^k}{p^\ell-d}-c_r\right)
+\sum_{s\in\intintuptoex{p}\setminus\set{r}}a_s\left(\frac{d_sp^k}{p^\ell-d}-c_s\right)\right)\\
&=B_1+mp\frac{1}{c}\left(\frac{d_rp^k}{p^\ell-d}-c_r\right)\\
&=B_1+\left(\SFcount{\Sf{D}}{r}\neq0\;?\;mp^K\frac{1}{cp^{K-1}}\left(\frac{d_rp^k}{p^\ell-d}-c_r\right):0\right)\\
&=B_1+mC_1p^K.
\end{align}
Thus, $\set{(a_r,n)\in(\Z_p)^2\mid\FFFDT(\FFf{F})[n]=\Sf{D}}=\set{\left(A_1+mp,B_1+mC_1p^K\right)\mid m\in\Z_p}$.

If, however, $a_r=A_2+mC_2p$ for some $m\in\Z_p$ (which again exactly covers all possible candidates for $a_r$ for $\FFf{F}$ to be in weak canonical form since $\Gcd{p,C_2\modulo p}=1$), then again by Corollary~\ref{CPerFormula} there is a unique $n\in\Z_p$ satisfying $\FFFDT(\FFf{F})[n]=\Sf{D}$ which is given by
\begin{align}
n&=\frac{1}{c}\left((A_2+mC_2p)\left(\frac{d_rp^k}{p^\ell-d}-c_r\right)
+\sum_{s\in\intintuptoex{p}\setminus\set{r}}a_s\left(\frac{d_sp^k}{p^\ell-d}-c_s\right)\right)\\
&=B_1-\frac{1}{c}\left(\SFcount{\Sf{D}}{r}\neq0\;?\;\left(\frac{B_1-B_1\modulo p^K}{C_1p^{K-1}}-\frac{mp}{C_1}\right)\left(C_1cp^{K-1}\right):0\right)\\
&=B_2+mp^K.
\end{align}
Thus, $\set{(a_r,n)\in(\Z_p)^2\mid\FFFDT(\FFf{F})[n]=\Sf{D}}=\set{\left(A_2+mC_2p,B_2+mp^K\right)\mid m\in\Z_p}$.

\noindent
\proofitem{(2)}%
First we observe that $c_0,\ldots,c_{p-1},d_0,\ldots,d_{p-1},c,d,A_1,B_1,C_1,K_1,L_1,A_2,B_2,C_2,K_2,L_2\in\Q\cap\Z_p$ by (1) and their respective definitions. Moreover, from (1) it follows that
\begin{align}
&\set{(a_r,n)\in\Z^2\mid\FFFDT(\FFf{F})[n]=\Sf{D}}\\
&\quad=\set{\left(A_1+mp,B_1+mC_1p^K\right)\mid m\in\Z_p}\cap\Z^2\\
&\quad=\set{\left(A_1+mp,B_1+mC_1p^K\right)\mid m\in\Z\land B_1+mC_1p^K\in\Z}\\
&\quad=\set{\left(A_1+yp,B_1+yC_1p^K\right)\mid y\in\Z\land\ex x\in\Z:Lx-LC_1p^Ky=LB_1}
\end{align}
where $L\ce\Lcm{K_1,L_1}$. B\'{e}zout's Lemma implies that the equation $Lx-LC_1p^Ky=LB_1$ has a solution $x,y\in\Z$ if and only if $\Gcd{L,LC_1p^K}$ divides $LB_1$, which in return is true if and only if $\Gcd{L,LC_1p^K}=1$, or equivalently $L=L_1$, respectively $K_1\mid L_1$. In this case the set of all solutions is given by
\begin{align}
&\set{(x,y)\in\Z^2\mid Lx-LC_1p^Ky=LB_1}=\set{((R_1B_1+mC_1p^K)L_1,(S_1B_1+m)L_1)\mid m\in\Z}
\end{align}
where $R_1,S_1\in\Z$ such that $L_1R_1-L_1C_1p^KS_1=1$. Consequently,
\begin{align}
&\set{(a_r,n)\in\Z^2\mid\FFFDT(\FFf{F})[n]=\Sf{D}}\\
&\quad=\set{\left(A_1+(S_1B_1+m)L_1p,B_1+(S_1B_1+m)L_1C_1p^K\right)\mid m\in\Z}.
\end{align}

Analogously we get,
\begin{align}
&\set{(a_r,n)\in\Z^2\mid\FFFDT(\FFf{F})[n]=\Sf{D}}\\
&\quad=\set{\left(A_2+mC_2p,B_2+mp^K\right)\mid m\in\Z_p}\cap\Z^2\\
&\quad=\set{\left(A_2+mC_2p,B_2+mp^K\right)\mid m\in\Z\land A_2+mC_2p\in\Z\land B_2\in\Z}\\
&\quad=\set{\left(A_2+yC_2p,B_2+yp^K\right)\mid y\in\Z\land\ex x\in\Z:Lx-LC_2py=LA_2\land B_2\in\Z}
\end{align}
where $L\ce\Lcm{K_2,L_2}$. As before B\'{e}zout's Lemma implies that the equation $Lx-LC_2py=LA_2$ has a solution $x,y\in\Z$ if and only if $\Gcd{L,LC_2p}$ divides $LA_2$, which in return is true if and only if $\Gcd{L,LC_2p}=1$, or equivalently $L=L_2$, respectively $K_2\mid L_2$. In this case the set of all solutions is given by
\begin{align}
&\set{(x,y)\in\Z^2\mid Lx-LC_2py=LA_2}=\set{((R_2A_2+mC_2p)L_2,(S_2A_2+m)L_2)\mid m\in\Z}
\end{align}
where $R_2,S_2\in\Z$ such that $L_2R_2-L_2C_2pS_2=1$. Consequently,
\begin{align}
&\set{(a_r,n)\in\Z^2\mid\FFFDT(\FFf{F})[n]=\Sf{D}}\\
&\quad=\set{\left(A_2+(S_2A_2+m)L_2C_2p,B_2+(S_2A_2+m)L_2p^K\right)\mid m\in\Z}.
\end{align}
\end{proof}

\noindent
As an example consider $\FFf{F}\ce\FFf{F}_C=(x,3x+1)$, $r=1$, and $\Sf{D}=(\Sf{P})^\infty\in\CoSequences(\SPboundedby{\intintuptoex{2}},\SPperiodic)$ and let
\begin{align}
U_0&\ce\sum_{i=0}^{\SFcount{\Sf{P}}{0}-1}2^{\SFposition{\Sf{P}}{0}[i]}3^{\SFcount{\Sf{P}}{1}-\SFposition{\Sf{P}}{0}[i]+i}\\
U_1&\ce\sum_{i=0}^{\SFcount{\Sf{P}}{1}-1}2^{\SFposition{\Sf{P}}{1}[i]}3^{\SFcount{\Sf{P}}{1}-i-1}\\
V&\ce2^{\abs{\Sf{P}}}-3^{\SFcount{\Sf{P}}{1}}.
\end{align}
Then $a_0=0$, $b_0=1$, and $b_1=3$. Furthermore, if $\Sf{D}\neq(0)^\infty$, simplifying all expressions in Corollary~\ref{CInvSystNum} then yields,
\begin{align}
K&=\SFposition{\Sf{P}}{1}[0]+1,&
A_1&=1,&
B_1&=\frac{U_1}{V},&
C_1&=\frac{U_1}{V}\frac{1}{2^{K-1}},&
L_1&=\frac{\abs{V}}{\Gcd{U_1,V}},
\end{align}
and $R_1,S_1\in\Z$ such that
\begin{align}
\frac{V}{\Gcd{U_1,V}}R_1-\frac{2U_1}{\Gcd{U_1,V}}S_1&=\Sgn{V}.
\end{align}
Furthermore,
\begin{align}
&\set{(a_1,n)\in\Z^2\mid\FFFDT(\FFf{F})[n]=\Sf{D}}\\
&\quad=\set{\left(A_1+(S_1B_1+m)L_12,B_1+(S_1B_1+m)L_1C_12^K\right)\mid m\in\Z}\\
&\quad=\set{\left(1+\left(S_1\frac{U_1}{V}+m\right)\frac{2\abs{V}}{\Gcd{U_1,V}},\frac{U_1}{V}+\left(S_1\frac{U_1}{V}+m\right)\frac{2\abs{V}}{\Gcd{U_1,V}}\frac{U_1}{V}\right)\mid m\in\Z}\\
&\quad=\set{\left(1,\frac{U_1}{V}\right)\left(1+\Sgn{V}\left(S_1\frac{2U_1}{\Gcd{U_1,V}}+2m\frac{V}{\Gcd{U_1,V}}\right)\right)\mid m\in\Z}\\
&\quad=\set{\left(1,\frac{U_1}{V}\right)(2m+R_1)\frac{\abs{V}}{\Gcd{U_1,V}}\mid m\in\Z}\\
&\quad=\set{\left(1,\frac{U_1}{V}\right)(2m+1)\frac{V}{\Gcd{U_1,V}}\mid m\in\Z}\\
\label{EInvA}
&\quad=\set{\left(\frac{V}{\Gcd{U_1,V}},\frac{U_1}{\Gcd{U_1,V}}\right)(2m+1)\mid m\in\Z}.
\end{align}

Analogously, if $r=0$ and $\Sf{D}\neq(1)^\infty$, then $a_1=1$, $b_0=1$, $b_1=3$,
\begin{align}
K&=\SFposition{\Sf{P}}{0}[0]+1,&
A_1&=0,&
B_1&=\frac{U_1}{V},&
C_1&=\frac{U_0}{V}\frac{1}{2^{K-1}},&
L_1&=\frac{\abs{V}}{\Gcd{U_0,V}},
\end{align}
and $R_1,S_1\in\Z$ such that
\begin{align}
\frac{V}{\Gcd{U_0,V}}R_1-\frac{2U_0}{\Gcd{U_0,V}}S_1&=\Sgn{V}.
\end{align}
In addition,
\begin{align}
\label{EInvB}
\set{(a_0,n)\in\Z^2\mid\FFFDT(\FFf{F})[n]=\Sf{D}}
&=\set{\left(\frac{V}{\Gcd{U_0,V}},\frac{U_0}{\Gcd{U_0,V}}\right)(2m+1)-1\mid m\in\Z}.
\end{align}
Note that $U_0-U_1=V$ and $\Gcd{U_0,V}=\Gcd{U_1,V}=\Gcd{U_0,U_1}$. A consequence of both Eqn.~(\ref{EInvA}) and Eqn.~(\ref{EInvB}) is that the Collatz conjecture is equivalent to
\begin{align}
\label{ECollatzDiv}
&\N\subseteq\FFFSoUPP(\FFf{F})\quad\land\quad\ex m\in\N:mV=U_0\Rightarrow\SFperiodic{\Sf{D}}\in\set{(0),(0,1),(1,0)}
\intertext{and respectively}
&\N\subseteq\FFFSoUPP(\FFf{F})\quad\land\quad\ex m\in\N:mV=U_1\Rightarrow\SFperiodic{\Sf{D}}\in\set{(0,1),(1,0)}.
\end{align}

\noindent
Variants of this result can be found in several publications on the original Collatz conjecture, such as \cite{Seifert:1988}.

\myparagraphtoc{When do all rational numbers have ultimately periodic digit expansions? Conjectures.}
In Corollary~\ref{CPerFormula} we proved that all ultimately periodic points of $(\Q\cap\Z_p)$-linear-polynomial $p$-adic systems are rational numbers. The converse question whether all rational numbers have ultimately periodic digit expansions with respect to a given $(\Q\cap\Z_p)$-linear-polynomial $p$-adic system $\FFf{F}$, i.e. whether $\FFf{F}$ is ultimately periodic on $\Q\cap\Z_p$, is incredibly difficult in general and sits right at the heart of the Collatz conjecture. The general framework of $p$-adic systems might help to shed some light on the true nature of the underlying difficulty, as it allows to discuss the question in a broader context. In this context we are able to formulate conjectures of increasing generality which are backed by computer experiments to a varying degree. Several generalizations of the original Collatz transformation which have been considered in the literature (cf. e.g. \cite{Chamberland:2003,Steiner:1981a,Steiner:1981b,BelagaMignotte:1998}) are covered by these general conjectures.

We begin by revisiting the Collatz conjecture itself.

\begin{conjecture}[Collatz]
\label{CoPeriodsA}
Let $\FFf{F}_C\ce(x,3x+1)\in\SoSystems{2}(\FFPlinearpolynomialcoefficients{\Z})$. Then,
\begin{align}
\fa n\in\N:\ex k\in\N:\FFFDT(\FFf{F}_C)[n][k]=1.
\end{align}
\end{conjecture}

\noindent
A slightly stronger version is given by

\begin{conjecture}
\label{CoPeriodsB}
Let $\FFf{F}_C\ce(x,3x+1)\in\SoSystems{2}(\FFPlinearpolynomialcoefficients{\Z})$. Then,

\begin{theoremtable}
\theoremitem{(1)}&$\FFPultimatelyperiodicon{\Q\cap\Z_2}(\FFf{F}_C)$\tabularnewline
\theoremitem{(2)}&$\fa n\in\N:(\SPultimatelyperiodic(\FFFDT(\FFf{F}_C)[n])\Rightarrow\ex k\in\N:\FFFDT(\FFf{F}_C)[n][k]=1).$
\end{theoremtable}
\end{conjecture}

\noindent
Considering both parts of the conjecture separately, there are several ways to generalize. The seemingly most arbitrary element is the definition of $\FFf{F}_C$. Why should $(x,3x+1)$ be in any way special among linear-polynomial $2$-adic systems? Experiments show that it probably isn't if we weaken (2).

\begin{conjecture}
\label{CoPeriodsC}
Let $2\leq p\in\N$, $\FFf{F}\in\SoSystems{p}(\FFPlinearpolynomialcoefficients{\Z})$, and $B\in\Z$ the product of all leading coefficients of the polynomials $\FFf{F}[r](x)$, $r\in\intintuptoex{p}$. Then,

\begin{theoremtable}
\theoremitem{(1)}&$\FFPultimatelyperiodicon{\Q\cap\Z_p}(\FFf{F})\quad\Leftrightarrow\quad\abs{B}<p^p$\tabularnewline
\theoremitem{(2)}&$\abs{\set{\SFperiodic{\FFFDT(\FFf{F})[n]}\mid n\in\Z}}<\infty.$
\end{theoremtable}
\end{conjecture}

\noindent
For this conjecture numerous computer experiments have been performed by the author and while it may take quite a long time for digit expansions to become periodic (especially for larger $p$), they all did eventually. One of the experiments that were done was the computation of the maximal lengths of the initial and periodic parts of the expansions of all integers in $\intintuptoin{1000}$ for all $\Z$-linear-polynomial $p$-adic systems satisfying $\abs{B}<p^p$ where all constant coefficients are equal to $0$, i.e. the computation of
\begin{align}
m_I(p)&\ce\max\setl{\abs{\SFinitial{\FFFDT(\FFf{F})[n]}}\mid\FFf{F}=(b_0x,\ldots,b_{p-1}x)\in\SoSystems{p}(\FFPlinearpolynomialcoefficients{\Z})}\\
\nonumber
&\phantom{\vphantom{a}\ce\max\setl{\abs{\SFinitial{\FFFDT(\FFf{F})[n]}}\mid\vphantom{a}}}
\abs{b_0\ldots b_{p-1}}<p^p\\
\nonumber
&\phantom{\vphantom{a}\ce\max\setl{\abs{\SFinitial{\FFFDT(\FFf{F})[n]}}\mid\vphantom{a}}}
\;\setr{n\in\intintuptoin{1000}}\\
m_P(p)&\ce\max\setl{\abs{\SFperiodic{\FFFDT(\FFf{F})[n]}}\mid\FFf{F}=(b_0x,\ldots,b_{p-1}x)\in\SoSystems{p}(\FFPlinearpolynomialcoefficients{\Z})}\\
\nonumber
&\phantom{\vphantom{a}\ce\max\setl{\abs{\SFinitial{\FFFDT(\FFf{F})[n]}}\mid\vphantom{a}}}
\abs{b_0\ldots b_{p-1}}<p^p\\
\nonumber
&\phantom{\vphantom{a}\ce\max\setl{\abs{\SFinitial{\FFFDT(\FFf{F})[n]}}\mid\vphantom{a}}}
\;\setr{n\in\intintuptoin{1000}}.
\end{align}
For $p\in\set{2,3,4}$ these values and the corresponding $p$-adic systems and starting values generating them are
\begin{align}
&\mathrlap{m_I(2)}\phantom{m_P(4)}=160
&&\text{for }\FFf{F}=(-3x,-x),&&n=284\\
&\mathrlap{m_P(2)}\phantom{m_P(4)}=19
&&\text{for }\FFf{F}=(x,-3x),&&n=609\\
&\mathrlap{m_I(3)}\phantom{m_P(4)}=52401
&&\text{for }\FFf{F}=(x,-26x,-x),&&n=796\\
&\mathrlap{m_P(3)}\phantom{m_P(4)}=3905
&&\text{for }\FFf{F}=(-13x,-x,2x),&&n=608\\
&\mathrlap{m_I(4)}\phantom{m_P(4)}=18481661
&&\text{for }\FFf{F}=(5x,-x,-51x,-x),&&n=818\\
&\mathrlap{m_P(4)}\phantom{m_P(4)}=3291996
&&\text{for }\FFf{F}=(-51x,-x,-5x,x),&&n=416.
\end{align}
Note that the conjecture states in particular that the constant coefficients of the polynomials have no influence on the question of ultimate periodicity on $\Q\cap\Z_p$. On  this aspect of the problem some results could be achieved which are presented further down in this section.

A further generalization, the status of which is less clear, considers more general coefficients for the linear polynomials defining the $p$-adic system.

\begin{conjecture}
\label{CoPeriodsD}
Let $2\leq p\in\N$, $\FFf{F}\in\SoSystems{p}(\FFPlinearpolynomialcoefficients{\Q\cap\Z_p})$, and $B\in\Q\cap\Z_p$ the product of all leading coefficients of the polynomials $\FFf{F}[r](x)$, $r\in\intintuptoex{p}$. Then,

\begin{theoremtable}
\theoremitem{(1)}&$\FFPultimatelyperiodicon{\Q\cap\Z_p}(\FFf{F})\quad\Leftrightarrow\quad B\in\Z\land\abs{B}<p^p$\tabularnewline
\theoremitem{(2)}&$\abs{\set{\SFperiodic{\FFFDT(\FFf{F})[n]}\mid n\in\Z}}<\infty.$
\end{theoremtable}
\end{conjecture}

\noindent
One of the observations often pointed out as being a first hint that the Collatz conjecture is indeed very hard to prove, is the fact that $27$ takes $70$ steps to reach $1$ under $\FFf{F}_C=(x,3x+1)$. In the context of the conjecture above we can do ``much worse'' as the following examples show. Let
\begin{align}
\FFf{F}_1&\ce(3/11x+2,-11x+1)\\
\FFf{F}_2&\ce(1/7x+2,21x+1)\\
\FFf{F}_3&\ce(1/5x-4,-15x+3).
\end{align}
Then,
\begin{align}
\abs{\SFinitial{\FFFDT(\FFf{F}_1)[27]}}&=816179238
&
\abs{\SFperiodic{\FFFDT(\FFf{F}_1)[27]}}&=5890445\\
\abs{\SFinitial{\FFFDT(\FFf{F}_2)[27]}}&=312815429
&
\abs{\SFperiodic{\FFFDT(\FFf{F}_2)[27]}}&=22014805908\\
\abs{\SFinitial{\FFFDT(\FFf{F}_3)[27]}}&=18966150
&
\abs{\SFperiodic{\FFFDT(\FFf{F}_3)[27]}}&=122858925930.
\end{align}
Furthermore,
\begin{align}
\SFperiodic{\FFFST(\FFf{F}_1)[27]}[0]&=\frac{292064}{11^3}
&
d_I(\FFf{F}_1)&=11^{25086}
&
d_P(\FFf{F}_1)&=11^{2912}\\
\SFperiodic{\FFFST(\FFf{F}_2)[27]}[0]&=\frac{197828}{7^3}
&
d_I(\FFf{F}_2)&=7^{9154}
&
d_P(\FFf{F}_2)&=7^{170632}\\
\SFperiodic{\FFFST(\FFf{F}_3)[27]}[0]&=\frac{101772}{5^6}
&
d_I(\FFf{F}_3)&=5^{3008}
&
d_P(\FFf{F}_3)&=5^{320048}
\end{align}
where $d_I(\FFf{F}_i)$ and $d_P(\FFf{F}_i)$ denote the largest denominators occurring in the initial and periodic parts of the sequence $\FFFST(\FFf{F}_i)[27]$ for $i\in\set{1,2,3}$. Another $2$-adic system that was considered in computer experiments is
\begin{align}
\FFf{F}_4&\ce(21/5x,5/7x+1)
\end{align}
which satisfies $\abs{\SFinitial{\FFFDT(\FFf{F}_4)[27]}}>10^{10}$. The denominator of $\FFFDT(\FFf{F}_4)^{\left(10^{10}\right)}(n)$ is $5^{12806}7^{119930}\approx4.1477678\cdot10^{110303}$. Its numerator is approximately $8.2260293\cdot10^{110305}$ which makes the entire fraction approximately equal to $0.0050422$. It appears hard to even guess whether $\FFFDT(\FFf{F}_4)[27]$ is ultimately periodic. Figure~\ref{FOrbitsA} and Figure~\ref{FOrbitsB} below show the developments of the magnitudes of the denominators in the sequences $\FFFST(\FFf{F}_i)[27]$, $i\in\set{1,2,3,4}$. While for $\FFf{F}_1,\FFf{F}_2,\FFf{F}_3$ these magnitudes can increase and decrease at any time, in the case of $\FFf{F}_4$ there is a tradeoff between the $5$-adic and $7$-adic valuations of the denominators of successive entries of the sequence $\FFFST(\FFf{F}_4)[27]$. If the $5$- and $7$-adic valuations of the denominator are both positive and $\FFf{F}_4[0]$ is applied, then the $5$-adic valuation of the denominator increases by $1$ while the $7$-adic valuation decreases by $1$. If, however, $\FFf{F}_4[1]$ is applied, it is the other way around. This means that the sum of the $5$- and $7$-adic valuations of the denominator can only ever change if one of the two is equal to $0$ which explains the shape of the graph showing this sum in Figure~\ref{FOrbitsB}. The consequence of this observation is that on the one hand the denominators of the sequence $\FFFST(\FFf{F}_4)[27]$ get large, which makes it unlikely that a period occurs, but on the other hand the sums of the $5$- and $7$-adic valuations of the denominator stay constant for a large number of steps, which increases the chances for the occurrence of a period. Which of these effects is stronger in the (infinitely) long run, remains an open question.

\begin{figure}[H]
\centering
\begin{overpic}[width=0.45\textwidth]{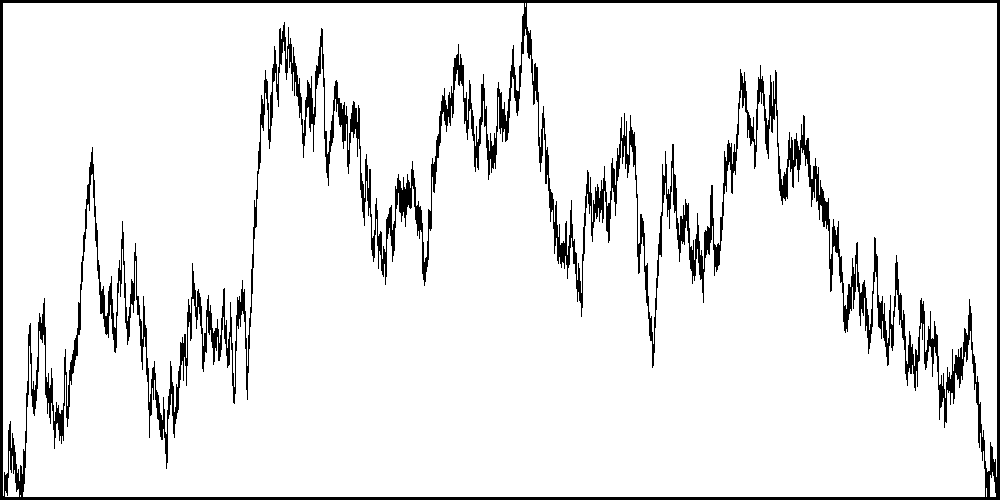}
\put(100.00,-3.00){\tiny$\mathllap{816179238}$}
\put(0.00,51.00){\tiny$25086$}
\end{overpic}
\qquad
\begin{overpic}[width=0.45\textwidth]{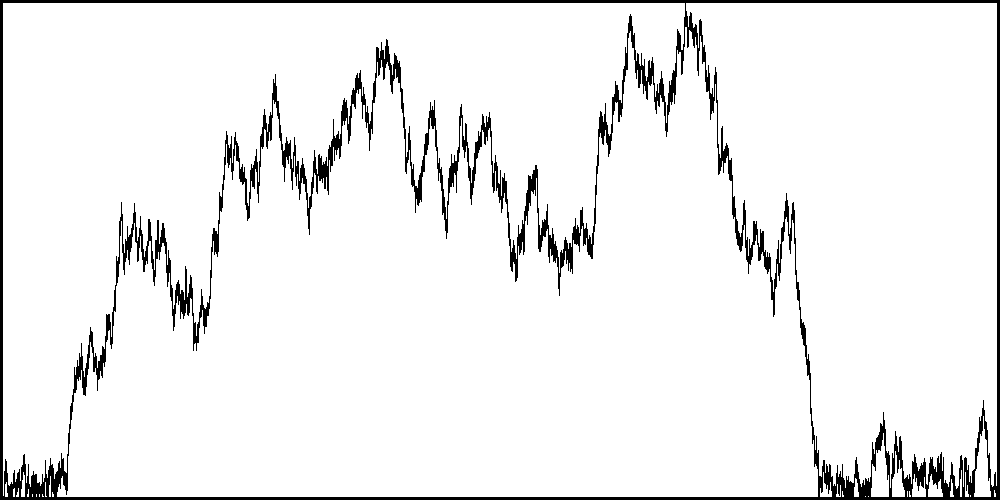}
\put(100.00,-3.00){\tiny$\mathllap{5890445}$}
\put(0.00,51.00){\tiny$2912$}
\end{overpic}\\[1.5\baselineskip]
\begin{overpic}[width=0.45\textwidth]{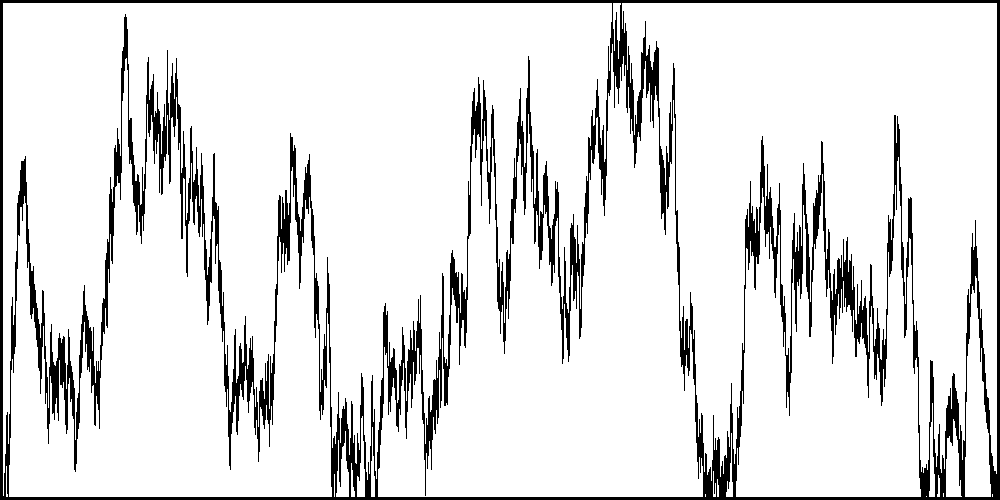}
\put(100.00,-3.00){\tiny$\mathllap{312815429}$}
\put(0.00,51.00){\tiny$9154$}
\end{overpic}
\qquad
\begin{overpic}[width=0.45\textwidth]{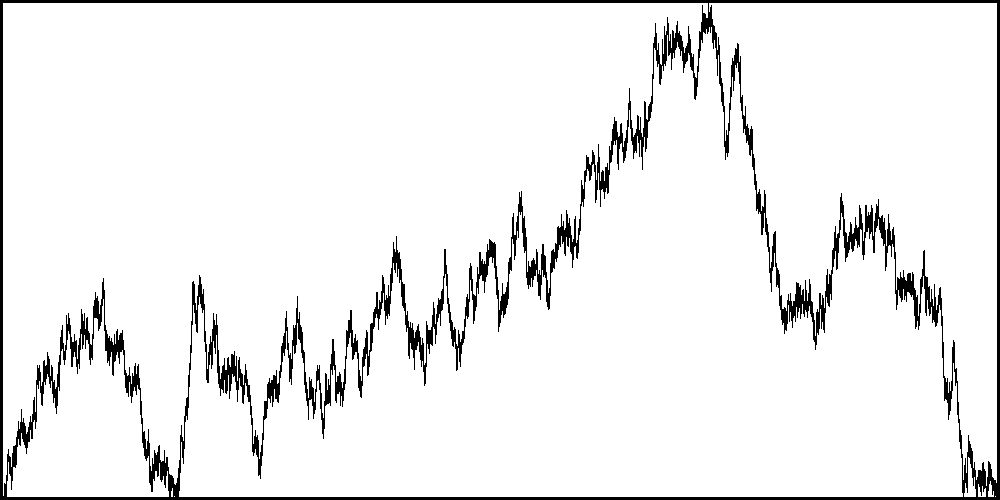}
\put(100.00,-3.00){\tiny$\mathllap{22014805908}$}
\put(0.00,51.00){\tiny$170632$}
\end{overpic}\\[1.5\baselineskip]
\begin{overpic}[width=0.45\textwidth]{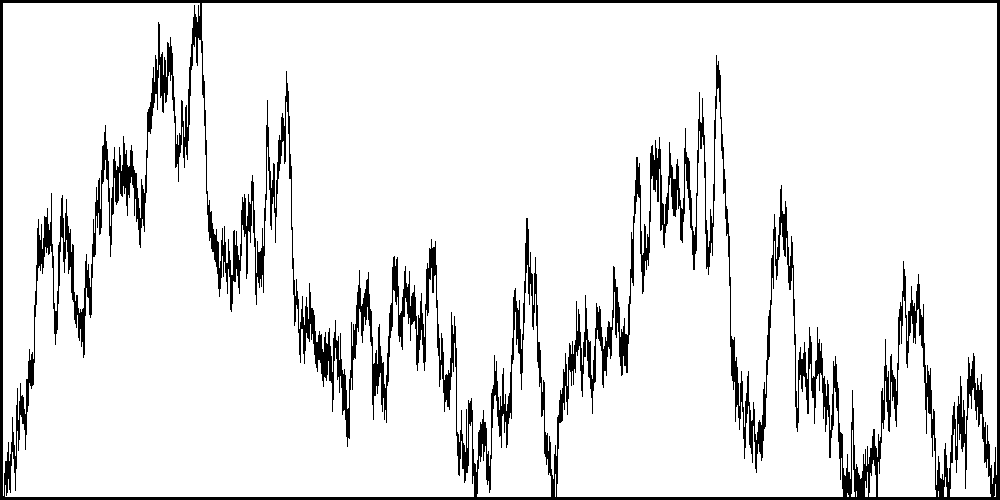}
\put(100.00,-3.00){\tiny$\mathllap{18966150}$}
\put(0.00,51.00){\tiny$3008$}
\end{overpic}
\qquad
\begin{overpic}[width=0.45\textwidth]{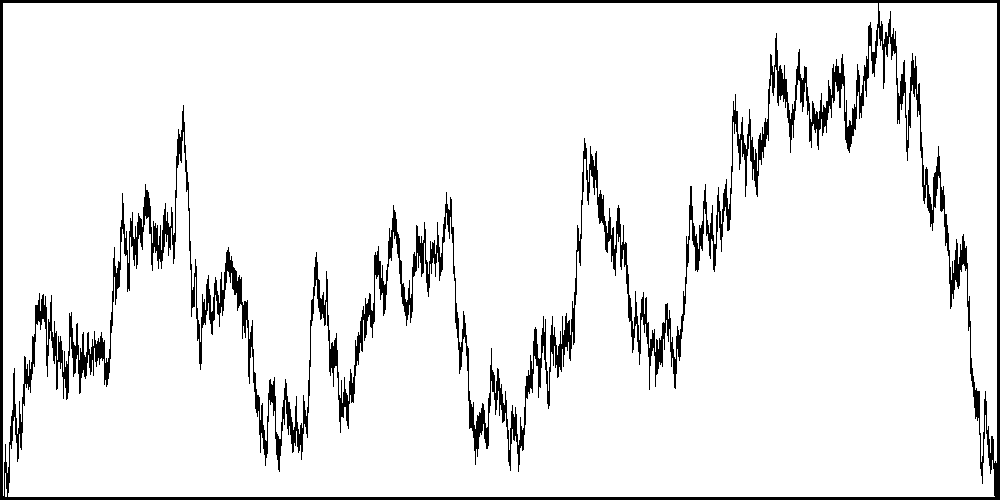}
\put(100.00,-3.00){\tiny$\mathllap{122858925930}$}
\put(0.00,51.00){\tiny$320048$}
\end{overpic}
\caption{The magnitudes (i.e. the $11$-, $7$-, and $5$-adic valuations respectively) of the denominators in the sequences $\FFFST(\FFf{F}_1)[27]$ (top row), $\FFFST(\FFf{F}_2)[27]$ (middle row), and $\FFFST(\FFf{F}_3)[27]$ (bottom row). The left column shows the initial parts and the right column shows the periodic parts of the respective sequences.}
\label{FOrbitsA}
\end{figure}

\begin{figure}[H]
\centering
\begin{overpic}[width=0.45\textwidth]{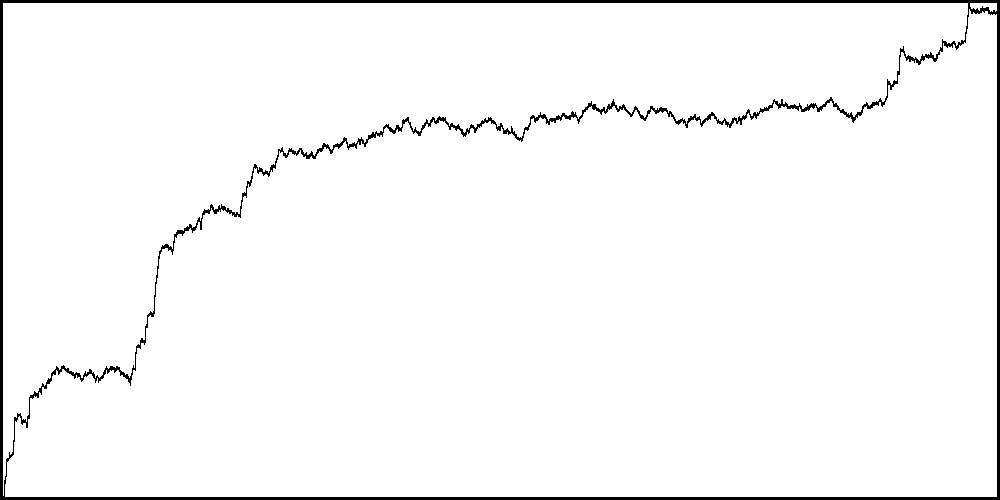}
\put(100.00,-3.00){\tiny$\mathllap{10000000000}$}
\put(0.00,51.00){\tiny$72649$}
\end{overpic}
\qquad
\begin{overpic}[width=0.45\textwidth]{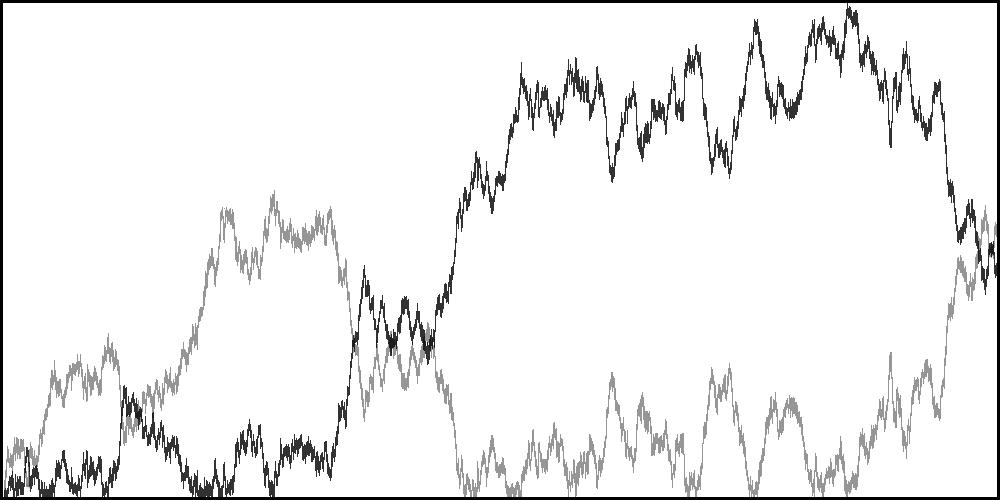}
\put(100.00,-3.00){\tiny$\mathllap{500000000}$}
\put(0.00,51.00){\tiny$36304$}
\end{overpic}\\[1.5\baselineskip]
\begin{overpic}[width=0.45\textwidth]{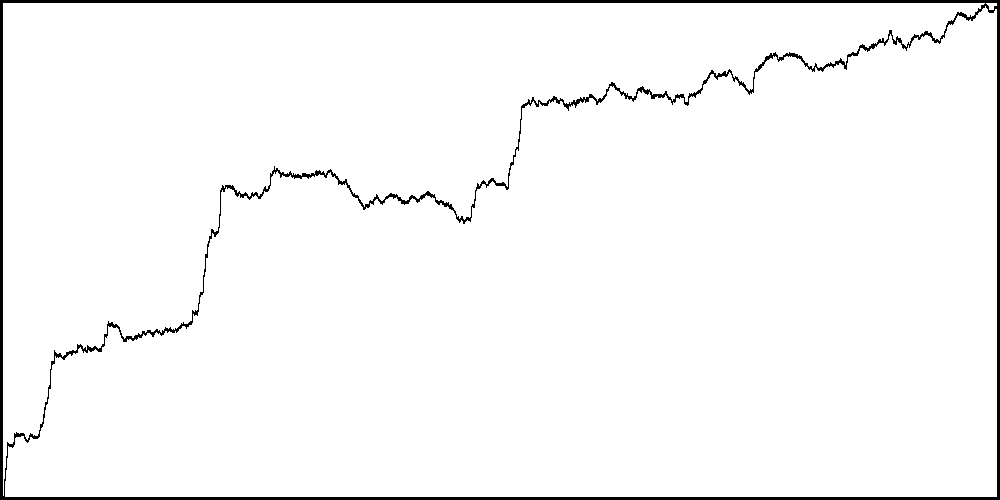}
\put(100.00,-3.00){\tiny$\mathllap{500000000}$}
\put(0.00,51.00){\tiny$18464$}
\end{overpic}
\qquad
\begin{overpic}[width=0.45\textwidth]{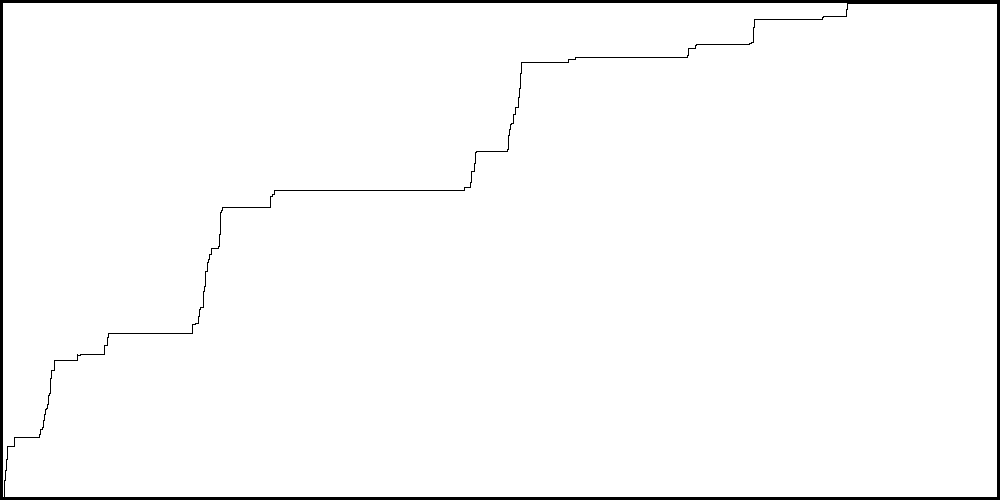}
\put(100.00,-3.00){\tiny$\mathllap{500000000}$}
\put(0.00,51.00){\tiny$36304$}
\end{overpic}
\caption{The two images in the left column show the magnitudes (i.e. the number of digits in base $35$) of the denominators of the first $10^{10}$ and $5\cdot10^8$ entries of the sequences $\FFFST(\FFf{F}_4)[27]$ respectively. The top right image shows the $5$- (dark gray) and $7$-adic valuations (light gray) of the denominators and the bottom right image shows their sums. It can be seen that changes (which are mostly increases) in the sum only occur, if one of the two valuations is equal to $0$.}
\label{FOrbitsB}
\end{figure}
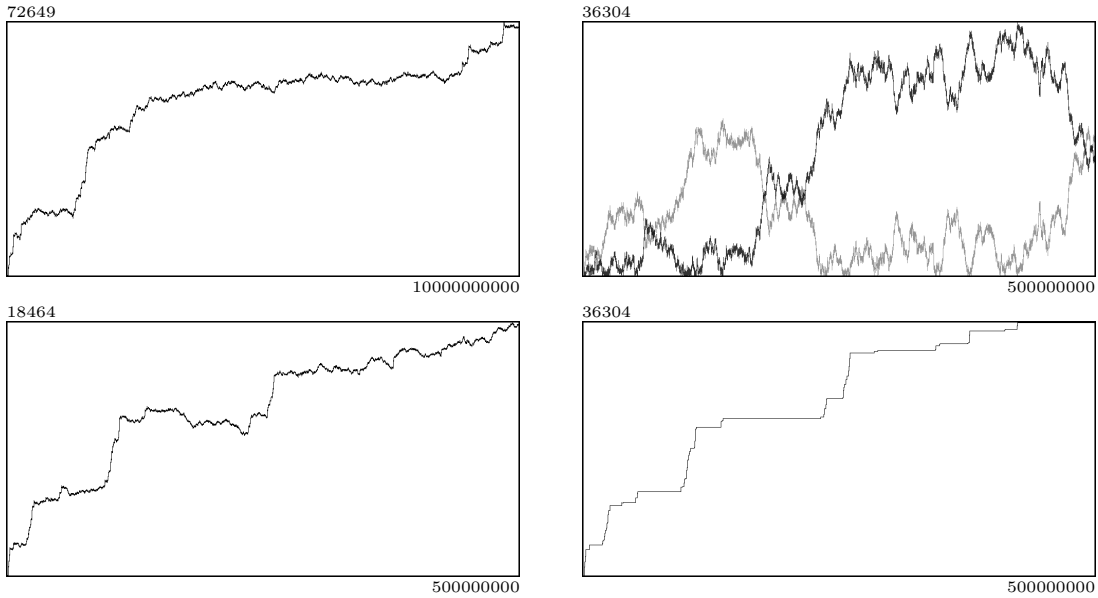

\noindent
Comparing the two conditions
\begin{align}
&B\in\Z\\
&\abs{B}<p^p
\end{align}
of Conjecture~\ref{CoPeriodsD}~(1) one can see that there is a conceptual difference between them that might be significant in explaining the difference between $\FFf{F}_1,\FFf{F}_2,\FFf{F}_3$ on the one hand and $\FFf{F}_4$ on the other. While the second condition does not include the extremal case $\abs{B}=p^p$ (indeed, by Lemma~\ref{LLinPolySuit}~(2) it could not even occur), the first condition does, in the sense explained above. It thus might be necessary to replace the condition $B\in\Z$ by a slightly stronger one which excludes the extremal case in some way. The condition $B\in\Z$ may be understood as a placeholder for a possibly slightly modified condition ``in the same spirit'', i.e. a condition which also involves only the absolute values of the linear coefficients of the polynomials in a ``simple'' way.

The following two conjectures state that we do not gain anything by increasing the degrees of the polynomials \ldots

\begin{conjecture}
\label{CoPeriodsE}
Let $2\leq p\in\N$, $\FFf{F}\in\SoSystems{p}(\FFPpolynomialcoefficients{\Q\cap\Z_p})$, and $B\in\Z_p$ the product of all leading coefficients of the polynomials $\FFf{F}[r](x)$, $r\in\intintuptoex{p}$. Then,

\begin{theoremtable}
\theoremitem{(1)}&$\FFPultimatelyperiodicon{\Q\cap\Z_p}(\FFf{F})\quad\Leftrightarrow\quad\FFf{F}\in\SoSystems{p}(\FFPlinearpolynomialcoefficients{\Q\cap\Z_p})\land B\in\Z\land\abs{B}<p^p$\tabularnewline
\theoremitem{(2)}&$\abs{\set{\SFperiodic{\FFFDT(\FFf{F})[n]}\mid n\in\Z}}<\infty.$
\end{theoremtable}
\end{conjecture}

\noindent
\ldots or by further extending the set from where to take the coefficients of the polynomials.

\begin{conjecture}
\label{CoPeriodsF}
Let $2\leq p\in\N$, $\FFf{F}\in\SoSystems{p}(\FFPpolynomial)$, and $B\in\Z_p$ the product of all leading coefficients of the polynomials $\FFf{F}[r](x)$, $r\in\intintuptoex{p}$. Then,

\begin{theoremtable}
\theoremitem{(1)}&$\FFPultimatelyperiodicon{\Q\cap\Z_p}(\FFf{F})\quad\Leftrightarrow\quad\FFf{F}\in\SoSystems{p}(\FFPlinearpolynomialcoefficients{\Q\cap\Z_p})\land B\in\Z\land\abs{B}<p^p$\tabularnewline
\theoremitem{(2)}&$\abs{\set{\SFperiodic{\FFFDT(\FFf{F})[n]}\mid n\in\Z}}<\infty.$
\end{theoremtable}
\end{conjecture}

\noindent
An overview of the conjectures above is given in Figure~\ref{FOvUltPer} below.

Note that if $\FFf{F}\in\SoSystems{p}$, then
\begin{align}
\FFPlinearpolynomial(\FFf{F})\Leftrightarrow\FFPpolynomialcoefficientsdegree{\Z_p}{1}(\FFf{F})
\end{align}
by Theorem~\ref{TFuncSuit}~(2) and Lemma~\ref{LLinPolySuit}~(2), and if $\FFf{F}\in\SoSystems{p}(\FFPweakcanonicalform,\FFPlinearpolynomialcoefficients{\Q\cap\Z_p})$, then
\begin{align}
&\FFPultimatelyperiodicon{\Q\cap\Z_p}(\FFf{F})\\
&\quad\Leftrightarrow[\Q\cap\Z_p\subseteq\FFFSoUPP](\FFf{F})\\
&\quad\Leftrightarrow
\set{\left(\frac{\FRFlinA{\FFf{F}}{\Sf{D}_P}}{p^{\abs{\Sf{D}_P}}-\FRFlinB{\FFf{F}}{\Sf{D}_P}}p^{\abs{\Sf{D}_I}}-\FRFlinA{\FFf{F}}{\Sf{D}_I}\right)\frac{1}{\FRFlinB{\FFf{F}}{\Sf{D}_I}}\mid\Sf{D}_I,\Sf{D}_P\in\CoSequences(\SPboundedby{\intintuptoex{p}},\SPfinite)}=\Q\cap\Z_p
\end{align}
by the ``In particular'' part of Corollary~\ref{CPerFormula}. By the same argument the original Collatz conjecture (Conjecture~\ref{CoPeriodsA}) is equivalent to the following divisibility question (cf. Eqn.~(\ref{ECollatzDiv})):
\begin{align}
\N\subseteq\FFFSoUPP(\FFf{F})\quad\land\quad\set{\Sf{D}\in\CoSequences(\SPboundedby{\intintuptoex{p}},\SPfinite)\mid\frac{\FRFlinA{\FFf{F}_C}{\Sf{D}}}{p^{\abs{\Sf{D}}}-\FRFlinB{\FFf{F}_C}{\Sf{D}}}\in\N}=\set{(0,1),(1,0)}.
\end{align}

\myparagraphtoc{Non-rational coefficients.}
A first general result on the question of ultimate periodicity on $\Q\cap\Z_p$ of polynomial $p$-adic systems where at least one coefficient of the polynomials is not a rational number, is given by the following theorem. It implies in particular that Conjecture~\ref{CoPeriodsE}~(1) and Conjecture~\ref{CoPeriodsF}~(1) are equivalent.

\begin{theorem}
\label{TGenPolyUPer}
Let $2\leq p\in\N$ and $\FFf{F}\in\SoSystems{p}(\FFPpolynomial,\neg\FFPpolynomialcoefficients{\Q\cap\Z_p})$. Then, $\neg\FFPultimatelyperiodicon{\Q\cap\Z_p}(\FFf{F})$.
\end{theorem}

\begin{proof}
Assume to the contrary that $\FFPultimatelyperiodicon{\Q\cap\Z_p}(\FFf{F})$ and let $r\in\intintuptoex{p}$ such that
\begin{align}
\FFf{F}[r]=a_0+a_1x+\cdots+a_dx^d\notin(\Q\cap\Z_p)[x].
\end{align}
If $\FFf{F}[r](a)=b$ for some $a\in\Q\cap(r+p\Z_p)$ and $b\in\Z_p\setminus\Q$, then $a\in\FFFSoUPP(\FFf{F})$ by $\FFPultimatelyperiodicon{\Q\cap\Z_p}(\FFf{F})$ and hence $(b-b\modulo p)/p=\FFf{F}(a)\in\FFFSoUPP(\FFf{F})$ which contradicts $\FFPultimatelyperiodicon{\Q\cap\Z_p}(\FFf{F})$, since $(b-b\modulo p)/p\in\Z_p\setminus\Q$. Thus, $\FFf{F}[r](\Q\cap(r+p\Z_p))\subseteq\Q\cap\Z_p$. In particular,
\begin{align}
b_i\ce\FFf{F}[r](r+ip)=a_0+a_1(r+ip)+\cdots+a_d(r+ip)^d\in\Q\cap\Z_p
\end{align}
for all $i\in\intint{1}{d+1}$. Thus, $a_0,\ldots,a_d$ solve a system of $d+1$ independent linear equations with coefficients in $\Q\cap\Z_p$. The unique solution of this system (which is $a_0,\ldots,a_d$) can be computed using Gaussian elimination which expresses the solution in terms of the coefficients of the linear equations using only the four basic arithmetical operations. Thus, $a_0,\ldots,a_d\in\Q\cap\Z_p$ which is a contradiction.
\end{proof}

\noindent
Note that the proof actually shows $[\FFFSoUPP\not\subseteq\Q\cap\Z_p](\FFf{F})$ which of course implies $\neg\FFPultimatelyperiodicon{\Q\cap\Z_p}(\FFf{F})$. In addition, one might conjecture that $[\Q\cap\Z_p\not\subseteq\FFFSoUPP](\FFf{F})$ also holds  under the assumptions of the theorem but a proof appears to be much harder.

\myparagraphtoc{The constant coefficients.}
A consequence of Theorem~\ref{TLinDirForm} is that, at least for $(\Q\cap\Z_2)$-linear-polynomial $2$-adic system, the constant coefficients of the linear polynomials have no influence on the question of whether all rational numbers have ultimately periodic digit expansions.

\begin{theorem}
\label{TConstIrr}
Let $\FFf{F}=(a_0'+b_0x,a_1'+b_1x),\FFf{G}=(a_0''+b_0x,a_1''+b_1x)\in\SoSystems{2}(\FFPweakcanonicalform,\FFPlinearpolynomial)$. Then,
\begin{align}
\label{EConstIrr}
\pi_{\FFf{F},\FFf{G}}(n)=\frac{a_0'a_1''-a_1'a_0''+n((b_0-2)a_1''-(b_1-2)a_0'')}{(b_0-2)a_1'-(b_1-2)a_0'}
\end{align}
for all $n\in\Z_2$. In particular, if $\FFPlinearpolynomialcoefficients{\Q\cap\Z_2}(\FFf{F})$ and $\FFPlinearpolynomialcoefficients{\Q\cap\Z_2}(\FFf{G})$, then $\pi_{\FFf{F},\FFf{G}}(\Q\cap\Z_2)=\Q\cap\Z_2$ and consequently
\begin{align}
\FFPultimatelyperiodicon{\Q\cap\Z_2}(\FFf{F})\Leftrightarrow\FFPultimatelyperiodicon{\Q\cap\Z_2}(\FFf{G})
\end{align}
by the ``In particular'' part of Lemma~\ref{LPeriodicOn}.
\end{theorem}

\begin{proof}
First we observe that it is sufficient to prove Eqn.~(\ref{EConstIrr}) on a dense subset of $\Z_2$, since $\pi_{\FFf{F},\FFf{G}}$ is a continuous function by Lemma~\ref{LPermProp}~(1) (and so is the right-hand side of Eqn.~(\ref{EConstIrr})). The subset we will consider is given by the set of all $2$-adic integers which have a periodic $\FFf{F}$-digit expansion which is dense in $\Z_2$, because $\FFf{F}$ has the block property.

Let $\Sf{D}\in\CoSequences(\SPboundedby{\intintuptoex{2}},\neg\SPempty,\SPfinite)$, and $m,n$ the unique elements of $\Z_p$ satisfying
\begin{align}
\FFFDT(\FFf{F})[m]&=\Sf{D}^\infty\\
\FFFDT(\FFf{G})[n]&=\Sf{D}^\infty
\end{align}
(cf. Lemma~\ref{LInfBlDTCompl}). Then, $\pi_{\FFf{F},\FFf{G}}(m)={\psi_\FFf{G}}^{-1}\circ\psi_\FFf{F}(m)=n$ by definition of $\pi_{\FFf{F},\FFf{G}}(n)$, and
\begin{align}
m&=\frac{\FRFlinA{\FFf{F}}{\Sf{D}}}{2^{\abs{\Sf{D}}}-\FRFlinB{\FFf{F}}{\Sf{D}}}\\
n&=\frac{\FRFlinA{\FFf{G}}{\Sf{D}}}{2^{\abs{\Sf{D}}}-\FRFlinB{\FFf{G}}{\Sf{D}}}
\end{align}
by Corollary~\ref{CPerFormula}. We are left to show that
\begin{align}
&\frac{a_0'a_1''-a_1'a_0''+m((b_0-2)a_1''-(b_1-2)a_0'')}{(b_0-2)a_1'-(b_1-2)a_0'}=n
\intertext{and respectively}
\label{EConstIrrGoal}
&a_0'a_1''(2^{\abs{\Sf{D}}}-B)+\FRFlinA{\FFf{F}}{\Sf{D}}((b_0-2)a_1''-(b_1-2)a_0'')\\
\nonumber
&\quad=a_1'a_0''(2^{\abs{\Sf{D}}}-B)+\FRFlinA{\FFf{G}}{\Sf{D}}((b_0-2)a_1'-(b_1-2)a_0')
\end{align}
where
\begin{align}
e&\ce\SFcount{\Sf{D}}{0}\\
o&\ce\SFcount{\Sf{D}}{1}\\
B&\ce\FRFlinB{\FFf{F}}{\Sf{D}}=\FRFlinB{\FFf{G}}{\Sf{D}}=\prod_{r=0}^1b_r^{\SFcount{\Sf{D}}{r}}=b_0^{\SFcount{\Sf{D}}{0}}b_1^{\SFcount{\Sf{D}}{1}}={b_0}^e{b_1}^o.
\end{align}
We compute
\begin{align}
\FRFlinA{\FFf{F}}{\Sf{D}}
&=\sum_{r=0}^1a_r'\sum_{i=0}^{\SFcount{\Sf{D}}{r}-1}2^{\SFposition{\Sf{D}}{r}[i]}\prod_{s=0}^1b_s^{\SFcount{\Sf{D}[\SFposition{\Sf{D}}{r}[i]+1,\abs{\Sf{D}}-1]}{s}}\\
&=a_0'\sum_{i=0}^{\SFcount{\Sf{D}}{0}-1}2^{\SFposition{\Sf{D}}{0}[i]}b_0^{\SFcount{\Sf{D}[\SFposition{\Sf{D}}{0}[i]+1,\abs{\Sf{D}}-1]}{0}}
b_1^{\SFcount{\Sf{D}[\SFposition{\Sf{D}}{0}[i]+1,\abs{\Sf{D}}-1]}{1}}+\vphantom{a}\\
\nonumber
&\phantom{\vphantom{a}=\vphantom{a}}
a_1'\sum_{i=0}^{\SFcount{\Sf{D}}{1}-1}2^{\SFposition{\Sf{D}}{1}[i]}b_0^{\SFcount{\Sf{D}[\SFposition{\Sf{D}}{1}[i]+1,\abs{\Sf{D}}-1]}{0}}
b_1^{\SFcount{\Sf{D}[\SFposition{\Sf{D}}{1}[i]+1,\abs{\Sf{D}}-1]}{1}}\\
&=a_0'\sum_{i=0}^{\SFcount{\Sf{D}}{0}-1}2^{\SFposition{\Sf{D}}{0}[i]}b_0^{\SFcount{\Sf{D}}{0}-i-1}
b_1^{\SFcount{\Sf{D}}{1}-\SFposition{\Sf{D}}{0}[i]+i}+\vphantom{a}\\
\nonumber
&\phantom{\vphantom{a}=\vphantom{a}}
a_1'\sum_{i=0}^{\SFcount{\Sf{D}}{1}-1}2^{\SFposition{\Sf{D}}{1}[i]}b_0^{\SFcount{\Sf{D}}{0}-\SFposition{\Sf{D}}{1}[i]+i}b_1^{\SFcount{\Sf{D}}{1}-i-1}\\
&=a_0'\sum_{i=0}^{e-1}2^{E_i}b_0^{e-i-1}
b_1^{o-E_i+i}+
a_1'\sum_{i=0}^{o-1}2^{O_i}b_0^{e-O_i+i}b_1^{o-i-1}
\intertext{where}
E_i&\ce\SFposition{\Sf{D}}{0}[i]\text{ for all $i\in\intintuptoex{e}$}\\
O_i&\ce\SFposition{\Sf{D}}{1}[i]\text{ for all $i\in\intintuptoex{o}$.}
\intertext{Analogously, we get}
\FRFlinA{\FFf{G}}{\Sf{D}}
&=a_0''\sum_{i=0}^{e-1}2^{E_i}b_0^{e-i-1}
b_1^{o-E_i+i}+
a_1''\sum_{i=0}^{o-1}2^{O_i}b_0^{e-O_i+i}b_1^{o-i-1}.
\end{align}
Let
\begin{align}
S_e&\ce\sum_{i=0}^{e-1}2^{E_i}b_0^{e-i-1}b_1^{o-E_i+i}\\
S_o&\ce\sum_{i=0}^{o-1}2^{O_i}b_0^{e-O_i+i}b_1^{o-i-1}.
\end{align}
Then, our goal Eqn.~(\ref{EConstIrrGoal}) is equivalent to
\begin{align}
&a_0'a_1''(2^{\abs{\Sf{D}}}-B)+\left(a_0'S_e+
a_1'S_o\right)((b_0-2)a_1''-(b_1-2)a_0'')\\
\nonumber
&\quad=a_1'a_0''(2^{\abs{\Sf{D}}}-B)+\left(a_0''S_e+
a_1''S_o\right)((b_0-2)a_1'-(b_1-2)a_0')
\intertext{which is again equivalent to}
&(2^{\abs{\Sf{D}}}-B+(b_0-2)S_e+(b_1-2)S_o)(a_0'a_1''-a_1'a_0'')=0.
\end{align}
Since we may assume without loss of generality that $a_0'a_1''-a_1'a_0''\neq0$ (and $\Z_2$ has no zero divisors), we need to show that
\begin{align}
\label{EConstIrrFinalGoal}
2^{\abs{\Sf{D}}}-B+(b_0-2)S_e+(b_1-2)S_o=0
\end{align}
which we will prove to be true for every $\Sf{D}\in\CoSequences(\SPboundedby{\intintuptoex{2}},\SPfinite)$ by induction on the length of $\Sf{D}$.

If $\abs{\Sf{D}}=0$ then $e=0$, $o=0$ and hence $B=1$, $S_e=0$, and $S_o=0$. Altogether, Eqn.~(\ref{EConstIrrFinalGoal}) is clearly true.

Now assume that Eqn.~(\ref{EConstIrrFinalGoal}) holds for some $\Sf{D}\in\CoSequences(\SPboundedby{\intintuptoex{2}},\SPfinite)$ and let $\Sf{D}'\in\CoSequences(\SPboundedby{\intintuptoex{2}},\SPfinite)$ such that $\Sf{D}'=\Sf{D}\cdot(d)$ where $d\in\set{0,1}$. Furthermore, let
\begin{align}
&e'\ce\SFcount{\Sf{D}'}{0},\;
o'\ce\SFcount{\Sf{D}'}{1},\;
B'\ce{b_0}^{e'}{b_1}^{o'}\\
&E_i'\ce\SFposition{\Sf{D}'}{0}[i]\text{ for all $i\in\intintuptoex{e'}$},\;
O_i'\ce\SFposition{\Sf{D}'}{1}[i]\text{ for all $i\in\intintuptoex{o'}$}\\
&S_e'\ce\sum_{i=0}^{e'-1}2^{E_i'}b_0^{e'-i-1}b_1^{o'-E_i'+i},\;
S_o'\ce\sum_{i=0}^{o'-1}2^{O_i'}b_0^{e'-O_i'+i}b_1^{o'-i-1}.
\end{align}
On the one hand, if $d=0$, then
\begin{align}
&e'=e+1,\;
o'=o,\;
B'=b_0B,\\
&E_i'=E_i\text{ for all $i\in\intintuptoex{e}$},\;
E_{e}'=\abs{\Sf{D}},\;
O_i'=O_i\text{ for all $i\in\intintuptoex{o}$},\\
&S_e'=b_0S_e+2^{\abs{\Sf{D}}},\;
S_o'=b_0S_o
\end{align}
and
\begin{align}
&2^{\abs{\Sf{D}'}}-B'+(b_0-2)S_e'+(b_1-2)S_o'\\
\nonumber
&\quad=2\cdot2^{\abs{\Sf{D}}}-b_0B+(b_0-2)\left(b_0S_e+2^{\abs{\Sf{D}}}\right)+(b_1-2)b_0S_o\\
&\quad=2\cdot2^{\abs{\Sf{D}}}+b_0\left(-B+(b_0-2)S_e+2^{\abs{\Sf{D}}}+(b_1-2)S_o\right)-2\cdot2^{\abs{\Sf{D}}}\\
&\quad=0.
\end{align}
On the other hand, if $d=1$, then
\begin{align}
&e'=e,\;
o'=o+1,\;
B'=b_1B,\\
&E_i'=E_i\text{ for all $i\in\intintuptoex{e}$},\;
O_i'=O_i\text{ for all $i\in\intintuptoex{o}$},\;
O_{o}'=\abs{\Sf{D}},\\
&S_e'=b_1S_e,\;
S_o'=b_1S_o+2^{\abs{\Sf{D}}}
\end{align}
and
\begin{align}
&2^{\abs{\Sf{D}'}}-B'+(b_0-2)S_e'+(b_1-2)S_o'\\
\nonumber
&\quad=2\cdot2^{\abs{\Sf{D}}}-b_1B+(b_0-2)b_1S_e+(b_1-2)\left(b_1S_o+2^{\abs{\Sf{D}}}\right)\\
&\quad=2\cdot2^{\abs{\Sf{D}}}+b_1\left(-B+(b_0-2)S_e+(b_1-2)S_o+2^{\abs{\Sf{D}}}\right)-2\cdot2^{\abs{\Sf{D}}}\\
&\quad=0
\end{align}
which completes the proof of Eqn.~(\ref{EConstIrr}).

For the ``In particular'' part we observe that $b_0\equiv b_1\equiv1\modulus{2}$, $a_0'\equiv a_0''\equiv0\modulus{2}$, and $a_1'\equiv a_1''\equiv1\modulus{2}$ by Theorem~\ref{TFuncSuit}~(2) and Lemma~\ref{LLinPolySuit}~(2) (note that $\FFf{F},\FFf{G}\in\SoSystems{2}(\FFPweakcanonicalform)$). Thus
\begin{align}
&a_0'a_1''-a_1'a_0''\equiv0\modulus{2}\\
&(b_0-2)a_1''-(b_1-2)a_0''\equiv1\modulus{2}\\
&(b_0-2)a_1'-(b_1-2)a_0'\equiv1\modulus{2}
\end{align}
and hence
\begin{align}
\pi:\Q\cap\Z_2&\to\Q\cap\Z_2\\
\nonumber
n&\mapsto\frac{a_0'a_1''-a_1'a_0''+n((b_0-2)a_1''-(b_1-2)a_0'')}{(b_0-2)a_1'-(b_1-2)a_0'}
\end{align}
is bijective.
\end{proof}

\noindent
A natural follow-up question is whether the ``In particular'' part of the previous theorem is true for general $(\Q\cap\Z_p)$-linear-polynomial $p$-adic systems, i.e. if the constant coefficients of the linear polynomials which define the $p$-adic systems can always be neglected when dealing with the question of whether the $p$-adic system is ultimately periodic on $\Q\cap\Z_p$. If true, this would be a first step in proving the general conjectures \ref{CoPeriodsC}~--~\ref{CoPeriodsF}. At least for $p=2$, which includes the Collatz case, this first step has already been made. For $p\geq3$ the situation becomes more difficult as there doesn't seem to exist a simple formula for $\pi_{\FFf{F},\FFf{G}}(n)$ in general. Nevertheless, we state the following conjecture which would probably be a good start for future work on this matter.

\begin{conjecture}
\label{CoConstIrr}
Let $2\leq p\in\N$ and $\FFf{F}=(a_0'+b_0x,\ldots,a_{p-1}'+b_{p-1}x),\FFf{G}=(a_0''+b_0x,\ldots,a_{p-1}''+b_{p-1}x)\in\SoSystems{p}(\FFPlinearpolynomialcoefficients{\Q\cap\Z_p})$. Then, $\pi_{\FFf{F},\FFf{G}}(\Q\cap\Z_p)=\Q\cap\Z_p$ and consequently
\begin{align}
\FFPultimatelyperiodicon{\Q\cap\Z_p}(\FFf{F})\Leftrightarrow\FFPultimatelyperiodicon{\Q\cap\Z_p}(\FFf{G})
\end{align}
by the ``In particular'' part of Lemma~\ref{LPeriodicOn}.
\end{conjecture}

\myparagraph{Swapping polynomials.}
In the last subsection we have established that, at least in the case $p=2$, the constant coefficients of the linear polynomials defining a $(\Q\cap\Z_p)$-linear-polynomial $p$-adic system $\FFf{F}$ have no influence on whether $\FFf{F}$ is ultimately periodic on $\Q\cap\Z_p$. Here we will prove, again for $p=2$, that the specific positions of the linear coefficients can also be neglected, which is another step closer to the proof of the general conjectures \ref{CoPeriodsC}~--~\ref{CoPeriodsF}. Before we can formulate our result, we need to add a little flexibility to the definition of $\pi_{\FFf{F},\FFf{G}}$ in that we allow the digits of the $\FFf{F}$-digit expansions to be permuted before interpreting them as $\FFf{G}$-digit expansions. For two $p$-adic systems $\FFf{F}$ and $\FFf{G}$ and for a bijective function $\sigma:\intintuptoex{p}\to\intintuptoex{p}$ (which extends naturally to $\CoSequences(\SPboundedby{\intintuptoex{p}})$, cf. the end of the subsection on sequences in Section~\ref{SNotDef}) let
\begin{align}
\label{Dpisigma}
\pi_{\FFf{F},\sigma,\FFf{G}}\ce{\psi_\FFf{G}}^{-1}\circ\sigma\circ\psi_\FFf{F}:\Z_p&\to\Z_p.
\end{align}
If $\sigma$ is the identity function, then clearly $\pi_{\FFf{F},\sigma,\FFf{G}}=\pi_{\FFf{F},\FFf{G}}$. Note that here we formulate our result for a specific choice for the constant coefficients. In combination with Theorem~\ref{TConstIrr} a general result involving arbitrary constant coefficients can easily be achieved since $\pi_{\FFf{G},\tau,\FFf{H}}\circ\pi_{\FFf{F},\sigma,\FFf{G}}=\pi_{\FFf{F},\tau\circ\sigma,\FFf{H}}$.

\begin{theorem}
\label{TOrderIrr}
Let $\FFf{F}=(b_0x,1+b_1x),\FFf{G}=(b_1x,1+b_0x)\in\SoSystems{2}(\FFPlinearpolynomial)$, and $\sigma:\set{0,1}\to\set{0,1}$, $0\mapsto1$, $1\mapsto0$. Then,
\begin{align}
\label{EOrderIrr}
\pi_{\FFf{F},\sigma,\FFf{G}}(n)=\frac{1+n(b_1-2)}{2-b_0}
\end{align}
for all $n\in\Z_2$. In particular, if $\FFPlinearpolynomialcoefficients{\Q\cap\Z_p}(\FFf{F})$ and $\FFPlinearpolynomialcoefficients{\Q\cap\Z_p}(\FFf{G})$, then $\pi_{\FFf{F},\sigma,\FFf{G}}(\Q\cap\Z_2)=\Q\cap\Z_2$ and consequently,
\begin{align}
\FFPultimatelyperiodicon{\Q\cap\Z_2}(\FFf{F})\Leftrightarrow\FFPultimatelyperiodicon{\Q\cap\Z_2}(\FFf{G})
\end{align}
by the ``In particular'' part of Lemma~\ref{LPeriodicOn}.
\end{theorem}

\noindent
Note that we actually need a slight generalization of Lemma~\ref{LPeriodicOn} to account for the permutation $\sigma$ in $\pi_{\FFf{F},\sigma,\FFf{G}}$. The proof can easily be adapted to show that the lemma also holds if $\pi_{\FFf{F},\FFf{G}}$ is replaced by $\pi_{\FFf{F},\sigma,\FFf{G}}$, if the set $B$ from the assumptions is stable under $\sigma$, i.e. if $\sigma(B)=B$. For $B\ce\set{\Sf{S}\in\CoSequences(\SPboundedby{\intintuptoex{p}})\mid\SPperiodic(\Sf{S})}$, $B\ce\set{\Sf{S}\in\CoSequences(\SPboundedby{\intintuptoex{p}})\mid\SPultimatelyperiodic(\Sf{S})}$, and $B\ce\set{\Sf{S}\in\CoSequences(\SPboundedby{\intintuptoex{p}})\mid\SPaperiodic(\Sf{S})}$ this is true for every $\sigma$ which implies that the ``In particular'' part of the lemma holds for $\pi_{\FFf{F},\sigma,\FFf{G}}$ without additional assumptions.

\begin{proof}[Proof of Theorem~\ref{TOrderIrr}]
As we did in the proof of Theorem~\ref{TConstIrr}, we again only consider a dense subset of $\Z_2$ (clearly, $\pi_{\FFf{F},\sigma,\FFf{G}}$ and the right-hand side of Eqn.~(\ref{EOrderIrr}) are also continuous functions) which will again be the set of all $2$-adic integers which have a periodic $\FFf{F}$-digit expansion.

Let $\Sf{D}\in\CoSequences(\SPboundedby{\intintuptoex{2}},\neg\SPempty,\SPfinite)$, and $m,n$ the unique elements of $\Z_p$ satisfying
\begin{align}
\FFFDT(\FFf{F})[m]&=\Sf{D}^\infty\\
\FFFDT(\FFf{G})[n]&=\sigma(\Sf{D}^\infty)
\end{align}
(cf. Lemma~\ref{LInfBlDTCompl}). Then, $\pi_{\FFf{F},\sigma,\FFf{G}}(m)={\psi_\FFf{G}}^{-1}\circ\sigma\circ\psi_\FFf{F}(m)=n$ by definition of $\pi_{\FFf{F},\sigma,\FFf{G}}(n)$, and
\begin{align}
m&=\frac{\FRFlinA{\FFf{F}}{\Sf{D}}}{2^{\abs{\Sf{D}}}-\FRFlinB{\FFf{F}}{\Sf{D}}}\\
n&=\frac{\FRFlinA{\FFf{G}}{\sigma(\Sf{D})}}{2^{\abs{\Sf{D}}}-\FRFlinB{\FFf{G}}{\sigma(\Sf{D})}}
\end{align}
by Corollary~\ref{CPerFormula}. Note that,
\begin{align}
\FRFlinB{\FFf{F}}{\Sf{D}}&=b_0^{\SFcount{\Sf{D}}{0}}b_1^{\SFcount{\Sf{D}}{1}}=b_1^{\SFcount{\sigma(\Sf{D})}{0}}b_0^{\SFcount{\sigma(\Sf{D})}{1}}=\FRFlinB{\FFf{G}}{\sigma(\Sf{D})}.
\end{align}
We are thus left to show that
\begin{align}
&\frac{1+m(b_1-2)}{2-b_0}=n
\intertext{and respectively}
\label{EOrderIrrGoal}
&2^{\abs{\Sf{D}}}-{b_0}^e{b_1}^o+(b_0-2)\FRFlinA{\FFf{G}}{\sigma(\Sf{D})}+(b_1-2)\FRFlinA{\FFf{F}}{\Sf{D}}=0
\end{align}
where
\begin{align}
e&\ce\SFcount{\Sf{D}}{0}\\
o&\ce\SFcount{\Sf{D}}{1}.
\end{align}
As in the proof of Theorem~\ref{TConstIrr} we compute
\begin{align}
\FRFlinA{\FFf{F}}{\Sf{D}}
&=\sum_{i=0}^{o-1}2^{O_i}b_0^{e-O_i+i}b_1^{o-i-1}\\
\FRFlinA{\FFf{G}}{\sigma(\Sf{D})}
&=\sum_{i=0}^{e-1}2^{E_i}b_0^{e-i-1}b_1^{o-E_i+i}
\intertext{where}
E_i&\ce\SFposition{\Sf{D}}{0}[i]\text{ for all $i\in\intintuptoex{e}$}\\
O_i&\ce\SFposition{\Sf{D}}{1}[i]\text{ for all $i\in\intintuptoex{o}$.}
\end{align}
Thus, our goal Eqn.~(\ref{EOrderIrrGoal}) is equivalent to
\begin{align}
2^{\abs{\Sf{D}}}-{b_0}^e{b_1}^o+(b_0-2)\sum_{i=0}^{e-1}2^{E_i}b_0^{e-i-1}b_1^{o-E_i+i}+(b_1-2)\sum_{i=0}^{o-1}2^{O_i}b_0^{e-O_i+i}b_1^{o-i-1}=0
\end{align}
which is identical to Eqn.~(\ref{EConstIrrFinalGoal}) which we already proved.
\end{proof}

\noindent
As we did after Theorem~\ref{TConstIrr}, we again formulate a conjecture on how to generalize the previous theorem.

\begin{conjecture}
\label{CoOrderIrr}
Let $2\leq p\in\N$, $\sigma:\intintuptoex{p}\to\intintuptoex{p}$ a bijective function, and $\FFf{F}=(b_0x,\ldots,b_{p-1}x),\FFf{G}=(b_{\sigma(0)}x,\ldots,b_{\sigma(p-1)}x)\in\SoSystems{p}(\FFPlinearpolynomial)$. Then, $\pi_{\FFf{F},\sigma,\FFf{G}}(\Q\cap\Z_p)=\Q\cap\Z_p$ and consequently,
\begin{align}
\FFPultimatelyperiodicon{\Q\cap\Z_p}(\FFf{F})\Leftrightarrow\FFPultimatelyperiodicon{\Q\cap\Z_p}(\FFf{G})
\end{align}
by the ``In particular'' part of (the mentioned generalization of) Lemma~\ref{LPeriodicOn}.
\end{conjecture}

We give an example which combines the statements of Theorem~\ref{TConstIrr} and Theorem~\ref{TOrderIrr} to illustrate how to compute $\pi_{\FFf{F},\sigma,\FFf{G}}(n)$ in a general situation.

\begin{example}
\label{EConstOrderIrr}
Let $\FFf{F}_1\ce(5x+4,-7x+19)$, $\FFf{F}_2\ce(5x,-7x+1)$, $\FFf{F}_3\ce(-7x,5x+1)$, $\FFf{F}_4\ce(-7x-12,5x+3)$, and $\sigma:\set{0,1}\to\set{0,1}$, $0\mapsto1$, $1\mapsto0$. Then, Theorem~\ref{TConstIrr} and Theorem~\ref{TOrderIrr} imply that
\begin{align}
\pi_{\FFf{F}_1,\FFf{F}_2}(n)&=\frac{4+3n}{93},&&
\pi_{\FFf{F}_2,\sigma,\FFf{F}_3}(n)&=\frac{-1+9n}{3},&&
\pi_{\FFf{F}_3,\FFf{F}_4}(n)&=\frac{-12-9n}{9}
\end{align}
and consequently
\begin{align}
\pi_{\FFf{F}_1,\sigma,\FFf{F}_4}(n)&=\pi_{\FFf{F}_3,\FFf{F}_4}(n)\circ\pi_{\FFf{F}_2,\sigma,\FFf{F}_3}(n)\circ\pi_{\FFf{F}_1,\FFf{F}_2}(n)=\frac{-35-3n}{31}.
\end{align}
Indeed, if $m\ce\frac{156065447}{59288775}$ and $n\ce-\frac{847767822}{612650675}$, then $\frac{-35-3m}{31}=n$ and
\begin{align}
\FFf{F}_1(m)&=(1,0,1,0)\cdot(1,1,0,1,0,0,0,1)^\infty\\
\FFf{F}_4(n)&=(0,1,0,1)\cdot(0,0,1,0,1,1,1,0)^\infty.
\end{align}
\end{example}

\myparagraphtoc{The contractive and expansive cases.}
In the previous subsections we have proven for $p=2$ and conjectured for $p\geq3$ that neither the constant coefficients nor the order of the linear coefficients of the polynomials defining a $(\Q\cap\Z_p)$-linear-polynomial $p$-adic system matter when it comes to the question of whether the $p$-adic system is ultimately periodic on $\Q\cap\Z_p$. These observations are closely related to the conjectures \ref{CoPeriodsC}~--~\ref{CoPeriodsF} which essentially state that the answer to this question only depends on the product of the linear coefficients. In the special cases where the $(\Q\cap\Z_p)$-linear-polynomial $p$-adic system is either contractive or expansive, we can even go further and prove some of the conjectures under these additional constraints. Among other things the following theorem characterizes contractive and expansive $(\Q\cap\Z_p)$-linear-polynomial $p$-adic systems.

\begin{theorem}
\label{TLinPolyCEM}
Let $2\leq p\in\N$, $\FFf{F}=(a_0+b_0x,\ldots,a_{p-1}+b_{p-1}x)\in\SoSystems{p}(\FFPlinearpolynomialcoefficients{\Q\cap\Z_p})$. Then,

\begin{theoremtable}[rlX]
\theoremitem{(1)}&$\FFPcontractive(\FFf{F})$&$\hspace{-3.pt}\quad\Leftrightarrow\quad\fa r\in\intintuptoex{p}:\abs{b_r}<p$\tabularnewline
\theoremitem{(2)}&$\FFPexpansive(\FFf{F})$&$\hspace{-3.pt}\quad\Leftrightarrow\quad\fa r\in\intintuptoex{p}:\abs{b_r}>p$.\tabularnewline
\theoremitem{(3)}&$\FFPmixed(\FFf{F})$&$\hspace{-3.pt}\quad\Leftrightarrow\quad\ex r\in\intintuptoex{p}:\abs{b_r}<p\land\ex r\in\intintuptoex{p}:\abs{b_r}>p$\tabularnewline
\theoremitem{(4)}&$\FFPcontractsdenominators(\FFf{F})$&$\hspace{-3.pt}\quad\Leftrightarrow\quad\fa r\in\intintuptoex{p}:a_r,b_r\in\Z$\tabularnewline
\theoremitem{(5)}&$\neg\FFPexpandsdenominators(\FFf{F})$.
\end{theoremtable}
\end{theorem}

\begin{proof}
Let
\begin{align}
M\ce\max\set{\abs{\frac{\abs{a_r}+p-1}{\abs{b_r}-p}}\mid r\in\intintuptoex{p}}.
\end{align}
Note that $M$ is well-defined since $\abs{b_r}\neq p$ for all $r\in\intintuptoex{p}$ by Theorem~\ref{TFuncSuit}~(2) and Lemma~\ref{LLinPolySuit}~(2).

\noindent
\proofitem[\qquad\qquad\quad ]{(1), ``$\Rightarrow$''}%
Assume to the contrary that $\abs{b_r}>p$ for some $r\in\intintuptoex{p}$ and let $m\in\N$ such that
\begin{align}
m>\frac{\abs{a_r}+p-1-r(\abs{b_r}-p)}{p(\abs{b_r}-p)}\quad\left(\Leftrightarrow\frac{\abs{b_r}(mp+r)-\abs{a_r}-p+1}{p}>mp+r\right).
\end{align}
Then,
\begin{align}
\abs{\FFf{F}(mp+r)}
&=\abs{\frac{a_r+b_r(mp+r)-(a_r+b_r(mp+r))\modulo p}{p}}\\
&\geq\frac{\abs{b_r}(mp+r)-\abs{a_r}-p+1}{p}\\
&>mp+r
=\abs{mp+r}
\end{align}
and hence $\FFf{F}$ cannot be contractive.

\noindent
\proofitem[\qquad\qquad\quad ]{(1), ``$\Leftarrow$''}%
We compute
\begin{align}
\abs{n}>M
&\Rightarrow\abs{n}\abs{\abs{b_{n\modulo p}}-p}>\abs{a_{n\modulo p}}+p-1\\
&\Leftrightarrow\abs{n}(p-\abs{b_{n\modulo p}})>\abs{a_{n\modulo p}}+p-1\\
&\Leftrightarrow\abs{n}p>\abs{a_{n\modulo p}}+\abs{b_{n\modulo p}n}+p-1\\
&\Rightarrow\abs{n}p>\abs{a_{n\modulo p}}+\abs{b_{n\modulo p}n}+(a_{n\modulo p}+b_{n\modulo p}n)\modulo p\\
&\Rightarrow\abs{n}>\abs{\frac{a_{n\modulo p}+b_{n\modulo p}n-(a_{n\modulo p}+b_{n\modulo p}n)\modulo p}{p}}=\FFf{F}(n).
\end{align}

\noindent
\proofitem[\qquad\qquad\quad ]{(2), ``$\Rightarrow$''}%
Assume to the contrary that $\abs{b_r}<p$ for some $r\in\intintuptoex{p}$ and let $m\in\N$ such that
\begin{align}
m>\frac{\abs{a_r}+p-1-r(p-\abs{b_r})}{p(p-\abs{b_r})}\quad \left(\Leftrightarrow\frac{\abs{b_r}(mp+r)+\abs{a_r}+p-1}{p}<mp+r\right).
\end{align}
Then,
\begin{align}
\abs{\FFf{F}(mp+r)}
&=\abs{\frac{a_r+b_r(mp+r)-(a_r+b_r(mp+r))\modulo p}{p}}\\
&\leq\frac{\abs{b_r}(mp+r)+\abs{a_r}+p-1}{p}\\
&<mp+r
=\abs{mp+r}
\end{align}
and hence $\FFf{F}$ cannot be expansive.

\noindent
\proofitem[\qquad\qquad\quad ]{(2), ``$\Leftarrow$''}%
We compute
\begin{align}
\abs{n}>M
&\Rightarrow\abs{n}\abs{\abs{b_{n\modulo p}}-p}>\abs{a_{n\modulo p}}+p-1\\
&\Leftrightarrow\abs{n}(\abs{b_{n\modulo p}}-p)>\abs{a_{n\modulo p}}+p-1\\
&\Leftrightarrow\abs{n}p<-\abs{a_{n\modulo p}}+\abs{b_{n\modulo p}n}-p+1\\
&\Rightarrow\abs{n}p<-\abs{a_{n\modulo p}}+\abs{b_{n\modulo p}n}-(a_{n\modulo p}+b_{n\modulo p}n)\modulo p\\
&\Rightarrow\abs{n}<\abs{\frac{a_{n\modulo p}+b_{n\modulo p}n-(a_{n\modulo p}+b_{n\modulo p}n)\modulo p}{p}}=\FFf{F}(n).
\end{align}
Note that we also could have referred to Theorem~\ref{TPolyExp} as (2), ``$\Leftarrow$'' is a special case of the theorem.

\noindent
\proofitem[\qquad\qquad\quad ]{(3)}%
Follows directly from (1) and (2).

\noindent
\proofitem[\qquad\qquad\quad ]{(4), ``$\Rightarrow$''}%
Assume to the contrary that $a_r=a/b$ with $(a,b)\in\Z\times\N$ coprime, $b_r=c/d$ with $(c,d)\in\Z\times\N$ coprime for some $r\in\intintuptoex{p}$, and $(b,d)\neq(1,1)$. Since $a/b=a_r\in\Z_p$ and $a$ and $b$ are coprime it follows that also $p$ and $b$ are coprime. Thus, there are $x,y\in\Z$ such that $xp+yb=1$ by B\'{e}zout's Lemma.

First we consider the case $d\neq1$. Let $n\ce r-rxp=ryb\in\Z$. Then, $n\in(r+p\Z)\cap b\Z$ and we get
\begin{align}
\FFf{F}(n)
&=\frac{an/b+c/d-(an/b+c/d)\modulo p}{p}
\end{align}
with $an/b-(an/b+c/d)\modulo p\in\Z$ and $c/d\in\Q\setminus\Z$ (since $d\neq1$). Thus, $\FFf{F}(n)\in\Q\setminus\Z$ but $n\in\Z$ which contradicts $\FFPcontractsdenominators(\FFf{F})$.

For the case $d=1$ let $n\ce r-rxp+p=ryb+p$.  Then, $n\in(r+p\Z)$ but $n\notin b\Z$ (since $b\neq 1$ and $p$ and $b$ are coprime) and we get $c/d-(an/b+c/d)\modulo p\in\Z$ and $an/b\in\Q\setminus\Z$ (since $n\notin b\Z$ and $a$ and $b$ are coprime). Thus, $\FFf{F}(n)\in\Q\setminus\Z$ but $n\in\Z$, which again contradicts $\FFPcontractsdenominators(\FFf{F})$.

\noindent
\proofitem[\qquad\qquad\quad ]{(4), ``$\Leftarrow$''}%
Follows directly from the definitions.

\noindent
\proofitem[\qquad\qquad\quad ]{(5)}%
Let $(a,b)\in\Z\times\N$ coprime with $a_0=a/b$ and $(c,d)\in\Z\times\N$ coprime with $b_0=c/d$. Furthermore, let $u,v\in\Z$ such that
\begin{align}
&apu+c\equiv0\modulus{\gcd(b,d)}\\
&\gcd(\gcd(b,d)v+u,d/\gcd(b,d))=1.
\end{align}
Such $u$ and $v$ exist because of two basic facts from elementary number theory, namely
\begin{align}
\fa m\in\N:\fa a,b\in\Z:&\quad\left(\ex x\in\Z:a+bx\equiv0\modulus{m}\quad\Leftrightarrow\quad\gcd(m,b)\mid a\right)\\
\fa a,b\in\Z:\fa 0\neq c\in\Z:&\quad\left(\ex x\in\Z:\gcd(a+bx,c)=1\quad\Leftrightarrow\quad\gcd(a,b,c)=1\right).
\end{align}
To see that the conditions of these two statements are satisfied, note that $\gcd(p,b)=1$ by Lemma~\ref{LLinPolySuit}, and hence $\gcd(\gcd(b,d),ap)=1\mid c$, and that $\gcd(\gcd(b,d),u)=1$ (because the equation $c+ux\equiv0\modulus{\gcd(b,d)}$ has a solution, viz. $ap$, and $\gcd(\gcd(b,d),c)=1$) and hence $\gcd(u,\gcd(b,d),d/\gcd(b,d))=1$. Using $u$ and $v$ we set
\begin{align}
n&\ce\frac{p(u+v\gcd(b,d))b/\gcd(b,d)}{d/\gcd(b,d)}\in\Q\cap p\Z_p.
\end{align}
By definition of $u$ and $v$ we get $\gcd(p(u+v\gcd(b,d))b/\gcd(b,d),d/\gcd(b,d))=1$, which implies that $n$ as given above is in lowest terms. Since $n\modulo p=0$ we get
\begin{align}
\FFf{F}(n)
&=\frac{1}{p}\left(\frac{an}{b}+\frac{c}{d}-\left(\frac{an}{b}+\frac{c}{d}\right)\modulo p\right)\\
&=\frac{1}{p}\left(\frac{(apu+c)/\gcd(b,d)+apv}{d/\gcd(b,d)}-\left(\frac{an}{b}+\frac{c}{d}\right)\modulo p\right).
\end{align}
Thus, the denominator of $\FFf{F}(n)$ in lowest terms is at most (in absolute value) $d/\gcd(b,d)$ (since $\gcd(b,d)$ divides $apu+c$ by definition of $u$) which is the denominator of $n$ in lowest terms.
\end{proof}

\noindent
For contractive $\Z$-linear-polynomial- and for expansive $(\Q\cap\Z_p)$-linear-polynomial $p$-adic systems this settles the question of ultimate periodicity on $\Q\cap\Z_p$ (cf. Conjecture~\ref{CoPeriodsC}).

\begin{corollary}
\label{CLinPolyCEM}
Let $2\leq p\in\N$, $\FFf{F}=(a_0+b_0x,\ldots,a_{p-1}+b_{p-1}x)\in\SoSystems{p}(\FFPlinearpolynomialcoefficients{\Q\cap\Z_p})$. Then,

\begin{theoremtable}[rlX]
\theoremitem{(1)}&$\fa r\in\intintuptoex{p}:a_r,b_r\in\Z\land\abs{b_r}<p$&$\hspace{-3.pt}\quad\Rightarrow\quad\FFPultimatelyperiodicon{\Q\cap\Z_p}(\FFf{F})$\tabularnewline
\theoremitem{(2)}&$\fa r\in\intintuptoex{p}:\abs{b_r}>p$&$\hspace{-3.pt}\quad\Rightarrow\quad\neg\FFPultimatelyperiodicon{\Q\cap\Z_p}(\FFf{F})$.\tabularnewline[0.5\baselineskip]
\end{theoremtable}
\end{corollary}

\begin{proof}
Follows from Theorem~\ref{TLinPolyCEM}, Lemma~\ref{LContrExp}, and the ``In particular'' part of Corollary~\ref{CPerFormula}.
\end{proof}

\myparagraphtoc{The mixed case.}
There are two mixed cases to be considered: a $(\Q\cap\Z_p)$-linear-polynomial $p$-adic system $\FFf{F}$ could be of mixed type (i.e. $\FFPmixed(\FFf{F})$) or mixing denominators (i.e. $\FFPmixesdenominators(\FFf{F})$). To the best knowledge of the author the question for ultimate periodicity on $\Q\cap\Z_p$ has not been settled for even a single such $\FFf{F}$. The most famous example is of course given by $\FFf{F}_C=(x,3x+1)$ of the Collatz conjecture which is of mixed type by Theorem~\ref{TLinPolyCEM}~(3). While it appears to be completely out of reach to answer whether or not $\FFPultimatelyperiodicon{\Q\cap\Z_2}$ holds for $\FFf{F}_C$ at the moment, the general framework that is $p$-adic systems, might provide examples that are easier to tackle without being ``trivial'' like in the contractive (e.g. $\FFf{F}_p=(x,\ldots,x)$, standard base $p$) or expansive (e.g. $\FFf{F}=(3x,3x+1)$) cases which are settled by Corollary~\ref{CLinPolyCEM}. Before we list some of these examples we will summarize what has already been achieved in this and in the previous section.

\begin{corollary}
\label{CConExpSummary}
Let $2\leq p\in\N$ and $\FFf{F}\in\SoSystems{p}(\FFPpolynomialcoefficients{\Q\cap\Z_p})$. Then,

\begin{theoremtable}
\theoremitem{(1)}&If $\FFf{F}[r]=a_r+b_rx$ with $a_r,b_r\in\Z$ and $\abs{b_r}<p$ for all $r\in\intintuptoex{p}$, then $\FFPultimatelyperiodicon{\Q\cap\Z_p}(\FFf{F})$\tabularnewline
\theoremitem{(2)}&If either $\FFf{F}[r]$ is of degree $2$ or higher or $\FFf{F}[r]=a_r+b_rx$ with $\abs{b_r}>p$ for all $r\in\intintuptoex{p}$, then $\neg\FFPultimatelyperiodicon{\Q\cap\Z_p}(\FFf{F})$.
\end{theoremtable}
\end{corollary}

\begin{proof}
Follows directly from Theorem~\ref{TPolyExp}, Theorem~\ref{TLinPolyCEM}, and Corollary~\ref{CLinPolyCEM}.
\end{proof}

\myparagraphtoc{Progress on conjectures: summary and open questions.}
Figure~\ref{FOvUltPer} below gives an overview of the status of the conjectures \ref{CoPeriodsC}~--~\ref{CoPeriodsF} and the mixed case. Related results and conjectures are given in Theorem~\ref{TConstIrr}, Conjecture~\ref{CoConstIrr}, Theorem~\ref{TOrderIrr}, Conjecture~\ref{CoOrderIrr}, Theorem~\ref{TLinPolyCEM}, and Corollary~\ref{CConExpSummary}.

\begin{figure}[H]
\centering
\begin{tikzpicture}
%\draw (0,0) ellipse (7cm and 5cm);
%\draw (0,0) ellipse (6cm and {5/7*6cm});
%\draw (0,0) ellipse (5cm and {5/7*5cm});
%\draw (0,0) ellipse (4cm and {5/7*4cm});
%\draw (0,0) ellipse (3cm and {5/7*3cm});
\draw[rounded corners=1.0cm] (-7.5cm, -0.6cm) rectangle (7.5cm, 9.2cm) {};
\draw[rounded corners=1.0cm] (-7.2cm, -0.5cm) rectangle (7.2cm, 8.2cm) {};
\draw[rounded corners=0.5cm] (-6.9cm, 6.2cm) rectangle (6.9cm, 7.2cm) {};
\draw[rounded corners=1.0cm] (-6.9cm, -0.4cm) rectangle (6.9cm, 6.1cm) {};
\draw[rounded corners=0.5cm] (-6.6cm, 4.1cm) rectangle (6.6cm, 5.1cm) {};
\draw[rounded corners=1.0cm] (-6.6cm, -0.3cm) rectangle (6.6cm, 4.0cm) {};
\draw[rounded corners=1.0cm] (-6.3cm, -0.2cm) rectangle (6.3cm, 3.0cm) {};
\draw[rounded corners=1.0cm] (-6.0cm, -0.1cm) rectangle (6.0cm, 2.0cm) {};
\draw[rounded corners=0.5cm] (-5.55cm, 0.0cm) rectangle (5.55cm, 1.0cm) {};
\node[text width=5.0cm,align=left] at (-4.25cm,8.7cm) {$\smash{\FFPpolynomial(\FFf{F})}$\vphantom{A}};
\node[text width=5.0cm,align=right] at (4.25cm,8.7cm) {\smash{Thm.~\ref{TGenPolyUPer}: $\neg\FFPultimatelyperiodicon{\Q_\cap\Z_p}(\FFf{F})$}\vphantom{A}};
\node[text width=1.0cm,align=center] at (-0.5cm,8.7cm) {\smash{(I)}\vphantom{A}};
\node[text width=5.0cm,align=left] at (-3.95cm,7.7cm) {$\smash{\FFPpolynomialcoefficients{\Q\cap\Z_p}(\FFf{F})}$\vphantom{A}};
\node[text width=5.0cm,align=right] at (3.95cm,7.7cm) {\smash{Con.~\ref{CoPeriodsE}: $\neg\FFPultimatelyperiodicon{\Q_\cap\Z_p}(\FFf{F})$}\vphantom{A}};
\node[text width=1.0cm,align=center] at (-0.5cm,7.7cm) {\smash{(II)}\vphantom{A}};
\node[text width=5.0cm,align=left] at (-3.8cm,6.7cm) {$\abs{b_r}>p$ if $\FFf{F}[r]$ \rlap{linear}\vphantom{A}};
\node[text width=5.0cm,align=right] at (3.8cm,6.7cm) {\smash{Cor.~\ref{CConExpSummary}: $\neg\FFPultimatelyperiodicon{\Q_\cap\Z_p}(\FFf{F})$}\vphantom{A}};
\node[text width=1.0cm,align=center] at (-0.5cm,6.7cm) {\smash{(III)}\vphantom{A}};
\node[text width=5.0cm,align=left] at (-3.65cm,5.6cm) {$\smash{\FFPlinearpolynomialcoefficients{\Q\cap\Z_p}(\FFf{F})}$\vphantom{A}};
\node[text width=5.0cm,align=right] at (3.65cm,5.6cm) {\smash{Con.~\ref{CoPeriodsD}: $\neg\FFPultimatelyperiodicon{\Q_\cap\Z_p}(\FFf{F})$}\vphantom{A}};
\node[text width=1.0cm,align=center] at (-0.5cm,5.6cm) {\smash{(IV)}\vphantom{A}};
\node[text width=5.0cm,align=left] at (-3.5cm,4.6cm) {$\smash{\abs{b_0},\ldots,\abs{b_{p-1}}>p}$\vphantom{A}};
\node[text width=5.0cm,align=right] at (3.5cm,4.6cm) {\smash{Cor.~\ref{CConExpSummary}: $\neg\FFPultimatelyperiodicon{\Q_\cap\Z_p}(\FFf{F})$}\vphantom{A}};
\node[text width=1.0cm,align=center] at (-0.5cm,4.6cm) {\smash{(V)}\vphantom{A}};
\node[text width=5.0cm,align=left] at (-3.35cm,3.5cm) {$\smash{\abs{b_0\cdots b_{p-1}}<p^p}$\vphantom{A}};
\node[text width=5.0cm,align=right] at (3.35cm,3.5cm) {\smash{Con.~\ref{CoPeriodsD}: $\neg\FFPultimatelyperiodicon{\Q_\cap\Z_p}(\FFf{F})$}\vphantom{A}};
\node[text width=1.0cm,align=center] at (-0.5cm,3.5cm) {\smash{(VI)}\vphantom{A}};
\node[text width=5.0cm,align=left] at (-3.05cm,2.5cm) {$\smash{b_0\cdots b_{p-1}\in\Z}$\vphantom{A}};
\node[text width=5.0cm,align=right] at (3.05cm,2.5cm) {\smash{Con.~\ref{CoPeriodsD}: $\FFPultimatelyperiodicon{\Q_\cap\Z_p}(\FFf{F})$}\vphantom{A}};
\node[text width=1.0cm,align=center] at (-0.5cm,2.5cm) {\smash{(VII)}\vphantom{A}};
\node[text width=5.0cm,align=left] at (-2.75cm,1.5cm) {$\smash{\FFPlinearpolynomialcoefficients{\Z}(\FFf{F})}$\vphantom{A}};
\node[text width=5.0cm,align=right] at (2.75cm,1.5cm) {\smash{Con.~\ref{CoPeriodsC}: $\FFPultimatelyperiodicon{\Q_\cap\Z_p}(\FFf{F})$}\vphantom{A}};
\node[text width=1.0cm,align=center] at (-0.5cm,1.5cm) {\smash{(VIII)}\vphantom{A}};
\node[text width=5.0cm,align=left] at (-2.45cm,0.5cm) {$\smash{\abs{b_0},\ldots,\abs{b_{p-1}}<p}$\vphantom{A}};
\node[text width=5.0cm,align=right] at (2.45cm,0.5cm) {\smash{Cor.~\ref{CConExpSummary}: $\FFPultimatelyperiodicon{\Q_\cap\Z_p}(\FFf{F})$}\vphantom{A}};
\node[text width=1.0cm,align=center] at (-0.5cm,0.5cm) {\smash{(IX)}\vphantom{A}};
\end{tikzpicture}
\caption{Overview of settled and open cases on the question of ultimate periodicity on $\Q\cap\Z_p$ of a $p$-adic system $\FFf{F}$, where $b_r$ is the linear coefficient of the polynomial $\FFf{F}[r]$ for $r\in\underline{\smash{p}}$.}
\label{FOvUltPer}
\end{figure}
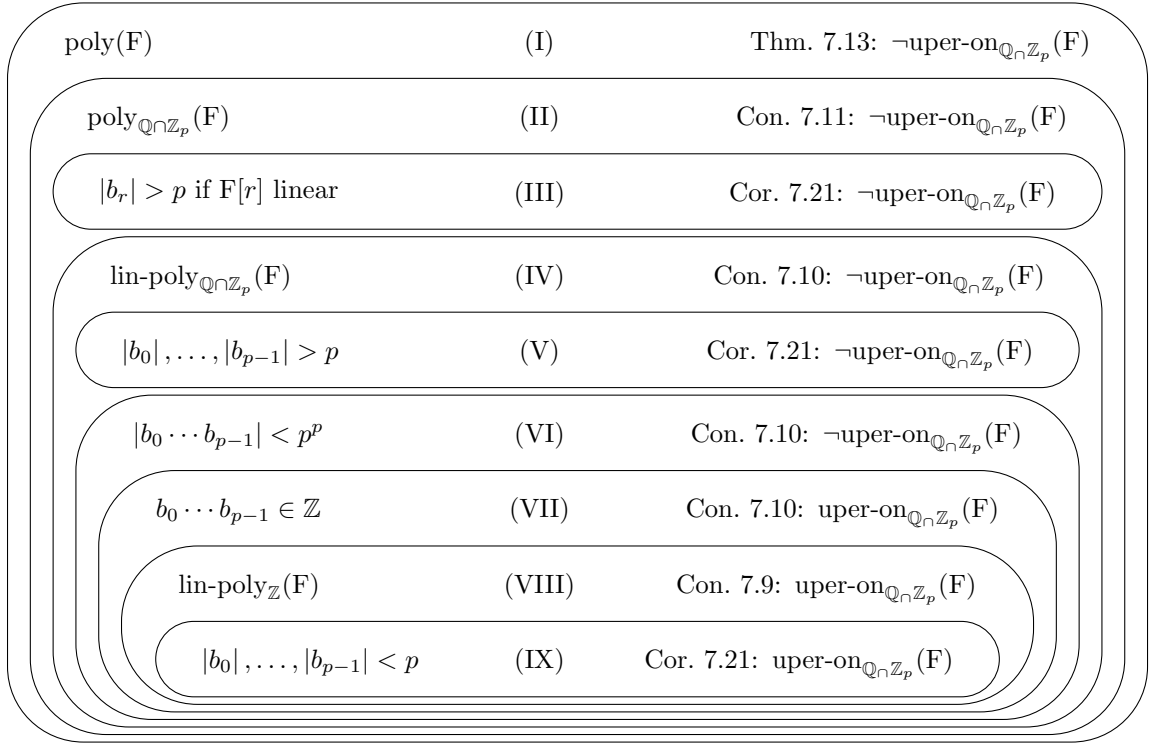

\noindent
Referring to the different regions in Figure~\ref{FOvUltPer} we give a list of possibly interesting examples of $p$-adic systems, some of which only just fall into the respective cases.

\bgroup
\setlength\LTleft{0pt}
\setlength\LTright{0pt}
\begin{longtable}{@{}ll@{\extracolsep{\fill}}l@{}}
(I)&$\FFf{F}=\left(a_0+b_0x,a_1+b_1x\right)$, where $\set{1}\neq\set{a_0,b_0,a_1,b_1}\subseteq\set{1,\sqrt{17},-\sqrt{17}}$\tabularnewline
&$\FFf{F}=\left(a_0+b_0x,\ldots,a_2+b_2x\right)$, where $\set{1}\neq\set{a_0,b_0,\ldots,a_2,b_2}\subseteq\set{1,\sqrt{-2},-\sqrt{-2}}$\tabularnewline
&$\FFf{F}=(a_0+b_0x,\ldots,a_4+b_4x)$, where $\set{1}\neq\set{a_0,b_0,\ldots,a_4,b_4}\subseteq\set{1,\I,-\I}$\tabularnewline
(II)&$\FFf{F}=(x^2+x)\cdot(x)^{p-1}$\tabularnewline
&$\FFf{F}=(x^{ap}+x)^{p-1}\cdot((p-1)x)$, where $a\in\N$ (say, $a=1000$)\tabularnewline
(III)&$\FFf{F}=(3x,x^2+x)$\tabularnewline
&$\FFf{F}=(x^2+x,3x)$\tabularnewline
&$\FFf{F}=(x^2+x,x^2+x)$\tabularnewline
(IV)&$\FFf{F}=((p^p+1)x)\cdot(x)^{p-1}$\tabularnewline
&$\FFf{F}=((p+(p=2\;?\;3:1))x)^{p-1}((p-1)x)$\tabularnewline
(V)&$\FFf{F}=((p+1)x)^p$\tabularnewline
&$\FFf{F}=((p^p+1)x)^p$\tabularnewline
(VI)&$\FFf{F}=(1/(p+1)x,(p-1)x)\cdot(x)^{p-2}$\tabularnewline
&$\FFf{F}=(1/(p+1)x)^{\floor{p/2}+1}\cdot((p+1)x)^{\ceiling{p/2}-1}$\tabularnewline
(VII)&$\FFf{F}=(1/(p+1)x,(p+1)x)\cdot(x)^{p-2}$\tabularnewline
&$\FFf{F}=(1/(p+1)x)^{\floor{p/2}}\cdot((p+1)x)^{\ceiling{p/2}}$\tabularnewline
(VIII)&$\FFf{F}=((p+1)x)\cdot(x)^{p-1}$\tabularnewline
&$\FFf{F}=((p+1)x)^{\floor{p/2}}\cdot((p-1)x)^{\ceiling{p/2}}$\tabularnewline
&$\FFf{F}=((p^p-1)x)\cdot(x)^{p-1}$\tabularnewline
(IX)&$\FFf{F}=(x)^p$ (standard base $p$)\tabularnewline
&$\FFf{F}=((p-1)x)^p$.
\end{longtable}
\egroup

\noindent
Note that (cf. Example~\ref{EHensel}),
\begin{align}
\smash{\sqrt{17}}&\in\set{\ldots0010111,\ldots1101001}\subseteq\Z_2\\
\smash{\sqrt{-2}}&\in\set{\ldots0200211,\ldots2022012}\subseteq\Z_3\\
\I&\in\set{\ldots2431212,\ldots2013233}\subseteq\Z_5.
\end{align}

\myparagraphtoc{Generalizations.}
The previous subsections of this section deal with the question of ultimate periodicity on a specific set ($\Q\cap\Z_p$) for a specific kind of $p$-adic systems (mostly $(\Q\cap\Z_p)$-linear-polynomial ones). While the set $\Q$ of rational numbers is certainly special among all dense subsets of $\Z_p$ and polynomial functions are also special among all functions on $\Z_p$, there is no obvious reason why this combination of ``subset of $\Z_p$'' and ``class of functions on $\Z_p$'' should be the only one to result in interesting patterns and relations. Analyzing previous results, the following generalization appears to be natural: for $2\leq p\in\N$ and $A\subseteq\SoSystems{p}$ let
\begin{align}
\label{DFFFSoPUAPP}
\FFFSoPP(A)\ce\bigcup_{\FFf{F}\in A}\FFFSoPP(\FFf{F}),\;
\FFFSoUPP(A)\ce\bigcup_{\FFf{F}\in A}\FFFSoUPP(\FFf{F}),\;
\FFFSoAPP(A)\ce\bigcup_{\FFf{F}\in A}\FFFSoAPP(\FFf{F}).
\end{align}
We say that an element $\FFf{F}$ of $A$ \emph{generates the periodic, ultimately periodic, or aperiodic points of $A$}, if $\FFFSoPP(\FFf{F})=\FFFSoPP(A)$, $\FFFSoUPP(\FFf{F})=\FFFSoUPP(A)$, or $\FFFSoAPP(\FFf{F})=\FFFSoAPP(A)$ respectively, and denote by $\FFFSoGenPP(A)$, $\FFFSoGenUPP(A)$, and $\FFFSoGenAPP(A)$ the respective sets of all these $\FFf{F}\in A$.

\noindent
By the ``In particular'' part of Corollary~\ref{CPerFormula} we get
\begin{align}
\FFFSoUPP(\SoSystems{p}(\FFPlinearpolynomialcoefficients{\Q\cap\Z_p}))=\Q\cap\Z_p
\end{align}
and Conjecture~\ref{CoPeriodsD} can be expressed as
\begin{align}
\FFFSoGenUPP(\SoSystems{p}(\FFPlinearpolynomialcoefficients{\Q\cap\Z_p}))&\overset{?}{=}\smash{\setl{\FFf{F}\in\SoSystems{p}(\FFPlinearpolynomialcoefficients{\Q\cap\Z_p})\;\big\vert}}\\
\nonumber
&\phantom{\vphantom{a}=\vphantom{a}}\qquad\FFf{F}=(a_0+b_0x,\ldots,a_{p-1}+b_{p-1}x)\\
\nonumber
&\phantom{\vphantom{a}=\vphantom{a}}\qquad b_0\cdots b_{p-1}\in\Z\\
\nonumber
&\phantom{\vphantom{a}=\vphantom{a}}\qquad\smash{\setr{\abs{b_0\cdots b_{p-1}}<p^p\vphantom{\FFPlinearpolynomialcoefficients{\Q\cap\Z_p}}}}.
\end{align}
Related results and conjectures are given in Theorem~\ref{TConstIrr}, Conjecture~\ref{CoConstIrr}, Theorem~\ref{TOrderIrr}, Conjecture~\ref{CoOrderIrr}, Theorem~\ref{TLinPolyCEM}, and Corollary~\ref{CLinPolyCEM}. Generalizing Conjecture~\ref{CoPeriodsD} one might try to study any of the sets $\FFFSoPP(A)$, $\FFFSoUPP(A)$, $\FFFSoAPP(A)$, $\FFFSoGenPP(A)$, $\FFFSoGenUPP(A)$, or $\FFFSoGenAPP(A)$ for other classes $A$ of $p$-adic systems, such as $\SoSystems{p}(\FFPpolynomialcoefficientsdegree{C}{D})$ where $C$ is a (possibly dense) subset of $\Z_p$, and $D\subseteq\Nz$ is a set of allowed degrees, some class of rational functions on $\Z_p$ (Section~\ref{SCharacterization}), $p$-adic systems defined by certain permutation polynomials (Section~\ref{SPermPoly}), etc.

\section{Permutation polynomials and trees of cycles}
\label{SPermPoly}
In Section~\ref{SCharacterization} we proved that ``almost all'' $\Z_p$-polynomial $p$-fibred functions are actually $p$-fibred systems (cf. Theorem~\ref{TFuncSuit} and Theorem~\ref{TPropPolyF}) giving us a multitude of examples. In addition to this class, the different interpretations of $p$-adic systems discussed in Section~\ref{SInterpretations} allow us to find even more $p$-adic systems that are essentially different from those we already know. This new class is again defined by polynomials but in a very different way. It turns out that $p$-adic permutations, which we proved to be just a different interpretation of $p$-adic systems (cf. Theorem~\ref{TPermEq}), can be polynomial functions and are thus defined by a single polynomial in $\Z_p[x]$ in this case (in the following we will use the terms ``polynomial'' and ``polynomial function'' interchangeably). These polynomials are exactly what is commonly known as permutation polynomials. A polynomial $f\in\Z_p[x]$ is called a \emph{$p$-permutation polynomial}, where $2\leq p\in\N$, if the following holds:
\begin{align}
\label{EPermPolyA}
&f_1\text{ is the identity on }\Z_p\slash p\Z_p\\
\label{EPermPolyB}
\fa k\in\N:\vphantom{a}&f_k\text{ is bijective},
\end{align}
where
\begin{align}
\label{Dpermpolyk}
f_k:\Z_p\slash p^k\Z_p&\to\Z_p\slash p^k\Z_p.\\
\nonumber
x+p^k\Z_p&\mapsto f(x)+p^k\Z_p
\end{align}
Please note that we slightly adapted the usual notion of permutation polynomials (cf. \cite{KellerOlsen:1968,LidlMullen:1988,LidlMullen:1993,Shallue:2012}) which usually are required only to satisfy the second condition but not the first. We need the first condition because $p$-adic permutations must satisfy Eqn.~(\ref{EPermA}). Using the following lemma which can also be found in \cite{KellerOlsen:1968}, it is very easy to test whether a given polynomial is a $p$-permutation polynomial.

\begin{lemma}
\label{LCharacPermPoly}
Let $2\leq p\in\N$ and $f\in\Z_p[x]$. Then $f$ is a $p$-permutation polynomial if and only if $f_1$ is the identity on $\Z_p\slash p\Z_p$ and $f_2$ is bijective.
\end{lemma}

If $\pi\ce f$, then Eqn.~(\ref{EPermPolyA}) and Eqn.~(\ref{EPermPolyB}) are clearly equivalent to Eqn.~(\ref{EPermA}) and Eqn.~(\ref{EPermB}) respectively which proves the following theorem.

\begin{theorem}
\label{TCharacPermPoly}
Let $2\leq p\in\N$ and $f\in\Z_p[x]$. Then $f$ is a $p$-permutation polynomial if and only if $f$ is a $p$-adic permutation (i.e. $f\in\SoPermutations{p}$). 
\end{theorem}

\noindent
Theorem~\ref{TPermEq} thus implies that for every $p$-permutation polynomial $f$ and every $p$-adic system $\FFf{G}$ there is a $p$-adic system $\FFf{F}$ such that $f=\pi_{\FFf{F},\FFf{G}}$. A natural follow-up question to this observation is whether all $p$-adic permutations are actually polynomial functions, i.e. $p$-permutation polynomials. The following example shows that this is not the case even if $\FFf{F}$ and $\FFf{G}$ are $\Z$-linear-polynomial.

\begin{example}
\label{ENotPermPoly}
Let $\FFf{F}\ce\FFf{F}_C=(x,3x+1)$, $\FFf{G}\ce\FFf{F}_2=(x,x-1)$, and $\pi\ce\pi_{\FFf{F},\FFf{G}}$. Then $\pi\in\SoPermutations{2}$ (since $\FFf{F}$ and $\FFf{G}$ are in $\SoSystems{2}$), $\pi(n)=-n/3$ for all $n\in\set{2^k\mid k\in\N}$, and $\pi(n)=-23n/9$ for all $n\in\set{2^k+2^{k-1}\mid k\in\N}$. If we assume that $\pi$ is a polynomial function then both $\pi(x)+x/3$ and $\pi(x)+23x/9$ have infinitely many roots. This implies (since $\Q_2$ does not contain zero divisors) that both polynomials are equal to $0$. But then $x/3$ is equal to $23x/9$ (as polynomials) which is a contradiction.
\end{example}

\noindent
In the other direction one might ask if every $p$-permutation polynomial $f$ can be written as $f=\pi_{\FFf{F},\FFf{G}}$ where both $\FFf{F}$ and $\FFf{G}$ are $\Z_p$-polynomial $p$-adic systems. This is not the case either, but there does not seem to be a proof as easy as the one given above for the other direction. Instead we will make use of the tree of cycles introduced at the end of Section~\ref{SInterpretations} to demonstrate that certain $p$-permutation polynomials cannot be represented in this way. The method we will use is quite general in nature and can probably be adapted to show that other classes of $p$-adic systems which may be found in the future, are also distinct from classes known up to this point.

\myparagraphtoc{Cycle trees.}
For a general $p$-adic permutation $\pi\in\SoPermutations{p}$ we recall some of the properties of its tree of cycles $\mathcal{G}(\pi)=(\mathcal{V}(\pi),\mathcal{E}(\pi))$ and the corresponding edge labeling $c(\pi):\mathcal{E}(\pi)\to\intintuptoex{p}$ as given in Corollary~\ref{CCycles}:

\begin{theoremtable}
$\bullet$&$\mathcal{G}(\pi)$ is a directed, infinite, rooted tree,\tabularnewline
$\bullet$&the out-degrees of all vertices are contained in $\intint{1}{p}$ and the out-degree of the root is $p$,\tabularnewline
$\bullet$&the labels of all outgoing edges of a given vertex sum up to $p$.\tabularnewline[0.5\baselineskip]
\end{theoremtable}

\noindent
We call any edge labeled tree satisfying all of the above properties a \emph{$p$-cycle tree}.

The first natural question in the context of $p$-adic systems that arises is whether every $p$-cycle tree can be realized as the tree of cycles of some $p$-adic permutation. This is indeed the case as the following theorem shows.

\begin{theorem}
\label{TPermFromTree}
Let $2\leq p\in\N$ and $\mathcal{G}=(\mathcal{V},\mathcal{E})$ be a $p$-cycle tree with edge labeling $c:\mathcal{E}\to\intint{1}{p}$. Let $o\in\mathcal{V}$ denote the root of $\mathcal{G}$ and for every vertex $o\neq v\in\mathcal{V}$ let $p(v)\in\mathcal{V}$ denote the predecessor of $v$, i.e. the unique vertex satisfying $p(v)v\in\mathcal{E}$. Furthermore, for every $v\in\mathcal{V}$ let $\Sf{S}(v)\in\CoSequences(\SPboundedby{\mathcal{V}})$ be a minimal sequence of all successors of $v$ in an arbitrary but fixed order, i.e. $v\Sf{S}(v)[i]\in\mathcal{E}$ for all $i\in\intintuptoex{\abs{\Sf{S}(v)}}$, $\set{u\in\mathcal{V}\mid vu\in\mathcal{E}}=\set{\Sf{S}(v)[i]\mid i\in\intintuptoex{\abs{\Sf{S}(v)}}}$, and $\abs{\set{u\in\mathcal{V}\mid vu\in\mathcal{E}}}=\abs{\Sf{S}(v)}$. We label the vertices of $\mathcal{G}$ by the following function (which we will also call $c$ as there is no danger of ambiguity),
\begin{align}
c:\mathcal{V}&\to\intint{1}{p}.\\
\nonumber
v&\mapsto
\begin{cases}
1&\text{if }v=o\\
c(p(v)v)&\text{if }v\neq o
\end{cases}
\intertext{Furthermore, we define the functions $o$ (which stands for ``offset'' and, again, is not at risk of being confused with the identically named root), $\Sf{P}$ (standing for ``positions''), and $\Sf{E}$ (standing for ``entries''),}
o:\mathcal{V}&\to\intintuptoex{p}\\
\nonumber
v&\mapsto
\begin{cases}
0&\text{if }v=o\\
\sum_{i=0}^{\SFposition{\Sf{S}(p(v))}{v}-1}c(\Sf{S}(p(v))[i])&\text{if }v\neq o
\end{cases}\\
\Sf{P}:\mathcal{V}&\to\CoSequences(\SPboundedby{\Nz})\\
\nonumber
v&\mapsto
\begin{cases}
(0)&\text{if }v=o\\
\prod_{i=0}^{c(v)-1}\left(\Sf{P}(p(v))+(o(v)+i)p^{k(v)-1}\right)&\text{if }v\neq o
\end{cases}\\
\Sf{E}:\mathcal{V}&\to\CoSequences(\SPboundedby{\intintuptoex{p}})\\
\nonumber
v&\mapsto
\begin{cases}
(0)&\text{if }v=o\\
\prod_{i=0}^{c(v)-1}(o(v)+i)^{\abs{\Sf{E}(p(v))}-1}\cdot(o(v)+(i+1)\modulo c(v))&\text{if }v\neq o
\end{cases}
\intertext{where the products mean products of sequences (i.e. their concatenation) and $k(v)$ denotes the layer of $v$, i.e. its distance from the root, i.e.,}
k:\mathcal{V}&\to\Nz.\\
\nonumber
v&\mapsto
\begin{cases}
0&\text{if }v=o\\
k(p(v))+1&\text{if }v\neq o
\end{cases}
\end{align}
Using $P$ and $E$ we define an infinite $p$-digit table $\STf{D}\in\SoTables{p}$ with domain $\Z_p$ and block property by writing the entries in $E$ at positions $P$ in the following way:
\begin{align}
\STf{D}\left[\Sf{P}(v)[i]+ap^{k(v)}\right]\left[\vphantom{p^{k(v)}}k(v)-1\right]=\Sf{E}(v)[i]
\end{align}
for all $o\neq v\in\mathcal{V}$, $i\in\intintuptoex{\abs{\Sf{P}(v)}}$, and $a\in\Z_p$. Let $\FFf{F}\in\SoSystems{p}(\FFPcanonicalform)$ be the $p$-adic system corresponding to $\STf{D}$ according to Eqn.~(\ref{EEqSummary}), $\pi\ce\pi_{\FFf{F},(x)^p}\in\SoPermutations{p}$, and
\begin{align}
\varphi:\mathcal{V}&\to\mathcal{V}(\pi).\\
\nonumber
v&\mapsto\left(k(v),[([\Sf{P}(v)[0]],\ldots,[\Sf{P}(v)[\abs{\Sf{P}(v)}-1]])]_{\sim_\sigma}\right)
\end{align}
Then, $\varphi$ is an isomorphism between $\mathcal{G}$ and $\mathcal{G}(\pi)$ that respects the labelings $c$ and $c(\pi)$. In particular, $(\mathcal{G},c)$ and $(\mathcal{G}(\pi),c(\pi))$ are isomorphic.
\end{theorem}

\noindent
To illustrate the workings of the theorem consider the example given in Figure~\ref{FPermFromTree} below.

\bgroup
\newcommand{\hdummy}{{\rule{2pt}{0pt}}}
\newcommand{\vdummy}{{\rule[-1.5pt]{0pt}{\fontcharht\font`0+3pt}}}
\newcommand{\vdummys}{{\rule[-0.35pt]{0pt}{\fontcharht\font`0+0.7pt}}}
\newcommand{\bdummy}{{\rule{0.4pt}{\fontcharht\font`0}}}
\definecolor{colorAl}{rgb}{0.7,0.7,0.7}
\definecolor{colorBl}{rgb}{0.8,0.8,0.8}
\definecolor{colorCl}{rgb}{0.9,0.9,0.9}
\definecolor{colorAd}{rgb}{0.4,0.4,0.4}
\definecolor{colorBd}{rgb}{0.5,0.5,0.5}
\definecolor{colorCd}{rgb}{0.6,0.6,0.6}

\def\temptable#1#2#3#4#5#6#7{
\begin{tabular}{cccccccccccccccccccccccccccccccccccccccccccc}
\vdummy&#7\bdummy&#7 &#7\hdummy&#7\hdummy
&#70&#70&#70&#70&#70&#70&#70&#70&#70&#70&#71&#71&#71&#71&#71&#71&#71&#71&#71&#71&#72
&#72&#72&#72&#72&#72&#72&#7\hdummy&#7\hdummy&#7 &#7 &#7 &#7 &#7 &#7 &#7 &#7 &#7\bdummy&\tabularnewline
\vdummy&#7\bdummy&#7 &#7&#7
&#70&#71&#72&#73&#74&#75&#76&#77&#78&#79&#70&#71&#72&#73&#74&#75&#76&#77&#78&#79&#70
&#71&#72&#73&#74&#75&#76&#7\hdummy&#7\hdummy&#7 &#7 &#7 &#7 &#7 &#7 &#7 &#7 &#7\bdummy&\tabularnewline
\hline
\vdummy&#7\bdummy&#7 &#7&#7
&#70&#71&#72&#70&#71&#72&#70&#71&#72&#70&#71&#72&#70&#71&#72&#70&#71&#72&#70&#71&#72
&#70&#71&#72&#70&#71&#72&#7\hdummy&#7\hdummy&#7 &#7 &#7 &#7 &#7 &#7 &#7 &#7 &#7\bdummy&\tabularnewline
\vdummy&#7\bdummy&#7 &#7&#7
&#70&#70&#70&#71&#71&#71&#72&#72&#72&#70&#70&#70&#71&#71&#71&#72&#72&#72&#70&#70&#70
&#71&#71&#71&#72&#72&#72&#7&#7&#7 &#7 &#7 &#7 &#7 &#7 &#7 &#7 &#7\bdummy&\tabularnewline
\vdummy&#7\bdummy&#7 &#7&#7
&#70&#70&#70&#70&#70&#70&#70&#70&#70&#71&#71&#71&#71&#71&#71&#71&#71&#71&#72&#72&#72
&#72&#72&#72&#72&#72&#72&#7&#7&#7 &#7 &#7 &#7 &#7 &#7 &#7 &#7 &#7\bdummy&\tabularnewline
\hline
\hline
\vdummy&#71&#7 &#7&#7
&#40&#5 &#6 &#10&#2 &#3 &#10&#2 &#3 &#10&#2 &#3 &#10&#2 &#3 &#10&#2 &#3 &#10&#2 &#3 
&#10&#2 &#3 &#10&#2 &#3 &#7&#7&#7 &#7 &#7 &#7 &#7 &#7 &#7 &#7 &#70&\tabularnewline
\vdummy&#72&#7 &#7&#7
&#4 &#51&#6 &#1 &#21&#3 &#1 &#21&#3 &#1 &#21&#3 &#1 &#21&#3 &#1 &#21&#3 &#1 &#21&#3 
&#1 &#21&#3 &#1 &#21&#3 &#7&#7&#7 &#7 &#7 &#7 &#7 &#7 &#7 &#7 &#71&\tabularnewline
\vdummy&#73&#7 &#7&#7
&#4 &#5 &#62&#1 &#2 &#32&#1 &#2 &#32&#1 &#2 &#32&#1 &#2 &#32&#1 &#2 &#32&#1 &#2 &#32
&#1 &#2 &#32&#1 &#2 &#32&#7&#7&#7 &#7 &#7 &#7 &#7 &#7 &#7 &#7 &#72&\tabularnewline
\vdummy&#7\bdummy&#7 &#7&#7
&#40&#51&#62&#10&#21&#32&#10&#21&#32&#10&#21&#32&#10&#21&#32&#10&#21&#32&#10&#21&#32
&#10&#21&#32&#10&#21&#32&#7&#7&#7 &#7 &#7 &#7 &#7 &#7 &#7 &#7 &#7\bdummy&\tabularnewline
\hline
\vdummy&#74&#7 &#7&#7
&#41&#5 &#6 &#42&#5 &#6 &#40&#5 &#6 &#11&#2 &#3 &#12&#2 &#3 &#10&#2 &#3 &#11&#2 &#3 
&#12&#2 &#3 &#10&#2 &#3 &#7&#7&#7 &#7 &#7 &#7 &#7 &#7 &#71&#72&#70&\tabularnewline
\vdummy&#75&#7 &#7&#7
&#4 &#50&#6 &#4 &#5 &#6 &#4 &#5 &#6 &#1 &#20&#3 &#1 &#2 &#3 &#1 &#2 &#3 &#1 &#20&#3 
&#1 &#2 &#3 &#1 &#2 &#3 &#7&#7&#7 &#7 &#7 &#7 &#7 &#7 &#7 &#7 &#70&\tabularnewline
\vdummy&#76&#7 &#7&#7
&#4 &#5 &#6 &#4 &#52&#6 &#4 &#51&#6 &#1 &#2 &#3 &#1 &#22&#3 &#1 &#21&#3 &#1 &#2 &#3 
&#1 &#22&#3 &#1 &#21&#3 &#7&#7&#7 &#7 &#7 &#7 &#7 &#7 &#7 &#72&#71&\tabularnewline
\vdummy&#77&#7 &#7&#7
&#4 &#5 &#61&#4 &#5 &#60&#4 &#5 &#6 &#1 &#2 &1#3&#1 &#2 &#30&#1 &#2 &#3 &#1 &#2 &#31
&#1 &#2 &#30&#1 &#2 &#3 &#7&#7&#7 &#7 &#7 &#7 &#7 &#7 &#7 &#71&#70&\tabularnewline
\vdummy&#78&#7 &#7&#7
&#4 &#5 &#6 &#4 &#5 &#6 &#4 &#5 &#62&#1 &#2 &#3 &#1 &#2 &#3 &#1 &#2 &#32&#1 &#2 &#3 
&#1 &#2 &#3 &#1 &#2 &#32&#7&#7&#7 &#7 &#7 &#7 &#7 &#7 &#7 &#7 &#72&\tabularnewline
\vdummy&#7\bdummy&#7 &#7&#7
&#41&#50&#61&#42&#52&#60&#40&#51&#62&#11&#20&#31&#12&#22&#30&#10&#21&#32&#11&#20&#31
&#12&#22&#30&#10&#21&#32&#7&#7&#7 &#7 &#7 &#7 &#7 &#7 &#7 &#7 &#7\bdummy&\tabularnewline
\hline
\vdummy&#79&#7 &#7&#7
&#40&#5 &#6 &#40&#5 &#6 &#41&#5 &#6 &#41&#5 &#6 &#41&#5 &#6 &#42&#5 &#6 &#42&#5 &#6 
&#42&#5 &#6 &#40&#5 &#6 &#7&#7&#70&#70&#71&#71&#71&#72&#72&#72&#70&\tabularnewline
\vdummy&#71&#70&#7&#7
&#4 &#51&#6 &#4 &#5 &#6 &#4 &#5 &#6 &#4 &#52&#6 &#4 &#5 &#6 &#4 &#5 &#6 &#4 &#50&#6 
&#4 &#5 &#6 &#4 &#5 &#6 &#7&#7&#7 &#7 &#7 &#7 &#7 &#7 &#71&#72&#70&\tabularnewline
\vdummy&#71&#71&#7&#7
&#4 &#5 &#6 &#4 &#50&#6 &#4 &#50&#6 &#4 &#5 &#6 &#4 &#5 &#6 &#4 &#5 &#6 &#4 &#5 &#6 
&#4 &#5 &#6 &#4 &#5 &#6 &#7&#7&#7 &#7 &#7 &#7 &#7 &#7 &#7 &#70&#70&\tabularnewline
\vdummy&#71&#72&#7&#7
&#4 &#5 &#6 &#4 &#5 &#6 &#4 &#5 &#6 &#4 &#5 &#6 &#4 &#51&#6 &#4 &#52&#6 &#4 &#5 &#6 
&#4 &#52&#6 &#4 &#51&#6 &#7&#7&#7 &#7 &#7 &#7 &#7 &#71&#72&#72&#71&\tabularnewline
\vdummy&#71&#73&#7&#7
&#4 &#5 &#60&#4 &#5 &#61&#4 &#5 &#6 &#4 &#5 &#61&#4 &#5 &#62&#4 &#5 &#6 &#4 &#5 &#62
&#4 &#5 &#60&#4 &#5 &#6 &#7&#7&#7 &#7 &#7 &#70&#71&#71&#72&#72&#70&\tabularnewline
\vdummy&#71&#74&#7&#7
&#4 &#5 &#6 &#4 &#5 &#6 &#4 &#5 &#60&#4 &#5 &#6 &#4 &#5 &#6 &#4 &#5 &#6 &#4 &#5 &#6 
&#4 &#5 &#6 &#4 &#5 &#6 &#7&#7&#7 &#7 &#7 &#7 &#7 &#7 &#7 &#7 &#70&\tabularnewline
\vdummy&#71&#75&#7&#7
&#4 &#5 &#6 &#4 &#5 &#6 &#4 &#5 &#6 &#4 &#5 &#6 &#4 &#5 &#6 &#4 &#5 &#61&#4 &#5 &#6 
&#4 &#5 &#6 &#4 &#5 &#6 &#7&#7&#7 &#7 &#7 &#7 &#7 &#7 &#7 &#7 &#71&\tabularnewline
\vdummy&#71&#76&#7&#7
&#4 &#5 &#6 &#4 &#5 &#6 &#4 &#5 &#6 &#4 &#5 &#6 &#4 &#5 &#6 &#4 &#5 &#6 &#4 &#5 &#6 
&#4 &#5 &#6 &#4 &#5 &#62&#7&#7&#7 &#7 &#7 &#7 &#7 &#7 &#7 &#7 &#72&\tabularnewline
\vdummy&#7\bdummy&#7 &#7&#7
&#40&#51&#60&#40&#50&#61&#41&#50&#60&#41&#52&#61&#41&#51&#62&#42&#52&#61&#42&#50&#62
&#42&#52&#60&#40&#51&#62&#7&#7&#7 &#7 &#7 &#7 &#7 &#7 &#7 &#7 &#7\bdummy&
\end{tabular}
}

\begin{figure}[H]
\bgroup
\def\arraystretch{0.0}
\setlength\tabcolsep{0pt}
\begin{minipage}{5.6855cm}
\includegraphics[width=5.6855cm]{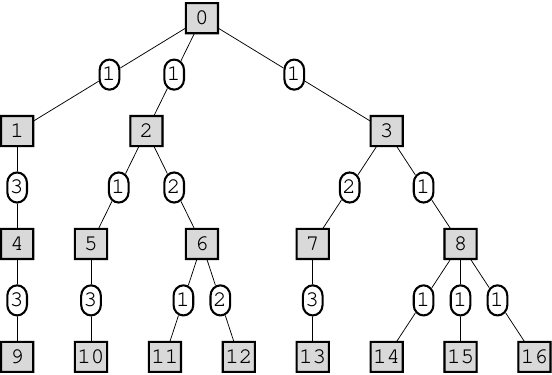}
\end{minipage}
$\rightarrow$
\temptable{\cellcolor{colorAd}}{\cellcolor{colorBd}}{\cellcolor{colorCd}}{\cellcolor{colorAl}}{\cellcolor{colorBl}}{\cellcolor{colorCl}}{}\llap{\temptable{\phantom}{\phantom}{\phantom}{\phantom}{\phantom}{\phantom}{\phantom}}
$\rightarrow$
\begin{tabular}{rc|cccc}
\vdots&\hdummy&\hdummy&\vdots&\vdots&\vdots\tabularnewline
\phantom{0}&\tabularnewline
\vdummys0&&&\cellcolor{colorAl}0&\cellcolor{colorAl}1&\cellcolor{colorAl}0\tabularnewline
\vdummys1&&&\cellcolor{colorBl}1&\cellcolor{colorBl}0&\cellcolor{colorBl}1\tabularnewline
\vdummys2&&&\cellcolor{colorCl}2&\cellcolor{colorCl}1&\cellcolor{colorCl}0\tabularnewline
\vdummys3&&&\cellcolor{colorAd}0&\cellcolor{colorAl}2&\cellcolor{colorAl}0\tabularnewline
\vdummys4&&&\cellcolor{colorBd}1&\cellcolor{colorBl}2&\cellcolor{colorBl}0\tabularnewline
\vdummys5&&&\cellcolor{colorCd}2&\cellcolor{colorCl}0&\cellcolor{colorCl}1\tabularnewline
\vdummys6&&&\cellcolor{colorAd}0&\cellcolor{colorAl}0&\cellcolor{colorAl}1\tabularnewline
\vdummys7&&&\cellcolor{colorBd}1&\cellcolor{colorBl}1&\cellcolor{colorBl}0\tabularnewline
\vdummys8&&&\cellcolor{colorCd}2&\cellcolor{colorCl}2&\cellcolor{colorCl}0\tabularnewline
\vdummys9&&&\cellcolor{colorAd}0&\cellcolor{colorAd}1&\cellcolor{colorAl}1\tabularnewline
\vdummys10&&&\cellcolor{colorBd}1&\cellcolor{colorBd}0&\cellcolor{colorBl}2\tabularnewline
\vdummys11&&&\cellcolor{colorCd}2&\cellcolor{colorCd}1&\cellcolor{colorCl}1\tabularnewline
\vdummys12&&&\cellcolor{colorAd}0&\cellcolor{colorAd}2&\cellcolor{colorAl}1\tabularnewline
\vdummys13&&&\cellcolor{colorBd}1&\cellcolor{colorBd}2&\cellcolor{colorBl}1\tabularnewline
\vdummys14&&&\cellcolor{colorCd}2&\cellcolor{colorCd}0&\cellcolor{colorCl}2\tabularnewline
\vdummys15&&&\cellcolor{colorAd}0&\cellcolor{colorAd}0&\cellcolor{colorAl}2\tabularnewline
\vdummys16&&&\cellcolor{colorBd}1&\cellcolor{colorBd}1&\cellcolor{colorBl}2\tabularnewline
\vdummys17&&&\cellcolor{colorCd}2&\cellcolor{colorCd}2&\cellcolor{colorCl}1\tabularnewline
\vdummys18&&&\cellcolor{colorAd}0&\cellcolor{colorAd}1&\cellcolor{colorAl}2\tabularnewline
\vdummys19&&&\cellcolor{colorBd}1&\cellcolor{colorBd}0&\cellcolor{colorBl}0\tabularnewline
\vdummys20&&&\cellcolor{colorCd}2&\cellcolor{colorCd}1&\cellcolor{colorCl}2\tabularnewline
\vdummys21&&&\cellcolor{colorAd}0&\cellcolor{colorAd}2&\cellcolor{colorAl}2\tabularnewline
\vdummys22&&&\cellcolor{colorBd}1&\cellcolor{colorBd}2&\cellcolor{colorBl}2\tabularnewline
\vdummys23&&&\cellcolor{colorCd}2&\cellcolor{colorCd}0&\cellcolor{colorCl}0\tabularnewline
\vdummys24&&&\cellcolor{colorAd}0&\cellcolor{colorAd}0&\cellcolor{colorAl}0\tabularnewline
\vdummys25&&&\cellcolor{colorBd}1&\cellcolor{colorBd}1&\cellcolor{colorBl}1\tabularnewline
\vdummys26&&&\cellcolor{colorCd}2&\cellcolor{colorCd}2&\cellcolor{colorCl}2\tabularnewline
\vdots&&&\vdots&\vdots&\vdots
\end{tabular}
\egroup
\caption{The image to the left shows layers $0$ to $3$ of a $3$-cycle tree. The table in the middle shows the corresponding sequences of entries $\Sf{E}(v)$ (right part) and the sequences of positions $\Sf{P}(v)$ (colored center part) for all $16$ vertices (excluding the root). For convenience, the top part of the table contains the indices of the respective positions in base $10$ and base $3$. The table to the right shows the resulting $3$-digit table $\STf{D}$ (cf. also Eqn.~(\ref{EBijDTSeq})).}
\label{FPermFromTree}
\end{figure}
\egroup

\noindent
The $3$-adic system $\FFf{F}$ corresponding to $\STf{D}$ in Figure~\ref{FPermFromTree} satisfies
\begin{align}
\FFf{F}(\underline{\smash{3^3}})\modulo3^2=(1,0,1,5,5,0,0,1,5,7,3,7,2,2,3,3,4,2,4,6,4,8,8,6,6,7,8)
\end{align}
(cf. Corollary~\ref{CComputeDT}~(2)), and the resulting $3$-adic permutation $\pi$ thus has the cycles (to improve readability we omit the square brackets indicating equivalence classes, i.e. we write $(0,3,6)$ for $[([0],[3],[6])]_{\sim_\sigma}$)
\begin{align}
\Sigma(\pi_0):&\;\;(0)\\
\Sigma(\pi_1):&\;\;(0),\;(1),\;(2)\\
\Sigma(\pi_2):&\;\;(0,3,6),\;(1),\;(4,7),\;(2,5),\;(8)\\
\Sigma(\pi_3):&\;\;(0,3,6,9,12,15,18,21,24),\;(1,10,19),\;(4,7),\;(13,16,22,25),\;(2,5,11,14,20,23),\\
\nonumber
&\;\;(8),\;(17),\;(26)
\end{align}
which coincide with the sequences $\Sf{P}(v)$ and also define the same tree as the one given in Figure~\ref{FPermFromTree}.

\begin{proof}[Proof of Theorem~\ref{TPermFromTree}]
We will prove the theorem by induction on the layer. It may be helpful to consult Figure~\ref{FPermFromTree} when following the argument. In the following we denote the set of vertices contained in the $k$-th layer of $\mathcal{G}$, $k\in\Nz$, by $\mathcal{V}_k=\set{v\in\mathcal{V}\mid k(v)=k}$ (analogously we set $\mathcal{V}(\pi)_k=\set{(\ell,\sigma)\in\mathcal{V}(\pi)\mid \ell=k}$).

The first thing we observe is that the $k$-th layer of the graph $\mathcal{G}(\pi)$ only depends on $\STf{D}\llbracket\intintuptoex{k}\rrbracket$, i.e. the columns $0$ to $k-1$ of $\STf{D}$. Additionally, by definition of $\STf{D}$ the $k$-th column $\STf{D}\llbracket k-1\rrbracket$ of $\STf{D}$ (note that we start at $0$ when indexing columns) is completely determined by ``the first $p^k$ entries'' of the column, i.e. the entries $\STf{D}[0][k-1],\ldots,\STf{D}[p^k-1][k-1]$ (the quotation marks are due to the fact that a column of $\STf{D}$ technically has entries for all $p$-adic integers on which the notion ``first'' does not make sense as there is no natural order on $\Z_p$). We thus define the sequence
\begin{align}
\Sf{S}_k\ce(\STf{D}[0][k-1],\ldots,\STf{D}[p^k-1][k-1])\in\CoSequences(\SPboundedby{\intintuptoex{p}},\SPlength{p^k})
\end{align}
for all $k\in\N$ (in the example given in Figure~\ref{FPermFromTree} these sequences are the lighter shaded parts of the columns of the rightmost table, i.e. $\Sf{S}_1=(0,1,2)$, $\Sf{S}_2=(1,0,1,2,2,0,0,1,2)$, and $\Sf{S}_3=(0,1,0,0,0,1,1,0,0,1,2,1,1,1,2,2,2,1,2,0,2,2,2,0,0,1,2)$). By definition of $\STf{D}$ we thus get
\begin{align}
\Sf{S}_k[j]=\STf{D}\left[\Sf{P}(v)[i]\right]\left[k-1\right]=\Sf{E}(v)[i]
\end{align}
for all $k\in\N$ and $j\in\intintuptoex{p^k}$ where $v\in\mathcal{V}_k$ and $i\in\intintuptoex{\abs{\Sf{P}(v)}}$ are the unique elements satisfying $j=\Sf{P}(v)[i]$ and whose existence and uniqueness we will prove in the following.

For that we first observe that $\Sf{P}(v)$ is $\intintuptoex{p^{k(v)}}$-bounded for all $v\in\mathcal{V}$ which can easily be shown by induction on $k$ using the fact that the labels of all vertices that share a parent vertex, sum up to $p$. In the same way it can be readily verified that $\sum_{v\in\mathcal{V}_k}\abs{\Sf{P}(v)}=p^k$ for all $k\in\Nz$. We are thus left to show that for all $k\in\Nz$, $\Sf{P}(v)$ and $\Sf{P}(w)$ don't have any common entries if $v$ and $w$ are distinct elements of $\mathcal{V}_k$. This follows again by induction on $k$ (in order to conclude that $\Sf{P}(v)$ and $\Sf{P}(w)$ don't share any entries if $v$ and $w$ have distinct parent vertices) and by definition of the offset function $o$ (in order to conclude that $\Sf{P}(v)$ and $\Sf{P}(w)$ don't share any entries if $v$ and $w$ have the same parent vertex). Altogether we just have proven that for all $k\in\Nz$ and $j\in\intintuptoex{p^k}$ there is a unique $v\in\mathcal{V}_k$ and $i\in\intintuptoex{\abs{\Sf{P}(v)}}$ satisfying $j=\Sf{P}(v)[i]$. Together with the fact that $\abs{\Sf{P}(v)}=\abs{\Sf{E}(v)}$ for all $v\in\mathcal{V}$ this implies that the sequences $\Sf{S}_k$, $k\in\N$ and consequently $\STf{D}$ are well-defined.

Next we note that the concatenation $\Sf{S}\ce\prod_{k=1}^\infty\Sf{S}_k$ of all sequences $\Sf{S}_k$ is exactly the infinite sequence corresponding to $\STf{D}$ in the sense of Eqn.~(\ref{EBijDTSeq}). Again by induction on the layer $k$ and by the definition of $o$ it can be shown that $\Sf{S}$ has the $(p,k)$-block property and thus $\STf{D}\in\SoTables{p}$.

We are now ready to proceed with the induction argument to show that $\varphi$ is an isomorphism between $(\mathcal{G}(\pi),c(\pi))$ (which we now know to be well-defined) and $(\mathcal{G},c)$.  Clearly, the zeroth layers of $(\mathcal{G},c)$ and $(\mathcal{G}(\pi),c(\pi))$ only consist of the respective roots which are mapped to each other by $\varphi$. Now assume that $\varphi$ is an isomorphism between the layers $0$ to $k-1$ of $(\mathcal{G},c)$ and $(\mathcal{G}(\pi),c(\pi))$ for some $k\in\N$ and let $v\in\mathcal{V}_{k-1}$ and $v'\ce\varphi(v)$. Then it follows from the induction hypothesis that $v'=(k-1,\Sf{P}(v))\in\mathcal{V}(\pi)_{k-1}$ (to economize notation we will identify $\Sf{P}(v)$ and $[([\Sf{P}(v)[0]],\ldots,[\Sf{P}(v)[\abs{\Sf{P}(v)}-1]])]_{\sim_\sigma}$). We are left to show that $\varphi$ maps the children of $v$ exactly to the children of $v'$ and also preserves the labels of all edges between $v$ and its children. For that let $w\in\mathcal{V}_k$ such that $vw\in\mathcal{E}$. We need to show that $w'\ce\varphi(w)=(k,\Sf{P}(w))\in\mathcal{V}(\pi)$, $v'w'\in\mathcal{E}(\pi)$ and $\abs{\Sf{P}(w)}/\abs{\Sf{P}(v)}=c(w)=c(vw)$. The latter part follows directly from the definition of $\Sf{P}$. To prove the other two statements we observe that
\begin{align}
\label{EPermFromTree}
{\psi_{(x)^p,k-1}}^{-1}\circ\psi_{\FFf{F},k-1}(\Sf{P}(v)[i])=\pi_{k-1}(\Sf{P}(v)[i])=\Sf{P}(v)[(i+1)\modulo\abs{\Sf{P}(v)}]
\end{align}
for all $i\in\intintuptoex{\abs{\Sf{P}(v)}}$, where
\begin{align}
\label{Dpsik}
\psi_{\FFf{F},k}:\Z_p\slash p^k\Z_p&\to\CoSequences(\SPboundedby{\intintuptoex{p}},\SPlength{k})\\
\nonumber
n&\mapsto\psi_\FFf{F}(n)[\intintuptoex{k}]
\end{align}
for any $p$-adic system $\FFf{F}$ and $k\in\Nz$ (again we omit the square brackets indicating equivalence classes to improve readability). Rearranging Eqn.~(\ref{EPermFromTree}) results in
\begin{align}
\psi_{\FFf{F},k-1}(\Sf{P}(v)[i])=\psi_{(x)^p,k-1}\left(\Sf{P}(v)[(i+1)\modulo\abs{\Sf{P}(v)}]\right)
\end{align}
for all $i\in\intintuptoex{\abs{\Sf{P}(v)}}$, i.e. the first $k-1$ digits of the standard base $p$ expansion of $\Sf{P}(v)[(i+1)\modulo\abs{\Sf{P}(v)}]\in\intintuptoex{p^{k-1}}$. Furthermore, from the definitions of $\Sf{P}$ and $\Sf{E}$ it follows that
\begin{align}
\Sf{P}(w)&=\prod_{i=0}^{c(w)-1}\left(P(v)+(o(w)+i)p^{k-1}\right)\\
\Sf{E}(w)&=\prod_{i=0}^{c(w)-1}(o(w)+i)^{\abs{\Sf{P}(v)}-1}\cdot(o(w)+(i+1)\modulo c(w)).
\end{align}
We thus get
\begin{align}
\psi_{\FFf{F},k}(\Sf{P}(w)[i])
&=\STf{D}[\Sf{P}(w)[i]][\intintuptoex{k}]\\
&=\STf{D}[\Sf{P}(v)[i\modulo\abs{\Sf{P}(v)}][\intintuptoex{k-1}]\cdot(\Sf{E}[w][i])\\
&=\psi_{\FFf{F},k-1}\left(\Sf{P}(v)[i\modulo\abs{\Sf{P}(v)}]\right)\cdot(\Sf{E}(w)[i])\\
&=\psi_{(x)^p,k-1}\left(\Sf{P}(v)[(i+1)\modulo\abs{\Sf{P}(v)}]\right)\cdot\\
\nonumber
&\phantom{=\vphantom{a}}\left(o(w)+(\floor{i/\abs{\Sf{P}(v)}}+(i\modulo\abs{\Sf{P}(v)}\neq\abs{\Sf{P}(v)}-1\;?\;0:1))\modulo c(w)\right)\\
&=\psi_{(x)^p,k-1}\left(\Sf{P}(v)[(i+1)\modulo\abs{\Sf{P}(v)}]\right)\cdot\left(o(w)+\floor{(i+1)/\abs{\Sf{P}(v)}}\modulo c(w)\right)\\
&=\psi_{(x)^p,k}\left(\Sf{P}(w)[(i+1)\modulo\abs{\Sf{P}(w)}]\right)
\end{align}
for all $i\in\intintuptoex{\abs{\Sf{P}(w)}}$, or equivalently,
\begin{align}
\pi_{k}(\Sf{P}(w)[i])={\psi_{(x)^p,k}}^{-1}\circ\psi_{\FFf{F},k}(\Sf{P}(w)[i])=\Sf{P}(w)[(i+1)\modulo\abs{\Sf{P}(w)}].
\end{align}
Since all elements of $\Sf{P}(w)$ are mutually distinct modulo $p^k$, $\Sf{P}(w)$ is a cycle of $\pi_k$ and consequently $w'\in\mathcal{V}(\pi)$. In addition, $\Sf{P}[w]\modulo p^{k-1}=\Sf{P}[v]^{c(w)}$ and thus $v'w'\in\mathcal{E}(\pi)$ which completes the proof.
\end{proof}

The following theorem characterizes completely the sets of all isomorphism classes of trees with up to $4$ layers which may occur as subtrees of trees of cycles of $2$-adic permutations defined by $\Z_2$-polynomial $2$-adic systems or $2$-permutation polynomials. This will allow also to show that the sets of all $p$-adic permutations defined by $\Z_p$-polynomial $p$-adic systems and $p$-permutation polynomials respectively, have both exclusive elements (that the two sets are not disjoint either follows from the simple observation that $\pi_{\FFf{F},\FFf{F}}(n)=n$ for all $p$-adic systems $\FFf{F}$ and $n\in\Z_p$).

\begin{theorem}
\label{TSubtrees}
For $2\leq p\in\N$ and $k\in\Nz$ let
\begin{align}
\mathrlap{S_{p,k}}\phantom{U_{p,k}}&\ce\big\{\text{isomorphism class of } T\mid\FFf{F},\FFf{G}\in\SoSystems{p}(\FFPpolynomialcoefficients{\Z_p})\\
\nonumber
&\phantom{\vphantom{a}\ce\big\{\text{isomorphism class of } T\mid\vphantom{a}}T\text{ full $k$-layer rooted subtree of $(\mathcal{G}(\pi_{\FFf{F},\FFf{G}}),c(\pi_{\FFf{F},\FFf{G}}))$}\big\}\\
\mathrlap{T_{p,k}}\phantom{U_{p,k}}&\ce\big\{\text{isomorphism class of } T\mid\FFf{F},\FFf{G}\in\SoSystems{p}(\FFPpolynomialcoefficients{\Z_p})\\
\nonumber
&\phantom{\vphantom{a}\ce\big\{\text{isomorphism class of } T\mid\vphantom{a}} T\text{ full $k$-layer rooted subtree of $(\mathcal{G}(\pi_{\FFf{F},\FFf{G}}),c(\pi_{\FFf{F},\FFf{G}}))$}\\
\nonumber
&\phantom{\vphantom{a}\ce\big\{\text{isomorphism class of } T\mid\vphantom{a}}\abs{\sigma}>1\text{ for root $(\ell,\sigma)$ of $T$}\big\}\\
\mathrlap{U_{p,k}}\phantom{U_{p,k}}&\ce\big\{\text{isomorphism class of } T\mid f\in\Z_p[x]\text{ $p$-permutation polynomial}\\
\nonumber
&\phantom{\vphantom{a}\ce\big\{\text{isomorphism class of } T\mid\vphantom{a}}T\text{ full $k$-layer rooted subtree of $(\mathcal{G}(f),c(f))$}\big\}\\
\mathrlap{V_{p,k}}\phantom{U_{p,k}}&\ce\big\{\text{isomorphism class of } T\mid f\in\Z_p[x]\text{ $p$-permutation polynomial}\\
\nonumber
&\phantom{\vphantom{a}\ce\big\{\text{isomorphism class of } T\mid\vphantom{a}}T\text{ full $k$-layer rooted subtree of $(\mathcal{G}(f),c(f))$}\\
\nonumber
&\phantom{\vphantom{a}\ce\big\{\text{isomorphism class of } T\mid\vphantom{a}}\abs{\sigma}>1\text{ for root $(\ell,\sigma)$ of $T$}\big\}.
\end{align}
Then,
\begin{align}
\abs{S_{2,2}}&=5&\abs{S_{2,3}}&=20&\abs{S_{2,4}}&=71\\
\abs{T_{2,2}}&=5&\abs{T_{2,3}}&=12&\abs{T_{2,4}}&=50\\
\abs{U_{2,2}}&=5&\abs{U_{2,3}}&=18&\abs{U_{2,4}}&=83\\
\abs{V_{2,2}}&=3&\abs{V_{2,3}}&=5&\abs{V_{2,4}}&=7
\end{align}
and all these sets are given in Figure~\ref{FSubtrees} where membership in the respective sets is indicated by the black boxes below each graph in the order $S$, $T$, $U$, $V$.
\end{theorem}

\noindent
Note that Figure~\ref{FSubtrees} lists all possible isomorphism classes of $2$-, $3$-, and $4$-layer rooted trees with out-degrees in $\set{1,2}$ of which there are $5$, $20$, and $230$ respectively. Also note that the edge labels are not shown as they are uniquely fixed by the graph itself in the case $p=2$.

\begin{figure}[H]
\centering
\includegraphics[width=\textwidth]{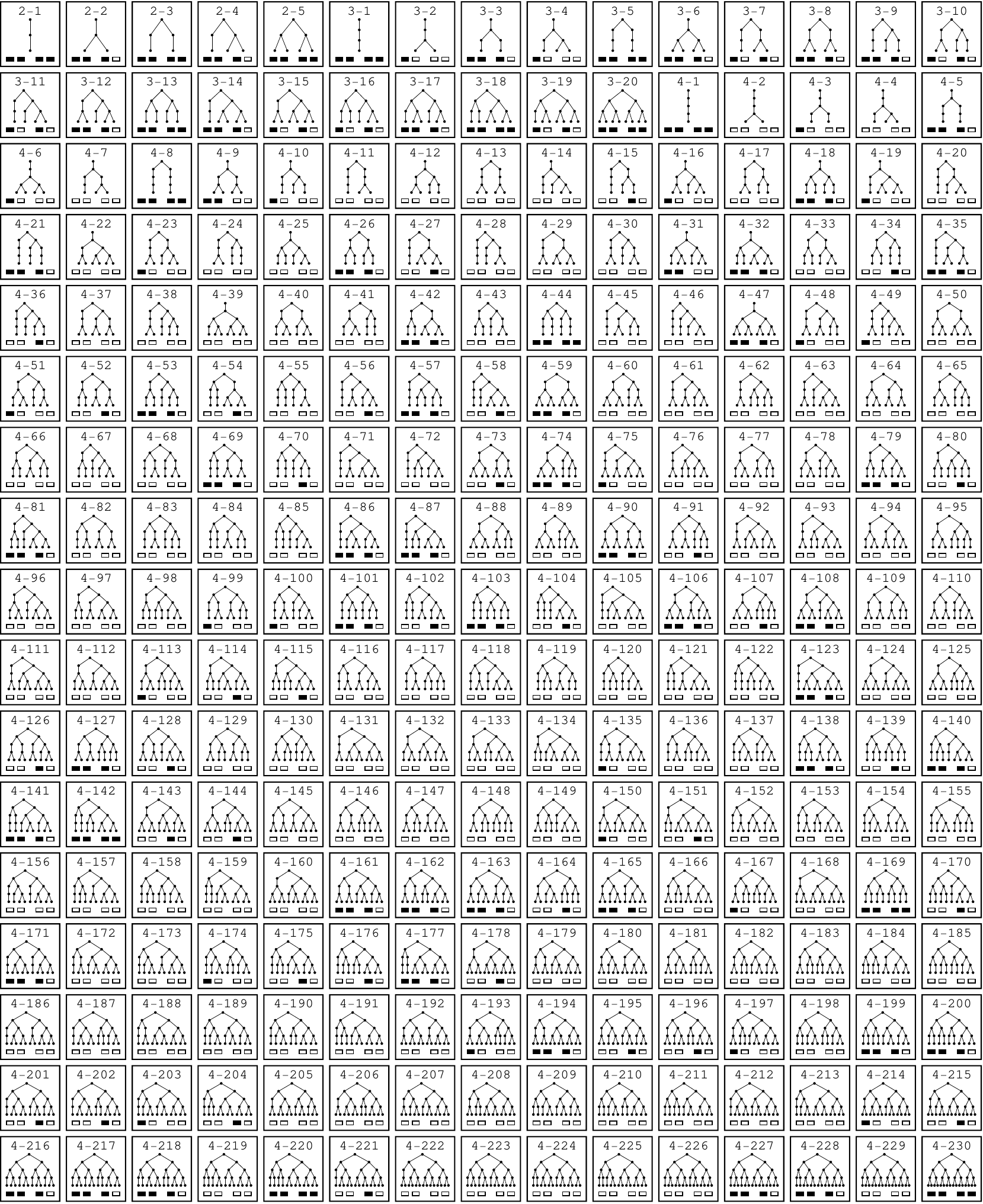}
\caption{The sets $S_{p,k}$, $T_{p,k}$, $U_{p,k}$, and $V_{p,k}$ for $p=2$ and $k\in\set{2,3,4}$.}
\label{FSubtrees}
\end{figure}

\begin{proof}[Proof of Theorem~\ref{TSubtrees}]
We begin with the easier part of the proof and show that all the isomorphism classes marked by black boxes in Figure~\ref{FSubtrees} actually are elements of the respective sets by listing examples of $2$-adic permutations $\pi_{\FFf{F},\FFf{G}}$, where $\FFf{F}$ and $\FFf{G}$ are $\Z_2$-polynomial $2$-adic systems, and $2$-adic permutations defined by $2$-permutation polynomials whose trees of cycles contain the given isomorphism classes as subtrees. These examples are summarized in the following table, where the $\Z_2$-polynomial $2$-adic system $\FFf{F}$ (for $S_{2,k}$ and $T_{2,k}$) and the $2$-permutation polynomial $f$ (for $U_{2,k}$ and $V_{2,k}$) are given and $\FFf{G}\ce(x,x-1)\in\SoSystems{2}$ is supposed to be fixed. Note that all claimed subtrees are found within the layers $0$ to $8$ of $\mathcal{G}(\pi_{\FFf{F},\FFf{G}})$ and $\mathcal{G}(f)$ respectively.

\bgroup
\fontsize{6}{0}\selectfont
\allowdisplaybreaks
\begin{align*}
\textbf{Subtree}&&&\mathbf{S_{2,k},\;\FFf{F}=\ldots}&&\mathbf{T_{2,k},\;\FFf{F}=\ldots}&&\mathbf{U_{2,k},\;f=\ldots}&&\mathbf{V_{2,k},\;f=\ldots}\\
\hline\\
2-1:&&&\smash{(x\!+\!2,x\!+\!1)}&&\smash{(x\!+\!2,x\!+\!1)}&&\smash{x\!+\!2}&&\smash{x\!+\!2}\\
2-2:&&&\smash{(x,x\!+\!1)}&&\smash{(3x\!+\!2,x\!+\!1)}&&\smash{3x}&&\\
2-3:&&&\smash{(x\!+\!2,x\!+\!1)}&&\smash{(3x\!+\!2,3x\!+\!3)}&&\smash{x\!+\!2}&&\smash{3x}\\
2-4:&&&\smash{(x,x\!+\!1)}&&\smash{(3x,x\!+\!1)}&&\smash{3x}&&\\
2-5:&&&\smash{(3x,x\!+\!1)}&&\smash{(x,x\!+\!1)}&&\smash{x}&&\smash{2x^2\!+\!3x\!+\!2}\\
3-1:&&&\smash{(x\!+\!2,x\!+\!1)}&&\smash{(x\!+\!2,x\!+\!1)}&&\smash{x\!+\!2}&&\smash{x\!+\!2}\\
3-2:&&&\smash{(3x\!+\!2,x\!+\!1)}&&&&&&\\
3-3:&&&\smash{(x\!+\!4,x\!+\!1)}&&\smash{(3x\!+\!2,x\!+\!3)}&&\smash{3x}&&\\
3-4:&&&\smash{(3x,x\!+\!1)}&&&&&&\\
3-5:&&&\smash{(x\!+\!2,x\!+\!1)}&&\smash{(x\!+\!4,x\!+\!1)}&&\smash{x\!+\!2}&&\smash{3x}\\
3-6:&&&\smash{(x,x\!+\!1)}&&\smash{(3x\!+\!2,x\!+\!1)}&&\smash{2x^2\!+\!3x\!+\!2}&&\\
3-7:&&&\smash{(3x\!+\!2,x\!+\!1)}&&&&\smash{2x^3\!+\!3x\!+\!2}&&\\
3-8:&&&\smash{(3x\!+\!2,3x\!+\!3)}&&\smash{(2x^2\!+\!x\!+\!4,x^2\!+\!x)}&&\smash{2x^2\!+\!3x\!+\!2}&&\\
3-9:&&&\smash{(3x\!+\!4,3x\!+\!3)}&&\smash{(x^2\!+\!x\!+\!2,x^2\!+\!x)}&&\smash{2x^2\!+\!x\!+\!2}&&\\
3-10:&&&\smash{(3x\!+\!4,x\!+\!1)}&&&&\smash{2x^3\!+\!x\!+\!4}&&\\
3-11:&&&\smash{(3x\!+\!2,x\!+\!3)}&&&&\smash{2x^3\!+\!x\!+\!2}&&\\
3-12:&&&\smash{(x,x\!+\!1)}&&\smash{(x\!+\!2,3x\!+\!1)}&&\smash{3x}&&\\
3-13:&&&\smash{(x\!+\!4,x\!+\!3)}&&\smash{(3x\!+\!2,3x\!+\!3)}&&\smash{x\!+\!4}&&\smash{2x^2\!+\!3x\!+\!2}\\
3-14:&&&\smash{(3x,3x\!+\!1)}&&\smash{(3x,x\!+\!1)}&&\smash{2x^2\!+\!x}&&\\
3-15:&&&\smash{(3x,x\!+\!1)}&&&&\smash{2x^3\!+\!x}&&\\
3-16:&&&\smash{(3x,x\!+\!1)}&&&&\smash{2x^3\!+\!3x}&&\\
3-17:&&&\smash{(3x,3x\!+\!1)}&&\smash{(2x^2\!+\!x,x^2\!+\!x)}&&\smash{2x^2\!+\!3x}&&\\
3-18:&&&\smash{(x,x\!+\!3)}&&\smash{(x^2\!+\!3x,x^2\!+\!x)}&&\smash{2x^2\!+\!x}&&\smash{x^4\!+\!3x^2\!+\!x\!+\!2}\\
3-19:&&&\smash{(x,3x\!+\!1)}&&&&\smash{2x^3\!+\!3x\!+\!4}&&\\
3-20:&&&\smash{(2x^2\!+\!x,x^2\!+\!x\!+\!2)}&&\smash{(x,x\!+\!1)}&&\smash{x}&&\smash{2x^3\!+\!2x^2\!+\!3x\!+\!2}\\
4-1:&&&\smash{(x\!+\!2,x\!+\!1)}&&\smash{(x\!+\!2,x\!+\!1)}&&\smash{x\!+\!2}&&\smash{x\!+\!2}\\
4-3:&&&\smash{(3x\!+\!2,x\!+\!3)}&&&&&&\\
4-5:&&&\smash{(x\!+\!4,x\!+\!1)}&&\smash{(3x\!+\!2,x\!+\!3)}&&\smash{3x}&&\\
4-6:&&&\smash{(3x,x\!+\!1)}&&&&&&\\
4-8:&&&\smash{(x\!+\!2,x\!+\!1)}&&\smash{(x\!+\!4,x\!+\!1)}&&\smash{x\!+\!2}&&\smash{3x}\\
4-9:&&&\smash{(3x\!+\!4,x\!+\!3)}&&\smash{(3x\!+\!4,x\!+\!3)}&&&&\\
4-10:&&&\smash{(3x,x\!+\!1)}&&&&&&\\
4-15:&&&&&&&\smash{2x^3\!+\!3x\!+\!2}&&\\
4-16:&&&\smash{(3x,x\!+\!1)}&&&&&&\\
4-18:&&&\smash{(3x\!+\!2,3x\!+\!3)}&&\smash{(4x^2\!+\!x\!+\!4,x^2\!+\!x)}&&\smash{7x}&&\\
4-19:&&&\smash{(3x,x\!+\!1)}&&&&&&\\
4-21:&&&\smash{(x\!+\!4,3x\!+\!1)}&&\smash{(x\!+\!4,3x\!+\!1)}&&\smash{2x^2\!+\!x\!+\!2}&&\\
4-23:&&&\smash{(3x,x\!+\!1)}&&&&&&\\
4-26:&&&\smash{(3x^2\!+\!3x\!+\!4,x^2\!+\!x)}&&\smash{(2x^2\!+\!x,x^2\!+\!x\!+\!2)}&&\smash{4x^3\!+\!2x^2\!+\!7x\!+\!2}&&\\
4-27:&&&&&&&\smash{2x^3\!+\!7x\!+\!2}&&\\
4-31:&&&\smash{(3x\!+\!2,x\!+\!1)}&&\smash{(3x\!+\!2,x\!+\!1)}&&&&\\
4-32:&&&\smash{(3x\!+\!4,x\!+\!1)}&&\smash{(3x\!+\!4,x\!+\!1)}&&\smash{x^4\!+\!x^2\!+\!x\!+\!6}&&\\
4-34:&&&&&&&\smash{2x^3\!+\!x\!+\!4}&&\\
4-35:&&&\smash{(3x,3x\!+\!1)}&&\smash{(2x^2\!+\!3x\!+\!4,x^2\!+\!x)}&&\smash{x^4\!+\!2x^3\!+\!x^2\!+\!3x\!+\!4}&&\\
4-36:&&&&&&&\smash{x^4\!+\!x^2\!+\!x\!+\!4}&&\\
4-42:&&&\smash{(x,3x\!+\!3)}&&\smash{(x,3x\!+\!3)}&&\smash{2x^2\!+\!3x\!+\!2}&&\\
4-44:&&&\smash{(x\!+\!4,x\!+\!3)}&&\smash{(3x\!+\!2,3x\!+\!3)}&&\smash{x\!+\!4}&&\smash{2x^2\!+\!3x\!+\!2}\\
4-47:&&&\smash{(x,x\!+\!1)}&&\smash{(x\!+\!2,3x\!+\!3)}&&\smash{2x^2\!+\!3x}&&\\
4-48:&&&\smash{(3x\!+\!2,x\!+\!3)}&&&&&&\\
4-49:&&&\smash{(x,3x\!+\!3)}&&&&&&\\
4-51:&&&\smash{(3x\!+\!4,x\!+\!1)}&&&&&&\\
4-52:&&&&&&&\smash{x^4\!+\!x^2\!+\!x\!+\!2}&&\\
4-53:&&&\smash{(3x\!+\!4,x\!+\!3)}&&\smash{(3x\!+\!4,x\!+\!3)}&&\smash{x^4\!+\!2x^3\!+\!3x^2\!+\!x\!+\!6}&&\\
4-54:&&&&&&&\smash{2x^3\!+\!5x\!+\!4}&&\\
4-56:&&&&&&&\smash{2x^3\!+\!x\!+\!2}&&\\
4-57:&&&\smash{(2x^2\!+\!3x,x^2\!+\!x\!+\!2)}&&\smash{(2x^2\!+\!3x,x^2\!+\!x\!+\!2)}&&\smash{2x^2\!+\!x\!+\!6}&&\\
4-58:&&&&&&&\smash{x^4\!+\!x^2\!+\!x}&&\\
4-59:&&&\smash{(3x\!+\!2,3x\!+\!1)}&&\smash{(2x^2\!+\!x,x^2\!+\!x)}&&\smash{4x^3\!+\!2x^2\!+\!3x\!+\!2}&&\\
4-69:&&&\smash{(2x^2\!+\!x\!+\!4,x^2\!+\!x\!+\!2)}&&\smash{(4x^2\!+\!3x\!+\!2,x^2\!+\!x\!+\!2)}&&\smash{x^5\!+\!3x^3\!+\!5x\!+\!4}&&\\
4-70:&&&&&&&\smash{x^5\!+\!x^3\!+\!2x^2\!+\!x\!+\!4}&&\\
4-73:&&&&&&&\smash{x^4\!+\!3x^2\!+\!3x\!+\!4}&&\\
4-74:&&&\smash{(2x^2\!+\!3x\!+\!4,x^2\!+\!x)}&&\smash{(2x^2\!+\!3x\!+\!4,x^2\!+\!x)}&&\smash{x^4\!+\!2x^3\!+\!3x^2\!+\!x\!+\!2}&&\\
4-75:&&&\smash{(3x\!+\!4,x\!+\!3)}&&&&&&\\
4-79:&&&\smash{(x\!+\!4,x\!+\!1)}&&\smash{(x\!+\!2,3x\!+\!1)}&&\smash{3x}&&\\
4-80:&&&&&&&\smash{4x^4\!+\!2x^3\!+\!6x^2\!+\!7x\!+\!6}&&\\
4-81:&&&\smash{(x,3x\!+\!1)}&&\smash{(x,3x\!+\!1)}&&\smash{x^4\!+\!2x^3\!+\!3x^2\!+\!x}&&\\
4-86:&&&\smash{(3x,x\!+\!1)}&&\smash{(3x,x\!+\!1)}&&\smash{x^4\!+\!2x^3\!+\!x^2\!+\!3x}&&\\
4-87:&&&\smash{(2x^2\!+\!x\!+\!4,x^2\!+\!x)}&&\smash{(4x^2\!+\!x,x^2\!+\!x)}&&\smash{2x^2\!+\!x}&&\\
4-90:&&&\smash{(x^2\!+\!3x\!+\!4,x^2\!+\!x\!+\!2)}&&\smash{(3x^2\!+\!3x\!+\!2,2x^2\!+\!3x\!+\!3)}&&\smash{x^4\!+\!2x^3\!+\!3x^2\!+\!3x\!+\!4}&&\\
4-91:&&&&&&&\smash{x^4\!+\!x^2\!+\!3x\!+\!4}&&\\
4-99:&&&\smash{(3x^2\!+\!x\!+\!4,x^2\!+\!x\!+\!2)}&&&&&&\\
4-100:&&&\smash{(3x\!+\!4,x\!+\!3)}&&&&&&\\
4-101:&&&\smash{(x^2\!+\!3x\!+\!2,x^2\!+\!x\!+\!2)}&&\smash{(2x^2\!+\!3x,x^2\!+\!3x)}&&\smash{x^4\!+\!2x^3\!+\!x^2\!+\!x\!+\!4}&&\\
4-102:&&&&&&&\smash{2x^3\!+\!3x}&&\\
4-103:&&&\smash{(3x\!+\!4,3x\!+\!3)}&&\smash{(3x\!+\!4,3x\!+\!3)}&&\smash{4x^2\!+\!x\!+\!4}&&\\
4-104:&&&&&&&\smash{x^5\!+\!x^3\!+\!2x^2\!+\!5x\!+\!4}&&\\
4-106:&&&\smash{(x,x\!+\!1)}&&\smash{(x\!+\!4,3x\!+\!1)}&&\smash{7x}&&\\
4-107:&&&&&&&\smash{4x^4\!+\!6x^3\!+\!6x^2\!+\!3x\!+\!6}&&\\
4-108:&&&\smash{(x\!+\!4,3x\!+\!3)}&&\smash{(x\!+\!4,3x\!+\!3)}&&\smash{x^4\!+\!2x^3\!+\!3x^2\!+\!x\!+\!4}&&\\
4-113:&&&\smash{(3x,x\!+\!1)}&&&&&&\\
4-114:&&&&&&&\smash{x^4\!+\!3x^2\!+\!3x}&&\\
4-115:&&&&&&&\smash{6x^3\!+\!5x}&&\\
4-123:&&&\smash{(3x\!+\!4,x\!+\!1)}&&\smash{(3x\!+\!4,x\!+\!1)}&&\smash{2x^2\!+\!5x}&&\\
4-126:&&&&&&&\smash{x^5\!+\!x^3\!+\!3x}&&\\
4-127:&&&\smash{(2x^2\!+\!3x\!+\!4,x^2\!+\!3x\!+\!4)}&&\smash{(3x^2\!+\!3x\!+\!2,2x^2\!+\!3x\!+\!1)}&&\smash{x^5\!+\!5x^4\!+\!x^3\!+\!7x^2\!+\!7x\!+\!4}&&\\
4-128:&&&&&&&\smash{x^4\!+\!x^2\!+\!7x\!+\!4}&&\\
4-135:&&&\smash{(3x,x\!+\!1)}&&&&&&\\
4-138:&&&\smash{(3x\!+\!4,x\!+\!1)}&&\smash{(3x\!+\!4,x\!+\!1)}&&\smash{x^5\!+\!3x^3\!+\!2x^2\!+\!3x}&&\\
4-139:&&&&&&&\smash{x^5\!+\!5x^4\!+\!7x^3\!+\!7x^2\!+\!x\!+\!4}&&\\
4-140:&&&\smash{(3x,x\!+\!1)}&&\smash{(3x,x\!+\!1)}&&\smash{x^4\!+\!2x^3\!+\!x^2\!+\!5x}&&\\
4-141:&&&\smash{(2x^2\!+\!x\!+\!4,x^2\!+\!x)}&&\smash{(2x^2\!+\!3x\!+\!2,x^2\!+\!x)}&&\smash{x^5\!+\!3x^3\!+\!x\!+\!4}&&\\
4-142:&&&\smash{(x,x\!+\!3)}&&\smash{(3x\!+\!4,3x\!+\!3)}&&\smash{5x}&&\smash{x^4\!+\!3x^2\!+\!x\!+\!2}\\
4-143:&&&&&&&\smash{x^4\!+\!x^2\!+\!x\!+\!6}&&\\
4-144:&&&&&&&\smash{2x^3\!+\!5x}&&\\
4-150:&&&\smash{(x,3x\!+\!1)}&&&&&&\\
4-151:&&&&&&&\smash{2x^3\!+\!x}&&\\
4-161:&&&\smash{(x^2\!+\!3x,x^2\!+\!x\!+\!2)}&&\smash{(3x^2\!+\!3x\!+\!2,4x^2\!+\!3x\!+\!1)}&&\smash{x^4\!+\!2x^3\!+\!3x^2\!+\!7x}&&\\
4-162:&&&\smash{(4x^2\!+\!x,x^2\!+\!x\!+\!2)}&&\smash{(2x^2\!+\!3x\!+\!2,2x^2\!+\!x\!+\!3)}&&\smash{x^5\!+\!3x^3\!+\!x}&&\\
4-163:&&&\smash{(2x^2\!+\!x\!+\!4,x^2\!+\!x)}&&\smash{(2x^2\!+\!x\!+\!4,x^2\!+\!x)}&&\smash{2x^2\!+\!3x}&&\\
4-164:&&&&&&&\smash{2x^3\!+\!7x}&&\\
4-165:&&&\smash{(2x^2\!+\!x\!+\!2,x^2\!+\!x\!+\!2)}&&\smash{(4x^2\!+\!x\!+\!2,x^2\!+\!3x\!+\!4)}&&\smash{x^5\!+\!3x^3\!+\!2x^2\!+\!3x\!+\!4}&&\\
4-167:&&&\smash{(x^2\!+\!x\!+\!2,x^2\!+\!x\!+\!2)}&&&&&&\\
4-169:&&&\smash{(3x^2\!+\!x\!+\!2,3x^2\!+\!3x\!+\!4)}&&\smash{(3x^2\!+\!x\!+\!2,3x^2\!+\!3x\!+\!4)}&&\smash{2x^2\!+\!3x}&&\smash{2x^2\!+\!3x}\\
4-170:&&&&&&&\smash{x^5\!+\!5x^4\!+\!3x^3\!+\!7x^2\!+\!5x\!+\!4}&&\\
4-171:&&&\smash{(3x\!+\!2,x\!+\!1)}&&\smash{(3x\!+\!2,x\!+\!1)}&&\smash{x^4\!+\!2x^3\!+\!x^2\!+\!x}&&\\
4-174:&&&\smash{(2x^2\!+\!3x,x^2\!+\!x\!+\!2)}&&&&&&\\
4-175:&&&&&&&\smash{x^4\!+\!x^2\!+\!7x}&&\\
4-176:&&&&&&&\smash{x^5\!+\!x^3\!+\!2x^2\!+\!5x}&&\\
4-177:&&&\smash{(x^2\!+\!3x,x^2\!+\!x)}&&\smash{(x^2\!+\!3x,x^2\!+\!x)}&&\smash{2x^2\!+\!x}&&\\
4-178:&&&&&&&\smash{6x^3\!+\!x}&&\\
4-193:&&&\smash{(x,3x\!+\!1)}&&&&&&\\
4-194:&&&\smash{(2x^2\!+\!x,x^2\!+\!x\!+\!2)}&&\smash{(x^2\!+\!3x\!+\!2,2x^2\!+\!3x\!+\!1)}&&\smash{x^5\!+\!3x^3\!+\!4x^2\!+\!5x}&&\\
4-195:&&&&&&&\smash{x^5\!+\!x^3\!+\!3x\!+\!4}&&\\
4-196:&&&&&&&\smash{2x^3\!+\!3x\!+\!4}&&\\
4-197:&&&\smash{(3x^2\!+\!x,x^2\!+\!x)}&&&&&&\\
4-199:&&&\smash{(x^2\!+\!x\!+\!2,x^2\!+\!x)}&&\smash{(x^2\!+\!x\!+\!2,x^2\!+\!x)}&&\smash{3x^5\!+\!x^4\!+\!7x^3\!+\!7x^2\!+\!7x}&&\\
4-200:&&&\smash{(3x^2\!+\!3x,x^2\!+\!x)}&&\smash{(3x^2\!+\!3x,x^2\!+\!x)}&&\smash{4x^3\!+\!5x}&&\\
4-201:&&&&&&&\smash{x^4\!+\!x^2\!+\!3x}&&\\
4-202:&&&&&&&\smash{x^5\!+\!x^3\!+\!6x^2\!+\!5x}&&\\
4-203:&&&\smash{(x^2\!+\!x,x^2\!+\!x)}&&&&&&\\
4-204:&&&&&&&\smash{x^5\!+\!x^3\!+\!6x^2\!+\!x}&&\\
4-214:&&&\smash{(x,3x\!+\!1)}&&&&&&\\
4-215:&&&&&&&\smash{2x^3\!+\!2x^2\!+\!5x}&&\\
4-216:&&&\smash{(3x^2\!+\!x\!+\!4,x^2\!+\!x)}&&\smash{(3x^2\!+\!x\!+\!4,x^2\!+\!x)}&&\smash{x^4\!+\!2x^3\!+\!3x^2\!+\!3x}&&\\
4-217:&&&\smash{(2x^2\!+\!3x\!+\!2,x^2\!+\!x\!+\!2)}&&\smash{(2x^2\!+\!3x\!+\!2,x^2\!+\!x\!+\!2)}&&\smash{x^5\!+\!3x^3\!+\!4x^2\!+\!x}&&\\
4-218:&&&\smash{(x^2\!+\!3x,x^2\!+\!x)}&&\smash{(x^2\!+\!3x,x^2\!+\!x)}&&\smash{2x^2\!+\!5x}&&\\
4-220:&&&\smash{(x^2\!+\!x\!+\!4,x^2\!+\!3x)}&&\smash{(x^2\!+\!x\!+\!4,x^2\!+\!3x)}&&\smash{2x^2\!+\!x}&&\smash{2x^3\!+\!2x^2\!+\!3x\!+\!2}\\
4-221:&&&&&&&\smash{x^5\!+\!x^3\!+\!2x^2\!+\!x}&&\\
4-227:&&&\smash{(2x^2\!+\!x,x^2\!+\!x\!+\!2)}&&\smash{(2x^2\!+\!3x\!+\!2,3x^2\!+\!x\!+\!4)}&&\smash{x^5\!+\!3x^3\!+\!5x}&&\\
4-228:&&&\smash{(x^2\!+\!x\!+\!2,x^2\!+\!3x)}&&\smash{(2x^2\!+\!x\!+\!2,x^2\!+\!3x\!+\!2)}&&\smash{4x^3\!+\!4x^2\!+\!x}&&\\
4-230:&&&\smash{(x,x\!+\!1)}&&\smash{(x,x\!+\!1)}&&\smash{x}&&\smash{2x^3\!+\!x\!+\!2}
\end{align*}
\egroup

In the other direction we need to show that all isomorphism classes not marked by black boxes cannot be realized as the isomorphism class of a subtree of the tree of cycles of a $2$-adic permutation of the respective type. To do so we define the auxiliary function
\begin{align}
\varphi_{\pi,k}:\Z_2&\to\set{0,1}\\
\nonumber
n&\mapsto\psi_{\FFf{F}_2,k+1}(\pi_{k+1}(n))[k]=\FFFDT((x,x-1))[\pi(n)][k]
\end{align}
for all $2$-adic permutations $\pi$ and all $k\in\Nz$, i.e. $\varphi_{\pi,k}(n)$ is the $k$-th digit (the digit corresponding to $2^k$) in the base $2$ expansion of $\pi(n)$ (note, $\FFf{F}_2=(x,x-1)$ and we omit the square brackets indicating equivalence classes to improve readability). Also, by Eqn.~(\ref{EPermB}) we get
\begin{align}
\varphi_{\pi,k}(n)
&=1-\varphi_{\pi,k}(n+2^k)\\
&=\varphi_{\pi,k}(n+2^{k+1})
\end{align}
for all $k\in\Nz$ and $n\in\Z_2$.

For arbitrary $\Z_2$-polynomial $2$-adic systems $\FFf{F}$ and $\FFf{G}$, $2$-permutation polynomial $f$, and $k\in\Nz$ we will prove the relations

\begin{theoremtable}
\theoremitem{(1)}&$\varphi_{\pi,k+2}(n_{000})+\varphi_{\pi,k+2}(n_{010})+\varphi_{\pi,k+3}(n_{100})+\varphi_{\pi,k+3}(n_{110})+\varphi_{\pi,k+3}(n_{101})+\varphi_{\pi,k+3}(n_{111})\equiv$\tabularnewline
%&$\varphi_{\pi,k+2}(n_{100})+\varphi_{\pi,k+2}(n_{110})+\varphi_{\pi,k+3}(n_{000})+\varphi_{\pi,k+3}(n_{010})+\varphi_{\pi,k+3}(n_{001})+\varphi_{\pi,k+3}(n_{011})\equiv$\tabularnewline
&$\varphi_{\pi,k+1}(n_{000})+\varphi_{\pi,k+1}(n_{100})\modulus{2}$\tabularnewline
\theoremitem{(2)}&$\varphi_{\mathrlap{f}\phantom{\pi},k+2}(n_{000})+\varphi_{\mathrlap{f}\phantom{\pi},k+2}(n_{010})+\varphi_{\mathrlap{f}\phantom{\pi},k+3}(n_{000})+\varphi_{\mathrlap{f}\phantom{\pi},k+3}(n_{010})+\varphi_{\mathrlap{f}\phantom{\pi},k+3}(n_{001})+\varphi_{\mathrlap{f}\phantom{\pi},k+3}(n_{011})\equiv$\tabularnewline
%&$\varphi_{\mathrlap{f}\phantom{\pi},k+2}(n_{100})+\varphi_{\mathrlap{f}\phantom{\pi},k+2}(n_{110})+\varphi_{\mathrlap{f}\phantom{\pi},k+3}(n_{100})+\varphi_{\mathrlap{f}\phantom{\pi},k+3}(n_{110})+\varphi_{\mathrlap{f}\phantom{\pi},k+3}(n_{101})+\varphi_{\mathrlap{f}\phantom{\pi},k+3}(n_{111})\equiv$\tabularnewline
&$0\modulus{2}$\tabularnewline
\theoremitem{(3)}&$k\geq1\Rightarrow$\tabularnewline
&\quad$\varphi_{f,k+1}(n_{000})+\varphi_{f,k+2}(n_{000})+\varphi_{f,k+3}(n_{010})+\varphi_{f,k+3}(n_{011})\equiv0\modulus{2}$\tabularnewline
%&(b) $\varphi_{f,k+1}(n_{100})+\varphi_{f,k+2}(n_{100})+\varphi_{f,k+3}(n_{110})+\varphi_{f,k+3}(n_{111})\equiv0\modulus{2}$\tabularnewline
\theoremitem{(4)}&$k\geq1\Rightarrow$\tabularnewline
&\quad$\left(\varphi_{f,k+1}(n_{000})+\varphi_{f,k+2}(n_{010})+\varphi_{f,k+3}(n_{000})\in\set{0,3}\;?\;0:1\right)\equiv$\tabularnewline
&\quad$\varphi_{f,k}(n_{000})+\varphi_{f,k+1}(n_{100})+\varphi_{f,k+2}(n_{000})+\varphi_{f,k+3}(n_{010})+\varphi_{f,k+4}(n_{000})+\varphi_{f,k+4}(n_{001})\modulus{2}$\tabularnewline[0.5\baselineskip]
\ignore{
\theoremitem{(4)}&$k\geq3\Rightarrow$\tabularnewline
&\leavevmode\rlap{(a)}\phantom{(b)} $\varphi_{f,k+0}(n_{000})+\varphi_{f,k+2}(n_{000})+\varphi_{f,k+3}(n_{100})+\varphi_{f,k+3}(n_{111})\equiv$\tabularnewline
&\leavevmode\phantom{(b)} $\varphi_{f,k+0}(n_{000})+\varphi_{f,k+1}(n_{000})\equiv$\tabularnewline
&\leavevmode\phantom{(b)} $0\modulus{2}\quad\lor$\tabularnewline
&\leavevmode\phantom{(b)} $\varphi_{f,k+2}(n_{000})+\varphi_{f,k+2}(n_{100})+\varphi_{f,k+3}(n_{100})+\varphi_{f,k+3}(n_{111})\equiv$\tabularnewline
&\leavevmode\phantom{(b)} $\varphi_{f,k+0}(n_{000})+\varphi_{f,k+1}(n_{000})\equiv$\tabularnewline
&\leavevmode\phantom{(b)} $1\modulus{2}$\tabularnewline
&(b) $\varphi_{f,k+0}(n_{000})+\varphi_{f,k+2}(n_{110})+\varphi_{f,k+3}(n_{000})+\varphi_{f,k+3}(n_{011})\equiv$\tabularnewline
&\leavevmode\phantom{(b)} $\varphi_{f,k+0}(n_{000})+\varphi_{f,k+1}(n_{100})\equiv$\tabularnewline
&\leavevmode\phantom{(b)} $0\modulus{2}\quad\lor$\tabularnewline
&\leavevmode\phantom{(b)} $\varphi_{f,k+2}(n_{010})+\varphi_{f,k+2}(n_{110})+\varphi_{f,k+3}(n_{000})+\varphi_{f,k+3}(n_{011})\equiv$\tabularnewline
&\leavevmode\phantom{(b)} $\varphi_{f,k+0}(n_{000})+\varphi_{f,k+1}(n_{100})\equiv$\tabularnewline
&\leavevmode\phantom{(b)} $1\modulus{2}$.\tabularnewline[0.5\baselineskip]
}
\end{theoremtable}
for all $n\in\Z_2$, where $\pi\ce\pi_{\FFf{F},\FFf{G}}$ and $n_{(a_02^0+\cdots+a_\ell2^\ell)}\ce n_{a_0\ldots a_\ell}\ce n+a_02^{k+0}+\cdots+a_\ell2^{k+\ell}$ for all $\ell\in\Nz$ and $a_0,\ldots,a_\ell\in\set{0,1}$.

To prove (1) we first show that if arbitrary $2$-adic permutations $\pi,\pi_1,\pi_2$ satisfy (1), then so do $\pi^{-1}$ and $\pi_2\circ\pi_1$. For all $n\in\Z_2$ there is an $s_n\in\intintuptoex{2^4}$ such that
\begin{align}
\psi_{\FFf{F}_2,k+4}(n_{(r)})=\psi_{\FFf{F}_2,k}(n)\cdot\psi_{\FFf{F}_2,4}(r+s_n)
\end{align}
for all $r\in\intintuptoex{2^4}$. Furthermore, if $\pi$ is a $2$-adic permutation it follows from the definition that there are unique $m_{\pi,n,[0]},\ldots,m_{\pi,n,[15]}\in\intintuptoex{2^4}$ such that
\begin{align}
\psi_{\FFf{F}_2,k+4}(\pi(n_{(r)}))=\psi_{\FFf{F}_2,k}(\pi(n))\cdot\psi_{\FFf{F}_2,4}(m_{\pi,n,r})
\end{align}
for all $r\in\intintuptoex{2^4}$ and the $m_{\pi,n,r}$ satisfy
\begin{align}
r\equiv s\modulus{2^\ell}\Leftrightarrow m_{\pi,n,r}\equiv m_{\pi,n,s}\modulus{2^\ell}
\end{align}
for all $\ell\in\intintuptoin{4}$ and $r,s\in\intintuptoex{2^\ell}$. Thus there is a bijective function
\begin{align}
\pi_{(n)}:\Z_2\slash2^4\Z_2&\to\Z_2\slash2^4\Z_2\\
\nonumber
r&\mapsto m_{\pi,n,r-s_n}
\end{align}
satisfying $\pi_{(n)}(r+s_n)=m_{\pi,n,r}$ for all $r\in\intintuptoex{2^4}$ and
\begin{align}
\label{ESubtreesA}
r\equiv s\modulus{2^\ell}\Leftrightarrow\pi_{(n)}(r)\equiv\pi_{(n)}(s)\modulus{2^\ell}
\end{align}
for all $\ell\in\intintuptoin{4}$ and $r,s\in\intintuptoex{2^\ell}$. Thus we get
\begin{align}
&\fa n\in\Z_2:\fa r\in\intintuptoex{2^4}:\psi_{\FFf{F}_2,k+4}(n_{(r)})=\psi_{\FFf{F}_2,k}(n)\cdot\psi_{\FFf{F}_2,4}(r+s_n)\\
&\quad\Rightarrow\fa n\in\Z_2:\psi_{\FFf{F}_2,k+4}(\pi^{-1}(\pi(n)))=\psi_{\FFf{F}_2,k}(\pi^{-1}(\pi(n)))\cdot\psi_{\FFf{F}_2,4}(\pi_{(n)}^{-1}(\pi_{(n)}(s_n)))\\
&\quad\Rightarrow\fa n\in\Z_2:\psi_{\FFf{F}_2,k+4}(\pi^{-1}(n))=\psi_{\FFf{F}_2,k}(\pi^{-1}(n))\cdot\psi_{\FFf{F}_2,4}(\pi_{(\pi^{-1}(n))}^{-1}(\pi_{(\pi^{-1}(n))}(s_{\pi^{-1}(n)})))\\
&\quad\Rightarrow\fa n\in\Z_2:\psi_{\FFf{F}_2,k+4}(\pi^{-1}(n))=\psi_{\FFf{F}_2,k}(\pi^{-1}(n))\cdot\psi_{\FFf{F}_2,4}(\pi_{(\pi^{-1}(n))}^{-1}(m_{\pi,\pi^{-1}(n),0}))\\
&\quad\Rightarrow\fa n\in\Z_2:\psi_{\FFf{F}_2,k+4}(\pi^{-1}(n))=\psi_{\FFf{F}_2,k}(\pi^{-1}(n))\cdot\psi_{\FFf{F}_2,4}(\pi_{(\pi^{-1}(n))}^{-1}(s_n))\\
&\quad\Rightarrow\fa n\in\Z_2:\fa r\in\intintuptoex{2^4}:\psi_{\FFf{F}_2,k+4}(\pi^{-1}(n_{(r)}))=\psi_{\FFf{F}_2,k}(\pi^{-1}(n))\cdot\psi_{\FFf{F}_2,4}(\pi_{(\pi^{-1}(n))}^{-1}(s_{n_{(r)}}))\\
&\quad\Rightarrow\fa n\in\Z_2:\fa r\in\intintuptoex{2^4}:\psi_{\FFf{F}_2,k+4}(\pi^{-1}(n_{(r)}))=\psi_{\FFf{F}_2,k}(\pi^{-1}(n))\cdot\psi_{\FFf{F}_2,4}(\pi_{(\pi^{-1}(n))}^{-1}(r+s_n)).
\end{align}
By the definitions we have the following identities
\begin{align}
\varphi_{\pi,k+\ell}(n_{(r)})
&=\psi_{\FFf{F}_2,4}(\pi_{(n)}(r+s_n))[\ell]\\
\varphi_{\pi^{-1},k+\ell}(n_{(r)})
&=\psi_{\FFf{F}_2,4}(\pi_{(\pi^{-1}(n))}^{-1}(r+s_n))[\ell]
\end{align}
for all $n\in\Z_p$, $r\in\intintuptoex{2^4}$, and $\ell\in\intintuptoin{3}$. The condition that $\pi$ satisfies (1) is thus equivalent to
\begin{align}
\label{ESubtreesB}
&\psi_{\FFf{F}_2,4}(\pi_{(n)}(0+s_n))[2]+\psi_{\FFf{F}_2,4}(\pi_{(n)}(2+s_n))[2]+\psi_{\FFf{F}_2,4}(\pi_{(n)}(1+s_n))[3]+\vphantom{a}\\
\nonumber
&\quad\psi_{\FFf{F}_2,4}(\pi_{(n)}(3+s_n))[3]+\psi_{\FFf{F}_2,4}(\pi_{(n)}(5+s_n))[3]+\psi_{\FFf{F}_2,4}(\pi_{(n)}(7+s_n))[3]\equiv\\
&\psi_{\FFf{F}_2,4}(\pi_{(n)}(1+s_n))[2]+\psi_{\FFf{F}_2,4}(\pi_{(n)}(3+s_n))[2]+\psi_{\FFf{F}_2,4}(\pi_{(n)}(0+s_n))[3]+\vphantom{a}\\
\nonumber
&\quad\psi_{\FFf{F}_2,4}(\pi_{(n)}(2+s_n))[3]+\psi_{\FFf{F}_2,4}(\pi_{(n)}(4+s_n))[3]+\psi_{\FFf{F}_2,4}(\pi_{(n)}(6+s_n))[3]\equiv\\
\label{ESubtreesC}
&\psi_{\FFf{F}_2,4}(\pi_{(n)}(0+s_n))[1]+\psi_{\FFf{F}_2,4}(\pi_{(n)}(1+s_n))[1]\modulus{2}
\end{align}
for all $n\in\Z_2$, and the condition that $\pi^{-1}$ satisfies (1) is equivalent to
\begin{align}
\label{ESubtreesD}
&\psi_{\FFf{F}_2,4}(\pi_{(n)}^{-1}(0+s_{\pi^{-1}(n)}))[2]+\psi_{\FFf{F}_2,4}(\pi_{(n)}^{-1}(2+s_{\pi^{-1}(n)}))[2]+\psi_{\FFf{F}_2,4}(\pi_{(n)}^{-1}(1+s_{\pi^{-1}(n)}))[3]+\vphantom{a}\\
\nonumber
&\quad\psi_{\FFf{F}_2,4}(\pi_{(n)}^{-1}(3+s_{\pi^{-1}(n)}))[3]+\psi_{\FFf{F}_2,4}(\pi_{(n)}^{-1}(5+s_{\pi^{-1}(n)}))[3]+\psi_{\FFf{F}_2,4}(\pi_{(n)}^{-1}(7+s_{\pi^{-1}(n)}))[3]\equiv\\
&\psi_{\FFf{F}_2,4}(\pi_{(n)}^{-1}(1+s_{\pi^{-1}(n)}))[2]+\psi_{\FFf{F}_2,4}(\pi_{(n)}^{-1}(3+s_{\pi^{-1}(n)}))[2]+\psi_{\FFf{F}_2,4}(\pi_{(n)}^{-1}(0+s_{\pi^{-1}(n)}))[3]+\vphantom{a}\\
\nonumber
&\quad\psi_{\FFf{F}_2,4}(\pi_{(n)}^{-1}(2+s_{\pi^{-1}(n)}))[3]+\psi_{\FFf{F}_2,4}(\pi_{(n)}^{-1}(4+s_{\pi^{-1}(n)}))[3]+\psi_{\FFf{F}_2,4}(\pi_{(n)}^{-1}(6+s_{\pi^{-1}(n)}))[3]\equiv\\
\label{ESubtreesE}
&\psi_{\FFf{F}_2,4}(\pi_{(n)}^{-1}(0+s_{\pi^{-1}(n)}))[1]+\psi_{\FFf{F}_2,4}(\pi_{(n)}^{-1}(1+s_{\pi^{-1}(n)}))[1]\modulus{2}
\end{align}
for all $n\in\Z_2$ (note that since the condition must be satisfied for all $n\in\Z_2$ and $\pi:\Z_2\to\Z_2$ is bijective, we may interchange $n$ and $\pi^{-1}(n)$ which is what we did here). There are only finitely many options for $\pi_{(n)}:\Z_2\slash2^4\Z_2\to\Z_2\slash2^4\Z_2$ which satisfy Eqn.~(\ref{ESubtreesA}) ($16384$, to be precise) and among those there are $8192$ which also satisfy Eqn.~(\ref{ESubtreesB})~--~(\ref{ESubtreesC}) for at least one choice for $s_n\in\intintuptoex{2^4}$. It can be easily checked using a computer that any of those $8192$ choices for $\pi_{(n)}$ satisfy Eqn.~(\ref{ESubtreesD})~--~(\ref{ESubtreesE}) for all choices for $s_{\pi^{-1}(n)}\in\intintuptoex{2^4}$. This completes the proof that if $\pi$ satisfies (1), then so does $\pi^{-1}$. On the other hand, if $\pi=\pi_2\circ\pi_1$ (which is a $2$-adic permutation by Theorem~\ref{TPermSubgroup}) then
\begin{align}
\psi_{\FFf{F}_2,k+4}(\pi_2\circ\pi_1(n_{(r)}))=\psi_{\FFf{F}_2,k}(\pi_2\circ\pi_1(n))\cdot\psi_{\FFf{F}_2,4}(\smash{(\pi_2)}_{(\pi_1(n))}\circ\smash{(\pi_1)}_{(n)}(r+s_n))
\end{align}
and thus 
\begin{align}
\varphi_{\pi_2\circ\pi_1,k+\ell}(n_{(r)})
&=\psi_{\FFf{F}_2,4}(\smash{(\pi_2)}_{(\pi_1(n))}\circ\smash{(\pi_1)}_{(n)}(r+s_n))[\ell]
\end{align}
for all $n\in\Z_p$, $r\in\intintuptoex{2^4}$, and $\ell\in\intintuptoin{3}$. Thus, if we set $\pi_{(n)}\ce\smash{(\pi_2)}_{(\pi_1(n))}\circ\smash{(\pi_1)}_{(n)}$ for all $n\in\Z_2$, the condition that $\pi_2\circ\pi_1$ satisfies (1) reads exactly as Eqn.~(\ref{ESubtreesB})~--~(\ref{ESubtreesC}) and we can again verify using a computer that if $\smash{(\pi_1)}_{(n)}$ and $\smash{(\pi_2)}_{(\pi_1(n))}$ are chosen among the $8192$ possible options mentioned above, then $\pi_{(n)}=\smash{(\pi_2)}_{(\pi_1(n))}\circ\smash{(\pi_1)}_{(n)}$ satisfies Eqn.~(\ref{ESubtreesB})~--~(\ref{ESubtreesC}) for all choices for $s_{n}\in\intintuptoex{2^4}$. This completes the proof that if $\pi_1$ and $\pi_2$ satisfies (1), then so does $\pi_2\circ\pi_1$. Thus it suffices to prove (1) for the case where $\FFf{G}=\FFf{F}_2$, because a general $\pi=\pi_{\FFf{F},\FFf{G}}$ can always be written as $\pi=\pi_{\FFf{G},\FFf{F}_2}^{-1}\circ\pi_{\FFf{F},\FFf{F}_2}$ (cf. also Lemma~\ref{LPermProp}~(3)).

If $\FFf{G}=\FFf{F}_2$, then
\begin{align}
\varphi_{\pi,k}(n)
&=\psi_{\FFf{F}_2,k+1}(\pi_{k+1}(n))[k]
=\psi_{\FFf{F}_2,k+1}({\psi_{\FFf{F}_2,k+1}}^{-1}(\psi_{\FFf{F},k+1}(n)))[k]
=\psi_{\FFf{F},k+1}(n)[k]\\
&=\FFFDT(\FFf{F})[n][k].
\end{align}
Thus (1) simplifies to
\begin{align}
\label{ESubtreesF}
&\FFFDT(\FFf{F})[n_{000}][k+2]+\FFFDT(\FFf{F})[n_{010}][k+2]+\FFFDT(\FFf{F})[n_{100}][k+3]+\FFFDT(\FFf{F})[n_{110}][k+3]\\
%\nonumber
%&\quad+\FFFDT(\FFf{F})[n_{101}][k+3]+\FFFDT(\FFf{F})[n_{111}][k+3]\equiv\\
%&\FFFDT(\FFf{F})[n_{100}][k+2]+\FFFDT(\FFf{F})[n_{110}][k+2]+\FFFDT(\FFf{F})[n_{000}][k+3]+\FFFDT(\FFf{F})[n_{010}][k+3]\\
\nonumber
&\quad+\FFFDT(\FFf{F})[n_{001}][k+3]+\FFFDT(\FFf{F})[n_{011}][k+3]\equiv\\
\label{ESubtreesG}
&\FFFDT(\FFf{F})[n_{000}][k+1]+\FFFDT(\FFf{F})[n_{100}][k+1]\modulus{2}.
\end{align}
The first step will be to prove the statement for the case $k=0$. This can easily be done with the aid of a computer, since Theorem~\ref{TReduceDegree} implies that it suffices to check the statement for all $\intintuptoex{2^4}$-polynomial $2$-adic systems with degree in $\intintuptoex{4}$. The next step is to prove the auxiliary property
\begin{align}
\label{ESubtreesH}
&\FFf{F}^k(n)-\FFf{F}^k(n+2^{k+\ell})\equiv\FFf{F}^k(n+2^k)-\FFf{F}^k(n+2^k+2^{k+\ell})\modulus{2^{\ell+2}}
\end{align}
for all $k,\ell\in\Nz$ and $n\in\Z_2$ which we will do by induction on $k$. The statement is clearly true for $k=0$. Assume that it is also true for $k\in\Nz$. Then,
\begin{align}
\FFf{F}^k(n)-\FFf{F}^k(n+2^{k+\ell+1})&\equiv\FFf{F}^k(n+2^k)-\FFf{F}^k(n+2^k+2^{k+\ell+1})\\
&\equiv\FFf{F}^k(n+2^{k+1})-\FFf{F}^k(n+2^{k+1}+2^{k+\ell+1})\modulus{2^{\ell+3}}.
\end{align}
If $A\ce\FFf{F}^k(n)$, $B\ce\FFf{F}^k(n+2^{k+\ell+1})$, $C\ce\FFf{F}^k(n+2^{k+1})$, and $D\ce\FFf{F}^k(n+2^{k+1}+2^{k+\ell+1})$ then
\begin{align}
&A-B-C+D\equiv0\modulus{2^{\ell+3}}
\intertext{by definition, and}
&A\equiv B\equiv C\equiv D\modulus{2}\\
&A\equiv B\modulus{2^{\ell+1}}\\
&C\equiv D\modulus{2^{\ell+1}}
\end{align}
by Theorem~\ref{TPropPolyF}~(1). Consequently,
\begin{align}
\FFf{F}(A)-\FFf{F}(B)-\FFf{F}(C)+\FFf{F}(D)
&=\frac{1}{2}\sum_{i=1}^da_i(A^i-B^i-C^i+D^i)
\end{align}
where $\FFf{F}[A\modulo2]=\sum_{i=0}^da_ix^i\in\Z_2[x]$ (note that for the whole theorem we define $0^0\ce1$ and we assume without loss of generality that $\FFf{F}$ is in weak canonical form). Our goal is to show that $\FFf{F}(A)-\FFf{F}(B)-\FFf{F}(C)+\FFf{F}(D)\equiv0\modulus{2^{\ell+2}}$ which obviously would follow if we could prove
\begin{align}
\label{ESubtreesI}
A^i-B^i-C^i+D^i\equiv0\modulus{2^{\ell+3}}
\end{align}
for all $i\in\N$ which we will do by induction on $i$. The statement is clearly true for $i=1$. Now assume that it is also true for some $i\in\N$ and let $S,T,U,V,W,X,Y\in\Z_2$ such that,
\begin{align}
2^{\ell+3}S&=A-B-C+D\\
2^{\ell+3}T&=A^i-B^i-C^i+D^i\\
2^{\ell+1}U&=A-B\\
2^{\ell+1}V&=C-D\\
2W&=B-C\\
2^{\ell+1}X&=A^i-B^i\\
2^{\ell+1}Y&=C^i-D^i.
\end{align}
It is easy to show that,
\begin{align}
&i=1\Rightarrow U=X\\
&i\neq1\Rightarrow(X\modulo2=1\Rightarrow B\modulo2=i\modulo2=1\Rightarrow U\modulo2=X\modulo2)
\end{align}
which implies
\begin{align}
X+U\sum_{j=0}^{i-1}B^j(B-2W)^{i-1-j}\equiv0\modulus{2}.
\end{align}
If we let $Z\in\Z_2$ such that
\begin{align}
2Z&=X+U\sum_{j=0}^{i-1}B^j(B-2W)^{i-1-j}
\end{align}
then
\begin{align}
A^{i+1}-B^{i+1}-C^{i+1}+D^{i+1}&=2^{\ell+3}(TD+WZ+SC^i+2^{\ell-1}X(U+V))\in2^{\ell+3}\Z_2
\end{align}
which completes the proof of Eqn.~(\ref{ESubtreesH}) and thus also the proof of Eqn.~(\ref{ESubtreesI}). Using the fact that Eqn.~(\ref{ESubtreesF})~--~Eqn.~(\ref{ESubtreesG}) is true for $k=0$ together with Eqn.~(\ref{ESubtreesH}), we are now able to prove Eqn.~(\ref{ESubtreesF})~--~Eqn.~(\ref{ESubtreesG}) for general $k$ which will complete the proof of relation (1). We do so again by using a computer to verify that
\begin{align}
&\STf{D}[\mathrlap{a}\phantom{x}][2]+\STf{D}[\mathrlap{c}\phantom{x}][2]+\STf{D}[\mathrlap{b}\phantom{x}][3]+\STf{D}[\mathrlap{d}\phantom{x}][3]+\STf{D}[\mathrlap{f}\phantom{x}][3]+\STf{D}[\mathrlap{h}\phantom{x}][3]\equiv
%&\STf{D}[\mathrlap{b}\phantom{x}][2]+\STf{D}[\mathrlap{d}\phantom{x}][2]+\STf{D}[\mathrlap{a}\phantom{x}][3]+\STf{D}[\mathrlap{c}\phantom{x}][3]+\STf{D}[\mathrlap{e}\phantom{x}][3]+\STf{D}[\mathrlap{g}\phantom{x}][3]\equiv\\
\STf{D}[\mathrlap{a}\phantom{x}][1]+\STf{D}[\mathrlap{b}\phantom{x}][1]\modulus{2}.
\end{align}
for all $\STf{D}\in\SoDigitTables{2}(\STPlength{4},\DTPblock)$ satisfying Eqn.~(\ref{ESubtreesF})~--~Eqn.~(\ref{ESubtreesG}) for $k=0$ (if $\FFFDT(\FFf{F})=\STf{D}$ there) and all $a,b,c,d,e,f,g,h\in\intintuptoex{2^4}$ satisfying
\begin{align}
&\mathrlap{a}\phantom{x}\not\equiv\mathrlap{b}\phantom{x}\not\equiv\mathrlap{c}\phantom{x}\not\equiv\mathrlap{d}\phantom{x}\not\equiv\mathrlap{e}\phantom{x}\not\equiv\mathrlap{f}\phantom{x}\not\equiv\mathrlap{g}\phantom{x}\not\equiv\mathrlap{h}\phantom{x}\modulus2\\
&\mathrlap{a}\phantom{x}\not\equiv\mathrlap{c}\phantom{x}\not\equiv\mathrlap{e}\phantom{x}\not\equiv\mathrlap{g}\phantom{x}\mathrlap{,}\phantom{\vphantom{a}\not\equiv\vphantom{a}}\mathrlap{b}\phantom{x}\not\equiv\mathrlap{d}\phantom{x}\not\equiv\mathrlap{f}\phantom{x}\not\equiv\mathrlap{h}\phantom{x}\modulus2^2\\
&\mathrlap{a}\phantom{x}\not\equiv\mathrlap{e}\phantom{x}\mathrlap{,}\phantom{\vphantom{a}\not\equiv\vphantom{a}}\mathrlap{b}\phantom{x}\not\equiv\mathrlap{f}\phantom{x}\mathrlap{,}\phantom{\vphantom{a}\not\equiv\vphantom{a}}\mathrlap{c}\phantom{x}\not\equiv\mathrlap{g}\phantom{x}\mathrlap{,}\phantom{\vphantom{a}\not\equiv\vphantom{a}}\mathrlap{d}\phantom{x}\not\equiv\mathrlap{h}\phantom{x}\modulus{2^3}\\
&\mathrlap{a}\phantom{x}-\mathrlap{c}\phantom{x}\equiv\mathrlap{b}\phantom{x}-\mathrlap{d}\phantom{x}\equiv\mathrlap{c}\phantom{x}-\mathrlap{e}\phantom{x}\equiv\mathrlap{d}\phantom{x}-\mathrlap{f}\phantom{x}\equiv\mathrlap{e}\phantom{x}-\mathrlap{g}\phantom{x}\equiv\mathrlap{f}\phantom{x}-\mathrlap{h}\phantom{x}\equiv\mathrlap{g}\phantom{x}-\mathrlap{a}\phantom{x}\equiv\mathrlap{h}\phantom{x}-\mathrlap{b}\phantom{x}\modulus{2^{3}}\\
&\mathrlap{a}\phantom{x}-\mathrlap{e}\phantom{x}\equiv\mathrlap{b}\phantom{x}-\mathrlap{f}\phantom{x}\equiv\mathrlap{c}\phantom{x}-\mathrlap{g}\phantom{x}\equiv\mathrlap{d}\phantom{x}-\mathrlap{h}\phantom{x}\equiv\mathrlap{e}\phantom{x}-\mathrlap{a}\phantom{x}\equiv\mathrlap{f}\phantom{x}-\mathrlap{b}\phantom{x}\equiv\mathrlap{g}\phantom{x}-\mathrlap{c}\phantom{x}\equiv\mathrlap{h}\phantom{x}-\mathrlap{d}\phantom{x}\modulus{2^{4}}.
\end{align}
Note that the first three conditions correspond to the block property and the last two correspond to Eqn.~(\ref{ESubtreesH}) for $\ell=1$ and $\ell=2$ if $a,\ldots,h=\FFf{F}^k(n_{000}),\ldots,\FFf{F}^k(n_{111})$. This completes the proof of relation (1).

To prove (2) we assume $k\geq1$ (we leave it to the reader to verify (2) for $k=0$ by, for example, using an adapted version of Lemma~\ref{LReduceDegree} to bound the degree of $f$), $f(x)=\sum_{i=0}^da_ix^i$, and $c\in\Z_2$ and compute
\begin{align}
f\left(n+c2^k\right)
&=\sum_{i=0}^da_i\sum_{j=0}^i\binom{i}{j}c^j2^{jk}n^{i-j}\\
&=\sum_{i=0}^da_i\sum_{j=0}^5\binom{i}{j}c^j2^{jk}n^{i-j}+\sum_{i=6}^da_i\sum_{j=6}^i\binom{i}{j}c^j2^{jk}n^{i-j}\\
&=\sum_{i=0}^5c^i2^{ik}\frac{f^{(i)}(n)}{i!}+2^{k+5}2^{5k-5}\sum_{i=6}^da_i\sum_{j=6}^i\binom{i}{j}c^j2^{(j-6)k}n^{i-j}\\
&\equiv\sum_{i=0}^5c^i2^{ik}\frac{f^{(i)}(n)}{i!}\modulus{2^{k+5}}
\end{align}
where $f^{(i)}$ denotes the $i$-th derivative of $f$. Thus,
\begin{align}
&\FFFDT(\FFf{F}_2)[f(n+c2^k)][k,k+4]\\
&\quad=\FFFDT(\FFf{F}_2)\left[\sum_{i=0}^5c^i2^{ik}\frac{f^{(i)}(n)}{i!}\right][k,k+4]\\
\label{ESubtreesJ}
&\quad=\FFFDT(\FFf{F}_2)\left[\frac{f(n)-f(n)\modulo2^k}{2^k}+cf'(n)+\left(k\leq4\;?\;\sum_{i=2}^5c^i2^{(i-1)k}\frac{f^{(i)}(n)}{i!}:0\right)\right][\intintuptoin{4}].
\end{align}
Consequently, (2) can be rewritten as
\begin{align}
&\varphi_{\mathrlap{f}\phantom{\pi},k+2}(n_{000})+\varphi_{\mathrlap{f}\phantom{\pi},k+2}(n_{010})+\varphi_{\mathrlap{f}\phantom{\pi},k+3}(n_{000})+\varphi_{\mathrlap{f}\phantom{\pi},k+3}(n_{010})+\vphantom{a}\\
\nonumber
&\quad\varphi_{\mathrlap{f}\phantom{\pi},k+3}(n_{001})+\varphi_{\mathrlap{f}\phantom{\pi},k+3}(n_{011})\equiv\\
%\nonumber
%&\varphi_{\mathrlap{f}\phantom{\pi},k+2}(n_{100})+\varphi_{\mathrlap{f}\phantom{\pi},k+2}(n_{110})+\varphi_{\mathrlap{f}\phantom{\pi},k+3}(n_{100})+\varphi_{\mathrlap{f}\phantom{\pi},k+3}(n_{110})+\vphantom{a}\\
%\nonumber
%&\quad\varphi_{\mathrlap{f}\phantom{\pi},k+3}(n_{101})+\varphi_{\mathrlap{f}\phantom{\pi},k+3}(n_{111})\equiv\\
\nonumber
&0\modulus{2}\Leftrightarrow\\
&\FFFDT(\FFf{F}_2)\left[f(n+0\cdot2^k)\right][k+2]+\FFFDT(\FFf{F}_2)\left[f(n+2\cdot2^k)\right][k+2]+\vphantom{a}\\
\nonumber
&\quad\FFFDT(\FFf{F}_2)\left[f(n+0\cdot2^k)\right][k+3]+\FFFDT(\FFf{F}_2)\left[f(n+2\cdot2^k)\right][k+3]+\vphantom{a}\\
\nonumber
&\quad\FFFDT(\FFf{F}_2)\left[f(n+4\cdot2^k)\right][k+3]+\FFFDT(\FFf{F}_2)\left[f(n+6\cdot2^k)\right][k+3]\equiv\\
%\nonumber
%&\FFFDT(\FFf{F}_2)\left[f(n+1\cdot2^k)\right][k+2]+\FFFDT(\FFf{F}_2)\left[f(n+3\cdot2^k)\right][k+2]+\vphantom{a}\\
%\nonumber
%&\quad\FFFDT(\FFf{F}_2)\left[f(n+1\cdot2^k)\right][k+3]+\FFFDT(\FFf{F}_2)\left[f(n+3\cdot2^k)\right][k+3]+\vphantom{a}\\
%\nonumber
%&\quad\FFFDT(\FFf{F}_2)\left[f(n+5\cdot2^k)\right][k+3]+\FFFDT(\FFf{F}_2)\left[f(n+7\cdot2^k)\right][k+3]\equiv\\
\nonumber
&0\modulus{2}\Leftrightarrow\\
\label{ESubtreesK}
&\FFFDT(\FFf{F}_2)\left[u_0\right][2]+\FFFDT(\FFf{F}_2)\left[u_0+2u_1\right][2]+\vphantom{a}\\
\nonumber
&\quad\FFFDT(\FFf{F}_2)\left[u_0\right][3]+\FFFDT(\FFf{F}_2)\left[u_0+2u_1+\left(k=1\;?\;8u_2:0\right)\right][3]+\vphantom{a}\\
\nonumber
&\quad\FFFDT(\FFf{F}_2)\left[u_0+4u_1\right][3]+\FFFDT(\FFf{F}_2)\left[u_0+6u_1+\left(k=1\;?\;8u_2:0\right)\right][3]\equiv\\
%\nonumber
%&\FFFDT(\FFf{F}_2)\left[u_0+u_1+\left(k\leq2\;?\;2^{k}u_2:0\right)+\left(k=1\;?\;4u_3:0\right)\right][2]+\vphantom{a}\\
%\nonumber
%&\quad\FFFDT(\FFf{F}_2)\left[u_0+3u_1+\left(k\leq2\;?\;2^{k}u_2:0\right)+\left(k=1\;?\;4u_3:0\right)\right][2]+\vphantom{a}\\
%\nonumber
%&\quad\FFFDT(\FFf{F}_2)\left[u_0+u_1+\left(k\leq3\;?\;2^{k}u_2:0\right)+\left(k=1\;?\;4u_3+8u_4:0\right)\right][3]+\vphantom{a}\\
%\nonumber
%&\quad\FFFDT(\FFf{F}_2)\left[u_0+3u_1+\left(k\leq3\;?\;9\cdot2^{k}u_2:0\right)+\left(k=1\;?\;12u_3+8u_4:0\right)\right][3]+\vphantom{a}\\
%\nonumber
%&\quad\FFFDT(\FFf{F}_2)\left[u_0+5u_1+\left(k\leq3\;?\;9\cdot2^{k}u_2:0\right)+\left(k=1\;?\;4u_3+8u_4:0\right)\right][3]+\vphantom{a}\\
%\nonumber
%&\quad\FFFDT(\FFf{F}_2)\left[u_0+7u_1+\left(k\leq3\;?\;2^{k}u_2:0\right)+\left(k=1\;?\;12u_3+8u_4:0\right)\right][3]\equiv\\
\nonumber
&0\modulus{2}
\end{align}
where $u_0\ce\left((f(n)-f(n)\modulo2^k)/2^k\right)\modulo2^4\in\intintuptoex{2^4}$, $u_1\ce f'(n)\modulo2^4\in\intintuptoex{2^4}\cap(2\Z+1)$ (if $f'(n)\in2\Z_2$ then $f(n+2^k)\equiv f(n)+2^{k+1}f'(n)/2\equiv f(n)\modulus{2^{k+1}}$ but $n+2^k\not\equiv n\modulus{2^{k+1}}$ which contradicts the assumption that $f$ is a $2$-permutation polynomial), and $u_2=\left(f''(n)/2\right)\modulo2^4\in\intintuptoex{2^4}$. It can easily be verified that Eqn.~(\ref{ESubtreesK}) holds for all $k\in\set{1,2,3,4}$, $u_0,u_2\in\intintuptoex{2^4}$, and $u_1\in\intintuptoex{2^4}\cap(2\Z+1)$ which completes the proof of (2).

The proofs of (3) and (4) can be done in an analogous fashion using again Eqn.~(\ref{ESubtreesJ}). The relations we need to verify in order to prove (3) are 
\begin{align}
&\FFFDT(\FFf{F}_2)\left[u_0\right][1]+\FFFDT(\FFf{F}_2)\left[u_0\right][2]+\FFFDT(\FFf{F}_2)\left[u_0+2u_1+\left(k=1\;?\;8u_2:0\right)\right][3]+\vphantom{a}\\
\nonumber
&\quad\FFFDT(\FFf{F}_2)\left[u_0+6u_1+\left(k=1\;?\;8u_2:0\right)\right][3]\equiv\\
\nonumber
&0\modulus{2}%\\
%&\FFFDT(\FFf{F}_2)\left[u_0+u_1+\left(k=1\;?\;2u_2:0\right)\right][1]+\vphantom{a}\\
%\nonumber
%&\quad\FFFDT(\FFf{F}_2)\left[u_0+u_1+\left(k\leq2\;?\;2^{k}u_2:0\right)+\left(k=1\;?\;4u_3:0\right)\right][2]+\vphantom{a}\\
%\nonumber
%&\quad\FFFDT(\FFf{F}_2)\left[u_0+3u_1+\left(k\leq3\;?\;9\cdot2^{k}u_2:0\right)+\left(k=1\;?\;12u_3+8u_4:0\right)\right][3]+\vphantom{a}\\
%\nonumber
%&\quad\FFFDT(\FFf{F}_2)\left[u_0+7u_1+\left(k\leq3\;?\;2^{k}u_2:0\right)+\left(k=1\;?\;12u_3+8u_4:0\right)\right][3]\equiv0\modulus{2}
\end{align}
for all $k\in\set{1,2,3,4}$, $u_0,u_2\in\intintuptoex{2^4}$, and $u_1\in\intintuptoex{2^4}\cap(2\Z+1)$. To prove (4) we need to show that
\begin{align}
&(\FFFDT(\FFf{F}_2)\left[u_0\right][1]+\FFFDT(\FFf{F}_2)\left[u_0+2u_1\right][2]+\FFFDT(\FFf{F}_2)\left[u_0\right][3]\in\set{0,3}\;?\;0:1)\equiv\\
\nonumber
&\FFFDT(\FFf{F}_2)\left[u_0\right][0]+\FFFDT(\FFf{F}_2)\left[u_0+u_1+\left(k=1\;?\;2u_2:0\right)\right][1]+\FFFDT(\FFf{F}_2)\left[u_0\right][2]+\vphantom{a}\\
\nonumber
&\quad\FFFDT(\FFf{F}_2)\left[u_0+2u_1+\left(k=1\;?\;8u_2:0\right)\right][3]+\FFFDT(\FFf{F}_2)\left[u_0\right][4]+\FFFDT(\FFf{F}_2)\left[u_0+4u_1\right][4]\modulus{2}
\end{align}
for all $k\in\set{1,2,3,4}$, $u_0,u_2\in\intintuptoex{2^5}$, and $u_1\in\intintuptoex{2^5}\cap(2\Z+1)$.

With the relations (1)~--~(4) at our disposal we are now able to prove the original statement of the theorem. It turns out that the isomorphism classes of subtrees of the tree of cycles $\mathcal{G}(\pi)$ of some $2$-adic permutation $\pi$ which satisfies the relations given in (1)~--~(4) are exactly those indicated by black boxes in Figure~\ref{FSubtrees}. Specifically, we will show,
\begin{align}
\mathrlap{S_{2,4}'\supseteq S_{2,4}''}\phantom{U_{2,4}'\supseteq U_{2,4}''}&\ce\big\{\text{isomorphism class of } T\mid\pi\in\SoPermutations{2}\text{ such that $\pi$ satisfies (1)}\\
\nonumber
&\phantom{\vphantom{a}\ce\big\{\text{isomorphism class of } T\mid\vphantom{a}}T\text{ full $4$-layer rooted subtree of $(\mathcal{G}(\pi),c(\pi))$}\big\}\\
\mathrlap{T_{2,4}'\supseteq T_{2,4}''}\phantom{U_{2,4}'\supseteq U_{2,4}''}&\ce\big\{\text{isomorphism class of } T\mid\pi\in\SoPermutations{2}\text{ such that $\pi$ satisfies (1)}\\
\nonumber
&\phantom{\vphantom{a}\ce\big\{\text{isomorphism class of } T\mid\vphantom{a}} T\text{ full $4$-layer rooted subtree of $(\mathcal{G}(\pi),c(\pi))$}\\
\nonumber
&\phantom{\vphantom{a}\ce\big\{\text{isomorphism class of } T\mid\vphantom{a}}\abs{\sigma}>1\text{ for root $(\ell,\sigma)$ of $T$}\big\}\\
\mathrlap{U_{2,4}'\supseteq U_{2,4}''}\phantom{U_{2,4}'\supseteq U_{2,4}''}&\ce\big\{\text{isomorphism class of } T\mid f\in\SoPermutations{2}\text{ such that $f$ satisfies (2), (3)}\\
\nonumber
&\phantom{\vphantom{a}\ce\big\{\text{isomorphism class of } T\mid\vphantom{a}}T\text{ full $4$-layer rooted subtree of $(\mathcal{G}(f),c(f))$}\big\}\\
\mathrlap{V_{2,4}'\supseteq V_{2,4}''}\phantom{U_{2,4}'\supseteq U_{2,4}''}&\ce\big\{\text{isomorphism class of } T\mid f\in\SoPermutations{2}\text{ such that $f$ satisfies (2), (3), (4)}\\
\nonumber
&\phantom{\vphantom{a}\ce\big\{\text{isomorphism class of } T\mid\vphantom{a}}T\text{ full $4$-layer rooted subtree of $(\mathcal{G}(f),c(f))$}\\
\nonumber
&\phantom{\vphantom{a}\ce\big\{\text{isomorphism class of } T\mid\vphantom{a}}\abs{\sigma}>1\text{ for root $(\ell,\sigma)$ of $T$}\big\}
\end{align}
where $S_{2,4}'$, $T_{2,4}'$, $U_{2,4}'$, and $V_{2,4}'$ are the sets of those $71$, $50$, $83$, and $7$ isomorphism classes of $4$-layer rooted trees with out-degrees in $\set{1,2}$ which are claimed to form the sets $S_{2,4}$, $T_{2,4}$, $U_{2,4}$, and $V_{2,4}$ respectively in Figure~\ref{FSubtrees}. Note that we already proved $S_{2,4}'\subseteq S_{2,4}$, $T_{2,4}'\subseteq T_{2,4}$, $U_{2,4}'\subseteq U_{2,4}$, $V_{2,4}'\subseteq V_{2,4}$ (by listing examples) and $S_{2,4}\subseteq S_{2,4}''$, $T_{2,4}\subseteq T_{2,4}''$, $U_{2,4}\subseteq U_{2,4}''$, $V_{2,4}\subseteq V_{2,4}''$ (by showing (1)~--~(4)).

We demonstrate the idea of proof by showing that tree $4$-$52$ from Figure~\ref{FSubtrees} does not belong to $S_{2,4}''$. All other trees can be dealt with in an analogous fashion. Figure~\ref{F4x52} shows a possible realization of tree $4$-$52$ and it is indicated in the caption that $k=4$ and $n\in\set{2,8}$ violate relation (1). In the following we will argue that any such possible realization of tree $4$-$52$ necessarily violates relation (1) and thus the isomorphism class of the tree cannot be contained in $S_{2,4}''$. We start by defining the auxiliary function
\begin{align}
\overline\varphi_{\pi,k}(n):\Z_2&\to\set{0,1}\\
\nonumber
n&\mapsto\left(\varphi_{\id_{\Z_2},k}(n)+\varphi_{\pi,k}(n)\right)\modulo2
\end{align}
for all $2$-adic permutations $\pi$ and all $k\in\Nz$, where $\id_{\Z_2}$ is the identity function on $\Z_2$. Then,
\begin{align}
\label{ESubtreesL}
\overline\varphi_{\pi,k}(n)&=\left(\FFFDT((x,x-1))[n][k]=\FFFDT((x,x-1))[\pi(n)][k]\;?\;0:1\right)
\end{align}
for all $n\in\Z_2$, i.e. $\overline\varphi_{\pi,k}(n)=0$ if the $k$-th binary digits of $n$ and $\pi(n)$ coincide, and $\overline\varphi_{\pi,k}(n)=1$ otherwise. Furthermore,
\begin{align}
&\overline\varphi_{\pi,k}(n)=\overline\varphi_{\pi,k}(n+2^k)
\end{align}
and if $\pi=\pi_{\FFf{F},\FFf{G}}$ for some $\Z_2$-polynomial $2$-adic systems $\FFf{F}$ and $\FFf{G}$ then,
\begin{align}
&\overline\varphi_{\pi,k+1}(n_{0})+\overline\varphi_{\pi,k+1}(n_{1})+\vphantom{a}\\
\nonumber
&\quad\overline\varphi_{\pi,k+2}(n_{00})+\overline\varphi_{\pi,k+2}(n_{01})+\vphantom{a}\\
\nonumber
&\quad\overline\varphi_{\pi,k+3}(n_{100})+\overline\varphi_{\pi,k+3}(n_{110})+\overline\varphi_{\pi,k+3}(n_{101})+\overline\varphi_{\pi,k+3}(n_{111})\\
\nonumber
&\equiv0\modulus{2}\\
&\overline\varphi_{\pi,k+1}(n_{0})+\overline\varphi_{\pi,k+1}(n_{1})+\vphantom{a}\\
\nonumber
&\quad\overline\varphi_{\pi,k+2}(n_{10})+\overline\varphi_{\pi,k+2}(n_{11})+\vphantom{a}\\
\nonumber
&\quad\overline\varphi_{\pi,k+3}(n_{000})+\overline\varphi_{\pi,k+3}(n_{010})+\overline\varphi_{\pi,k+3}(n_{001})+\overline\varphi_{\pi,k+3}(n_{011})\\
\nonumber
&\equiv0\modulus{2}
\end{align}
for all $k\in\Nz$ and $n\in\Z_2$ by (1). Consequently,
\begin{align}
\label{ESubtreesM}
&\smash[b]{\sum_{i=0}^{\ell-1}\Big(}\overline\varphi_{\pi,k+1}(n[i]_{0})+\overline\varphi_{\pi,k+1}(n[i]_{1})+\vphantom{a}\\
\nonumber
&\hphantom{\sum_{i=0}^{\ell-1}\Big(}\overline\varphi_{\pi,k+2}(n[i]_{d[i]0})+\overline\varphi_{\pi,k+2}(n[i]_{d[i]1})+\vphantom{a}\\
\nonumber
&\hphantom{\sum_{i=0}^{\ell-1}\Big(}\overline\varphi_{\pi,k+3}(n[i]_{e[i]00})+\overline\varphi_{\pi,k+3}(n[i]_{e[i]10})+\overline\varphi_{\pi,k+3}(n[i]_{e[i]01})+\overline\varphi_{\pi,k+3}(n[i]_{e[i]11})\smash{\Big)}\\
\nonumber
&\equiv0\modulus{2}
\end{align}
for all $\ell\in\N$, $\Sf{n}\in\CoSequences(\SPboundedby{2^k},\SPlength{\ell})$, and $\Sf{d},\Sf{e}\in\CoSequences(\SPboundedby{\set{0,1}},\SPlength{\ell})$ with $\Sf{e}=1-\Sf{d}$.

Using Eqn.~(\ref{ESubtreesM}) we will now show that tree $4$-$52$ from Figure~\ref{FSubtrees} does not belong to $S_{2,4}''$. For this purpose let $\pi\in\SoPermutations{2}$ and $k\in\Nz$ such that $(\mathcal{G}(\pi),c(\pi))$ contains a subtree $T$ which is isomorphic to tree $4$-$52$. Note that in the example given in Figure~\ref{F4x52} we have $k=4$ and $\sigma_0=(0,2,8,10)$ (to improve readability we omit the square brackets indicating equivalence classes). Let $v_1,\ldots,v_{15}$ denote the remaining vertices of $T$, ordered in a way that is compatible (regarding graph isomorphy) with the ordering given in Figure~\ref{F4x52}, and let $\sigma_0,\ldots,\sigma_{15}$ denote the corresponding cycles. If we set $\ell\ce\abs{\sigma_0}$, then $\abs{\sigma_1}=\abs{\sigma_2}=\abs{\sigma_4}=\abs{\sigma_5}=\ell$, $\abs{\sigma_3}=\abs{\sigma_6}=\abs{\sigma_7}=\abs{\sigma_8}=\abs{\sigma_9}=\abs{\sigma_{12}}=\abs{\sigma_{13}}=\abs{\sigma_{14}}=\abs{\sigma_{15}}=2\ell$, and $\abs{\sigma_{10}}=\abs{\sigma_{11}}=4\ell$ by Corollary~\ref{CCycles}. Furthermore, by the defining properties of $p$-adic permutations (cf. Eqn.~(\ref{EPermA}), Eqn.~(\ref{EPermB}), and Theorem~\ref{TCycles}) we get the following structural properties of the tables representing the binary expansions of the cycles corresponding to the vertices of $T$ (cf. Figure~\ref{F4x52}):

\begin{theoremtable}
$\bullet$&The block below the first row of every child vertex (gray and dark gray parts) is either a copy of the whole table of its parent vertex (if it has a sibling) or two such copies next to each other (if it is an only child).\tabularnewline
$\bullet$&The first rows (light gray parts) of two siblings are ones' complements of each other and the first row of an only child has two parts of equal lengths which are ones' complements of each other.\tabularnewline
$\bullet$&On top of every ``parent block'' the first row of every child vertex $v$ which has a sibling has an even number $c(v)$ of entries which differ from the respective following entries (cyclically) and the first row of every only child has an odd number of such entries (in Figure~\ref{F4x52} the numbers $c(v)$ are given next to the respective vertex). If $v=(k,\sigma)$, $c(v)=\sum_{i=0}^{l-1}\overline\varphi_{\pi,k-1}(\sigma[i])$ by Eqn.~(\ref{ESubtreesL}).\tabularnewline[0.5\baselineskip]
\end{theoremtable}

\noindent
Let $\Sf{n}\in\CoSequences(\SPboundedby{2^k},\SPlength{\ell})$ such that $\sigma_0=[([\Sf{n}[0]],\ldots,[\Sf{n}[\ell-1]])]_{\sim_\sigma}$ ($\Sf{n}$ is one of the sequences $(0,2,8,10)$, $(2,8,10,0)$, $(8,10,0,2)$, $(10,0,2,8)$ in the example given in Figure~\ref{F4x52}). Furthermore, let $\Sf{d}\in\CoSequences(\SPboundedby{\set{0,1}},\SPlength{\ell})$ be the corresponding top row of $v_1$ (i.e. $(0,1,0,1)$, $(1,0,1,0)$, $(0,1,0,1)$, or $(1,0,1,0)$ in the example given in Figure~\ref{F4x52}) and set $\Sf{e}\ce1-\Sf{d}$. Since $v_3$ is an only child and $v_4$, $v_6$, $v_{12}$, and $v_{14}$ all have siblings, $c(v_3)+c(v_4)+c(v_6)+c(v_{12})+c(v_{14})$ is odd. But by definition of $\Sf{n}$, $\Sf{d}$, and $\Sf{e}$, this sum is equal to the sum given in Eqn.~(\ref{ESubtreesM}) which is even if $\pi$ satisfies (1). Consequently, tree $4$-$52$ cannot belong to $S_{2,4}''$ as claimed.

\bgroup
\newcommand{\ab}{\allowbreak}
\newcommand\colorfbox[2]{{\color{#1}\fbox{#2}}}
\setlength{\fboxsep}{2pt}
\setlength{\fboxrule}{1.5pt}

\begin{figure}[H]
\centering
\begin{overpic}[width=0.9\textwidth]{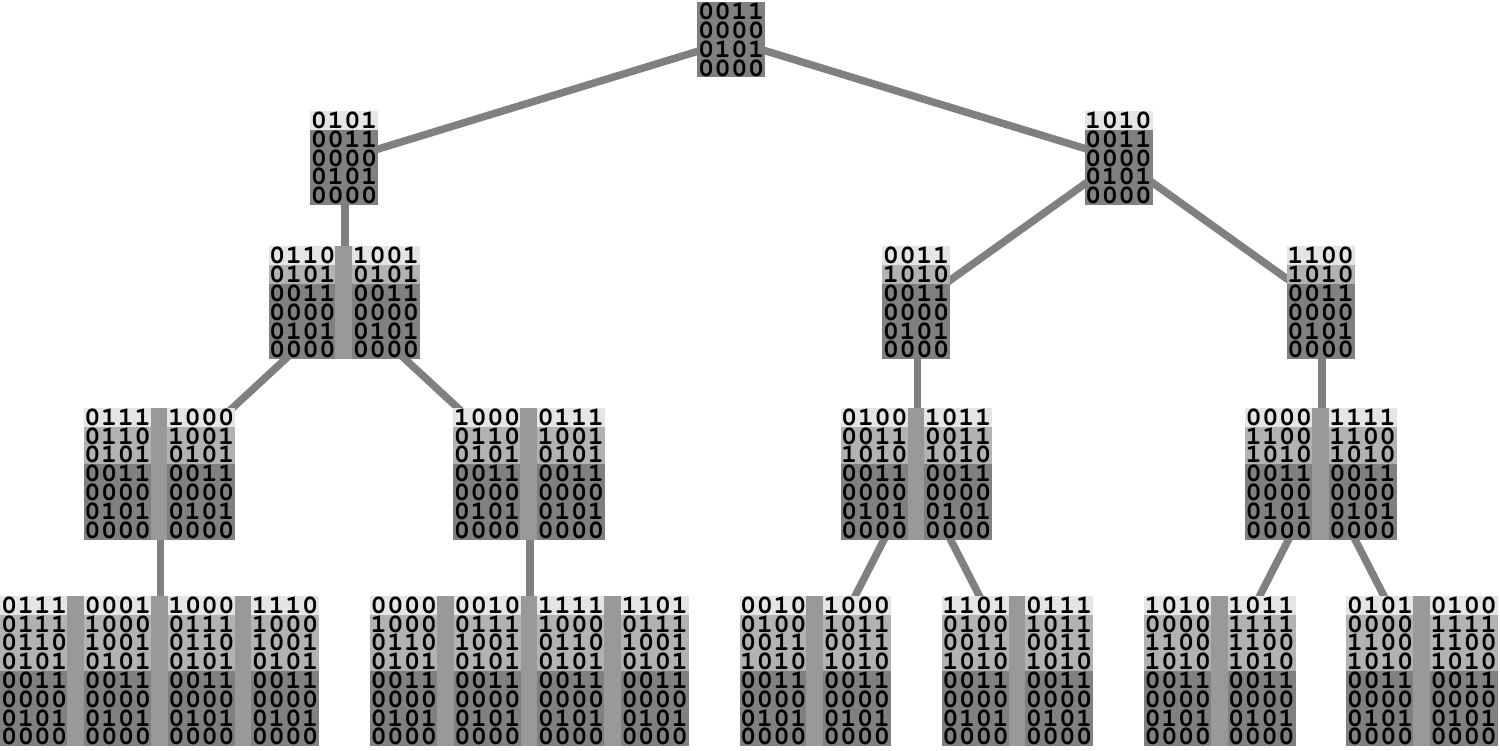}
%\put(47.00,43.71){$v_0$}
%\put(44.29,47.00){$v_0$}
%\put(50.95,47.00){$v_0$}
\put(47.62,43.21){$v_0$}
%\put(23.00,35.19){$v_1$}
%\put(18.56,37.00){$v_1$}
%\put(25.14,37.00){$v_1$}
\put(23.85,34.69){$v_1\text{\scriptsize:\,4}$}
%\put(73.00,35.19){$v_2$}
%\put(70.12,39.00){$v_2$}
%\put(76.76,39.00){$v_2$}
\put(73.44,34.69){$v_2\text{\scriptsize:\,4}$}
\put(17.00,26.07){\colorfbox{gray!50}{\phantom{\rule{17.12pt}{28.93pt}}}}
%\put(21.00,24.92){$v_3$}
%\put(15.71,27.00){$v_3$}
%\put(27.94,27.00){$v_3$}
\put(21.825,24.42){$v_3\text{\scriptsize:\,3}$}
\put(57.91,26.07){\colorfbox{gray!50}{\phantom{\rule{17.40pt}{28.93pt}}}}
%\put(62.00,24.92){$v_4$}
%\put(56.60,27.00){$v_4$}
%\put(63.28,27.00){$v_4$}
\put(61.94,24.42){$v_4\text{\scriptsize:\,2}$}
%\put(89.00,24.92){$v_5$}
%\put(83.61,27.00){$v_5$}
%\put(90.25,27.00){$v_5$}
\put(88.93,24.42){$v_5\text{\scriptsize:\,2}$}
\put(4.67,14.04){\colorfbox{gray!50}{\phantom{\rule{38.85pt}{33.76pt}}}}
%\put(12.00,12.90){$v_6$}
%\put(3.39,16.00){$v_6$}
%\put(15.62,16.00){$v_6$}
\put(11.505,12.40){$v_6\text{\scriptsize:\,2}$}
%\put(36.00,12.90){$v_7$}
%\put(27.97,16.00){$v_7$}
%\put(40.26,16.00){$v_7$}
\put(36.115,12.40){$v_7\text{\scriptsize:\,2}$}
%\put(60.00,12.90){$v_8$}
%\put(53.85,16.00){$v_8$}
%\put(66.08,16.00){$v_8$}
\put(59.965,12.40){$v_8\text{\scriptsize:\,3}$}
%\put(87.00,12.90){$v_9$}
%\put(80.82,16.00){$v_9$}
%\put(93.04,16.00){$v_9$}
\put(86.93,12.40){$v_9\text{\scriptsize:\,1}$}
%\put(9.00,-0.88){$v_{10}$}
%\put(-3.21,3.00){$v_{10}$}
%\put(21.21,3.00){$v_{10}$}
\put(9.00,-1.38){$v_{10}\text{\scriptsize:\,3}$}
%\put(34.00,-0.88){$v_{11}$}
%\put(21.51,3.00){$v_{11}$}
%\put(45.85,3.00){$v_{11}$}
\put(33.68,-1.38){$v_{11}\text{\scriptsize:\,3}$}
\put(48.39,0.26){\colorfbox{gray!50}{\phantom{\rule{38.83pt}{38.58pt}}}}
%\put(53.00,-0.88){$v_{12}$}
%\put(46.08,3.00){$v_{12}$}
%\put(59.33,3.00){$v_{12}$}
\put(52.705,-1.38){$v_{12}\text{\scriptsize:\,4}$}
%\put(66.00,-0.88){$v_{13}$}
%\put(59.55,3.00){$v_{13}$}
%\put(72.82,3.00){$v_{13}$}
\put(66.185,-1.38){$v_{13}\text{\scriptsize:\,4}$}
\put(75.35,0.26){\colorfbox{gray!50}{\phantom{\rule{38.88pt}{38.58pt}}}}
%\put(80.00,-0.88){$v_{14}$}
%\put(73.01,3.00){$v_{14}$}
%\put(86.30,3.00){$v_{14}$}
\put(79.655,-1.38){$v_{14}\text{\scriptsize:\,6}$}
%\put(93.00,-0.88){$v_{15}$}
%\put(86.54,3.00){$v_{15}$}
%\put(99.79,3.00){$v_{15}$}
\put(93.165,-1.38){$v_{15}\text{\scriptsize:\,6}$}
\end{overpic}\\
$ $
\caption{A possible scenario for obtaining tree $4$-$52$ from Figure~\ref{FSubtrees}. The vertices $v_0,\ldots,v_{15}$ correspond to the cycles $\sigma_0=(0,\ab2,\ab8,\ab10)$ of $\pi_4$, $\sigma_1=(0,\ab18,\ab8,\ab26)$, $\sigma_2=(16,\ab2,\ab24,\ab10)$ of $\pi_5$, $\sigma_3=(0,\ab50,\ab40,\ab26,\ab32,\ab18,\ab8,\ab58)$, $\sigma_4=(16,\ab2,\ab56,\ab42)$, $\sigma_5=(48,\ab34,\ab24,\ab10)$ of $\pi_6$, $\sigma_6=(0,\ab114,\ab104,\ab90,\ab96,\ab18,\ab8,\ab58)$, $\sigma_7=(64,\ab50,\ab40,\ab26,\ab32,\ab82,\ab72,\ab122)$, $\sigma_8=(16,\ab66,\ab56,\ab42,\ab80,\ab2,\ab120,\ab106)$, $\sigma_9=(48,\ab34,\ab24,\ab10,\ab112,\ab98,\ab88,\ab74)$ of $\pi_7$, and $\sigma_{10}=(0,\ab242,\ab232,\ab218,\ab96,\ab18,\ab8,\ab186,\ab128,\ab114,\ab104,\ab90,\ab224,\ab146,\ab136,\ab58)$, $\sigma_{11}=(64,\ab50,\ab40,\ab26,\ab32,\ab82,\ab200,\ab122,\ab192,\ab178,\ab168,\ab154,\ab160,\ab210,\ab72,\ab250)$, $\sigma_{12}=(16,\ab66,\ab184,\ab42,\ab208,\ab2,\ab120,\ab106)$, $\sigma_{13}=(144,\ab194,\ab56,\ab170,\ab80,\ab130,\ab248,\ab234)$, $\sigma_{14}=(176,\ab34,\ab152,\ab10,\ab240,\ab98,\ab216,\ab202)$, $\sigma_{15}=(48,\ab162,\ab24,\ab138,\ab112,\ab226,\ab88,\ab74)$ of $\pi_8$, written in base $2$ with most significant digits being in the top rows. The tree violates $\varphi_{\pi,k+2}(n_{000})+\varphi_{\pi,k+2}(n_{010})+\varphi_{\pi,k+3}(n_{100})+\varphi_{\pi,k+3}(n_{110})+\varphi_{\pi,k+3}(n_{101})+\varphi_{\pi,k+3}(n_{111})\equiv\varphi_{\pi,k+1}(n_{000})+\varphi_{\pi,k+1}(n_{100})\modulus{2}$ (relation (1)) for $k=4$ and $n\in\set{2,8}$ ($1+0+0+0+1+0\not\equiv1+0\modulus{2}$ and $0+0+1+1+0+0\not\equiv1+0\modulus{2}$) and can thus not be realized as a subtree of $(\mathcal{G}(\pi_{\FFf{F},\FFf{G}}),c(\pi_{\FFf{F},\FFf{G}}))$ for any $\FFf{F},\FFf{G}\in\SoSystems{2}(\FFPpolynomialcoefficients{\Z_2})$.}
\label{F4x52}
\end{figure}

\egroup
\end{proof}

Theorem~\ref{TSubtrees} finally allows us to prove that there is a $p$-permutation polynomial $f$ which cannot be written as $f=\pi_{\FFf{F},\FFf{G}}$ where $\FFf{F}$ and $\FFf{G}$ are $\Z_p$-polynomial $p$-adic systems as the following example shows.

\begin{example}
\label{ENotPolySyst}
Let $f(x)=2x^3+x+2\in\Z_2[x]$. Then $f$ is a $2$-permutation polynomial by Lemma~\ref{LCharacPermPoly} and $(\mathcal{G}(f),c(f))$ contains a subtree which is isomorphic to tree $4$-$15$ from Figure~\ref{FSubtrees}. Thus $f\neq\pi_{\FFf{F},\FFf{G}}$ for all $\Z_2$-polynomial $2$-adic systems $\FFf{F}$ and $\FFf{G}$ by Theorem~\ref{TSubtrees}.
\end{example}

\bgroup
\newcommand{\e}{\!=\!}
\newcommand{\p}{\!+\!}
\newcommand{\m}{\!-\!}

\setlength{\tabcolsep}{0pt}

\fontsize{6}{6}\selectfont

\begin{figure}[H]
\begin{tabularx}{\textwidth}{>{\hsize=1.0\hsize}X>{\hsize=1.0\hsize}X>{\hsize=1.0\hsize}X>{\hsize=1.0\hsize}X}
\centering\includegraphics[width=0.21\textwidth]{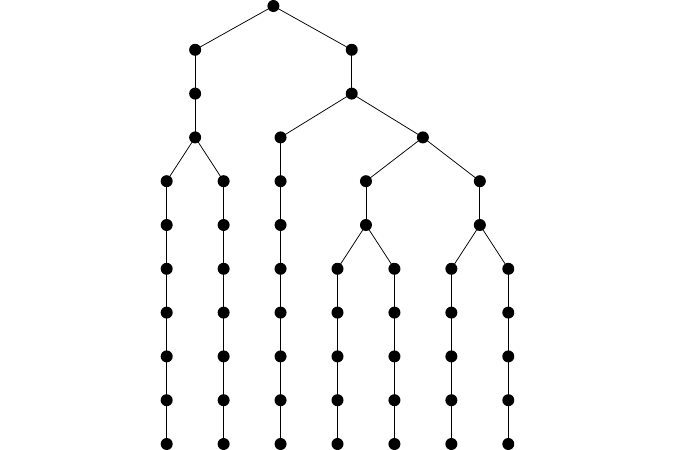}&
\centering\includegraphics[width=0.21\textwidth]{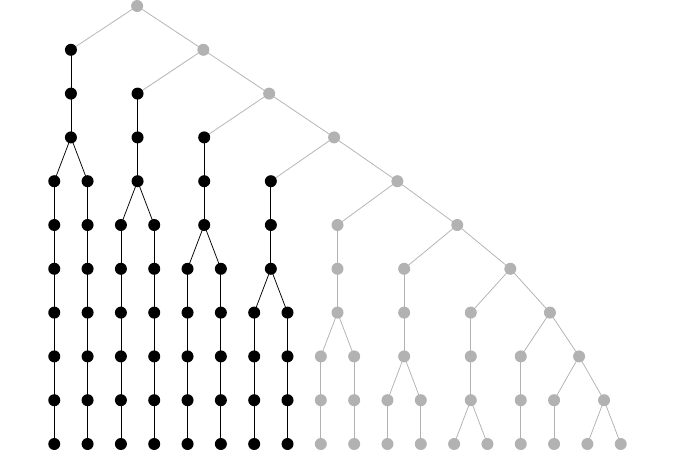}&
\centering\includegraphics[width=0.21\textwidth]{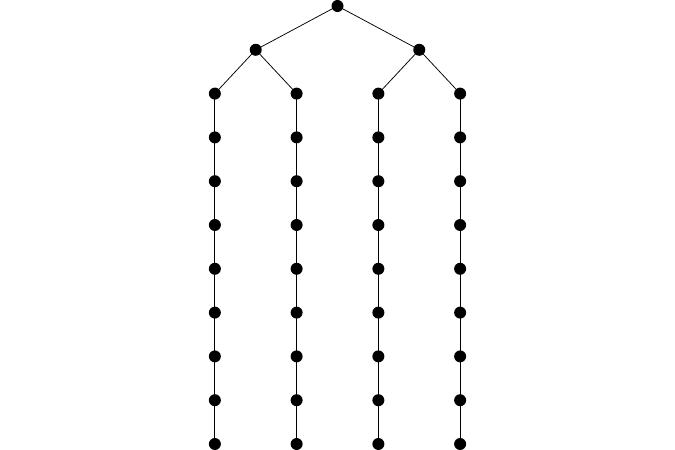}&
\centering\includegraphics[width=0.21\textwidth]{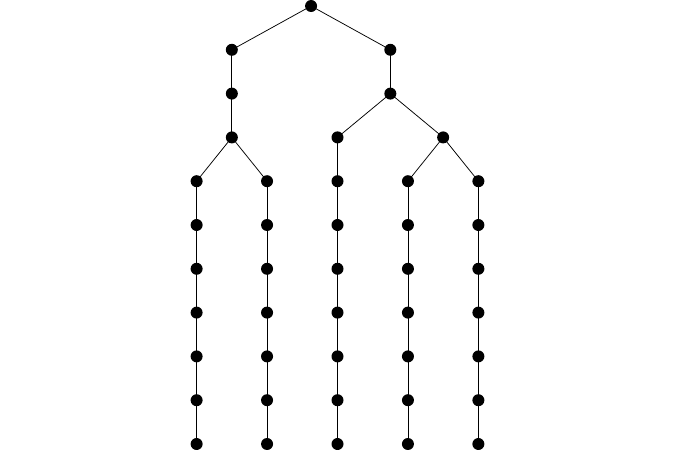}\tabularnewline[0.5\baselineskip]
\centering
$\begin{array}{l}
\mathrlap{\FFf{F}}\phantom{\FFf{G}}\e(x,3x\p1)\\
\FFf{G}\e(x\m2,x\p5)
\end{array}$
&\centering
$\begin{array}{l}
\mathrlap{\FFf{F}}\phantom{\FFf{G}}\e(x,3x\p1)\\
\FFf{G}\e(x,x\m7)
\end{array}$
&\centering
$\begin{array}{l}
\mathrlap{\FFf{F}}\phantom{\FFf{G}}\e(x,5x\p1)\\
\FFf{G}\e(x\m4,x\p1)
\end{array}$
&\centering
$\begin{array}{l}
\mathrlap{\FFf{F}}\phantom{\FFf{G}}\e(5x\p6,-7x\p1)\\
\FFf{G}\e(-3x\p+8,7x\p9)
\end{array}$\tabularnewline
\tabularnewline[0.5\baselineskip]
\centering\includegraphics[width=0.21\textwidth]{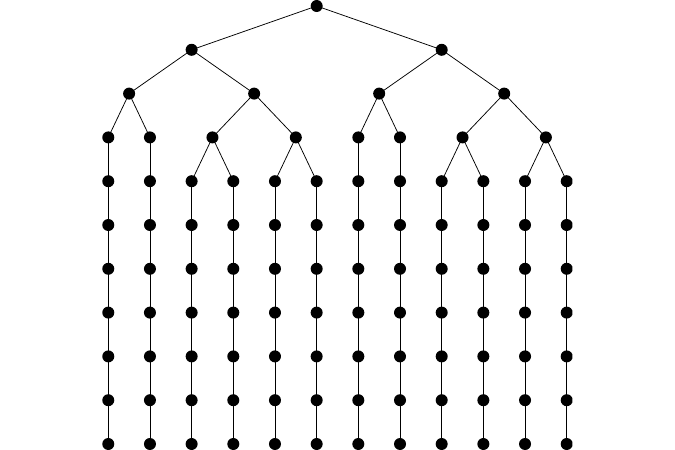}&
\centering\includegraphics[width=0.21\textwidth]{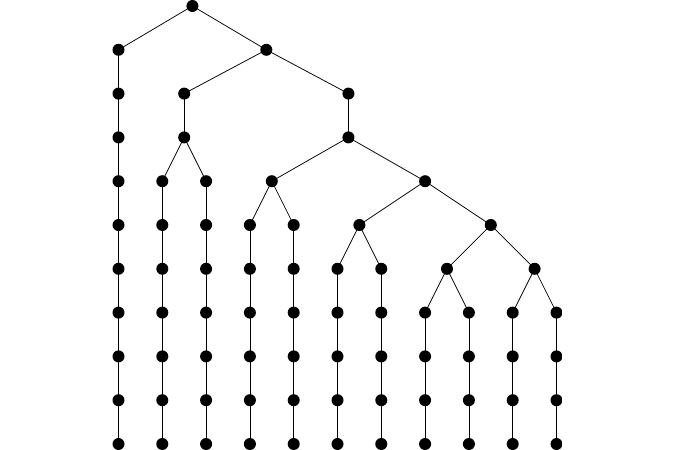}&
\centering\includegraphics[width=0.21\textwidth]{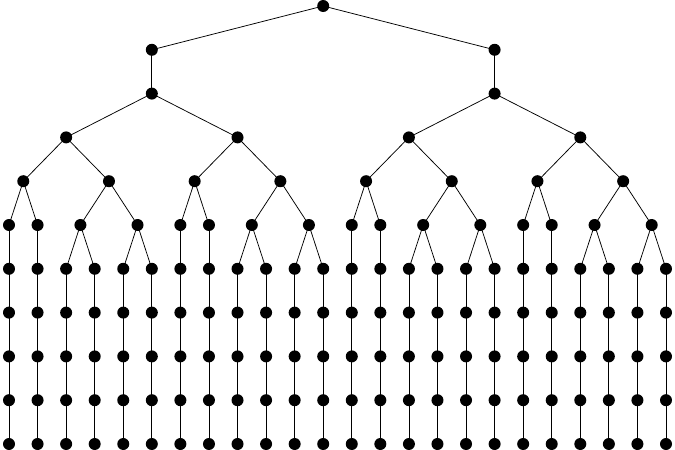}&
\centering\includegraphics[width=0.21\textwidth]{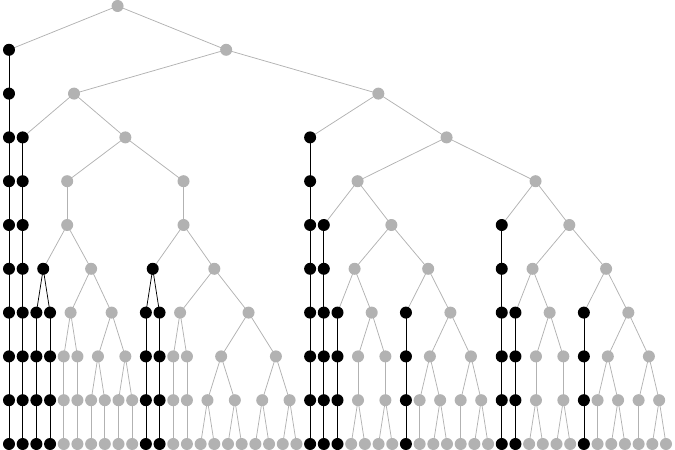}\tabularnewline[0.5\baselineskip]
\centering
$\begin{array}{l}
\mathrlap{\FFf{F}}\phantom{\FFf{G}}\e(-x\p6,3x\p3)\\
\FFf{G}\e(-5x\m2,-x\p7)
\end{array}$
&\centering
$\begin{array}{l}
\mathrlap{\FFf{F}}\phantom{\FFf{G}}\e(x\p4,-5x\p9)\\
\FFf{G}\e(7x\p6,5x\p7)
\end{array}$
&\centering
$\begin{array}{l}
\mathrlap{\FFf{F}}\phantom{\FFf{G}}\e(-9x\m6,-7x\p7)\\
\FFf{G}\e(9x\p4,7x\p7)
\end{array}$
&\centering
$\begin{array}{l}
\mathrlap{\FFf{F}}\phantom{\FFf{G}}\e(-3x\m8,5x\p3)\\
\FFf{G}\e(7x\m6,-x\p5)
\end{array}$\tabularnewline
\tabularnewline[0.5\baselineskip]
\centering\includegraphics[width=0.21\textwidth]{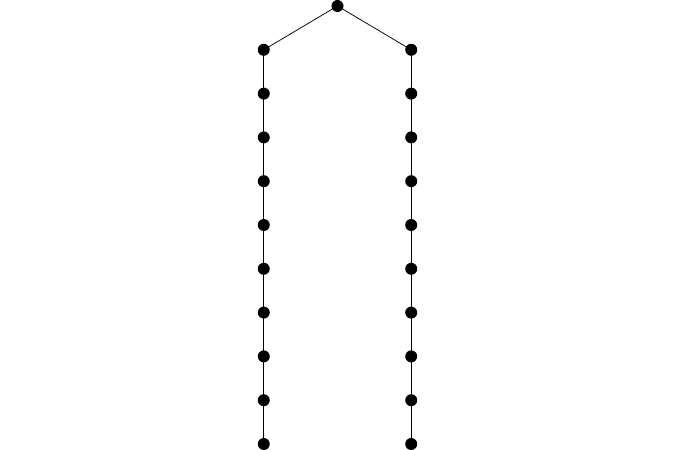}&
\centering\includegraphics[width=0.21\textwidth]{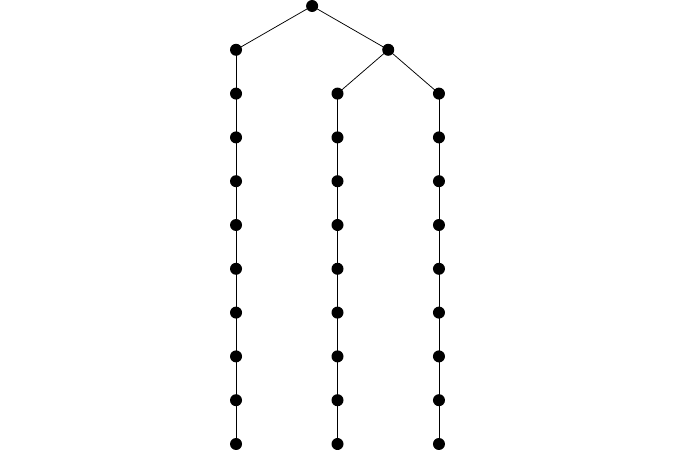}&
\centering\includegraphics[width=0.21\textwidth]{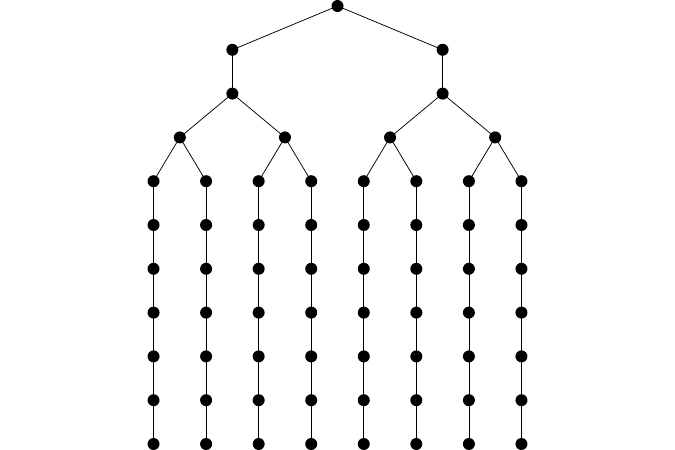}&
\centering\includegraphics[width=0.21\textwidth]{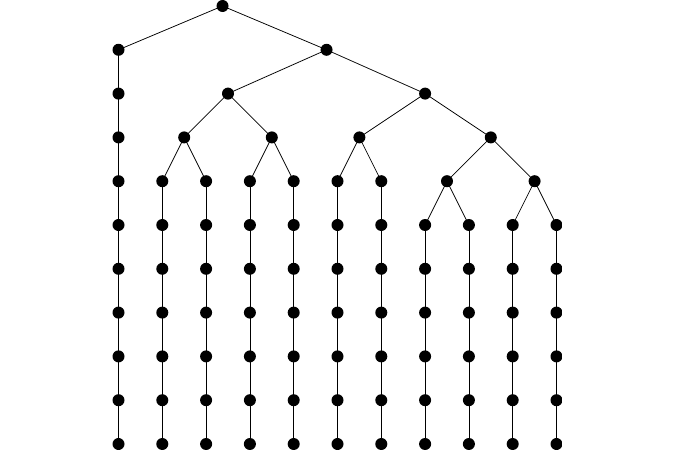}\tabularnewline[0.5\baselineskip]
\centering
$\begin{array}{l}
\mathrlap{\FFf{F}}\phantom{\FFf{G}}\e(5x^4\p5x^3\m5x^2\p7x\p6,\\
\phantom{\FFf{G}\e(}{-6}x^4\p2x^3\m5x^2\m x)\\
\FFf{G}\e(-6x^4\m2x^3\m6x^2\p9x\p8\\
\phantom{\FFf{G}\e(}2x^4\p6x^3\p6x^2\m x\p7)
\end{array}$
&\centering
$\begin{array}{l}
\mathrlap{\FFf{F}}\phantom{\FFf{G}}\e(4x^4\p7x^3\p x^2\m5x,\\
\phantom{\FFf{G}\e(}{-2}x^4\m4x^3\p4x^2\m3x\p3)\\
\FFf{G}\e(4x^4\p3x^3\m x^2\m3x\p2\\
\phantom{\FFf{G}\e(}8x^4\p3x^3\m4x\m5)
\end{array}$
&\centering
$\begin{array}{l}
\mathrlap{\FFf{F}}\phantom{\FFf{G}}\e(8x^4\m8x^3\m3x\p4,\\
\phantom{\FFf{G}\e(}{-}x^4\m4x^3\m5x^2\p9x+1)\\
\FFf{G}\e(-7x^4\p x^3\m x^2\m3x\p6\\
\phantom{\FFf{G}\e(}{-5}x^4\m x^3\p3x^2\p4x\m3)
\end{array}$
&\centering
$\begin{array}{l}
\mathrlap{\FFf{F}}\phantom{\FFf{G}}\e(-4x^4\p3x^3\m x^2\m x\p2,\\
\phantom{\FFf{G}\e(}{-}x^4\m9x^3\m4x^2\m4x\p2)\\
\FFf{G}\e(-5x^4\p6x^2\m x\p8\\
\phantom{\FFf{G}\e(}7x^4\p2x^3\m2x^2\m9x\m2)
\end{array}$\tabularnewline
\tabularnewline[0.5\baselineskip]
\centering\includegraphics[width=0.21\textwidth]{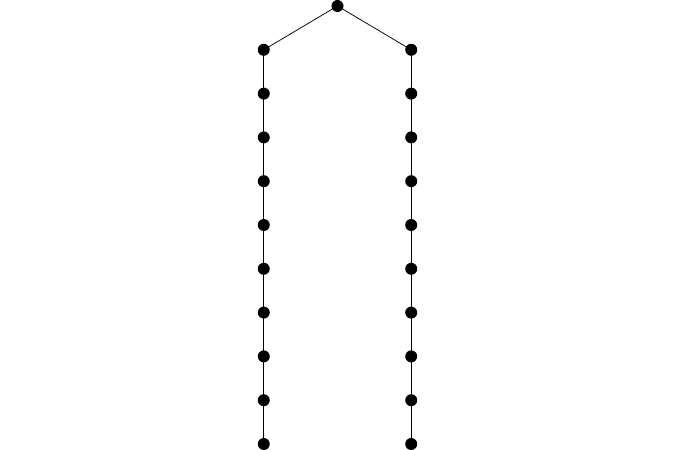}&
\centering\includegraphics[width=0.21\textwidth]{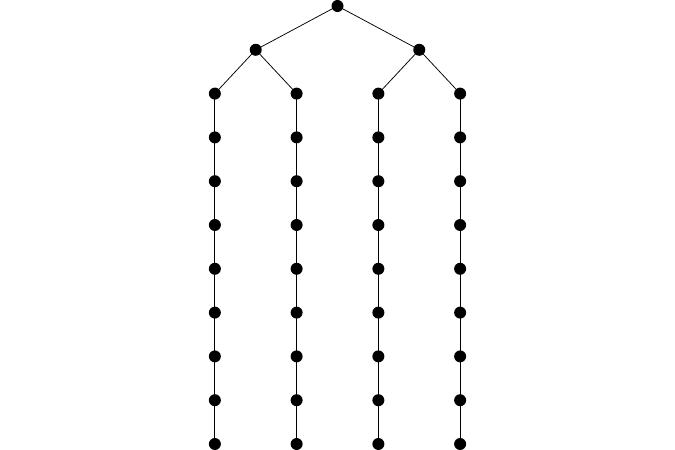}&
\centering\includegraphics[width=0.21\textwidth]{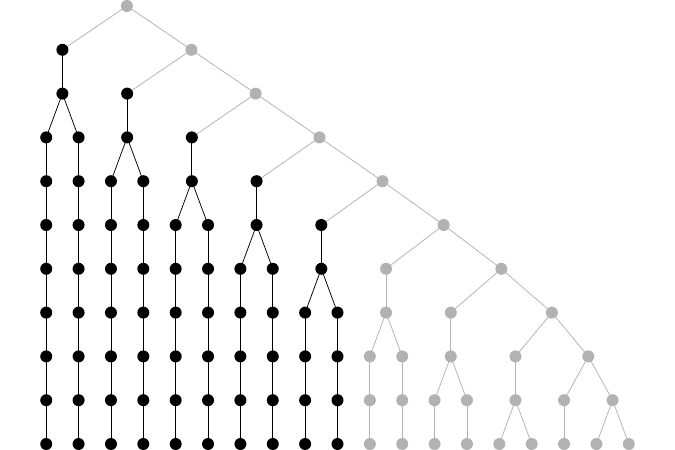}&
\centering\includegraphics[width=0.21\textwidth]{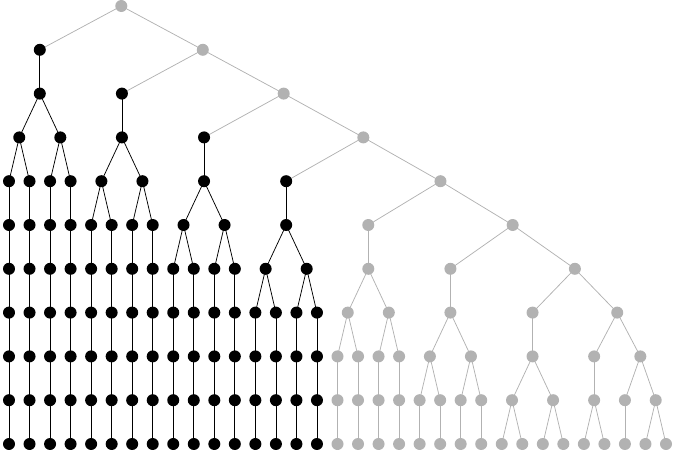}\tabularnewline[0.5\baselineskip]
\centering$f\e5x\m6$&
\centering$f\e x\p4$&
\centering$f\e3x\m4$&
\centering$f\e7x\p4$\tabularnewline
\tabularnewline[0.5\baselineskip]
\centering\includegraphics[width=0.21\textwidth]{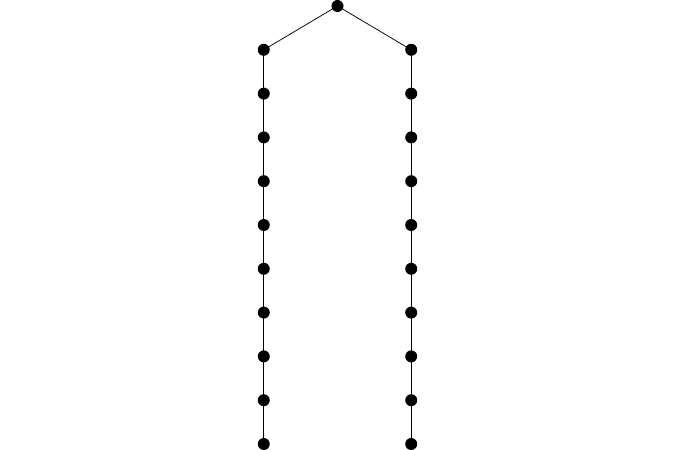}&
\centering\includegraphics[width=0.21\textwidth]{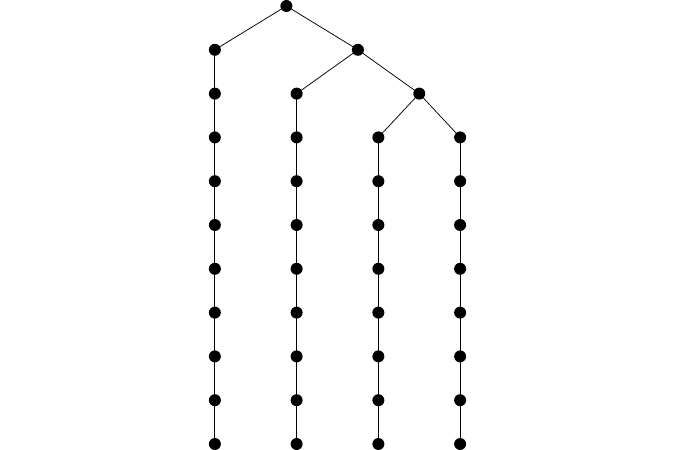}&
\centering\includegraphics[width=0.21\textwidth]{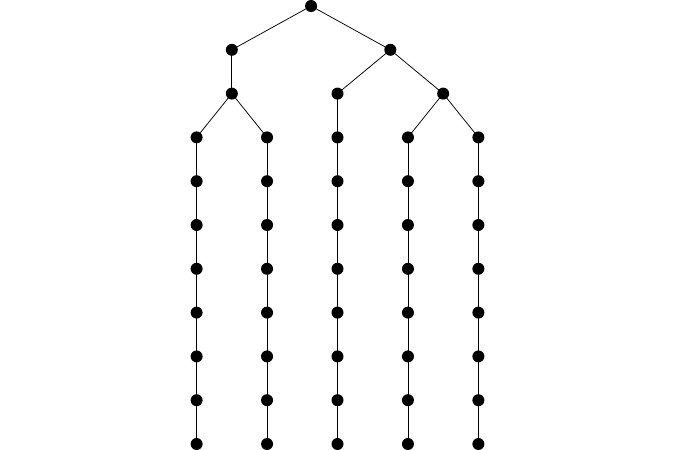}&
\centering\includegraphics[width=0.21\textwidth]{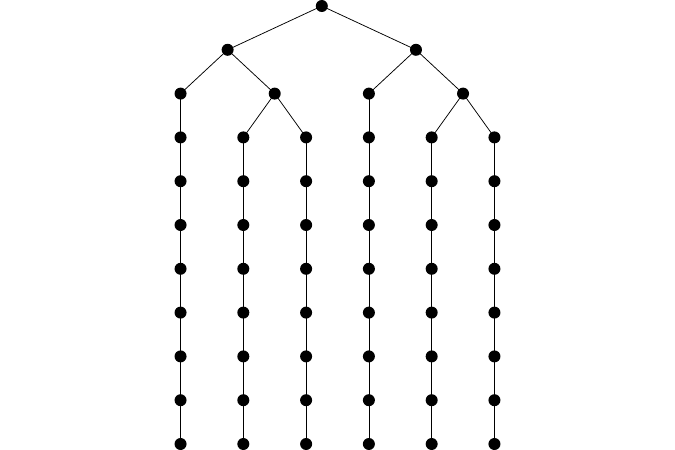}\tabularnewline[0.5\baselineskip]
\centering$f\e-3x^4\p2x^3\m x^2\m9x\m6$&
\centering$f\e-3x^4\p4x^3\m x^2\p7x\p6$&
\centering$f\e x^4\p2x^3\m9x^2\p x\m4$&
\centering$f\e-3x^4\p2x^3\p9x^2\p x\p4$\tabularnewline
\tabularnewline[0.5\baselineskip]
\centering\includegraphics[width=0.21\textwidth]{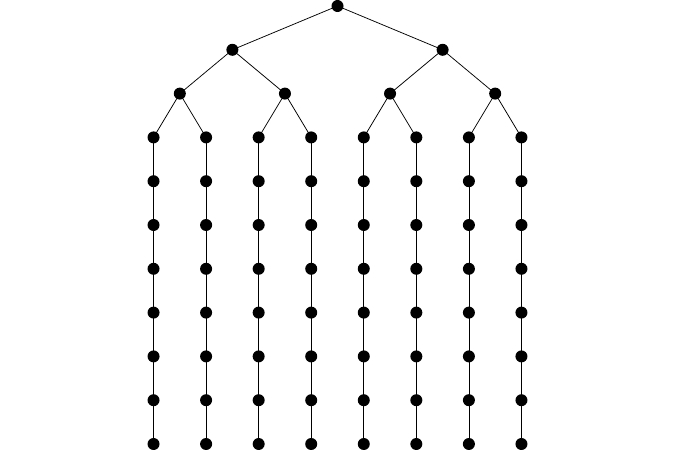}&
\centering\includegraphics[width=0.21\textwidth]{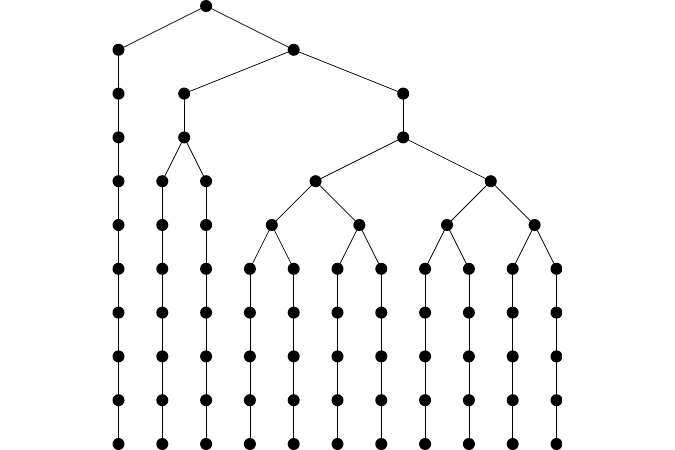}&
\centering\includegraphics[width=0.21\textwidth]{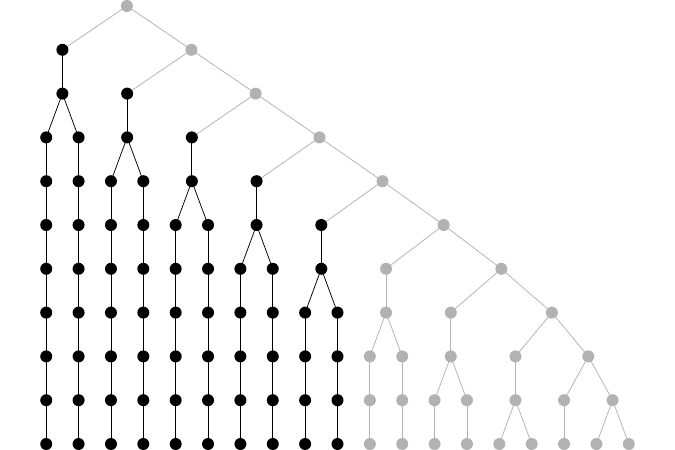}&
\centering\includegraphics[width=0.21\textwidth]{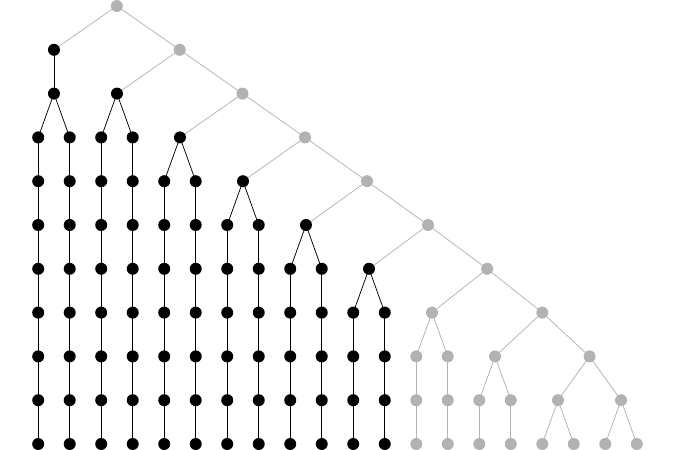}\tabularnewline[0.5\baselineskip]
\centering$f\e4x^4\m4x^2\p x\p8$&
\centering$f\e-9x^4\m6x^3\m5x^2\p7x\p2$&
\centering$f\e4x^4\p4x^2\p3x\m8$&
\centering$f\e-4x^4\m2x^3\m3x\m8$
\end{tabularx}

\caption{Several examples of $2$-cycle trees of the form $\mathcal{G}(\pi_{\FFf{F},\FFf{G}})$, respectively $\mathcal{G}(f)$. Infinite subtrees rooted in black vertices are completely characterized due to Corollary~\ref{CSubtrees}.}
\label{FInfiniteTrees}
\end{figure}

\egroup

In certain situations the following corollary of Theorem~\ref{TSubtrees} allows for the characterization of the complete (infinite) tree $(\mathcal{G}(\pi_{\FFf{F},\FFf{G}}),c(\pi_{\FFf{F},\FFf{G}}))$ where $\FFf{F}$ and $\FFf{G}$ are $\Z_2$-polynomial $2$-adic systems, respectively $(\mathcal{G}(f),c(f))$ where $f$ is a $2$-permutation polynomial. 

\begin{corollary}[No ``Y'' property]
\label{CSubtrees}
Let $\FFf{F},\FFf{G}\in\SoSystems{2}(\FFPpolynomialcoefficients{\Z_2})$, $f$ a $2$-permutation polynomial, and $\pi=\pi_{\FFf{F},\FFf{G}}$ or $\pi=f$. Furthermore, let $\ell\in\N$ and $(v_1,\ldots,v_\ell)$ be a (directed) path in $\mathcal{G}(\pi)$ with $\deg_o(v_i)=1$ (out-degree) for all $i\in\intint{1}{\ell-1}$ and $deg_o(v_\ell)=2$ (i.e. the graph induced by the vertices $v_1,\ldots,v_\ell$ and the two children of $v_\ell$ looks like the letter ``Y''). Then, $\ell\leq3$.
\end{corollary}

\begin{proof}
Follows directly from the absence of black boxes underneath tree $4$-$2$ in Figure~\ref{FSubtrees}.
\end{proof}

\noindent
Informally, the previous corollary states that if a branch of a $2$-cycle tree defined by $\Z_2$-polynomial $2$-adic systems or by $2$-permutation polynomials doesn't split for $3$ consecutive times, it will never split. Thus, several of the trees shown in Figure~\ref{FInfiniteTrees} are completely characterized and others are characterized at least in (infinite) parts.

\newcommand{\qi}[2][\quad\qquad ]{{\rlap{\textbf{#2}}#1}}

\section{Open questions and problems}
\label{SQuestions}
To provide a possible roadmap for future investigations of $p$-adic systems in many different directions we give a list, by no means exhaustive, of potentially interesting questions and problems below. Whenever we refer to ``classes of $p$-adic systems'' (or, analogously, ``classes of $p$-digit tables with block property'' or ``classes of $p$-adic permutations'', cf. Section~\ref{SInterpretations}), we intend this to be understood as any meaningful collection of $p$-adic systems that is described in this article or that will be found during future investigations of $p$-adic systems. Examples of such classes are

\begin{theoremtable}[rrX]
$\bullet$&\multicolumn{2}{X}{the class $\SoSystems{p}(\FFPpolynomialcoefficientsdegree{A}{D})$ of $A$-polynomial $p$-adic systems with degree in $D$ with some natural choices for $A$ and $D$, such as $A=\N,\Nz,\Z,\Q\cap\Z_p,\Z_p,\Q_p,\ldots$ and $D=\set{d},\intintuptoin{d},\N,\ldots$}\tabularnewline
$\bullet$&\multicolumn{2}{X}{$p$-adic systems defined by rational functions or power series, again with possible restrictions to the occurring coefficients and degrees}\tabularnewline
$\bullet$&\multicolumn{2}{X}{$p$-adic systems of the form $\FFf{F}_\STf{D}$ (cf. Theorem~\ref{TASeqDTbl}) where $\STf{D}$ belongs to some class of $p$-digit tables with block property, such as $p$-digit tables defined by sequences like the Thue-Morse sequence (p.~\pageref{DThueMorse})}\tabularnewline
$\bullet$&\multicolumn{2}{X}{$p$-adic systems of the form ${\Pi_\FFf{G}}^{-1}(f)$ where $\FFf{G}$ is a fixed $p$-adic system or itself ranging over some class of $p$-adic systems and $f$ belongs to some class of $p$-adic permutations, such as}\tabularnewline
&$\circ$&$p$-permutation polynomials with possible restrictions to the occurring coefficients and degrees\tabularnewline
&$\circ$&$p$-adic permutations of the form $\pi_{\FFf{F},\FFf{G}}$ with $\FFf{F}$ and $\FFf{G}$ again belonging to some classes of $p$-adic systems\tabularnewline
&$\circ$&finite products $\FFf{F}_1\circ_\FFf{G}\ldots\circ_\FFf{G}\FFf{F}_\ell$ of $p$-adic systems from a certain class (cf. Theorem~\ref{TPermSubgroup} et seq.) with a fixed or bounded number $\ell$ of operands\tabularnewline
&$\circ$&the closure under $\circ_\FFf{G}$ of any union of previously mentioned classes of $p$-adic permutations\tabularnewline
$\bullet$&\multicolumn{2}{X}{any of the above, with bounds on $p$ (fixed value, range, only primes, etc.) or restrictions imposed by demanding additional properties, such as being contractive, expansive, of mixed type, avoiding, periodic, ultimately periodic, or aperiodic on some given set, etc.}\tabularnewline
\end{theoremtable}
$ $\\[-0.5\baselineskip]
With these examples of classes of $p$-adic systems at hand we are ready to provide the announced list of open questions and problems.\\[0.25\baselineskip]
\qi{1)}%
Investigate $p$-adic systems from the perspective of them being number systems. For a $p$-adic system $\FFf{F}$ let $\Sf{S}\star_\FFf{F}\Sf{T}\ce\psi_\FFf{F}\left(\psi_\FFf{F}^{-1}(\Sf{S})\star\psi_\FFf{F}^{-1}(\Sf{T})\right)$ for all $\Sf{S},\Sf{T}\in\CoSequences(\SPboundedby{\intintuptoex{p}},\neg\SPfinite)$ and any operation $\star\in\set{+,-,\cdot}$. If $\FFf{F}=\FFf{F}_p$, there are efficient algorithms for the computation of $\Sf{S}\star_\FFf{F}\Sf{T}$. Are there other choices for $\FFf{F}$ for which useful algorithms can be found? Furthermore, if $\FFf{F}=\FFf{F}_p$ there is no known efficient algorithm for the computation of a prime factor of some $n\in\N$ from its $\FFf{F}$-digit expansion $\psi_\FFf{F}(n)$. Are there other choices for $\FFf{F}$ for which efficient algorithms can be found?\\[0.25\baselineskip]
\qi{2)}%
Prove Conjecture~\ref{ConPropRatPolyF} on the characterization of all (weakly) $(p,r)$-suitable rational functions (cf. also Theorem~\ref{TPropRatF} which proves a special case of the conjecture). More generally, characterize all (weakly) $(p,r)$-suitable analytic functions. Generalize Theorem~\ref{TCharacAvPolyF} and characterize all $(p,r)$-avoiding polynomial functions in $\Q_p[x]$, all $(p,r)$-avoiding rational functions, or even all $(p,r)$-avoiding analytic functions.\\[0.25\baselineskip]
\qi{3)}%
Example~\ref{ENoCharacWBlk} demonstrates that the weak block property for $p$-fibred functions does not permit a necessary and sufficient characterization that only considers the functions $\FFf{F}[0],\ldots,\FFf{F}[p-1]$ independently from one another. Furthermore, Example~\ref{ENoCharacWBlkEx} shows that there is a $p$-digit table with weak block property that is the $p$-digit table of a $p$-fibred function but cannot be realized as the $p$-digit table of a $p$-fibred function, whose entries are weakly $(p,r)$-suitable functions. Is there a predicate $P$ on the set of functions on $\Z_p$ other than being weakly $(p,r)$-suitable which satisfies that a $p$-digit table that is the $p$-digit table of a $p$-fibred function has the weak block property if and only if it can be realized as the $p$-digit table of a $p$-fibred function whose entries satisfy the predicate $P$?\\[0.25\baselineskip]
\qi{4)}%
For every $p$-adic permutation $\pi$ and every $p$-adic system $\FFf{G}$ there is a unique $p$-adic system $\FFf{F}$ such that $\pi=\pi_{\FFf{F},\FFf{G}}$ (cf. Theorem~\ref{TPermEq}). If $\pi$ belongs to a certain subclass of $p$-adic permutations, are there particularly ``nice'' choices for $\FFf{F}$ and $\FFf{G}$? As examples consider the $2$-adic permutations \begin{align}
f(x)&=10x^2-3x+4\\
g(x)&=-2x^2+7x-6
\end{align}
from the class of $2$-permutation polynomials. If we set (cf. Example~\ref{EHensel})
\begin{align}
\FFf{F}_1&\ce\left(\frac{\sqrt{200x^2-60x-71}+3}{10},\frac{\sqrt{200x^2-60x-91}+3}{10}\right)\\
\FFf{F}_2&\ce\left(\frac{\sqrt{200x^2-60x+90-161}+3}{10},\frac{\sqrt{3(200x^2-60x+90)-161}+3}{10}\right)\\
\FFf{F}_3&\ce\left(\frac{\sqrt{8x^2-28x+25}+7}{2},\frac{\sqrt{8x^2-28x+29}+7}{2}\right)\\
\FFf{F}_4&\ce\left(\frac{\sqrt{8x^2-28x+22+3}+7}{2},\frac{\sqrt{3(8x^2-28x+22)+3}+7}{2}\right)\\
\FFf{G}_1&\ce(x,x)\\
\FFf{G}_2&\ce(x,3x+1),
\end{align}
then $f=\pi_{\FFf{F}_1,\FFf{G}_1}=\pi_{\FFf{F}_2,\FFf{G}_2}$ and $g=\pi_{\FFf{F}_3,\FFf{G}_1}=\pi_{\FFf{F}_4,\FFf{G}_2}$. Under which conditions can $\FFf{F}$ and $\FFf{G}$ be chosen to be both polynomial or from some other fixed class of $p$-adic systems?\\[0.25\baselineskip]
\qi{5)}%
Study the relation between $p$-adic systems and known sequences which have the $(p,k)$-block property such as the (slightly modified) Thue-Morse sequence (cf. p.~\pageref{DThueMorse}).\\[0.25\baselineskip]
\qi{6)}%
Investigate the group structure of $\left(\SoSystems{p},\circ_\FFf{G}\right)$, respectively $\left(\SoPermutations{p},\circ\right)$. What do the subgroups gained from forming the closure of any of the classes of $p$-adic systems under $\circ_\FFf{G}$ look like? Are any of the classes of $p$-adic systems already closed under $\circ_\FFf{G}$? If $\FFf{F}_1$ and $\FFf{F}_2$ are $p$-adic systems, what is the relation between the sets of periodic, ultimately periodic, or aperiodic points of $\FFf{F}_1$, $\FFf{F}_2$, and their product $\FFf{F}_1\circ_\FFf{G}\FFf{F}_2$? Can $\FFf{F}_C=(x,3x+1)$ be written as the product of other (possibly polynomial) $2$-adic systems whose sets of ultimately periodic points are known? More generally, can specific $p$-adic systems or even all $p$-adic systems from a certain class be written as the product of ``nice'' $p$-adic systems (e.g. whose sets of periodic, ultimately periodic, or aperiodic points are known, which are contractive, expansive, avoiding, etc.)?\\[0.25\baselineskip]
\qi{7)}%
Hensel's Lemma can be used to show that certain real or complex numbers that are defined by polynomial equations (such as $\sqrt{2}$ or $\I$) have counterparts within $\Z_p$ for some $2\leq p\in\N$. Do the generalizations of Hensel's Lemma (Theorem~\ref{TSuitFUnRoot} and Theorem~\ref{TAvoidFRFUnFP}) have similar applications, possibly with respect to other classes of functions?\\[0.25\baselineskip]
\qi{8)}%
Further investigate trees of cycles. What are the possible finite subtrees of trees of cycles of classes of $p$-adic permutations other than those covered by Theorem~\ref{TSubtrees} (especially for $p\geq3$ and $k\geq4$ there)? Is it possible to characterize all trees of cycles of $p$-adic permutations of the form $\pi_{\FFf{F},\FFf{G}}$, where $\FFf{F}$ and $\FFf{G}$ are $\Z_p$-polynomial $p$-adic systems, or $\pi=f$ for some $p$-permutation polynomial $f\in\Z_p[x]$ by extending the results of Theorem~\ref{TSubtrees} (cf. Figure~\ref{FSubSec} below)? If $\pi_1$ and $\pi_2$ are $p$-adic permutations, what can be said about the relation between the trees of cycles $(\mathcal{G}(\pi_1),c(\pi_1))$, $(\mathcal{G}(\pi_2),c(\pi_2))$, and $(\mathcal{G}(\pi_1\circ\pi_2),c(\pi_1\circ\pi_2))$? ``Having identical trees of cycles'' defines an equivalence relation on the set of all $p$-adic systems. Theorem~\ref{TPermFromTree} gives an explicit construction of at least one $p$-adic permutation from a given equivalence class (given by its shared tree of cycles). Find a full characterization of all $p$-adic permutations in a given equivalence class. Does every equivalence class contain elements of a specific class of $p$-adic systems and can they too be characterized? What can be said about the relation between trees of cycles $(\mathcal{G}(\pi_{\FFf{F},\FFf{G}}),c(\pi_{\FFf{F},\FFf{G}}))$ and the sets of periodic, ultimately periodic, and aperiodic points of $\FFf{F}$ and $\FFf{G}$? Conjecture~\ref{CoPeriodsC} states that $\FFf{F}_2=(x,x-1)$ and $\FFf{F}_C=(x,3x+1)$ have identical sets of ultimately periodic points ($\Q\cap\Z_2$), but $\FFf{F}_2$ and $\FFf{F}=(x,5x+1)$ do not (cf. also the ``In particular'' of Theorem~\ref{TConstIrr}). Can trees of cycles shed some light on why this is the case (cf. the first three trees in the first row of Figure~\ref{FInfiniteTrees})?

\begin{figure}[H]
\centering
\includegraphics[width=0.9\textwidth]{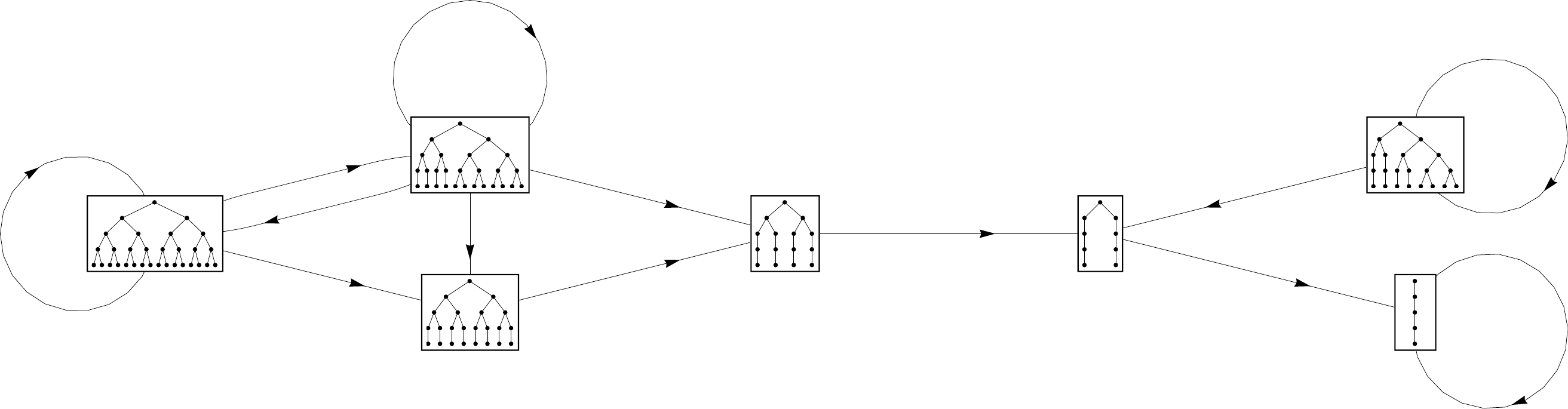}
\caption{The seven elements of $V_{2,4}$ of Theorem~\ref{TSubtrees}. A directed edge from tree $S$ to tree $T$ indicates that $S$ can be extended by $T$ by one layer. It can be seen that tree $4$-$1$ (on the bottom right) is a terminal object of the graph which is essentially the statement of Corollary~\ref{CSubtrees} on trees of cycles of $2$-permutation polynomials (no ``Y'' property). By studying $V_{2,k}$ for $k\geq4$ more terminal objects may be found allowing for a full characterization of all trees of cycles $(\mathcal{G}(\pi),c(\pi))$ where $\pi=f$ for some $2$-permutation polynomial $f$. The corresponding tree representing the extensibility of the $50$-element set $T_{2,4}$ has also only one terminal object (tree $4$-$1$ again, which is also part of the statement of Corollary~\ref{CSubtrees}). Here too it may be possible to find other terminal objects by analyzing $T_{2,k}$ for larger $k$.}
\label{FSubSec}
\end{figure}

\noindent
\qi{9)}%
Prove conjectures \ref{CoPeriodsA}~--~\ref{CoPeriodsF} or at least achieve first non-trivial results on the question of ultimate periodicity of linear-polynomial $p$-adic systems of mixed type, such as proving

\begin{theoremtable}
$\bullet$&$\FFFDT((x^2+x,x))[n]$ aperiodic for some concrete $n\in\Z$\tabularnewline
$\bullet$&$\FFFDT((x^{1000000p}+x)^{p-1}\cdot((p-1)x))[n]$ aperiodic for some concrete $2\leq p\in\N$ and $n\in\Z$\tabularnewline
$\bullet$&$\FFFDT((1000001x,x))[n]$ aperiodic for some concrete $n\in\Z$\tabularnewline
$\bullet$&$\FFFDT(((p^{1000000p}+1)x)^{p-1}\cdot((p-1)x))[n]$ aperiodic for some concrete $2\leq p\in\N$ and $n\in\Z$\tabularnewline
$\bullet$&$\FFFDT((1/(p+1)x,(p-1)x)\cdot(x)^{p-2})[n]$ aperiodic for some concrete $2\leq p\in\N$ and $n\in\Z$\tabularnewline
$\bullet$&$\FFPultimatelyperiodicon{\Q\cap\Z_p}(((p+1)x)\cdot(x)^{p-1})$ for some concrete $2\leq p\in\N$\tabularnewline
$\bullet$&$\FFPultimatelyperiodicon{\Q\cap\Z_p}((1/3x,3x))$\tabularnewline
$\bullet$&$\FFPultimatelyperiodicon{\Q\cap\Z_2}((1/(p-1)x)\cdot((p-1)x)^{p-1})$ for some concrete $3\leq p\in\N$\tabularnewline[0.5\baselineskip]
\end{theoremtable}
(cf. also Corollary~\ref{CConExpSummary}, Figure~\ref{FOvUltPer}, and the subsequent list of examples of $p$-adic systems). Is $\FFFDT((21/5x,5/7x+1))[27]$ ultimately periodic or aperiodic? If it is aperiodic, try to formulate the correct version of the condition $B\in\Z$ in Conjecture~\ref{CoPeriodsD} (cf. the discussion of the issue between Conjecture~\ref{CoPeriodsD} and Conjecture~\ref{CoPeriodsE}). Show $[\Q\cap\Z_p\not\subseteq\FFFSoUPP](\FFf{F})$ for every polynomial $p$-adic system $\FFf{F}$ which is not also $(\Q\cap\Z_p)$-polynomial (cf. Theorem~\ref{TGenPolyUPer} and the subsequent comment).\\[0.25\baselineskip]
\qi{10)}%
Prove Conjecture~\ref{CoConstIrr} and Conjecture~\ref{CoOrderIrr} for $p\geq3$: the constant coefficients and the order of the linear coefficients of $(\Q\cap\Z_p)$-linear-polynomial $p$-adic systems are irrelevant for the question of ultimate periodicity on $\Q\cap\Z_p$. Considering the condition $\abs{B}<p^p$ in Conjecture~\ref{CoPeriodsC}, prove for $p\geq2$ that the signs of the linear coefficients of $(\Q\cap\Z_p)$-linear-polynomial $p$-adic systems are irrelevant for the question of ultimate periodicity on $\Q\cap\Z_p$. Do other classes of $p$-adic systems (especially $(\Q\cap\Z_p)$-polynomial $p$-adic systems of larger degrees) have similar symmetries which also fix the sets of periodic, ultimately periodic, or aperiodic points?\\[0.25\baselineskip]
\qi{11)}%
What can be said about the sets $\FFFSoPP(\FFf{F})$, $\FFFSoUPP(\FFf{F})$, and $\FFFSoAPP(\FFf{F})$ for a $p$-adic system $\FFf{F}$ from a specific class? Do these sets have any structure, symmetries, invariants, etc.? The second generalization of Hensel's Lemma (Theorem~\ref{TAvoidFRFUnFP}) reveals structural properties of the set $\FFFSoPP(\FFf{F})$ if $\FFf{F}$ is an avoiding $\Z_p$-polynomial $p$-adic system, as the example following Theorem~\ref{TCharacAvPolyF} shows: if $\FFf{F}=(7x^3-4x^2+x-6,3x^7-x+1,5x^4+4x-1)$, then the set of periodic points of $\FFf{F}$ is equal to the set of all fixed points of arbitrary compositions of the polynomial functions $\FFf{F}[0]/3$, $\FFf{F}[1]/3$, and $\FFf{F}[2]/3$. $\FFFSoPP(\FFf{F})$ can be interpreted as a ``generalized zero set'' defined by three polynomials in $\Q_3[x]$.\\[0.25\baselineskip]
\qi{12)}%
Study any of the sets $\FFFSoPP(A)$, $\FFFSoUPP(A)$, $\FFFSoAPP(A)$, $\FFFSoGenPP(A)$, $\FFFSoGenUPP(A)$, or $\FFFSoGenAPP(A)$ for any class $A$ of $p$-adic systems (cf. the subsection ``Generalizations'' of Section~\ref{SLinPoly}). Do these sets have any structure, symmetries, invariants, etc.? Theorem~\ref{TConstIrr} and Theorem~\ref{TOrderIrr} provide first results in this direction: for $\FFf{F}$ from the class $A$ of $(\Q\cap\Z_2)$-linear-polynomial $2$-adic systems , $\FFFSoUPP(\FFf{F})$ is invariant under the change of constant coefficients or the change of the order of linear coefficients. What is
\begin{align}
&\FFFSoGenUPP(\set{\FFf{F}})\cap\SoSystems{p}(\FFPpolynomialcoefficients{\Q\cap\Z_p}),
\end{align}
i.e. for which $(\Q\cap\Z_p)$-polynomial $p$-adic systems $\FFf{G}$ does one get $\FFFSoUPP(\FFf{F})=\FFFSoUPP(\FFf{G})$, where $\FFf{F}=(x,5x+1)$, $\FFf{F}=(5x,5x+1)$, $\FFf{F}=(x,x^2+x)$, $\FFf{F}=(x^2+x,x^2+x)$, or $\FFf{F}$ is some other concrete $(\Q\cap\Z_p)$-polynomial $p$-adic system? This can be seen as the inverse problem of the previous question 11): for a fixed ``generalized zero set'' $Z\subseteq\Z_p$, what can be said about the $p$-adic systems from a certain class whose sets of periodic points are equal to the given $Z$? Is there a generalized Galois theory in this setting?\\[0.25\baselineskip]
\qi{13)}%
If $\FFf{F}$ is a $(\Q\cap\Z_p)$-linear-polynomial $p$-adic system, then $[\FFFSoUPP\subseteq\Q\cap\Z_p](\FFf{F})$ by the ``In particular'' part of Corollary~\ref{CPerFormula} (cf. also Conjecture~\ref{CoPeriodsD} on when we have $[\FFFSoUPP=\Q\cap\Z_p](\FFf{F})$). Find any such $\FFf{F}$ and a $p$-adic integer $n$ whose $\FFf{F}$-digit expansion is aperiodic and equal to any known sequence in $\CoSequences(\SPboundedby{\intintuptoex{p}},\neg\SPfinite)$ (like the real base $p$ expansion of some irrational number or the Thue-Morse sequence). Specifically, what is $\pi_{(x,3x+1),(x,x-1)}(\sqrt{17})$ (cf. Example~\ref{EHensel})?\\[0.25\baselineskip]
\qi{14)}%
For every $2\leq p\in\N$, every countable subset $A$ of $\Z_p$ and every finite $\intintuptoex{p}$-bounded sequence $\Sf{S}$ there is a $p$-adic system $\FFf{F}$ such that the periodic part of the $\FFf{F}$-digit expansion of every $n$ in $A$ is cyclically equivalent to $\Sf{S}$ (cf. the definition of $\sim_\sigma$ on p.~\pageref{Dcyceq}). In order to find such an $\FFf{F}$, pick any countable subset $B$ of $\Z_p$ which is dense in $\Z_p$ and construct a $p$-digit table $\STf{D}$ by fixing an ultimately periodic $\STf{D}$-digit expansion with period $\Sf{S}$ (or an aperiodic $\STf{D}$-digit expansion if $\Sf{S}$ is empty) for the elements of $A\cup B$ one at a time, in a way that is compatible (regarding the block property) with what has already been fixed. After $\STf{D}$ has been constructed, let $\STf{E}\in\SoTables{p}$ be its unique extension by Lemma~\ref{LDTExt} and let  $\FFf{F}\ce\FFf{F}_\STf{E}$ be the $p$-adic system corresponding to $\STf{E}$ according to Theorem~\ref{TASeqDTbl}. For specific choices for $A$ and $\Sf{S}$, are there ``nice'' $p$-adic systems $\FFf{F}$ with the described property? As an example consider the non-standard ternary (cf. \cite{vandeWoestijne:2008,Kenyon:2015}) system $\FFf{F}=(x,x+1,x-1)$ which has the property that all integers (i.e. $A=\Z$) have an ultimately periodic $\FFf{F}$-digit expansion with period $(0)$ (i.e. $\Sf{S}=(0)$). This can be proven by verifying that the $\FFf{F}$-sequences of all integers $n$ with $\abs{n}\leq2$ are ultimately periodic with period $(0)$ (cf. the definition of $M$ in the proof of Theorem~\ref{TLinPolyCEM} and the ``In particular'' part of Lemma~\ref{LContrExp}~(1)). The ``niceness'' of $\FFf{F}$ in this case is of course given by the fact that $\FFf{F}$ is $\Z$-linear-polynomial. Is there a ``nice'' (from a specific class, closed on the integers, etc.) $p$-adic system $\FFf{F}$ which has the property that all rational numbers (i.e. $A=\Q\cap\Z_p$) have an ultimately periodic $\FFf{F}$-digit expansion with period $(0)$?\\[0.25\baselineskip]
\qi{15)}%
If $\FFf{F}$ and $\FFf{G}$ are concrete $p$-adic system systems and $f$ is a concrete $p$-permutation polynomial, plotting any of the permutations $\pi_k$ of $\intintuptoex{p^k}$ (we identify $\Z_p\slash p^k\Z_p$ and $\intintuptoex{p^k}$) for $\pi\ce\pi_{\FFf{F},\FFf{G}}$ or $\pi=f$ often reveals intriguing patterns, as Figure~\ref{FPermutations} shows. Study such permutations with regard to randomness and discrepancy (see \cite{DrmotaTichy:1997} for an introduction to discrepancy theory and notions of randomness).
\begin{figure}[H]
\centering
\includegraphics[height=6.1cm]{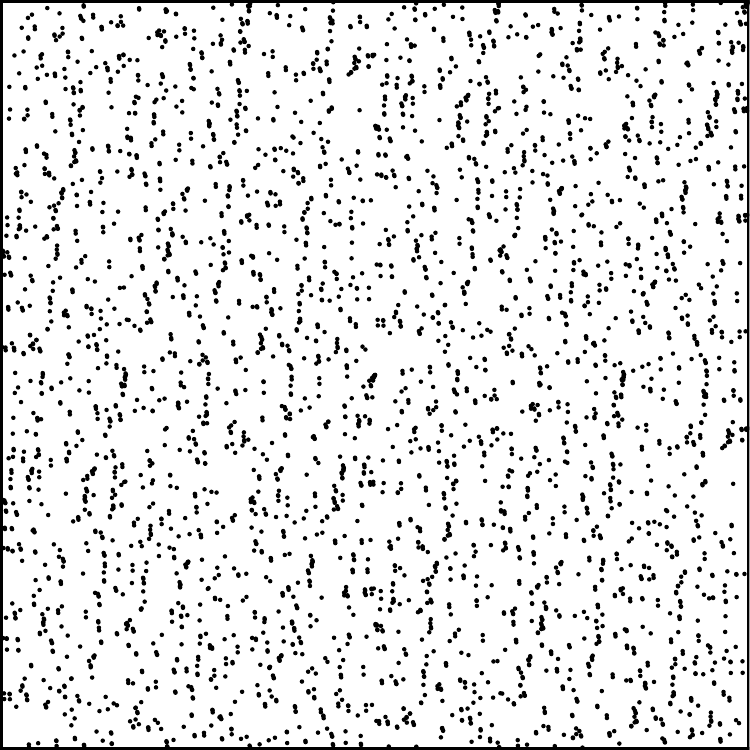}
\qquad
\includegraphics[height=6.1cm]{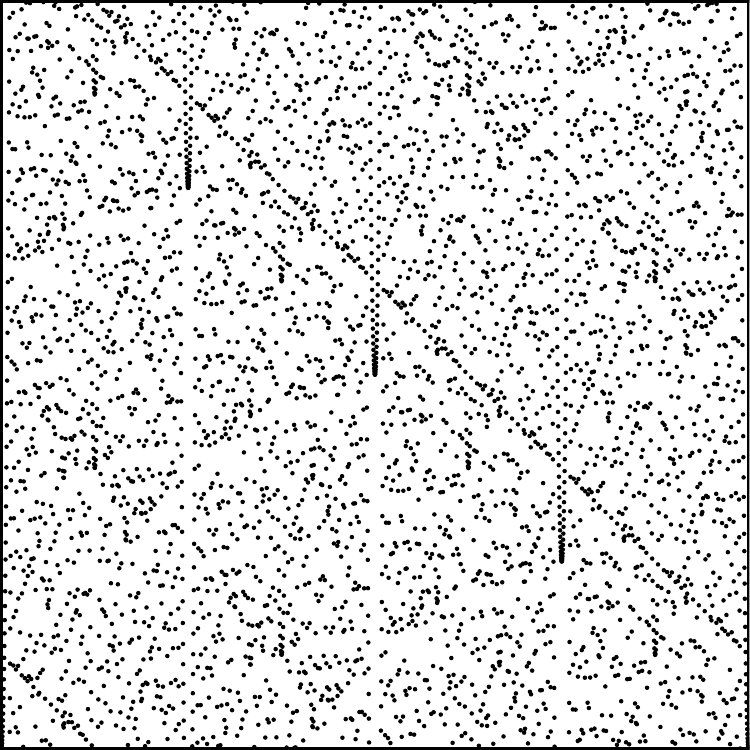}\\[1.5\baselineskip]
\includegraphics[height=6.1cm]{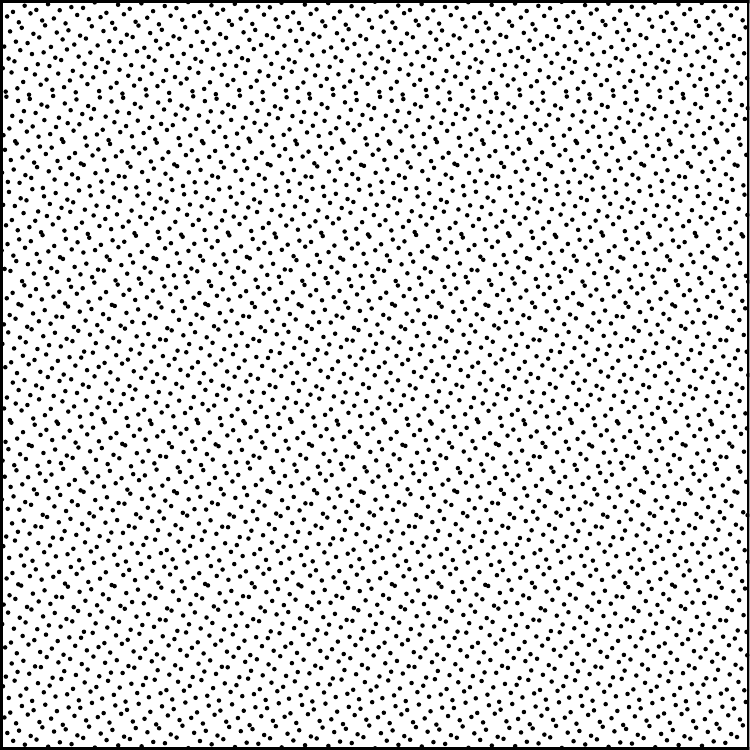}
\qquad
\includegraphics[height=6.1cm]{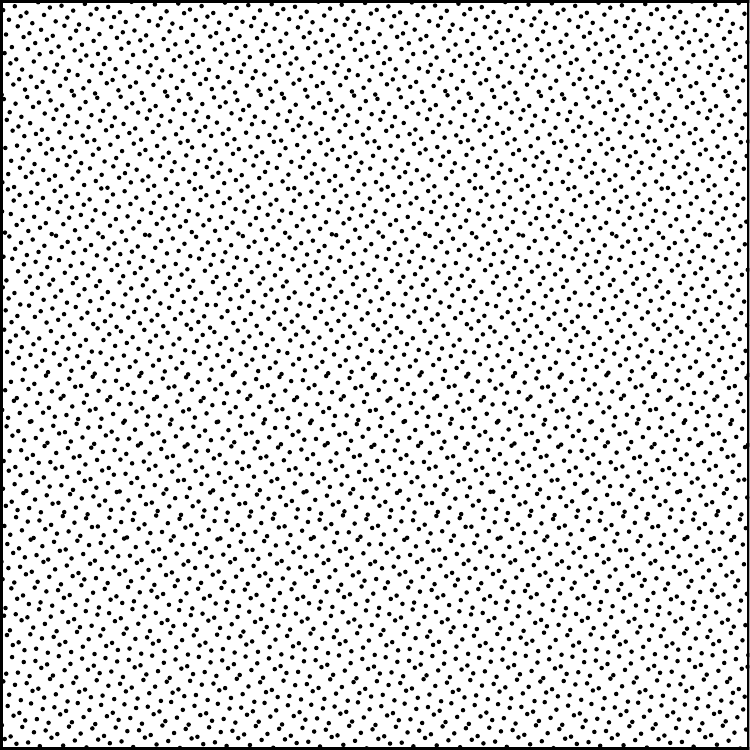}
\caption{The permutation $\pi_{12}:\underline{2^{12}}\to\underline{2^{12}}$ for $\pi=\pi_{(x,3x+1),(x,x-1)}$ (top left) and $\pi=2x^2+3x+2$ (top right). The respective images in the bottom row show the permutations $\operatorname{inv}\circ\pi_{12}:\underline{2^{12}}\to\underline{2^{12}}$ where $\operatorname{inv}:\underline{2^{12}}\to\underline{2^{12}}$ inverts the $12$ digits of the binary expansion of its input (e.g. $\operatorname{inv}(318)=1992$).}
\label{FPermutations}
\end{figure}

\section*{Acknowledgment}
\noindent
The author is supported by the FWF stand-alone project P 30205-NBL ``Arithmetic Dynamical Systems, Polynomials, and Polytopes'' (\url{http://marioweitzer.com/fwf_adspp}).

\bibliographystyle{siam}
\bibliography{p-adic_systems_-_weitzer}

\section*{Appendix}
\myparagraphtoc{$p$-adic pseudo-valuations.}
For $2\leq p\in\N$ let
\begin{align}
\tag{A.1}
\nu_p:\Q_p&\to\Z\cup\set{\infty}.\\
\nonumber
x&\mapsto
\begin{cases}
\max\set{k\in\Nz\mid x/p^k\in\Z_p}&\text{if }x\in\Z_p\setminus\set{0}\\
-\min\set{k\in\N\mid xp^k\in\Z_p}&\text{if }x\in\Q_p\setminus\Z_p\\
\infty&\text{if }x=0
\end{cases}
\end{align}
If $p$ is a prime then $\nu_p$ is the $p$-adic valuation. If $p$ is not a prime then $\nu_p$ is not a valuation as the multiplicative property $\nu_p(ab)=\nu_p(a)+\nu(b)$ is violated in general (not only due to the existence of zero divisors but even if $p$ is a prime power: $\nu_4(2\cdot2)=1\neq0=\nu_4(2)+\nu_4(2)$). It is well-known that if $q_1,\ldots,q_s\in\P$ are the distinct prime factors of $2\leq p\in\N$ then $\Q_p$ and $\Q_{q_1}\times\cdots\times\Q_{q_s}$ are isomorphic with an isomorphism given by
\begin{align}
\tag{A.2}
\varphi_p:\Q_p&\to\Q_{q_1}\times\cdots\times\Q_{q_s}.\\
\nonumber
\sum_{k=j}^\infty a_kp^k&\mapsto\left(\sum_{k=j}^\infty a_k(p/q_1)^kq_1^k,\ldots,\sum_{k=j}^\infty a_k(p/q_s)^kq_s^k\right)
=\left(\sum_{k=j_1}^\infty a^{(1)}_kq_1^k,\ldots,\sum_{k=j_s}^\infty a^{(s)}_kq_s^k\right)
\end{align}
We thus may extend the $q$-adic valuation $\nu_q$ to $\Q_p$ for every $q\in\set{q_1,\ldots,q_s}$ by
\begin{align}
\tag{A.3}
\nu_q:\Q_p&\to\Z\cup\set{\infty}.\\
\nonumber
\sum_{k=j}^\infty a_kp^k&\mapsto\nu_q\left(\sum_{k=j}^\infty a_k(p/q)^kq^k\right)
\end{align}
$\nu_p$ and $\nu_{q_1},\ldots,\nu_{q_s}$ satisfy the following properties ($x,y\in\Q_p$, $k\in\Z$):

\begin{theoremtable}
$\bullet$&$\nu_p(x)=\inf\set{\floor{\nu_q(x)/\nu_q(p)}\mid q\in\set{q_1,\ldots,q_s}}$\tabularnewline
$\bullet$&$x=0\Leftrightarrow\nu_p(x)=\infty\Leftrightarrow\fa q\in\set{q_1,\ldots,q_s}:\nu_q(x)=\infty$\tabularnewline
$\bullet$&$x\in p^k\Z_p\Leftrightarrow\nu_p(x)\geq k\Leftrightarrow\fa q\in\set{q_1,\ldots,q_s}:\nu_q(x)\geq\nu_q(p)k$\tabularnewline
$\bullet$&$\fa q\in\set{q_1,\ldots,q_s}:\nu_q(xy)=\nu_q(x)+\nu_q(y)$\tabularnewline
$\bullet$&$\fa q\in\set{q_1,\ldots,q_s}:\nu_q(x+y)\geq\inf\set{\nu_q(x),\nu_q(y)}$\tabularnewline
$\bullet$&$\fa q\in\set{q_1,\ldots,q_s}:\nu_q(x)\neq\nu_q(y)\lor\nu_q(x)=\infty\lor\nu_q(y)=\infty\Rightarrow\nu_q(x+y)=\inf\set{\nu_q(x),\nu_q(y)}$.\tabularnewline[0.5\baselineskip]
\end{theoremtable}

\newcommand{\ti}[2]{{\textbf{#1~\ref{#2}} (p.~\pageref{#2}):}}

\myparagraphtoc{List of theorems.}
$ $\\[0.25\baselineskip]
\ti{Lemma}{LFFCanForm}
Every $p$-fibred function has a unique canonical form. $\sim_p$ is an equivalence relation on $\SoFibredFunctions{p}$ and the canonical forms constitute a complete set of representatives.\\[0.25\baselineskip]
\ti{Lemma}{LOrdFuncBlock}
A sufficient condition for the weak block property and a necessary and sufficient condition for the block property of closed $p$-fibred functions when interpreted as ordinary functions.\\[0.25\baselineskip]
\ti{Example}{EOrdFuncBlock}
The sufficient condition for the weak block property in Lemma~\ref{LOrdFuncBlock} neither is necessary.\\[0.25\baselineskip]
\ti{Theorem}{TDTFF}
Characterization of all $p$-fibred functions which define a given $p$-digit table $\STf{D}$ in terms of the sets $\STf{D}(n)$, $n\in\STFdomain(\STf{D})$.\\[0.25\baselineskip]
\ti{Lemma}{LFinBlDTCompl}
If $\STf{D}$ is a $p$-digit table of length $k$ which has the block property at $k$, then every given sequence of length $k$ with entries in $\intintuptoex{p}$ can be found exactly once among the initial parts of length $k$ of the $\STf{D}$-digit expansions of any CRS modulo $p^k$.\\[0.25\baselineskip]
\ti{Lemma}{LSBlDTBl}
If a $p$-digit table whose domain contains a CRS modulo $p^\ell$ has the block property at $\ell$ and the weak block property at $k\leq\ell$, then it also has the block property at $k$.\\[0.25\baselineskip]
\ti{Theorem}{TDnStructure}
An analysis of the structure of the sets $\STf{D}(n)$ for a $p$-digit table $\STf{D}$ under the assumption of various (weak) block properties.\\[0.25\baselineskip]
\ti{Corollary}{CComputeDT}
Simplifying the computation of $\FFf{F}$-digit expansions for a closed $p$-fibred function $\FFf{F}$ which has various (weak) block properties.\\[0.25\baselineskip]
\ti{Lemma}{LInfBlDTEq}
If $\STf{D}$ is an infinite $p$-digit table with block property, then the function mapping an element of the domain of $\STf{D}$ to its $\STf{D}$-digit expansion is injective.\\[0.25\baselineskip]
\ti{Corollary}{CInitPer}
For a $p$-adic system $\FFf{F}$ the lengths of the initial and periodic parts of the $\FFf{F}$-sequence and the $\FFf{F}$-digit expansion of some $p$-adic integer coincide.\\[0.25\baselineskip]
\ti{Lemma}{LInfBlDTCompl}
If $\STf{D}$ is an infinite $p$-digit table with domain $\Z_p$ and block property, then the function mapping an element of $\Z_p$ to its $\STf{D}$-digit expansion is bijective (infinite version of Lemma~\ref{LFinBlDTCompl} and specialization of Lemma~\ref{LInfBlDTEq}). In particular, the sets $\STf{D}(n)$ are singletons.\\[0.25\baselineskip]
\ti{Theorem}{TASeqDTbl}
For every infinite $p$-digit table with domain $\Z_p$ and block property $\STf{D}$ there is a unique $p$-fibred function $\FFf{F}$ in canonical form whose $\FFf{F}$-digit table coincides with $\STf{D}$.\\[0.25\baselineskip]
\ti{Example}{EWBlockF}
An example of a $p$-digit table with domain $\Z_p$ and weak block property which cannot be expressed as the $\FFf{F}$-digit table of any $p$-fibred function $\FFf{F}$.\\[0.25\baselineskip]
\ti{Lemma}{LEquDT}
If two $p$-digit tables with weak block property and equal domain coincide on some subset of their shared domain that is dense in $\Z_p$, then the two $p$-digit tables coincide as a whole.\\[0.25\baselineskip]
\ti{Lemma}{LDTExt}
Explicit construction of the unique $p$-digit table $\STf{E}$ with domain $\Z_p$ and weak block property which extends a given $p$-digit table $\STf{D}$ with weak block property, whose domain contains a subset which is dense in $\Z_p$.\\[0.25\baselineskip]
\ti{Corollary}{CFFwbDTExt}
If a $p$-fibred function $\FFf{G}$ with weak block property extends to all of $\Z_p$ another $p$-fibred function $\FFf{F}$ whose domain is dense in $\Z_p$, it does so in accordance with the unique extension of the corresponding $p$-digit tables as given in Lemma~\ref{LDTExt}.\\[0.25\baselineskip]
\ti{Lemma}{LPermProp}
Basic properties of permutations of $\Z_p$ of the form $\pi_{\FFf{F},\FFf{G}}$ if $\FFf{F}$ and $\FFf{G}$ are $p$-adic systems.\\[0.25\baselineskip]
\ti{Theorem}{TPermEq}
For every $p$-adic permutation $\pi$ and every $p$-adic system $\FFf{G}$ there is a unique (up to $\sim_p$) $p$-adic system $\FFf{F}$ so that $\pi=\pi_{\FFf{F},\FFf{G}}$.\\[0.25\baselineskip]
\ti{Theorem}{TPermSubgroup}
The set of all $p$-adic permutations forms a subgroup of the set of all permutations of $\Z_p$ with respect to composition.\\[0.25\baselineskip]
\ti{Lemma}{LGroupProp}
Basic properties of the isomorphism $\Pi_\FFf{G}$ which transports the group structure on the set of $p$-adic permutations to the set of $p$-adic systems and an explicit formula for the resulting group operation on $p$-adic systems.\\[0.25\baselineskip]
\ti{Example}{EGroupProp}
An example illustrating the group operation on $p$-adic systems analyzed in Lemma~\ref{LGroupProp}.\\[0.25\baselineskip]
\ti{Theorem}{TCycles}
If $\pi$ is a $p$-adic permutation, then every cycle $\Sf{S}$ of $\pi_k$ splits into up to $p$ cycles of $\pi_{k+1}$ which are congruent to $\Sf{S}$ modulo $p^k$ (entry-wise, cyclically) and every cycle of $\pi_{k+1}$ is a ``child'' of some cycle of $\pi_k$ in this way.\\[0.25\baselineskip]
\ti{Corollary}{CCycleLengths}
If $\pi$ is a $p$-adic permutation, then the prime factors of the lengths of all cycles of $\pi_k$ are contained in $\intintuptoin{p}$.\\[0.25\baselineskip]
\ti{Corollary}{CCycles}
Basic properties of the edge labeled graph $(\mathcal{G}(\pi),c(\pi))$ defined by a $p$-adic permutation $\pi$.\\[0.25\baselineskip]
\ti{Theorem}{TFuncSuit}
A sufficient condition for the weak block property and a necessary and sufficient condition for the block property of a closed $p$-fibred function $\FFf{F}$ in terms of weak $(p,r)$-suitability and $(p,r)$-suitability of the functions $\FFf{F}[r]$, $r\in\intintuptoex{p}$.\\[0.25\baselineskip]
\ti{Example}{ENoCharacWBlk}
The weak block property of a closed $p$-fibred function $\FFf{F}$ neither does (like the block property) permit a necessary and sufficient characterization that only considers the functions $\FFf{F}[0],\ldots,\FFf{F}[p-1]$ independently from one another.\\[0.25\baselineskip]
\ti{Example}{ENoCharacWBlkEx}
An example of a $p$-digit table with weak block property which is the $p$-digit table of a $p$-fibred function but cannot be realized as the $p$-digit table of a $p$-fibred function whose entries are weakly $(p,r)$-suitable functions.\\[0.25\baselineskip]
\ti{Corollary}{CComputeDTStrong}
A stronger version of statement (3) of Corollary~\ref{CComputeDT}, which loosens the condition that $\FFf{F}$ must have the block property at $k$ to the weak block property S at $k$.\\[0.25\baselineskip]
\ti{Lemma}{LSuitFProp}
Basic properties of $(p,r)$-suitable functions.\\[0.25\baselineskip]
\ti{Corollary}{CSuitFProp}
A $p$-adic system $\FFf{F}$ is surjective and $p$-to-one (as a function on $\Z_p$) and its restriction $\FFf{F}\vert_{r+p\Z_p}$, $r\in\intintuptoex{p}$, is surjective and one-to-one.\\[0.25\baselineskip]
\ti{Theorem}{TSuitFProd}
If $g$ is a weakly $(p,r)$-suitable function at $\intintuptoin{k}$ satisfying $\gcd(p,g(n)\modulo p)=1$ for all $n$ in its domain, then the product $fg$ is (weakly) $(p,r)$-suitable at $\intintuptoin{k}$ if and only if $f$ is (weakly) $(p,r)$-suitable at $\intintuptoin{k}$.\\[0.25\baselineskip]
\ti{Theorem}{TPropPolyF}
Characterization of (weakly) $(p,r)$-suitable polynomial functions in $\Z_p[x]$.\\[0.25\baselineskip]
\ti{Lemma}{LRatPolyIntegral}
Characterization of $(p,r)$-integral polynomial functions in $\Q_p[x]$.\\[0.25\baselineskip]
\ti{Theorem}{TPropRatPolyF}
Characterization of (weakly) $(p,r)$-suitable polynomial functions in $\Q_p[x]$.\\[0.25\baselineskip]
\ti{Corollary}{CPropPolyF}
Every $\Z_p$-polynomial $p$-fibred function has the weak property S and it has the block property if and only if it has the block property at $k$ for any $k\geq2$.\\[0.25\baselineskip]
\ti{Corollary}{CPropRatPolyF}
For every polynomial $p$-fibred function $\FFf{F}$ a $K\in\Nz$ is constructed such that $\FFf{F}$ has the weak property S if and only if it has the weak property S at $\intintuptoin{K+1}$ and $\FFf{F}$ has the block property if and only if it has the block property at $\intintuptoin{K+3}$.\\[0.25\baselineskip]
\ti{Corollary}{CPolySDTExt}
The $p$-digit table of the extension of a $p$-fibred function $\FFf{F}$ with domain $\Z$ defined by polynomial functions in $\Z[x]$ obtained by extending the domain from $\Z$ to $\Z_p$ coincides with the unique extension of the $p$-digit table of $\FFf{F}$ as given by Lemma~\ref{LDTExt}.\\[0.25\baselineskip]
\ti{Lemma}{LReduceDegree}
Construction of a polynomial $g\in\intintuptoex{p^k}[x]$ with degree less than $k$ which coincides modulo $p^k$ on $r+p\Z_p$ with a given polynomial $f\in\Z_p[x]$.\\[0.25\baselineskip]
\ti{Theorem}{TReduceDegree}
Construction of a $\intintuptoex{p^k}$-polynomial $p$-fibred function $\FFf{G}$ with degree in $\intintuptoex{k}$ satisfying $\FFFDT(\FFf{F})\llbracket\intintuptoex{k}\rrbracket=\FFFDT(\FFf{G})\llbracket\intintuptoex{k}\rrbracket$ for a given $\Z_p$-polynomial $p$-fibred function $\FFf{F}$.\\[0.25\baselineskip]
\ti{Theorem}{TPropRatF}
Characterization of (weakly) $(p,r)$-suitable rational functions $f=g/h$ with $g,h\in\Z_p[x]$ and $\gcd(p,h(r)\modulo p)=1$ for all $n\in r+p\Z_p$.\\[0.25\baselineskip]
\ti{Corollary}{CPropRatF}
Every $p$-fibred function defined by rational functions of the kind treated in Theorem~\ref{TPropRatF} has the weak property S and it has the block property if and only if it has the block property at $k$ for any $k\geq2$.\\[0.25\baselineskip]
\ti{Conjecture}{ConPropRatPolyF}
Conjecture on the characterization of (weakly) $(p,r)$-suitable rational functions $f=g/h$ with $g,h\in\Z_p[x]$ but without the condition $\gcd(p,h(r)\modulo p)=1$ for all $n\in r+p\Z_p$.\\[0.25\baselineskip]
\ti{Lemma}{LHensel}
Hensel's Lemma: a polynomial $f\in\Z_p[x]$ has a unique root in $r+p\Z_p$ if $f(r)\modulo p=0$ and $f'(r)\modulo p\neq0$.\\[0.25\baselineskip]
\ti{Theorem}{TSuitFUnRoot}
Generalization of Hensel's Lemma: a general function $f\in\Z_p[x]$ has a unique root in $r+p\Z_p$ if $f(r+p\Z_p)\subseteq p\Z_p$ and $f$ is $(p,r)$-suitable.\\[0.25\baselineskip]
\ti{Lemma}{LShiftedSuitF}
The sum of a function which is weakly $(p,r)$-suitable at $k$ and a linear polynomial is weakly $(p,r)$-suitable at $k$, and the sum of a function which is $(p,r)$-suitable at $\intintuptoin{k}$ and a linear polynomial whose linear coefficient is in $p\Z_p$ is $(p,r)$-suitable at $\intintuptoin{k}$.\\[0.25\baselineskip]
\ti{Example}{EShiftedSuitF}
Lemma~\ref{LShiftedSuitF} cannot be generalized by loosening the condition ``$(p,r)$-suitable at $\intintuptoin{k}$'' by requiring ``$(p,r)$-suitable at $k$'' instead. \\[0.25\baselineskip]
\ti{Example}{EHensel}
Examples of applications of Hensel's Lemma to prove that certain real or complex numbers defined by polynomial equations have counterparts within $\Z_p$.\\[0.25\baselineskip]
\ti{Theorem}{TSuitFUnRootEq}
Stronger version of Theorem~\ref{TSuitFUnRoot}: A general function $f\in\Z_p[x]$ which maps $r+p\Z_p$ to $p\Z_p$ is $(p,r)$-suitable if and only if $f$ has a unique root $a$ in $r+p\Z_p$ and $\Gcd{p,g(r)\modulo p}=1$, where $g\in\Z_p[x]$ so that $f(x)=(x-a)g(x)$.\\[0.25\baselineskip]
\ti{Example}{ESuitFUnRootEq}
Even if $p$ is prime the condition $\Gcd{p,g(r)\modulo p}=1$ in Theorem~\ref{TSuitFUnRootEq} cannot be dropped.\\[0.25\baselineskip]
\ti{Lemma}{LFRF}
Relation between the application of a $p$-fibred rational function $\FRFf{R}$ and the corresponding $p$-fibred function $\FFf{F}\ce\FRFFint(\FRFf{R}\vert_{\Z_p\cap\FRFFdomain(\FRFf{R})})$.\\[0.25\baselineskip]
\ti{Lemma}{LEqSEqD}
Relation between the application of a $p$-fibred rational function $\FRFf{R}$ and the corresponding $p$-fibred function $\FFf{F}\ce\FRFFint(\FRFf{R}\vert_{\Z_p})$ under the condition that $\FRFf{R}$ is avoiding.\\[0.25\baselineskip]
\ti{Example}{EEqSEqD}
The assumptions of Lemma~\ref{LEqSEqD} cannot be loosened, even if the entries of $\FRFf{R}$ are polynomials.\\[0.25\baselineskip]
\ti{Theorem}{TAvoidFRFUnFP}
Generalization of Hensel's Lemma: under natural technical conditions, the function $\FRFF{\FRFf{R}}{\Sf{D}}$ has a unique fixed point in $\Z_p$ for every avoiding $p$-fibred rational function $\FRFf{R}$ and every sequence of digits $\Sf{D}$.\\[0.25\baselineskip]
\ti{Theorem}{TCharacAvPolyF}
Characterization of $(p,r)$-avoiding polynomial functions in $\Z_p[x]$.\\[0.25\baselineskip]
\ti{Lemma}{LPeriodicOn}
Characterization of when two $p$-adic systems $\FFf{F}$ and $\FFf{G}$ are periodic, ultimately periodic, or aperiodic on the same sets using $\pi_{\FFf{F},\FFf{G}}$.\\[0.25\baselineskip]
\ti{Lemma}{LContrExp}
Consequences of a $p$-fibred function being contractive or expansive for periodic and ultimately periodic digit expansions.\\[0.25\baselineskip]
\ti{Theorem}{TPolyExp}
$(\Q\cap\Z_p)$-polynomial $p$-adic systems where each polynomial is either of degree $2$ or higher or has a linear coefficient greater than $p$ in absolute value, are expansive, and $(\Q\cap\Z_p)$-polynomial $p$-adic systems that are contractive are linear-polynomial.\\[0.25\baselineskip]
\ti{Lemma}{LLinPolySuit}
Analysis of (weak) $(p,r)$-suitability and $(p,r)$-avoidance of linear polynomials.\\[0.25\baselineskip]
\ti{Theorem}{TLinDirForm}
Explicit formula for $\FRFF{\FRFf{R}}{\Sf{D}}(n)$ if $\FRFf{R}$ is a linear-polynomial $p$-fibred rational function.\\[0.25\baselineskip]
\ti{Corollary}{CPerFormula}
Explicit formula for the unique $p$-adic integer having a given ultimately periodic $\FFf{F}$-digit expansion for a given linear-polynomial $p$-adic system $\FFf{F}$.\\[0.25\baselineskip]
\ti{Corollary}{CBegFormula}
Explicit formula for the unique element of $\intintuptoex{p^{\abs{\Sf{D}}}}$ having a given initial $\FFf{F}$-digit expansion for a given linear-polynomial $p$-adic system $\FFf{F}$.\\[0.25\baselineskip]
\ti{Corollary}{CInvSyst}
Explicit construction of all linear-polynomial $p$-adic systems $\FFf{F}$ for which a given $p$-adic integer $n$ has a given ultimately periodic digit expansion $\Sf{D}$.\\[0.25\baselineskip]
\ti{Corollary}{CInvSystNum}
Explicit construction of all pairs $(\FFf{F},n)$ of linear-polynomial $p$-adic systems and $p$-adic integers for which a given ultimately periodic digit expansion $\Sf{D}$ coincides with the $\FFf{F}$-digit expansion of $n$.\\[0.25\baselineskip]
\ti{Conjecture}{CoPeriodsA}
The original Collatz conjecture.\\[0.25\baselineskip]
\ti{Conjecture}{CoPeriodsB}
Generalization of the Collatz conjecture: $\FFf{F}_C$ is ultimately periodic on $\Q\cap\Z_2$ and ultimately periodic orbits of natural numbers end up at $1$.\\[0.25\baselineskip]
\ti{Conjecture}{CoPeriodsC}
Variant of the the Collatz conjecture for $\Z$-linear-polynomial $p$-adic systems.\\[0.25\baselineskip]
\ti{Conjecture}{CoPeriodsD}
Variant of the the Collatz conjecture for $(\Q\cap\Z_p)$-linear-polynomial $p$-adic systems.\\[0.25\baselineskip]
\ti{Conjecture}{CoPeriodsE}
Variant of the the Collatz conjecture for $(\Q\cap\Z_p)$-polynomial $p$-adic systems.\\[0.25\baselineskip]
\ti{Conjecture}{CoPeriodsF}
Variant of the the Collatz conjecture for polynomial $p$-adic systems.\\[0.25\baselineskip]
\ti{Theorem}{TGenPolyUPer}
Conjecture~\ref{CoPeriodsE}~(1) and Conjecture~\ref{CoPeriodsF}~(1) are equivalent.\\[0.25\baselineskip]
\ti{Theorem}{TConstIrr}
Explicit formula for $\pi_{\FFf{F},\FFf{G}}(n)$ if $\FFf{F}$ and $\FFf{G}$ are linear-polynomial $2$-adic systems with matching linear coefficients, which implies that the constant coefficients of $(\Q\cap\Z_2)$-linear-polynomial $2$-adic systems have no influence on the question of whether all rational numbers have ultimately periodic digit expansions.\\[0.25\baselineskip]
\ti{Conjecture}{CoConstIrr}
Conjecture that the constant coefficients of $(\Q\cap\Z_p)$-linear-polynomial $p$-adic systems have no influence on the question of whether all rational numbers have ultimately periodic digit expansions.\\[0.25\baselineskip]
\ti{Theorem}{TOrderIrr}
Explicit formula for $\pi_{\FFf{F},\sigma,\FFf{G}}(n)$ if $\FFf{F}$ and $\FFf{G}$ are linear-polynomial $2$-adic systems with swapped linear coefficients, which implies that the order of the linear coefficients of $(\Q\cap\Z_2)$-linear-polynomial $2$-adic systems has no influence on the question of whether all rational numbers have ultimately periodic digit expansions.\\[0.25\baselineskip]
\ti{Conjecture}{CoOrderIrr}
Conjecture that the order of the linear coefficients of $(\Q\cap\Z_p)$-linear-polynomial $p$-adic systems has no influence on the question of whether all rational numbers have ultimately periodic digit expansions.\\[0.25\baselineskip]
\ti{Example}{EConstOrderIrr}
Application of the formulas for $\pi_{\FFf{F},\FFf{G}}(n)$ and $\pi_{\FFf{F},\sigma,\FFf{G}}(n)$ given in Theorem~\ref{TConstIrr} and Theorem~\ref{TOrderIrr}.\\[0.25\baselineskip]
\ti{Theorem}{TLinPolyCEM}
Characterization of contractive, expansive, and mixed type $(\Q\cap\Z_p)$-linear-polynomial $p$-adic systems.\\[0.25\baselineskip]
\ti{Corollary}{CLinPolyCEM}
Answer to the question of ultimate periodicity on $\Q\cap\Z_p$ for contractive $\Z$-linear-polynomial- and for expansive $(\Q\cap\Z_p)$-linear-polynomial $p$-adic systems.\\[0.25\baselineskip]
\ti{Corollary}{CConExpSummary}
Summary of results on the question of ultimate periodicity on $\Q\cap\Z_p$ for $(\Q\cap\Z_p)$-polynomial $p$-adic systems.\\[0.25\baselineskip]
\ti{Lemma}{LCharacPermPoly}
Characterization of $p$-permutation polynomials.\\[0.25\baselineskip]
\ti{Theorem}{TCharacPermPoly}
A polynomial $f\in\Z_p[x]$ is a $p$-permutation polynomial if and only if it is a $p$-adic permutation.\\[0.25\baselineskip]
\ti{Example}{ENotPermPoly}
There are $p$-adic permutations defined by $\Z$-linear-polynomial $p$-adic systems which are neither $p$-permutation polynomials.\\[0.25\baselineskip]
\ti{Theorem}{TPermFromTree}
Every $p$-cycle tree can be realized as the tree of cycles of some $p$-adic permutation.\\[0.25\baselineskip]
\ti{Theorem}{TSubtrees}
Characterization of the sets of all isomorphism classes of trees with up to $4$ layers which may occur as subtrees of trees of cycles of $2$-adic permutations defined by $\Z_2$-polynomial $2$-adic systems or by $2$-permutation polynomials.\\[0.25\baselineskip]
\ti{Example}{ENotPolySyst}
There are $p$-permutation polynomials which are neither $p$-adic permutations defined by $\Z_p$-polynomial $p$-adic systems.\\[0.25\baselineskip]
\ti{Corollary}{CSubtrees}
No ``Y'' property: the trees of cycles of $2$-adic permutations defined by $\Z_2$-polynomial $2$-adic systems or by $2$-permutation polynomials do not contain ``Y''-shaped subtrees.

\newpage

\myparagraphtoc{List of symbols in order of first appearance.}
$ $

\bgroup
\setlength\LTleft{0pt}
\setlength\LTright{0pt}
\begin{longtable}{@{\extracolsep{\fill}}llr@{}}
$A(\mathcal{P})$
&elements of $A$ satisfying predicates in $\mathcal{P}$
&\pageref{DSetPred}\tabularnewline
$\intint{a}{b}$
&integer interval $\set{n\in\Z\mid a\leq n\leq b}$
&\pageref{DIntInt}\tabularnewline
$\intintuptoex{a}$
&integer interval $\intint{0}{a-1}$
&\pageref{DIntInt}\tabularnewline
$\intintuptoin{a}$
&integer interval $\intint{0}{a}$
&\pageref{DIntInt}\tabularnewline
$\CoSequences$
&class of sequences
&\pageref{DCoSequences}\tabularnewline
$\abs{\Sf{S}}$
&length/size of $\Sf{S}$
&\pageref{DSlength}\tabularnewline
$\Sf{S}[A]$
&subsequence of $\Sf{S}$, indices in $A$
&\pageref{DSsubseq}\tabularnewline
$\Sf{S}[i,j]$
&subsequence of $\Sf{S}$, indices in $\intint{i}{j}$
&\pageref{DSsubseqshort}\tabularnewline
$\SPlength{A}(\Sf{S})$
&$\Sf{S}$ has length in $A$
&\pageref{DSPlength}\tabularnewline
$\SPfinite(\Sf{S})$
&$\Sf{S}$ is finite
&\pageref{DSPfinite}\tabularnewline
$\SPempty(\Sf{S})$
&$\Sf{S}$ is empty
&\pageref{DSPempty}\tabularnewline
$\SPboundedby{A}(\Sf{S})$
&$\Sf{S}$ is $A$-bounded
&\pageref{DSPboundedby}\tabularnewline
$\SPprefix{\Sf{T}}(\Sf{S})$
&$\Sf{S}$ has prefix $\Sf{T}$
&\pageref{DSPprefix}\tabularnewline
$\SPsuffix{\Sf{T}}(\Sf{S})$
&$\Sf{S}$ has suffix $\Sf{T}$
&\pageref{DSPsuffix}\tabularnewline
$\Sf{S}\cdot\Sf{T}$
&product/concatenation of $\Sf{S}$ and $\Sf{T}$
&\pageref{DSmult}\tabularnewline
$\Sf{S}^n$
&$n$-th power of $\Sf{S}$
&\pageref{DSpow}\tabularnewline
$\Sf{S}^\infty$
&infinite periodic sequence with period $\Sf{S}$
&\pageref{DSpowinf}\tabularnewline
$\SFinitial{\Sf{S}}$
&initial part of $\Sf{S}$
&\pageref{DSFinitial}\tabularnewline
$\SFperiodic{\Sf{S}}$
&periodic part of $\Sf{S}$
&\pageref{DSFperiodic}\tabularnewline
$\SPperiodic(\Sf{S})$
&$\Sf{S}$ is (purely) periodic
&\pageref{DSPperiodic}\tabularnewline
$\SPultimatelyperiodic(\Sf{S})$
&$\Sf{S}$ is ultimately periodic
&\pageref{DSPultimatelyperiodic}\tabularnewline
$\SPaperiodic(\Sf{S})$
&$\Sf{S}$ is aperiodic
&\pageref{DSPaperiodic}\tabularnewline
$f(\Sf{S})$
&entry-wise application of $f$ to $\Sf{S}$
&\pageref{DSfunc}\tabularnewline
$\CoSequenceTables$
&class of sequence tables
&\pageref{DCoSequenceTables}\tabularnewline
$\STFdomain(\STf{S})$
&domain of $\STf{S}$
&\pageref{DSTFdomain}\tabularnewline
$\abs{\STf{S}}$
&length/size of $\STf{S}$
&\pageref{DSTlength}\tabularnewline
$\STf{S}[n]$
&$n$-th row of $\STf{S}$/$\STf{S}$-sequence of $n$
&\pageref{DSTentry}\tabularnewline
$\STf{S}\vert_A$
&restriction of $\STf{S}$ to $A$
&\pageref{DSTrestriction}\tabularnewline
$\STf{S}\llbracket A\rrbracket$
&subtable of $\STf{S}$, indices in $A$
&\pageref{DSTsubtable}\tabularnewline
$\STf{S}\llbracket i,j\rrbracket$
&subtable of $\STf{S}$, indices in $\intint{i}{j}$
&\pageref{DSTsubtableshort}\tabularnewline
$\STf{S}\cdot\STf{T}$
&product/concatenation of $\STf{S}$ and $\STf{T}$
&\pageref{DSTmultpowfunc}\tabularnewline
$\STf{S}^n$
&$n$-th power of $\STf{S}$
&\pageref{DSTmultpowfunc}\tabularnewline
$\STf{S}^\infty$
&infinite periodic table with period $\STf{S}$
&\pageref{DSTmultpowfunc}\tabularnewline
$f(\STf{S})$
&entry-wise application of $f$ to $\STf{S}$
&\pageref{DSTmultpowfunc}\tabularnewline
$\STPdomain{A}(\STf{S})$
&$\STf{S}$ has domain $A$
&\pageref{DSTPdomain}\tabularnewline
$\STPlength{A}(\STf{S})$
&$\STf{S}$ has length in $A$
&\pageref{DSTPlength}\tabularnewline
$\STPfinite(\STf{S})$
&$\STf{S}$ is finite
&\pageref{DSTPfinite}\tabularnewline
$\STPempty(\STf{S})$
&$\STf{S}$ is empty
&\pageref{DSTPempty}\tabularnewline
$\STPboundedby{A}(\STf{S})$
&$\STf{S}$ is $A$-bounded
&\pageref{DSTPboundedby}\tabularnewline
$\SoDigitTables{p}$
&set of $p$-digit tables
&\pageref{DSoDigitTables}\tabularnewline
$\STf{D}[n]$
&$\STf{D}$-digit expansion of $n$
&\pageref{DDTentry}\tabularnewline
$\STf{D}[n][k]$
&$k$-th digit of $n$ with respect to $\STf{D}$
&\pageref{DDTentryentry}\tabularnewline
$\DTPweakblockat{K}(\STf{D})$
&$\STf{D}$ has the weak block property at $K$
&\pageref{DDTPweakblockat}\tabularnewline
$\DTPweakblock(\STf{D})$
&$\STf{D}$ has the weak block property
&\pageref{DDTPweakblock}\tabularnewline
$\DTPblockat{K}(\STf{D})$
&$\STf{D}$ has the block property at $K$
&\pageref{DDTPblockat}\tabularnewline
$\DTPblock(\STf{D})$
&$\STf{D}$ has the block property
&\pageref{DDTPblock}\tabularnewline
$\SoFibredFunctions{p}$
&set of $p$-fibred functions
&\pageref{DSoFibredFunctions}\tabularnewline
$\FFFdomain(\FFf{F})$
&domain of $\FFf{F}$
&\pageref{DFFFdomain}\tabularnewline
$\FFf{F}(n)$
&application of $\FFf{F}$ to $n$
&\pageref{EFibredFunction}\tabularnewline
$\modulo$
&modulo function
&\pageref{Dmodulo}\tabularnewline
$\FFf{F}\vert_A$
&restriction of $\FFf{F}$ to $A$
&\pageref{DFFrestriction}\tabularnewline
$\FFf{F}\sim_p\FFf{G}$
&equivalence of $\FFf{F}$ and $\FFf{G}$
&\pageref{EEquiv}\tabularnewline
$\FFPcanonicalform(\FFf{F})$
&$\FFf{F}$ is in canonical form
&\pageref{EPCanF}\tabularnewline
$\FFPweakcanonicalform(\FFf{F})$
&$\FFf{F}$ is in weak canonical form
&\pageref{EPWCanF}\tabularnewline
$\FFPdomain{A}(\FFf{F})$
&$\FFf{F}$ has domain $A$
&\pageref{DFFPdomain}\tabularnewline
$\FFPboundedby{A}(\FFf{F})$
&$\FFf{F}$ is $A$-bounded
&\pageref{DFFPboundedby}\tabularnewline
$\FFPclosed(\FFf{F})$
&$\FFf{F}$ is closed
&\pageref{DFFPclosed}\tabularnewline
$\FFFST(\FFf{F})$
&$\FFf{F}$-sequence table
&\pageref{DFFFST}\tabularnewline
$\FFFDT(\FFf{F})$
&$\FFf{F}$-digit table
&\pageref{DFFFDT}\tabularnewline
$\FFFST(\FFf{F})[n]$
&$\FFf{F}$-sequence of $n$
&\pageref{DFseq}\tabularnewline
$\FFFDT(\FFf{F})[n]$
&$\FFf{F}$-digit expansion of $n$
&\pageref{DFdigexp}\tabularnewline
$\FFFDT(\FFf{F})[n][k]$
&$k$-th digit of $n$ with respect to $\FFf{F}$
&\pageref{DFdigit}\tabularnewline
$\FFPweakblockat{K}(\FFf{F})$
&$\FFf{F}$ has the weak block property at $K$
&\pageref{DFPblock}\tabularnewline
$\FFPweakblock(\FFf{F})$
&$\FFf{F}$ has the weak block property
&\pageref{DFPblock}\tabularnewline
$\FFPblockat{K}(\FFf{F})$
&$\FFf{F}$ has the block property at $K$
&\pageref{DFPblock}\tabularnewline
$\FFPblock(\FFf{F})$
&$\FFf{F}$ has the block property
&\pageref{DFPblock}\tabularnewline
$\SoSystems{p}$
&set of $p$-adic systems
&\pageref{DSoSystems}\tabularnewline
$\FFf{F}_C$
&Collatz transformation
&\pageref{ECollatz}\tabularnewline
$\FFf{F}_2$
&binary transformation
&\pageref{EBinary}\tabularnewline
$(P(x)\;?\;f(x):g(x))$
&conditional function
&\pageref{Dcondfunc}\tabularnewline
$\SoFunctions{p}$
&set of functions with $p$-block property
&\pageref{DSoFunctions}\tabularnewline
$\SoTables{p}$
&set of $p$-digit tables with block property
&\pageref{DSoTables}\tabularnewline
$\STf{D}(n)$
&application of $\STf{D}$ to $n$
&\pageref{DDTfunc}\tabularnewline
$\powerset(A)$
&powerset of $n$
&\pageref{Dpowerset}\tabularnewline
$R(n)$
&application of $R$ to $n$
&\pageref{DCRSfunc}\tabularnewline
$\FFf{F}_\STf{D}$
&unique $p$-adic system defined by given $p$-digit table $\STf{D}$
&\pageref{DFD}\tabularnewline
$\SPweakblock{p}{k}(\Sf{S})$
&$\Sf{S}$ has the weak $(p,k)$-block property
&\pageref{DSPweakblock}\tabularnewline
$\SPblock{p}{k}(\Sf{S})$
&$\Sf{S}$ has the $(p,k)$-block property
&\pageref{DSPblock}\tabularnewline
$\psi_\FFf{F}(n)$
&$\FFf{F}$-digit expansion of $n$
&\pageref{Dpsi}\tabularnewline
$\pi_{\FFf{F},\FFf{G}}(n)$
&number whose $\FFf{G}$-digit expansion equals the $\FFf{F}$-digit expansion of $n$
&\pageref{Dpi}\tabularnewline
$\FFf{F}_p$
&$p$-ary transformation
&\pageref{DFp}\tabularnewline
$\SoPermutations{p}$
&set of $p$-adic permutations
&\pageref{DSoPermutations}\tabularnewline
$\pi_k$
&$\pi$ modulo $p^k$
&\pageref{Dpik}\tabularnewline
$\Pi_{\FFf{G}}$
&group isomorphism between $\SoSystems{p}\slash_{\sim_p}$ and $\SoPermutations{p}$ with respect to $\FFf{G}$
&\pageref{Dgroupisom}\tabularnewline
$\circ_\FFf{G}$
&group operation on $\SoSystems{p}$ transported by $\Pi_{\FFf{G}}$
&\pageref{Dgroupop}\tabularnewline
$\sigma(\Sf{S},s)$
&cyclic shift of $\Sf{S}$ by $s$ steps
&\pageref{ECycShift}\tabularnewline
$\Sigma(\pi)$
&set of cycles of $\pi$
&\pageref{DSocycles}\tabularnewline
$\Sf{S}\sim_\sigma\Sf{T}$
&cyclical equivalence of $\Sf{S}$ and $\Sf{T}$
&\pageref{Dcyceq}\tabularnewline
$\abs{[\Sf{S}]_{\sim_\sigma}}$
&length/size of $[\Sf{S}]_{\sim_\sigma}$
&\pageref{Dcyclelength}\tabularnewline
$\mathcal{V}(\pi)$
&set of vertices of tree of cycles of $\pi$
&\pageref{Dtreeofcyclesvertices}\tabularnewline
$\mathcal{E}(\pi)$
&set of edges of tree of cycles of $\pi$
&\pageref{Dtreeofcyclesedges}\tabularnewline
$\mathcal{G}(\pi)$
&tree of cycles of $\pi$
&\pageref{Dtreeofcycles}\tabularnewline
$c(\pi)$
&edge labeling of tree of cycles of $\pi$
&\pageref{Dtreeofcyclescolor}\tabularnewline
$\FPweaklysuitableat{p}{r}{K}(f)$
&$f$ is weakly $(p,r)$-suitable at $K$
&\pageref{DFPweaklysuitableat}\tabularnewline
$\FPweaklysuitable{p}{r}(f)$
&$f$ is weakly $(p,r)$-suitable
&\pageref{DFPweaklysuitable}\tabularnewline
$\FPsuitableat{p}{r}{K}(f)$
&$f$ is $(p,r)$-suitable at $K$
&\pageref{DFPsuitableat}\tabularnewline
$\FPsuitable{p}{r}(f)$
&$f$ is $(p,r)$-suitable
&\pageref{DFPsuitable}\tabularnewline
$\DTPweakblockFat{K}(\STf{D})$
&$\STf{D}$ has the weak block property F at $K$
&\pageref{DDTPweakblockFat}\tabularnewline
$\DTPweakblockF(\STf{D})$
&$\STf{D}$ has the weak block property F
&\pageref{DDTPweakblockF}\tabularnewline
$\DTPweakblockSat{K}(\STf{D})$
&$\STf{D}$ has the weak block property S at $K$
&\pageref{DDTPweakblockSat}\tabularnewline
$\DTPweakblockS(\STf{D})$
&$\STf{D}$ has the weak block property S
&\pageref{DDTPweakblockS}\tabularnewline
$\FFPweakblockFat{K}(\FFf{F})$
&$\FFf{F}$ has the weak block property F at $K$
&\pageref{DFFPweakblockFat}\tabularnewline
$\FFPweakblockF(\FFf{F})$
&$\FFf{F}$ has the weak block property F
&\pageref{DFFPweakblockF}\tabularnewline
$\FFPweakblockSat{K}(\FFf{F})$
&$\FFf{F}$ has the weak block property S at $K$
&\pageref{DFFPweakblockSat}\tabularnewline
$\FFPweakblockS(\FFf{F})$
&$\FFf{F}$ has the weak block property S
&\pageref{DFFPweakblockS}\tabularnewline
$\FPintegral{p}{r}(f)$
&$f$ is $(p,r)$-integral
&\pageref{DFPintegral}\tabularnewline
$\FFPpolynomialcoefficientsdegree{A}{D}(\FFf{F})$
&$\FFf{F}$ is $A$-polynomial with degree in $D$
&\pageref{DFFPpolynomialcoefficientsdegree}\tabularnewline
$\FFPpolynomialcoefficients{A}(\FFf{F})$
&$\FFf{F}$ is $A$-polynomial
&\pageref{DFFPpolynomialcoefficients}\tabularnewline
$\FFPpolynomial(\FFf{F})$
&$\FFf{F}$ is polynomial
&\pageref{DFFPpolynomial}\tabularnewline
$\FFPlinearpolynomialcoefficients{A}(\FFf{F})$
&$\FFf{F}$ is $A$-linear-polynomial
&\pageref{DFFPlinearpolynomialcoefficients}\tabularnewline
$\FFPlinearpolynomial(\FFf{F})$
&$\FFf{F}$ is linear-polynomial
&\pageref{DFFPlinearpolynomial}\tabularnewline
$\SoFibredRationalFunctions{p}$
&set of $p$-fibred rational functions
&\pageref{DSoFibredRationalFunctions}\tabularnewline
$\FRFFdomain(\FRFf{R})$
&domain of $\FRFf{R}$
&\pageref{DFRFFdomain}\tabularnewline
$\FRFf{R}\vert_A$
&restriction of $\FRFf{R}$ to $A$
&\pageref{DFRFrestriction}\tabularnewline
$\FRFPdomain{A}(\FRFf{R})$
&$\FRFf{R}$ has domain $A$
&\pageref{DFRFPdomain}\tabularnewline
$\FRFPboundedby{A}(\FRFf{R})$
&$\FRFf{R}$ is $A$-bounded
&\pageref{DFRFPboundedby}\tabularnewline
$\FRFPclosed(\FRFf{R})$
&$\FRFf{R}$ is closed
&\pageref{DFRFPclosed}\tabularnewline
$\FRFPintegral(\FRFf{R})$
&$\FRFf{R}$ is integral
&\pageref{DFRFPintegral}\tabularnewline
$\FRFF{\FRFf{R}}{\Sf{D}}(x)$
&application of $\FRFF{\FRFf{R}}{\Sf{D}}$ to $x$
&\pageref{EFibredRFunction}\tabularnewline
$\FRFFST{\Sf{D}}(\FRFf{R})$
&$\FRFf{R}$-sequence table with respect to $\Sf{D}$
&\pageref{DFRFFST}\tabularnewline
$\FRFFST{\Sf{D}}(\FRFf{R})(n)$
&$\FRFf{R}$-sequence of $n$ with respect to $\Sf{D}$
&\pageref{DFRFFSTentry}\tabularnewline
$\FRFPpolynomialcoefficientsdegree{A}{D}(\FRFf{R})$
&$\FRFf{R}$ is $A$-polynomial with degree in $D$
&\pageref{DFRFPpolynomialcoefficientsdegree}\tabularnewline
$\FRFPpolynomialcoefficients{A}(\FRFf{R})$
&$\FRFf{R}$ is $A$-polynomial
&\pageref{DFRFPpolynomialcoefficients}\tabularnewline
$\FRFPpolynomial(\FRFf{R})$
&$\FRFf{R}$ is polynomial
&\pageref{DFRFPpolynomial}\tabularnewline
$\FRFPlinearpolynomialcoefficients{A}(\FRFf{R})$
&$\FRFf{R}$ is $A$-linear-polynomial
&\pageref{DFRFPlinearpolynomialcoefficients}\tabularnewline
$\FRFPlinearpolynomial(\FRFf{R})$
&$\FRFf{R}$ is linear-polynomial
&\pageref{DFRFPlinearpolynomial}\tabularnewline
$\FRFFint(\FRFf{R})$
&$p$-fibred function corresponding to $\FRFf{R}$
&\pageref{DFRFFint}\tabularnewline
$\FPavoiding{p}{r}(f)$
&$f$ is $(p,r)$-avoiding
&\pageref{DFPavoiding}\tabularnewline
$\FRFPavoiding(\FRFf{R})$
&$\FRFf{R}$ is avoiding
&\pageref{DFRFPavoiding}\tabularnewline
$\FFPavoiding(\FFf{F})$
&$\FFf{F}$ is avoiding
&\pageref{DFFPavoiding}\tabularnewline
$\FFFSoPP(\FFf{F})$
&set of periodic points of $\FFf{F}$
&\pageref{DFFFSoPP}\tabularnewline
$\FFFSoUPP(\FFf{F})$
&set of ultimately periodic points of $\FFf{F}$
&\pageref{DFFFSoUPP}\tabularnewline
$\FFFSoAPP(\FFf{F})$
&set of aperiodic points of $\FFf{F}$
&\pageref{DFFFSoAPP}\tabularnewline
$\FFPsatisfies{E}(\FFf{F})$
&$\FFf{F}$ satisfies $E$
&\pageref{DFFPsatisfies}\tabularnewline
$\FFPperiodicon{A}(\FFf{F})$
&$\FFf{F}$ is periodic on $A$
&\pageref{DFFPperiodicon}\tabularnewline
$\FFPultimatelyperiodicon{A}(\FFf{F})$
&$\FFf{F}$ is ultimately periodic on $A$
&\pageref{DFFPultimatelyperiodicon}\tabularnewline
$\FFPaperiodicon{A}(\FFf{F})$
&$\FFf{F}$ is aperiodic on $A$
&\pageref{DFFPaperiodicon}\tabularnewline
$\FFPcontractive(\FFf{F})$
&$\FFf{F}$ is contractive
&\pageref{DFFPcontractive}\tabularnewline
$\FFPexpansive(\FFf{F})$
&$\FFf{F}$ is expansive
&\pageref{DFFPexpansive}\tabularnewline
$\FFPmixed(\FFf{F})$
&$\FFf{F}$ is of mixed type
&\pageref{DFFPmixed}\tabularnewline
$\FFPcontractsdenominators(\FFf{F})$
&$\FFf{F}$ contracts denominators
&\pageref{DFFPcontractsdenominators}\tabularnewline
$\FFPexpandsdenominators(\FFf{F})$
&$\FFf{F}$ expands denominators
&\pageref{DFFPexpandsdenominators}\tabularnewline
$\FFPmixesdenominators(\FFf{F})$
&$\FFf{F}$ mixes denominators
&\pageref{DFFPmixesdenominators}\tabularnewline
$\FRFlinA{\FFf{F}}{\Sf{D}}$
&constant coefficient of $p^{\abs{\Sf{D}}}\FRFF{\FRFf{R}}{\Sf{D}}(x)\in\Z_p[x]$
&\pageref{DFRFlinA}\tabularnewline
$\FRFlinB{\FFf{F}}{\Sf{D}}$
&linear coefficient of $p^{\abs{\Sf{D}}}\FRFF{\FRFf{R}}{\Sf{D}}(x)\in\Z_p[x]$
&\pageref{DFRFlinB}\tabularnewline
$\pi_{\FFf{F},\sigma,\FFf{G}}(n)$
&variant of $\pi_{\FFf{F},\FFf{G}}(n)$ involving swapping of digits as specified by $\sigma$
&\pageref{Dpisigma}\tabularnewline
$\FFFSoPP(A)$
&union of sets of periodic points of $\FFf{F}$ for all $\FFf{F}\in A$
&\pageref{DFFFSoPUAPP}\tabularnewline
$\FFFSoUPP(A)$
&union of sets of ultimately periodic points of $\FFf{F}$ for all $\FFf{F}\in A$
&\pageref{DFFFSoPUAPP}\tabularnewline
$\FFFSoAPP(A)$
&union of sets of aperiodic points of $\FFf{F}$ for all $\FFf{F}\in A$
&\pageref{DFFFSoPUAPP}\tabularnewline
$f_k$
&$f$ modulo $k$
&\pageref{Dpermpolyk}\tabularnewline
$\psi_{\FFf{F},k}$
&$\psi_\FFf{F}$ modulo $k$
&\pageref{Dpsik}
\end{longtable}
\egroup

\end{document}